%% file: main.tex
\documentclass[10pt]{amsproc}
\usepackage[utf8]{inputenc}
\input{header}

\title[$\THH$, $\BP{n}$, and a topological Sen operator]{Topological Hochschild homology, truncated Brown-Peterson spectra, and a topological Sen operator}
\author{S. K. Devalapurkar}
\address{1 Oxford St, Cambridge, MA 02139}
\email{sdevalapurkar@math.harvard.edu, \today}
\thanks{Part of this work was done when the author was supported by the PD Soros Fellowship and NSF DGE-2140743. The present article is far from being a final version, so any comments and suggestions for improvement are greatly appreciated. I'll post major updates to the arXiv, but I'll upload minor edits to my website; so please see there for the most up-to-date version.}

\begin{document}

\maketitle


\begin{abstract}
    In this article, we study the topological Hochschild homology of $\E{3}$-forms of truncated Brown-Peterson spectra, taken relative to certain Thom spectra $X(p^n)$ (introduced by Ravenel and used by Devinatz-Hopkins-Smith in the proof of the nilpotence theorem). We prove analogues of B\"okstedt's calculations $\THH(\FF_p) \simeq \FF_p[\Omega S^3]$ and $\THH(\Z_p) \simeq \Z_p[\Omega S^3\pdb{3}]$. We also construct a topological analogue of the Sen operator of Bhatt-Lurie-Drinfeld, and study a higher chromatic extension. The behavior of these ``topological Sen operators'' is dictated by differentials in the Serre spectral sequence for Cohen-Moore-Neisendorfer fibrations.
\end{abstract}


\tableofcontents
\newpage

\section{Introduction}
\input{introduction/intro}
\newpage


\section{Calculation of $\THH$}
\subsection{Review of $X(p^n)$}
\input{bokstedt/xpn-review}
\subsection{Computation of $\THH$ relative to $X(p^n)$}
\input{bokstedt/thh-calculation}
\subsection{Variant: $\THH$ over a deeper base}\label{defining-jp}
\input{bokstedt/over-laurent}
\newpage

\section{The topological Sen operator}\label{applications-thh}
\input{sen/topological-sen}
\subsection{Some calculations of $\THH$ relative to $X(p)$ and $\Theta$}
\input{sen/some-calculations}
\subsection{Relation to the $\ptl$-de Rham complex}
\input{sen/relation-to-p-dr}
\subsection{Aside: the Segal conjecture}\label{aside-segal}
\input{sen/segal}
\subsection{Aside: Cartier isomorphism}
\input{sen/cartier}
\newpage

\section{Relationship to the moduli stack of formal groups}\label{mfg-sen}
\subsection{Incarnation of the topological Sen operator over $\Mfg$}
\input{mfg/algebraic-sen}
\subsection{Comparing $\THH$ relative to $T(n)$ and $T(n+1)$}
\input{mfg/pth-powers}
\newpage

\appendix
\counterwithin{lemma}{section}
\section{Analogues for $\ko$ and $\tmf$}\label{analogues-ko-tmf}
\input{appendix/ko-tmf-analogues}
\newpage
\section{Alternative calculation of $\diffr_{\Z/p^n}$}\label{alt-pf-zpn}
\input{appendix/diffracted-hodge-of-Zpn}
\newpage
\section{Cartier duals of $W[F^n]$ and $W^\times[F^n]$}\label{cartier-duals}
\input{appendix/cartier-duals}
\newpage

\bibliographystyle{alpha}
\bibliography{main}
\end{document}

%% file: header.tex
\usepackage[dvipsnames]{xcolor}

\usepackage{etoolbox}

\makeatletter
\patchcmd{\@thm}{\let\thm@indent\indent}{\let\thm@indent\noindent}{}{}
\patchcmd{\@thm}{\thm@headfont{\scshape}}{\thm@headfont{\bfseries}}{}{}

\usepackage[T1]{fontenc}
\usepackage{enumitem}
\usepackage{caption}

\usepackage{mathtools}

\let\U\undefined

\usepackage{multirow}
\usepackage{longtable}

\usepackage{amsmath}
\usepackage{amsfonts}
\usepackage{amssymb}
\usepackage{float}
\usepackage{amsthm}
\usepackage{comment}
\usepackage{amscd}
\usepackage{amsxtra}
\usepackage{epsfig}
\usepackage{epigraph}
\usepackage{pgfplots}

\usepackage{stmaryrd}

\usepackage{listings}
\usepackage{color}

\definecolor{dkgreen}{rgb}{0,0.6,0}
\definecolor{gray}{rgb}{0.5,0.5,0.5}
\definecolor{mauve}{rgb}{0.58,0,0.82}

\usepackage{color}
\usepackage[all]{xy}
\usepackage{verbatim}

\usepackage{eucal}
\usepackage{mathrsfs}

\usepackage{graphicx}
\usepackage{tikz-cd}
\usepackage{quiver}

\usepackage{url}
\usepackage{verbatim}

\usepackage{xy}
\xyoption{all}
\newcommand{\xar}[1]{\xrightarrow{{#1}}}

\usepackage{amsthm}

\usepackage{caption} 
\makeatletter
\def\@tocline#1#2#3#4#5#6#7{\relax
  \ifnum #1>\c@tocdepth 
  \else
    \par \addpenalty\@secpenalty\addvspace{#2}%
    \begingroup \hyphenpenalty\@M
    \@ifempty{#4}{%
      \@tempdima\csname r@tocindent\number#1\endcsname\relax
    }{%
      \@tempdima#4\relax
    }%
    \parindent\z@ \leftskip#3\relax \advance\leftskip\@tempdima\relax
    \rightskip\@pnumwidth plus4em \parfillskip-\@pnumwidth
    #5\leavevmode\hskip-\@tempdima
      \ifcase #1
       \or\or \hskip 1em \or \hskip 2em \else \hskip 3em \fi%
      #6\nobreak\relax
    \hfill\hbox to\@pnumwidth{\@tocpagenum{#7}}\par
    \nobreak
    \endgroup
  \fi}
\makeatother

\usepackage{stackengine}
\stackMath

\usepackage[noabbrev,capitalize]{cleveref} 

\crefformat{nul}{(#2#1#3)}
\Crefformat{nul}{(#2#1#3)}

\crefname{section}{\S}{\S\S}
\crefname{subsection}{\S}{\S\S}
\crefname{axioms}{Axiom}{Axioms}
\crefname{exercise}{Exercise}{Exercises}
\crefname{exercisenum}{Exercise}{Exercises}
\crefname{construction}{Construction}{Constructions}
\crefname{problem}{Problem}{Problems}
\crefname{theorem}{Theorem}{Theorems}
\crefname{definition}{Definition}{Definitions}
\crefname{prop}{Proposition}{Propositions}
\crefname{lemma}{Lemma}{Lemmas}
\crefname{example}{Example}{Examples}
\crefname{examplealph}{Example}{Examples}
\crefname{corollary}{Corollary}{Corollaries}
\crefname{nonexample}{Nonexample}{Nonexamples}
\crefname{equation}{}{}
\crefname{summary}{Summary}{Summaries}
\crefname{recollection}{Recollection}{Recollections}
\Crefname{recollection}{Recollection}{Recollections}
\Crefname{nonexample}{Nonexample}{Nonexamples}
\Crefname{corollary}{Corollary}{Corollaries}
\Crefname{corollary}{Corollary}{Corollaries}
\Crefname{axioms}{Axiom}{Axioms}
\Crefname{exercise}{Exercise}{Exercises}
\Crefname{exercisenum}{Exercise}{Exercises}
\Crefname{construction}{Construction}{Constructions}
\Crefname{problem}{Problem}{Problems}
\Crefname{theorem}{Theorem}{Theorems}
\Crefname{definition}{Definition}{Definitions}
\Crefname{prop}{Proposition}{Propositions}
\Crefname{lemma}{Lemma}{Lemmas}
\Crefname{example}{Example}{Examples}
\Crefname{examplealph}{Example}{Examples}
\Crefname{section}{\S}{\S\S}
\Crefname{subsection}{\S}{\S\S}

\DeclareMathOperator{\Fun}{Fun}
\DeclareMathOperator{\Hom}{Hom}

\DeclareMathOperator{\Ext}{Ext}

\DeclareMathOperator{\spec}{Spec}
\DeclareMathOperator{\spf}{Spf}

\DeclareMathOperator{\cofib}{cofib}
\DeclareMathOperator{\fib}{fib}

\newcommand{\Fr}{\mathrm{Frob}}
\DeclareMathOperator{\QCoh}{\mathrm{QCoh}}

\DeclareMathOperator{\colim}{colim}
\DeclareMathOperator{\hocolim}{hocolim}

\DeclareMathOperator{\Tot}{Tot}

\DeclareMathOperator{\Sym}{Sym}

\newtheorem{lemma}{Lemma}[subsection]

\newtheorem{corollary}[lemma]{Corollary}

\newtheorem{theorem}[lemma]{Theorem}

\newtheorem*{corno}{Corollary}
\newtheorem*{thmno}{Theorem}
\newtheorem*{propo}{Proposition}
\newtheorem{prop}[lemma]{Proposition}
\newtheorem{conjecture}[lemma]{Conjecture}
\newtheorem*{conjectureno}{Conjecture}

\newtheorem*{conj-moore}{Conjecture~\ref{moore-splitting}}
\newtheorem*{conj-cent}{Conjecture~\ref{centrality-conj}}
\newtheorem*{conj-tmf}{Conjecture~\ref{tmf-conj}}
\newtheorem*{thm-main}{Theorem~\ref{main-thm}}
\newtheorem*{cor-bpn}{Corollary~\ref{bpn}}
\newtheorem*{thm-string}{Theorem~\ref{mstring}}

\theoremstyle{definition}

\newtheorem{definition}[lemma]{Definition}
\newtheorem{warning}[lemma]{Warning}
\newtheorem{example}[lemma]{Example}

\newtheorem{construction}[lemma]{Construction}
\newtheorem{remark}[lemma]{Remark}
\newtheorem{notation}[lemma]{Notation}
\newtheorem{variant}[lemma]{Variant}

\newtheorem{recall}[lemma]{Recollection}

\newtheorem{slogan}[lemma]{Slogan}
\newtheorem{observe}[lemma]{Observation}

\renewcommand{\AA}{\mathbf{A}}
\newcommand{\FF}{\mathbf{F}}
\newcommand{\Z}{\mathbf{Z}}
\newcommand{\QQ}{\mathbf{Q}}
\newcommand{\cc}{\mathbf{C}}

\newcommand{\cC}{\mathscr{C}}
\newcommand{\cd}{\mathscr{D}}
\newcommand{\RR}{\mathbf{R}}
\newcommand{\Deltab}{\mathbf{\Delta}}

\newcommand{\Lk}{L_{K(n)}}

\newcommand{\Sp}{\mathrm{Sp}}

\newcommand{\Mfg}{\M_\mathrm{FG}}

\newcommand{\Mell}{\M_\mathrm{ell}}

\newcommand{\co}{\mathscr{O}}
\newcommand{\ce}{\mathscr{E}}

\newcommand{\GG}{\mathbf{G}}

\newcommand{\LMod}{\mathrm{LMod}}

\newcommand{\Mod}{\mathrm{Mod}}

\newcommand{\cL}{\mathscr{L}}

\newcommand{\CP}{{\mathbf{C}P}}
\newcommand{\HHP}{\mathbf{H}P}
\newcommand{\RP}{\mathbf{R}P}
\newcommand{\OP}{\mathbf{O}P}

\newcommand{\Map}{\mathrm{Map}}

\newcommand{\Lone}{L_{K(1)}}

\newcommand{\cg}{\mathscr{G}}

\newcommand{\cA}{\mathscr{A}}

\newcommand{\Sq}{\mathrm{Sq}}

\newcommand{\Eoo}{{\mathbf{E}_\infty}}

\newcommand{\CAlg}{\mathrm{CAlg}}

\newcommand{\cf}{\mathscr{F}}

\newcommand{\ev}{\mathrm{ev}}

\newcommand{\cN}{\mathscr{N}}

\newcommand{\ul}[1]{\underline{#1}}
\newcommand{\ol}[1]{\overline{#1}}

\newcommand{\E}[1]{{\mathbf{E}_{{#1}}}}
\newcommand{\mmod}{/\!\!/}

\newcommand{\id}{\mathrm{id}}

\newcommand{\GL}{\mathrm{GL}}
\newcommand{\BGL}{\mathrm{BGL}}

\newcommand{\MO}{\mathrm{MO}}
\newcommand{\MU}{\mathrm{MU}}
\newcommand{\BU}{\mathrm{BU}}
\renewcommand{\O}{\mathrm{O}}
\newcommand{\SO}{\mathrm{SO}}
\newcommand{\BO}{\mathrm{BO}}

\newcommand{\ku}{\mathrm{ku}}
\newcommand{\ko}{\mathrm{ko}}

\newcommand{\bo}{\mathrm{bo}}

\newcommand{\tmf}{\mathrm{tmf}}

\newcommand{\BP}[1]{\mathrm{BP}\langle{#1}\rangle}
\newcommand{\U}{\mathrm{U}}
\newcommand{\SU}{\mathrm{SU}}
\newcommand{\BPP}{\mathrm{BP}}

\renewcommand{\H}{\mathrm{H}}

\newcommand{\BSU}{\mathrm{BSU}}
\newcommand{\BSpin}{\mathrm{BSpin}}

\newcommand{\MSpin}{\mathrm{MSpin}}
\newcommand{\MString}{\mathrm{MString}}
\newcommand{\BString}{\mathrm{BString}}

\newcommand{\Spin}{\mathrm{Spin}}

\newcommand{\F}{\mathrm{F}}
\newcommand{\B}{{B}}

\newcommand{\W}{{W}}
\newcommand{\HH}{\mathrm{HH}}

\newcommand{\HP}{\mathrm{HP}}
\newcommand{\THH}{\mathrm{THH}}
\newcommand{\TC}{\mathrm{TC}}
\newcommand{\TP}{\mathrm{TP}}

\newcommand{\op}{\mathrm{op}}

\newcommand{\crys}{\mathrm{crys}}

\newcommand{\dR}{\mathrm{dR}}

\newcommand{\fr}[1]{\mathfrak{#1}}

\newcommand{\pr}{\mathrm{pr}}

\newcommand{\gr}{\mathrm{gr}}

\newcommand{\fil}{\mathrm{fil}}

\newcommand{\cQ}{\mathscr{Q}}

\newcommand{\bH}{\mathbf{H}}

\newcommand{\cM}{\mathcal{M}}

\newcommand{\ppar}{{(p)}}

\newcommand{\conj}{\mathrm{conj}}

\renewcommand{\log}{\mathrm{log}}

\newcommand{\M}{\cM}

\newcommand{\cZ}{\mathscr{Z}}

\renewcommand{\S}{Section }

\newcommand{\WW}{\mathbf{W}}

\newcommand{\NN}{\mathbf{N}}

\makeatletter
\providecommand{\leftsquigarrow}{%
  \mathrel{\mathpalette\reflect@squig\relax}%
}
\newcommand{\reflect@squig}[2]{%
  \reflectbox{$\m@th#1\rightsquigarrow$}%
}
\makeatother

\newcommand{\bull}{\bullet}

\newcommand{\bD}{\mathbf{D}}

\newcommand{\pw}[1]{[\![#1]\!]}
\newcommand{\ls}[1]{(\!(#1)\!)}

\newcommand{\wt}{\mathrm{wt}}

\newcommand{\Cp}{{\Z/p}}
\newcommand{\Cpn}{{\Z/m}}
\newcommand{\Cn}{{\Z/n}}

\newcommand{\Efr}[1]{{\mathbf{E}_{{#1}}^\mathrm{fr}}}

\newcommand{\sh}{\mathrm{sh}}

\renewcommand{\tilde}{\widetilde}

\newcommand{\Lag}{\mathrm{Lag}}

\newcommand{\Gr}{\mathrm{Gr}}

\newcommand{\BSp}{\mathrm{BSp}}

\newcommand{\Syn}{\mathrm{Syn}}

\newcommand{\mot}{\mathrm{mot}}

\usepackage{slashed}
\newcommand{\diffr}{\widehat{\Omega}^{\slashed{D}}}

\usepackage{relsize}
\usepackage[bbgreekl]{mathbbol}
\usepackage{amsfonts}
\DeclareSymbolFontAlphabet{\mathbb}{AMSb} 
\DeclareSymbolFontAlphabet{\mathbbl}{bbold}
\newcommand{\prism}{{\mathbbl{\Delta}}}
\newcommand{\prismht}{\overline{\prism}}

\newcommand{\MSp}{\mathrm{MSp}}

\newcommand{\ptl}{{\tilde{p}}}

\newcommand{\pdb}[1]{\langle{#1}\rangle}
\newcommand{\Ainf}{A_\mathrm{inf}}

\newcommand{\WCart}{\mathrm{WCart}}
\newcommand{\HT}{\mathrm{HT}}

%% file: introduction/intro.tex
\subsection{Summary}

Fix a prime $p$. A fundamental calculation of B\"okstedt's \cite{bokstedt-thh} says that $\pi_\ast \THH(\FF_p)$ is isomorphic to a polynomial ring $\FF_p[\sigma]$ with $|\sigma|=2$. Recent work of Hahn-Wilson shows that this polynomiality phenomenon persists at higher heights, provided one works relative to $\MU$ instead of the sphere. Namely, \cite[Theorem E]{hahn-wilson-bpn} states that if $\BP{n}$ is an $\E{3}$-form of the truncated Brown-Peterson spectrum, then $\pi_\ast \THH(\BP{n}/\MU)$ is a polynomial algebra over $\pi_\ast \BP{n}$ on generators in even degree. Moreover, the first such generator is the double suspension $\sigma^2(v_{n+1})$. 

In this article, we will show that the ``polynomial $\THH$'' phenomenon persists if one instead considers $\THH$ relative to the Ravenel spectra $X(p^n)$, introduced in \cite{ravenel-loc} and used by \cite{dhs-i} in the proof of the nilpotence theorem. Motivated by \cite{bpn-thom}, the thesis of this article is that many statements involving the study of $\FF_p$- or $\Z_p$-algebras relative to the sphere spectrum admit natural generalizations when studying $\BP{n-1}$- or $\BP{n}$-algebras relative to $X(p^n)$. Many of the results presented here were motivated by the perspective that there should be a chromatic analogue of integral $p$-adic Hodge theory (where $p$ is replaced by the chromatic element $v_n$; see \cref{bpn-vs-tj})\footnote{I'd also like to direct the reader to \url{https://www.royalacademy.org.uk/art-artists/work-of-art/prismatic-colour-wheel}; but I hope our \cref{bpn-vs-tj} is more mathematically informative!}.

The $\Efr{2}$-ring $X(p^n)$ is the Thom spectrum of the $\Efr{2}$-map $\Omega \SU(p^n) \to \Omega \SU \simeq \BU$, so that $X(1) = S^0$ and $X(\infty) = \MU$. Just as $\MU_\ppar$ splits as a direct sum of shifts of $\BPP$, the spectrum $X(p^n)_\ppar$ splits into a direct sum of shifts of an $\E{1}$-ring denoted\footnote{This is \textit{not} the telescope of a $v_n$-self map! See \cref{clash-of-notation}.} $T(n)$. If $\cC$ is a left $X(p^n)$-linear $\infty$-category, then \cite[Corollary 2.9 and Corollary 3.7]{framed-e2} ensures that it makes sense to define the relative topological Hochschild homology $\THH(\cC/X(p^n))$, and furthermore that $\THH(\cC/X(p^n))$ admits an $S^1$-action.\footnote{We warn the reader that even if $\cC$ admits the structure of a monoidal $\infty$-category, $\THH(\cC/X(p^n))$ rarely inherits any multiplicative structure from $\cC$, since $X(p^n)$ does not admit the structure of an $\E{3}$-ring (see \cref{Xn-e3}).}

Our main result is an analogue of B\"okstedt's calculation. If $R$ is a ring spectrum, let $R[\B\Delta_n]$ denote the free $R$-module whose homotopy groups are isomorphic to a divided power algebra $\pi_\ast(R)\pdb{y_{i} | 1\leq i \leq p^n-1, i\neq p^k}$ where $|y_j| = 2j$.\footnote{The contribution $\B\Delta_n$ plays essentially no practical/meaningful role in this article. Its appearance in the equivalences below can be removed if $T(n)\subseteq X(p^n)_\ppar$ admits the structure of an $\Efr{2}$-algebra. We strongly believe this to be possible (enough to state it as \cref{tn-e2}!), so we suggest the reader ignore $\B\Delta_n$ --- and simultaneously replace $X(p^n)$ by $T(n)$ --- on a first pass.}
Morally, $R[\B\Delta_n]$ is the $R$-chains on the ``classifying space of $\prod_{i=1}^n \SU(p^i-1)/\SU(p^{i-1})$''; so, if $X$ is another space, we will write $R[\B\Delta_n\times X]$ to denote $R[\B\Delta_n] \otimes_R R[X]$.
Fix an $\E{3}$-form of the truncated Brown-Peterson spectrum $\BP{n-1}$ (which exists by \cite[Theorem A]{hahn-wilson-bpn}). Motivated by the results of \cite{bpn-thom}, and using the calculations of \cite{angeltveit-rognes, homological-hfpss}, we show:

\begin{thmno}[\cref{thh-calculations}(a)]
There is a $p$-complete equivalence 
$$\THH(\BP{n-1}/X(p^n)) \simeq \BP{n-1}[\B\Delta_n \times \Omega S^{2p^n+1}]$$
of $\BP{n-1}$-modules; in particular, there is a $p$-complete isomorphism
$$\pi_\ast \THH(\BP{n-1}/X(p^n)) \simeq \pi_\ast \BP{n-1}[\B\Delta_n][\theta_n],$$
where $\theta_n\in \pi_{2p^n} \THH(\BP{n-1}/X(p^n))$ is $\sigma^2(v_n)$.

Moreover, there are $p$-complete isomorphisms
\begin{align*}
    \pi_\ast \TC^-(\BP{n-1}/X(p^n)) & \cong \pi_\ast (\BP{n}[\B\Delta_n])\pw{\hbar}[\theta_n]/(\theta_n \hbar - v_n), \\
    \pi_\ast \TP(\BP{n-1}/X(p^n)) & \cong \pi_\ast (\BP{n}^{tS^1}[\B\Delta_n]),
\end{align*}
where $\hbar\in \pi_{-2} \BP{n}^{hS^1}$. Under the map $\TP(\BP{n-1}/X(p^n)) \to \TP(\BP{n-1}/\MU)$, the image of $v_n\in \pi_{2p^n-2} \TP(\BP{n-1}/X(p^n))$ can be identified with the image of $v_n\in \pi_{2p^n-2} \MU^{tS^1}$ under the map $\MU^{tS^1} \to \TP(\BP{n-1}/\MU)$.
\end{thmno}
\begin{remark}
If $T(n)\subseteq X(p^n)_\ppar$ admits the structure of an $\Efr{2}$-algebra (\cref{tn-e2}), then \cref{thh-calculations} would give the cleaner statements that $\THH(\BP{n-1}/T(n)) \simeq \BP{n-1}[\Omega S^{2p^n+1}]$, and that $\pi_\ast \TP(\BP{n-1}/T(n)) \cong \pi_\ast \BP{n}^{tS^1}$.
The map $\pi_\ast \THH(\BP{n-1}/T(n)) \to \pi_\ast \THH(\BP{n-1}/\MU)$ is injective, and exhibits the source as the submodule $\pi_\ast \BP{n-1}[\sigma^2(v_n)]$ of $\pi_\ast \THH(\BP{n-1}/\MU)$.
\end{remark}
\cref{thh-calculations} implies the following result, which, for $n=0$, is a very special case of the main result of \cite{petrov-vologodsky}:

\begin{corno}[\cref{graded-drw-lift}]
Let $R = \BP{n}[\Z_{\geq 0}^j]$ be a flat polynomial ring over $\BP{n}$, viewed as a $\Z_{\geq 0}^j$-graded $\Efr{2}$-$\BP{n}$-algebra.
Then there is a  $p$-complete isomorphism of $\Z_{\geq 0}^j$-graded modules equipped with a map from $\pi_\ast \BP{n}^{tS^1}[\B\Delta_n] \cong \pi_\ast \TP(\BP{n-1}/X(p^n))$:
$$\pi_\ast \TP^\gr((R/v_n)/X(p^n)) \cong \pi_\ast \HP^\gr(R/\BP{n})[\B\Delta_n].$$
Here, the superscript $\gr$ denotes the Tate construction taken in \emph{$\Z_{\geq 0}^j$-graded} spectra.
\end{corno}

\begin{remark}
\cref{thh-calculations} quickly implies redshift for $K(\BP{n-1})$ (see \cref{k-thy-redshift}). 
When $n=0$, the first part of \cref{thh-calculations}(a) recovers B\"okstedt's calculation of $\THH(\FF_p)$, since $\BP{-1} = \FF_p$ and $X(0) = S^0$. When $p=2$, the statement of \cref{thh-calculations} can be simplified using \cite[Remark 3.1.9]{bpn-thom}; for instance, we obtain the following \textit{additive} equivalences and isomorphisms: for $n=1$, we have
\begin{align*}
    \THH(\Z_2/T(1)) \simeq \Z_2[\sigma^2(v_1)],\ & \pi_\ast \TP(\Z_2/T(1))^\wedge_2 \simeq \pi_\ast (\ku^{tS^1})^\wedge_2.
\end{align*}
Since $\tmf_1(3)$ is a form of $\BP{2}$ by \cite{lawson-naumann}, for $n=2$, we have
\begin{align*}
    \THH(\ku^\wedge_2/T(2)) \simeq \ku^\wedge_2[\sigma^2(v_2)],\ & \pi_\ast \TP(\ku^\wedge_2/T(2))^\wedge_2 \simeq \pi_\ast (\tmf_1(3)^{tS^1})^\wedge_2.
\end{align*}
\end{remark}

We also prove an analogue of B\"okstedt's calculation \cite{bokstedt-thh} of $\THH(\Z_p)$:
\begin{thmno}[\cref{thh-calculations}(b)]
There is an equivalence of $\BP{n}$-modules
$$\THH(\BP{n}/X(p^n))^\wedge_p \cong \BP{n}[\B\Delta_n]^\wedge_p \oplus \left(\bigoplus_{j\geq 1} \Sigma^{2jp^{n+1}-1} \BP{n}[\B\Delta_n]/pj\right)^\wedge_p.$$
Moreover, $\pi_{2p^{n+1}-3} \TC^-(\BP{n}/X(p^n))^\wedge_p$ detects the class $\sigma_n\in \pi_{2p^{n+1}-3} X(p^n)$ from \cite[Lemma 3.1.12]{bpn-thom}.
\end{thmno}

\input{introduction/modified-plot}

\begin{remark}
If one replaces $X(p^n)$ in the left-hand side of \cref{thh-calculations}(b) with $X(p^{n+1}-1)$, the only change to the right-hand side is that $\B\Delta_n$ is replaced by $\B\Delta_{n+1}$.
Let us mention the following mild variant of \cref{thh-calculations} (see \cref{modulo-moore}): the $\FF_p[v_{n-j}, \cdots, v_{n-1}]$-module $\pi_\ast \THH(\BP{n-1}/X(p^j))/(p,\cdots,v_{n-1-j})$ is isomorphic to the tensor product of $\BP{n-1}[\Omega S^{2p^n+1} \times \B\Delta_j]/(p,\cdots,v_{n-1-j})_\ast$ with an exterior algebra on classes $\lambda_{j+1}, \cdots, \lambda_n$, where $|\lambda_m| = 2p^m-1$.
We also prove an analogue of \cref{thh-calculations} for $\ko$ and $\tmf$ in \cref{analogues-ko-tmf}. For example, if the spectra $A$ and $B$ \cite[Section 3]{bpn-thom} lift to $\Efr{2}$-rings, there are $2$-complete equivalences
\begin{align*}
    \THH(\ko/A) & \simeq \ko \oplus \left(\bigoplus_{j\geq 1} \Sigma^{8j - 1} \ko/2j\right), \\
    \THH(\tmf/B) & \simeq \tmf \oplus \left(\bigoplus_{j\geq 1} \Sigma^{16j - 1} \tmf/2j\right).
\end{align*}
\end{remark}
\begin{remark}
If \cref{tn-e2} (or rather, a weaker version which only asks that $T(n)$ admit the structure of an $\Efr{2}$-ring) were true, then the contribution of $\B\Delta_n$ could be eliminated from \cref{thh-calculations}(b): namely, there would be a $p$-complete equivalence
$$\THH(\BP{n}/T(n)) \simeq \BP{n} \oplus \bigoplus_{j\geq 1} \Sigma^{2jp^{n+1}-1} \BP{n}/pj.$$
\end{remark}
We warn the reader that all the equivalences proved above are only additive, so one cannot directly use them to study the stacks associated to $\THH$ (defined via the even filtration of \cite{even-filtr}). 
As a perhaps more digestible example of this phenomenon (see \cref{T2 and y2 relationship}), note that since $\FF_2$ is the Thom spectrum of an $\E{1}$-map $\U(2) \to \BGL_1(\ku)$, there is an equivalence $\HH(\FF_2/\ku) \simeq \FF_2[\BU(2)]$; however, this cannot be upgraded to an equivalence of $\FF_2$-algebras, since the right-hand side is not even obviously a ring!

In \cref{frobenius-thh}, \cref{frob-in-topology}, and \cref{similarity-ehp}, we use the EHP sequence to explain the similarity in the calculation of $\THH(\BP{n}/X(p^{n+1}))$ and $\THH(\BP{n-1}/X(p^n))$ given by \cref{thh-calculations}. This discussion in fact yields the following more general structural uniformity in the truncated Brown-Peterson spectra (see \cref{bpn-vs-tj} for a visual illustration): 
\begin{slogan}[\cref{mod-vj-vn1} and \cref{mod-jvn-thh} for precise statements]\label{uniformity-slogan}
If $n\geq j-1$, the structure of $\BP{n}$ as an $\E{1}$-$X(p^j)$-algebra (i.e., $\THH(\BP{n}/X(p^j))$) mirrors the structure of $\BP{n-1}$ as an $\E{1}$-$X(p^{j-1})$-algebra (i.e., $\THH(\BP{n-1}/X(p^{j-1}))$), which in turn mirrors the structure of $\BP{n-j}$ as an $\E{1}$-algebra over the sphere (i.e., $\THH(\BP{n-j})$).
\end{slogan}

Let $\cC$ be a left $X(p^n)$-linear $\infty$-category. Then, the descent spectral sequence from $\THH$ relative to $X(p^n)$ to $\THH$ relative to $X(p^n-1)$ runs:
\begin{thmno}[\cref{topological-sen}]
There is a map $\Theta_\cC: \Sigma^{-2p^n} \THH(\cC/X(p^n)) \to \THH(\cC/X(p^n))$ such that there is a cofiber sequence
\begin{equation}\label{intro-sen}
    \THH(\cC/X(p^n-1)) \xar{\iota} \THH(\cC/X(p^n)) \xar{\Theta_\cC} \Sigma^{2p^n} \THH(\cC/X(p^n)),
\end{equation}
where the map $\iota$ is $S^1$-equivariant, and the cofiber of $\iota$ is (at least nonequivariantly) identified with $\Sigma^{2p^n} \THH(\cC/X(p^n))$.
\end{thmno}
\begin{remark}
Motivated by \cite{apc, drinfeld-prism}, we dub the map $\Theta_\cC$ the \textit{topological Sen operator}; its construction is motivated by the work of \cite{bpn-thom} relating $\BP{n}$ to Cohen-Moore-Neisendorfer type fiber sequences \cref{b-fib-phin}. When $\cC = \LMod_{\BP{n-1}}$, \cref{thh-calculations} implies that the map $\Theta$
sends 
$$\Theta: \theta_n^j \mapsto jp\theta_n^{j-1}.$$
When $n=1$, it therefore behaves like the Sen operator on the diffracted Hodge complex of $\Z_p$ which computes $\prismht_{\Z_p}\{\ast\}$.
\end{remark}
\begin{remark}
In \cref{analogues-ko-tmf} (see \cref{quaternionic-sen}), we describe a quaternionic analogue of \cref{intro-sen}, obtained by replacing $X(n)$ by the Thom spectrum $X_\bH(n)$ of the tautological symplectic bundle over $\Omega(\SU(2n)/\Sp(n))$ obtained via the map $\Omega(\SU(2n)/\Sp(n)) \to \Omega(\SU/\Sp) \simeq \BSp$ given by Bott periodicity.
\end{remark}

In \cref{xp-polynomial}, we define an $\Efr{2}$-ring $J(p)$ which admits an $\Efr{2}$-map $J(p) \to X(p)$ such that $\THH(T(1)/J(p)) \simeq T(1)[J_{p-1}(S^2)]$. The underlying $\E{1}$-ring of $J(p)$ is $S[\Z] = S[t^{\pm 1}]$ with $|t|=0$, but they differ as $\Efr{2}$-rings. The \textit{raison d'\^etre} for $J(p)$ is that $\THH(\Z_p/J(p))$ is polynomial on a class $x$ in degree $2$ which is a $p$th root of $\theta\in \pi_{2p} \THH(\Z_p/X(p))$. More precisely, there is an equivalence $\THH(\Z_p/J(p)) \simeq \Z_p[\Omega S^3]$ such that the map $\THH(\Z_p/J(p)) \to \THH(\Z_p/X(p))$ is induced by $\Z_p$-chains of the Hopf map $\Omega S^3 \to \Omega S^{2p+1}$. In \cref{higher-jp}, we also construct two $\Efr{2}$-rings (as Thom spectra over $\Omega\U(2)$ and $\Omega\Spin(4)$) which play the role of $J(p)$ for $\ku$ when $p=2$.

We construct the following cofiber sequence analogous to \cref{intro-sen} for any $J(p)$-linear $\infty$-category $\cC$:
$$\THH(\cC) \xar{\iota} \THH(\cC/J(p)) \xar{\Theta_\cC'} \Sigma^{2} \THH(\cC/J(p)).$$
It turns out that upon reducing the above cofiber sequence mod $p$, one obtains the following important example:
\begin{example}
If $\cC$ is a $\Z_p$-linear $\infty$-category, there is a cofiber sequence (see \cref{topological-sen-jpn})
\begin{equation}\label{intro-mod-p-sen}
    \THH(\cC) \otimes_{\Z_p} \FF_p \xar{\iota} \THH(\cC\otimes_{\Z_p} \FF_p) \xar{\Theta'} \Sigma^{2} \THH(\cC\otimes_{\Z_p} \FF_p).
\end{equation}
When $\cC = \Mod_{\Z_p}$, the effect of the map $\Theta'$ on homotopy is given by the map $\FF_p[\sigma] \to \Sigma^2 \FF_p[\sigma]$ which sends $\sigma^j\mapsto j\sigma^{j-1}$.
There is also a cofiber sequence
\begin{equation}\label{hp-cofib}
    \THH(\cC)^{t\Cp} \otimes_{\Z_p} \FF_p \xar{\iota} \HP(\cC\otimes_{\Z_p} \FF_p/\FF_p) \xar{\Theta'} \HP(\cC\otimes_{\Z_p} \FF_p/\FF_p).
\end{equation}
If $\cC = \Mod_R$ for an animated $\Z_p$-algebra $R$, we expect the maps in \cref{intro-mod-p-sen} and \cref{hp-cofib} to respect the motivic filtrations. Taking $\gr^i_\mot[-2i]$ would then produce the following cofiber sequences involving the associated graded pieces of the Nygaard filtration on the prismatic cohomologies of $R$ and $R/p$:
\begin{align*}
    (\cN^i \hat{\prism}_R)/p & \to \F^\conj_i \dR_{(R/p)/\FF_p} \to \F^\conj_{i-1} \dR_{(R/p)/\FF_p}, \\
    \prismht_R/p & \to \dR_{(R/p)/\FF_p} \to \dR_{(R/p)/\FF_p}.
\end{align*}
Such cofiber sequences on Hodge-Tate cohomology do indeed exist, and can be constructed purely algebraically using the methods of \cite{bhatt-mathew-syntomic} and \cite[Proposition 6.4.8]{apc}; see \cref{bhatt-mathew-cofiber} and \cref{dr-modp}.
\end{example}

We also show by explicit calculation:
\begin{propo}[\cref{polynomial-t1} and \cref{a1-inv} for precise statements]
There is an isomorphism $\pi_\ast \TP(\Z_p[t]/X(p))\cong \pi_\ast \HP(\BP{1}[t]/\BP{1})$.

Furthermore, the map $\TP^\gr(\BP{n-1}[t]/X(p^n)) \to \TP(\BP{n-1}/X(p^n))$ is an equivalence after $K(n)$-localization, and \cref{conjecture-thh} implies that (up to a Nygaard-type completion) $\Lk \TP(-/X(p^n))$ is $\AA^1$-invariant.
\end{propo}
We also have:
\begin{conjectureno}[\cref{rel-jp-diffracted} and \cref{p-tilde-t1}]
Let $R$ be an animated $\Z_p$-algebra, and let $\F^\conj_\star \diffr_R$ denote the conjugate-filtered ($p$-completed) diffracted Hodge complex of \cite[Construction 4.7.1]{apc}. Then $\THH(R/J(p))$ admits a motivic filtration such that $\gr^i_\mot \THH(R/J(p)) \simeq (\F^\conj_i \diffr_R)[2i]$, and such that the map $\Theta'_R: \THH(R/J(p)) \to \Sigma^2 \THH(R/J(p))$ respects the motivic filtration and induces the map $\Theta + i: \F^\conj_i \diffr_R \to \F^\conj_{i-1} \diffr_R$ on $\gr^i_\mot$.

Similarly, $\THH(R/X(p))$ admits a motivic filtration such that $\gr^i_\mot \THH(R/X(p)) \simeq (\F^\conj_{pi} \diffr_R)[2pi] \otimes_R R[\BSU(p-1)]$.
Moreover, $\TP(R/X(p))$ admits a motivic filtration $\F^\star_\mot \TP(R/X(p))$ such that $\gr^i_\mot \TP(R/X(p)) \simeq \hat{\prism}_{R/\Z_p\pw{\ptl}}[2i] \otimes_R \epsilon^R$, where $\hat{\prism}_{R/\Z_p\pw{\ptl}}$ is the Nygaard completion of $\ptl \Omega_R$.
\end{conjectureno}


In \cref{applications-thh}, we supplement \cref{rel-jp-diffracted} with some examples (such as $R$ being a $p$-complete perfectoid ring, $R = \Z/p^n$ for odd $p$, $R$ being a complete DVR of mixed characteristic $(0,p)$, and $R = \Z_p[t]$).
\begin{remark}
For the case $R = \Z/p^n$, we give ``two'' calculations of the diffracted Hodge complex $\diffr_{\Z/p^n}$; one uses abstract properties of the diffracted Hodge complex (and was explained to us by Bhatt), and the other (provided in \cref{alt-pf-zpn}) is via concrete calculations in the ring $W(\Z_p)$. In particular, in \cref{diffr-zpn}, we refine the calculation of \cite[Example 5.15]{prismatization} to show that there is an equivalence $\WCart^\HT_{\Z/p^n} \cong \GG_a^\sharp/\GG_m^\sharp$ of stacks over $\Z/p^n$.
\end{remark}
In \cref{aside-segal}, we also study an analogue of the Segal conjecture for $\THH$ relative to $J(p)$ and $T(n)$. One interesting consequence (\cref{reg-noeth-jp-segal}) is that if $R$ is a $p$-torsionfree discrete commutative ring such that $R/p$ is regular Noetherian and $L\Omega^n_R = 0$ for $n\gg 0$, then \cite[Remark 4.7.4]{apc} and \cref{rel-jp-diffracted} imply that $R$ satisfies a version of the Segal conjecture for $\THH$ relative to $J(p)$.

In \cref{general-cartier}, we prove an analogue of the Cartier isomorphism in Hochschild homology for a flat polynomial algebra over any $\E{2}$-ring, and show that it specializes to homotopical analogues of several known examples of the Cartier isomorphism. (This is quite likely well-known to some experts, but we could not find a source.)
\begin{propo}[\cref{general-cartier}]
Let $R$ be an $\E{2}$-ring. Then there is a $S^1$-equivariant map $\fr{C}: \HH(R^{t\Cp}[t]/R^{t\Cp}) \to \HH(R[t]/R)^{t\Cp}$ sending $t\mapsto t^p$, where $S^1$ acts on $\HH(R[t]/R)^{t\Cp}$ via the residual $S^1/\mu_p$-action, and on $\HH(R^{t\Cp}[t]/R^{t\Cp})$ via the diagonal action on Hochschild homology and residual $S^1/\mu_p$-action on $R^{t\Cp}$. If $t$ is given weight $1$, then $\fr{C}$ induces an $S^1$-equivariant equivalence $\HH(R^{t\Cp}[t]/R^{t\Cp})_{\wt \leq m} \to (\HH(R[t]/R)_{\wt\leq mp})^{t\Cp}$ of \textit{graded} $R^{t\Cp}$-modules.
\end{propo}

In \cref{mfg-sen}, we describe the topological Sen operator from the perspective of the moduli stack $\Mfg$ of formal groups.
We begin by describing an algebraic analogue of $\THH$. This is given by an Adams-Novikov analogue of the B\"okstedt spectral sequence: if $R$ is a $p$-local homotopy commutative ring such that $\MU_\ast(R)$ is concentrated in even degrees, one can define a stack $\cM_R$ whose coherent cohomology is the $E_2$-page of the Adams-Novikov spectral sequence for $R$ (see \cite[Chapter 9]{tmf}). 
\begin{propo}[\cref{anss-bokstedt}]
If $\gr^\bull_\ev \THH(R)$ denotes the associated graded of the \textit{even filtration} of \cite{even-filtr} on $\THH(R)$ (which recovers the motivic filtration on $\THH(R)$ of \cite{bms-ii} when $R$ is quasisyntomic), then there is a spectral sequence:
$$\pi_\ast \HH(\cM_R/\Mfg) \Rightarrow \pi_\ast \gr^\bull_\ev \THH(R).$$
There is also an analogue for relative $\THH$.
\end{propo}
This spectral sequence behaves essentially like the B\"okstedt spectral sequence in most examples. In particular, if $R \to R'$ is a map of $p$-local homotopy commutative rings whose $\MU$-homologies are concentrated in even degrees, then $\HH(\cM_R/\cM_{R'})$ can be viewed as the ``Adams-Novikov-B\"okstedt associated graded'' of $\gr^\bull_\ev \THH(R/R')$.
Motivated by this perspective, we describe an analogue of the topological Sen operator of \cref{topological-sen} as a Gauss-Manin connection on stacks related to $\Mfg$ (see \cref{algebraic-sen}):
\begin{thmno}[\cref{algebraic-sen} and \cref{gm-sharp}]
The stack $\cM_{T(n)}$ is isomorphic to the moduli stack of graded $p$-typical formal groups equipped with a $p$-typical coordinate of order $\leq p^n$. Moreover, the Adams-Novikov analogue of \cref{topological-sen} is a fiber sequence
$$\HH(X/\cM_{T(n-1)}) \to \HH(X/\cM_{T(n)}) \xar{\Theta_\mot} \Sigma^{2p^n,p^n} \HH(X/\cM_{T(n)})$$
associated to any stack $X \to \cM_{T(n)}$, where $\Sigma^{n,w}$ denotes a shift by homological degree $n$ and weight $w$.

Similarly, there is a fiber sequence
$$\HH(X/\Mfg) \to \HH(X/\cM_{J(p)}) \xar{\Theta_\mot} \Sigma^{2,1} \HH(X/\cM_{J(p)})$$
associated to any stack $X \to \cM_{J(p)}$.
\end{thmno}
\begin{remark}
In \cref{analogues-ko-tmf} (\cref{stack-XHn} and \cref{quaternionic-sen}), we also study a quaternionic analogue of the above fiber sequence. This description crucially relies on the twistor fibration $\CP^{2n-1} \to \HHP^{n-1}$, which is given in coordinates by the map $[z_1: \cdots: z_{2n}] \mapsto [z_1 + z_2 \mathbf{j} : \cdots: z_{2n-1} + z_{2n} \mathbf{j}]$.
\end{remark}

\subsection{Some complements}

In \cref{conjecture-thh}, we suggest that the identification of $\pi_\ast \TP(\BP{n-1}/X(p^n))$ can be extended to an equivalence $\TP(\BP{n-1}/X(p^n)) \simeq \BP{n}^{tS^1}[\B\Delta_n]$ of spectra:
\begin{conjectureno}[\cref{conjecture-thh}]
The spectrum $\THH(\BP{n-1}/X(p^n))$ admits the structure of an $S^1$-equivariant $\BP{n}$-module, and 
the isomorphism $\pi_\ast \TP(\BP{n-1}/T(n)) \cong \pi_\ast \BP{n}^{tS^1}$ lifts to an equivalence of spectra $\TP(\BP{n-1}/T(n)) \simeq \BP{n}^{tS^1}$. 
\end{conjectureno}
This discussion suggests viewing the pair $(\pi_0 \BP{n}^{tS^1}, (\frac{[p](\hbar)}{\hbar}))$ as a higher chromatic analogue of the crystalline prism $(\Z_p, (p))$, where $\hbar$ is the complex orientation of $\BP{n}$ and $[p](\hbar)$ is its $p$-series. Note that the pair $(\pi_0 \BP{n}^{tS^1}, (\frac{[p](\hbar)}{\hbar}))$ has no reason to naturally admit the structure of a prism.

Finally, it would be interesting to know whether \cref{uniformity-slogan} can be used to prove \cite[Conjecture 6.1]{lee-thh}. A first step in this direction would be to show that the topological Sen operators on $\THH(\BP{n}/X(p^j))$, $\THH(\BP{n-1}/X(p^{j-1}))$, ..., and $\THH(\BP{n-j})$ can also be matched up under the structural uniformity of \cref{uniformity-slogan}. (Also see \cref{mod-p-vn}.)

This article suggests several directions in which the work presented here can be extended; we have recorded these as \cref{tn-e2}, \cref{conjecture-thh}, \cref{jpn-conjecture}, \cref{rel-jp-diffracted}, the closely related \cref{p-tilde-t1}, \cref{c2-equiv-analogues}, and \cref{a-b-e2}. We wish to emphasize that, unlike \cite[Theorem A and Corollary B]{bpn-thom}, the main results of this article are entirely unconditional, and can be viewed as (in our opinion, substantial) evidence for the conjectures presented here and in \cite{bpn-thom}.

\subsection{Acknowledgements}
I'm grateful to Ben Antieau, Elden Elmanto, Jeremy Hahn, Ishan Levy, Sasha Petrov, Arpon Raksit, and Andy Senger for conversations on these and related topics; to Bhargav Bhatt for explaining \cref{ht-div-pow} to me; to Akhil Mathew for telling me about the cofiber sequence \cref{dr-modp}; and to Andy Baker for a discussion about the spectra $X_\bH(n)$ from \cref{XH-n}. Some of the ideas in this article started during a visit to Northwestern in March 2022, and I'm especially grateful to Ben Antieau for the opportunity to visit; I would have never been able to understand \cite{apc} --- more generally, this subject area --- were it not for him. I would also like to thank my advisors Dennis Gaitsgory and Mike Hopkins for their advice, support, and influence on me.

%% file: introduction/modified-plot.tex
\def\stacktype{S} 
\begin{figure}[H]
\hspace*{-3.5cm}
\begin{tikzcd}
	& {} &&& {} && {} && {} \\
	& {\scriptstyle{n=2}} & 
	{\mathrm{BP}\langle 2\rangle/S^0}
	&& {\mathrm{BP}\langle 2\rangle/T(1)} && \mathbin{\stackon[0pt]{\mathrm{BP}\langle 2\rangle/T(2)}{\tmf/B}} & \cdots & {\mathrm{BP}\langle 2\rangle/\BPP} \\
	\\
	& {\scriptstyle{n=1}} & \mathbin{\stackon[0pt]{j/S^0}{\mathrm{BP}\langle 1\rangle/S^0}} && \mathbin{\stackon[0pt]{\mathrm{BP}\langle 1\rangle/T(1)}{\ko/A}} && {\mathrm{BP}\langle 1\rangle/T(2)} & \cdots & {\mathrm{BP}\langle 1\rangle/\BPP} \\
	&&&&&& {} \\
	& {\scriptstyle{n=0}} & {\mathbf{Z}_p/S^0} && {\mathbf{Z}_p/T(1)} & {} & \cdots & \cdots & {\Z_p/\BPP} \\
	{} && {} && {} && {} && {} & {} \\
	& {\scriptstyle{n=-1}} & {\mathbf{F}_p/S^0} & {} & \cdots && \cdots & \cdots & {\FF_p/\BPP} \\
	& {}
	\arrow["{\Theta_{T(1)}}", Rightarrow, from=6-5, to=6-3]
	\arrow["{\Theta_{T(1)}}", Rightarrow, from=4-5, to=4-3]
	\arrow["{\Theta_{T(1)}}", Rightarrow, from=2-5, to=2-3]
	\arrow["{\Theta_{T(2)}}", Rightarrow, from=2-7, to=2-5]
	\arrow["{\Theta_{T(2)}}", Rightarrow, from=4-7, to=4-5]
	\arrow["{\Theta_\BPP}", Rightarrow, from=2-9, to=2-8]
	\arrow["{\Theta_\BPP}", Rightarrow, from=4-9, to=4-8]
	\arrow["{\Theta_\BPP}", Rightarrow, from=6-9, to=6-8]
	\arrow["{\Theta_\BPP}", Rightarrow, from=8-9, to=8-8]
	\arrow["{T(m)}"'{pos=1}, shift left=5, from=7-1, to=7-10]
	\arrow[color={rgb,255:red,255;green,51;blue,54}, no head, from=4-2, to=1-5]
	\arrow[color={rgb,255:red,255;green,51;blue,54}, no head, from=6-2, to=1-7]
	\arrow[color={rgb,255:red,255;green,51;blue,54}, no head, from=8-2, to=2-8]
	\arrow["{\mathrm{TP}}"', shift left=4, dashed, from=4-7, to=2-7]
	\arrow["{\mathrm{TP}, \scriptstyle{q\Omega, A\Omega}}"'{pos=0.7}, shift left=4, dashed, from=6-5, to=4-5]
	\arrow["{\mathrm{TP}, \scriptstyle{\mathrm{crys}}}", shift left=4, dashed, from=8-3, to=6-3]
	\arrow["{\mathrm{TP}, \prism}", shift left=4, dashed, from=6-3, to=4-3]
	\arrow["{\mathrm{BP}\langle n\rangle}"{pos=1}, shift right=5, from=9-2, to=1-2]
	\arrow["{\text{Mod }v_1,\text{ is submodule} \text{ gen.~by }\theta_1^{pk}}", color={rgb,255:red,54;green,51;blue,255}, from=4-7, to=6-5]
	\arrow["{\text{Mod }p,\text{ is submodule} \text{ gen.~by }\sigma^{pk}}", color={rgb,255:red,54;green,51;blue,255}, from=6-5, to=8-3]
\end{tikzcd}
\captionsetup{width=\linewidth}
\caption[]{Heuristic picture suggested by this article, where we have assumed for simplicity that $T(m)$ admits the structure of a framed $\E{2}$-ring. 
{\begin{itemize}[wide]
    \item The spectra sandwiched between diagonal lines of slope $1$ (partitioned by a red line) display similar structural behaviour. Here, $A$ and $B$ are studied in \cite{mahowald-thom} (where $A$ is denoted $X_5$), \cite[Construction 3.1]{tmf-witten}, and \cite{hopkins-mahowald-orientations}.
    \item The horizontal double arrows indicate the topological Sen operators of \cref{topological-sen}, i.e., the descent spectral sequence for the map $\THH(-/T(n-1)) \to \THH(-/T(n))$. This is closely related to the Cohen-Moore-Neisendorfer map $\Omega^2 S^{2p^n+1} \to S^{2p^n-1}$.
    \item The (slightly offset) vertical dashed lines going from $(n,n-1)$ to $(n,n)$ indicate the $p$-completed isomorphism $\pi_\ast \TP(\BP{n-1}/T(n)) \cong \pi_\ast \BP{n}^{tS^1}$ of \cref{thh-calculations}. The other vertical arrow from $(0,0)$ to $(0,1)$ is the identification of $\THH(\Z_p)$ with $\tau_{\geq 0}(j^{t\Cp})$, which will appear in future work with Arpon Raksit. (Here, $j$ is the connective complex image-of-J spectrum.) This equivalence is already predicted by the pioneering work of B\"okstedt-Madsen in \cite{bokstedt-madsen}.
    \item The downwards-sloping blue arrows indicate that $\THH(\BP{n}/T(n+1))/v_n$ is a submodule of $\THH(\BP{n-1}/T(n))$ generated by $\theta_n^{pk}$ for $k\geq 0$. See \cref{frobenius-thh}, \cref{frob-in-topology}, and \cref{similarity-ehp} for an explanation of this phenomenon using the EHP sequence.
    \item The columns continue infinitely far out (i.e., $\THH(\BP{n-1}/T(m))$ for $m>n$). However, the drawing is truncated because these terms do not detect any more information than $\THH(\BP{n-1}/T(n))$ itself. The ``exception'' is the final column, where the descent from $\THH(\BP{n-1}/\BPP)$ to $\THH(\BP{n-1})$ can be described algebro-geometrically via
    the $p$-typical Witt ring scheme.
\end{itemize}
}}
\label{bpn-vs-tj}
\end{figure}

%% file: bokstedt/xpn-review.tex
\begin{definition}[{Ravenel, \cite[Section 3]{ravenel-loc}}]
Let $X(n)$ denote the Thom spectrum of the $\E{2}$-map $\Omega \SU(n) \subseteq \BU \xrightarrow{J} \B\GL_1(S)$, where the first map arises from Bott periodicity.
\end{definition}
\begin{example}
The $\E{2}$-ring $X(1)$ is the sphere spectrum, while $X(\infty)$ is $\MU$. Since the map $\Omega \SU(n) \to \BU$ is an equivalence in dimensions $\leq 2n-2$, the same is true for the map $X(n) \to \MU$; the first dimension in which $X(n)$ has an element in its homotopy which is not detected by $\MU$ is $2n-1$. In other words, writing $\pi_\ast \MU = \Z[b_1, b_2, \cdots]$ with $|b_i| = 2i$, the classes $b_1, \cdots, b_{n-1}$ lift to $X(n)$; there is an inclusion $\Z[b_1, \cdots, b_{n-1}] \subseteq \pi_\ast X(n)$.
\end{example}
\begin{remark}\label{Xn-e3}
The $\E{2}$-structure on $X(n)$ does \emph{not} extend to an $\E{3}$-structure (see \cite[Example 1.5.31]{lawson-dl}).
\end{remark}
After localizing at a prime $p$, the spectrum $\MU$ splits as a wedge of suspensions of $\BPP$; this splitting comes from the Quillen idempotent on $\MU$. The same is true of the $X(n)$ spectra, as explained in \cite[Section 6.5]{green}: a multiplicative map $X(n)_{(p)}\to X(n)_{(p)}$ is determined by a polynomial $f(x) = \sum_{0\leq i\leq n-1} a_i x^{i+1}$, with $a_0 = 1$ and $a_i\in\pi_{2i}(X(n)_{(p)})$. One can use this to define a truncated form of the Quillen idempotent $\epsilon_n$ on $X(n)_{(p)}$ (see \cite[Proposition 1.3.7]{hopkins-thesis}), and thereby obtain a summand of $X(n)_{(p)}$. We summarize the necessary results in the following theorem.
\begin{theorem}\label{tn-def}
Let $n$ be such that $p^n\leq k\leq p^{n+1}-1$. Then $X(k)_{(p)}$ splits as a wedge of suspensions of the spectrum $T(n) = \epsilon_{p^n}\cdot X(p^n)_{(p)}$.
\begin{itemize}
    \item $T(n)$ admits the structure of an $\E{1}$-ring such that the map $T(n) \to X(p^n)$ is a map of $\E{1}$-rings (see \cite[Section 7.5]{tn-e1}).
	\item The map $T(n) \to \BPP$ is an equivalence in dimensions $\leq |v_{n+1}|-2$, so there is an indecomposable element $v_i\in \pi_\ast T(n)$ which maps to an indecomposable element in $\pi_\ast \BPP$ for $0\leq i\leq n$. In particular (by (a)), there is an inclusion $\Z_\ppar[v_1, \cdots, v_n] \subseteq \pi_\ast T(n)$.
	\item The map $T(n) \to \BPP$ induces the inclusion $\BPP_\ast T(n) = \BPP_\ast[t_1,\cdots,t_n] \subseteq \BPP_\ast(\BPP)$ on $\BPP$-homology, and the inclusions $\FF_2[\zeta_1^2, \cdots, \zeta_n^2] \subseteq \FF_2[\zeta_1^2, \zeta_2^2, \cdots]$ and $\FF_p[\zeta_1, \cdots, \zeta_n] \subseteq \FF_p[\zeta_1, \zeta_2, \cdots]$ on mod $2$ and mod $p$ homology, respectively.
\end{itemize}
\end{theorem}
\begin{example}\label{definition-of-t1}
The $\E{1}$-ring $T(1)$ is the Thom spectrum of the $\E{1}$-map $\Omega S^{2p-1} \to \BGL_1(S)$ which detects $\alpha_1\in \pi_{2p-2} \BGL_1(S) \cong \pi_{2p-3} S$ on the bottom cell of $\Omega S^{2p-1}$. Since $p\alpha_1 = 0$, a nullhomotopy of $p\alpha_1$ defines a class $v_1\in \pi_{2p-2} T(1)$. Under the unit map $T(1) \to \BPP$, this class is sent to the eponymous class $v_1\in \pi_{2p-2} \BPP$.
\end{example}
\begin{warning}\label{clash-of-notation}
Unfortunately, \cref{tn-def} leads to an egregious clash of notation, since $T(n)$ is also often used to denote the telescope of a $v_n$-self map of a finite type $n$ spectrum. In this article, we will \textit{only} use $T(n)$ to mean the $\E{1}$-ring from \cref{tn-def}. We propose using the notation $\mathrm{Tel}(n)$ to denote the telescope of a $v_n$-self map.
\end{warning}
\begin{notation}
If $R$ is a commutative ring, we write $\Lambda_R(x)$ to denote an exterior $R$-algebra on a class $x$, and $R\pdb{x}$ to denote a divided power $R$-algebra on a class $x$. The notation $\gamma_j(x)$ denotes the $j$th divided power of $x$, so that $j! \gamma_j(x) = x^j$. We will also often write $R\pdb{x}$ to denote the underlying $R$-module of $R\pdb{x}$.
\end{notation}
\begin{construction}
Define a space $\Delta_n$ by 
$$\Delta_n = \prod_{i=1}^n \SU(p^i-1)/\SU(p^{i-1}),$$
and let $\ol{\Delta}_i$ denote the $i$th term in this product.
If $R$ is a ring spectrum, write $R[\Omega \Delta_n]$ to denote the $\E{2}$-polynomial $R$-algebra $R[x_i | 1\leq i \leq p^n-1, i\neq p^k-1]$, where $|x_i| = 2i$. 
Let $R[\B\Delta_n]$ denote the $2$-fold bar construction of the augmentation $R[\Omega \Delta_n] \to R$, so that it is an $\E{2}$-$R$-coalgebra whose homotopy groups are isomorphic to $\pi_\ast(R)\pdb{y_{i} | 1\leq i \leq p^n-1, i\neq p^k}$ where $|y_j| = 2j$. As mentioned in the introduction, $R[\B\Delta_n]$ morally should be viewed as the $R$-chains on the ``classifying space of $\prod_{i=1}^n \SU(p^i-1)/\SU(p^{i-1})$''; to this end, if $X$ is another space, we will write $R[\B\Delta_n\times X]$ to denote $R[\B\Delta_n] \otimes_R R[X]$; and if $R$ is a discrete ring, we will often write $\H_\ast(\B\Delta_n; R)$ to denote $\pi_\ast R[\B\Delta_n]$. The factor $R[\B\Delta_n]$ will primarily be an unfortunate annoyance in this article.
Note that $\Delta_1 = \SU(p-1)$. Then, we have $X(p^n) = T(n)[\Omega \Delta_n]$ and $X(p^n-1) = T(n-1)[\Omega \Delta_n]$, so that
\begin{align*}
    \H_\ast(X(2^n); \FF_2) & \cong \FF_2[\zeta_1^2, \cdots, \zeta_n^2] \otimes_{\FF_p} \H_\ast(\Omega \Delta_n; \FF_p), \\
    \H_\ast(X(p^n); \FF_p) & \cong \FF_p[\zeta_1, \cdots, \zeta_n] \otimes_{\FF_p} \H_\ast(\Omega \Delta_n; \FF_p),
\end{align*}
and similarly for $X(p^n-1)$.
\end{construction}
It is believed that $T(n)$ admits more structure (see also \cite{ak-quigley} for some discussion):
\begin{conjecture}\label{tn-e2}
    The $\cQ_1$-ring structure on $T(n)$ extends to a \textit{framed} $\E{2}$-ring structure.
\end{conjecture}
\begin{remark}\label{t2-thom}
When $p=2$, both $X(2) = T(1)$ and $T(2)$ admit the structure of $\Efr{2}$-algebras by \cite[Remark 3.8]{framed-e2}: they are Thom spectra of $\U$-bundles over $\Omega \Sp(1) \simeq \Omega S^3$ and $\Omega \Sp(2)$, respectively. These $\U$-bundles are defined via double loops of the the composite
$$\BSp(n) \to \BSU(2n) \to \BSU \simeq \B^3 \U.$$
\end{remark}
\begin{prop}[{\cite[Corollary 2.9 and Corollary 3.7]{framed-e2}}]\label{xn-framed-e2}
The $\E{2}$-structure on $X(n)$ refines to an $\Efr{2}$-structure.
\end{prop}
\begin{corollary}
Let $\cC$ be an $X(n)$-linear $\infty$-category. Then $\THH(\cC/X(n))$ acquires the structure of an $S^1$-equivariant spectrum with an $S^1$-equivariant unit map $X(n) \to \THH(\cC/X(n))$.
\end{corollary}

%% file: bokstedt/thh-calculation.tex
Unless explicitly stated otherwise, all fiber sequences in this section (as well as the following sections) will be localized at $p$.
\begin{recall}\label{homology-bpn}
There are isomorphisms
\begin{align*}
    \H_\ast(\BP{n-1}; \FF_2) & \cong \FF_2[\zeta_1^2, \cdots, \zeta_n^2, \zeta_{n+1}, \cdots] \\
    & \cong \H_\ast(T(n); \FF_2) \otimes_{\FF_2} \FF_2[\zeta_j | j\geq n+1], \\
    \H_\ast(\BP{n-1}; \FF_p) & \cong \Lambda_{\FF_p}[\tau_j | j \geq n] \otimes_{\FF_p} \FF_p[\zeta_1, \zeta_2, \cdots] \\
    & \cong \H_\ast(T(n); \FF_p) \otimes_{\FF_p} \FF_p[\zeta_j | j\geq n+1] \otimes_{\FF_p} \Lambda_{\FF_p}[\tau_j | j \geq n], p>2.
\end{align*}
We note that the ``$Q_0$-Margolis homology'' of $\H_\ast(\BP{n-1}; \FF_2)$ (i.e., the homology of $\Sq^1$ viewed as a differential acting on $\H_\ast(\BP{n-1}; \FF_2)$) is precisely $\H_\ast(T(n); \FF_2)$, because $\Sq^1$ is a derivation and $\Sq^1(\zeta_j) = \zeta_{j-1}^2$.
\end{recall}
\begin{recall}\label{hahn-wilson-review}
We need to recall some results from \cite{hahn-wilson-bpn}. First, \cite[Theorem A]{hahn-wilson-bpn} tells us that there exists an $\E{3}$-form of $\BP{n}$. First, \cite[Theorem 2.5.4]{hahn-wilson-bpn} states that $\pi_\ast \THH(\BP{n-1}/\MU)$ is isomorphic to a polynomial algebra over $\pi_\ast \BP{n-1}$ on infinitely many generators, the first of which is denoted $\sigma^2(v_n)$. The class $\sigma^2(v_n)$ lives in degree $2p^n$. Finally, \cite[Theorem 5.0.1]{hahn-wilson-bpn} states that there is an isomorphism $\pi_\ast \TC^-(\BP{n-1}/\MU) \simeq (\pi_\ast \THH(\BP{n-1}/\MU))\pw{\hbar}$ of $\Z_p[v_1, \cdots, v_{n-1}]$-algebras. Moreover, under the map $\MU^{hS^1} \to \TC^-(\BP{n-1}/\MU)$, the class $v_n \in \pi_\ast \MU^{hS^1} \cong (\pi_\ast \MU)\pw{\hbar}$ is sent to $\sigma^2(v_n)\hbar$. In particular, $\pi_\ast \TC^-(\BP{n-1}/\MU)$ detects the classes $p, v_1, \cdots, v_{n-1}, v_n := \sigma^2(v_n)\hbar$. Similarly, $\pi_\ast \TP(\BP{n-1}/\MU)$ detects the classes $p, \cdots, v_n$ under the map $\MU^{tS^1} \to \TP(\BP{n-1}/\MU)$, and $\pi_\ast \THH(\BP{n-1}/\MU)^{t\Cp}$ detects the classes $p, \cdots, v_{n-1}$ under the map $\MU^{t\Cp} \to \THH(\BP{n-1}/\MU)^{t\Cp}$.
\end{recall}
\begin{notation}
If $R$ is a complex-oriented ring spectrum, we will write $\hbar$ to denote the complex orientation of $R$, viewed as a class in $\pi_{-2} R^{hS^1}$. The motivation for this notation comes from geometric representation theory (in the case where $R$ is a $\Z_p$-algebra), where the complex orientation $\hbar\in \H^2(\CP^\infty; R)$ plays the role of a quantization parameter.
\end{notation}
The main result of this section is the following analogue of B\"okstedt's theorem on $\THH(\FF_p)$ and $\THH(\Z_p)$.
\begin{theorem}\label{thh-calculations}
Fix $\E{3}$-forms of the truncated Brown-Peterson spectra $\BP{n-1}$ and $\BP{n}$.
We have:
\begin{enumerate}
    \item There is a $p$-complete equivalence of $\BP{n-1}$-modules:
    $$\THH(\BP{n-1}/X(p^n)) \simeq \BP{n-1}[\B\Delta_n\times \Omega S^{2p^n+1}].$$
    Write $\theta_n\in \pi_{2p^n} \THH(\BP{n-1}/X(p^n))$ to denote the class corresponding to the map $E: S^{2p^n} \to \Omega S^{2p^n+1}$. Under the $S^1$-equivariant map $\THH(\BP{n-1}/X(p^n)) \to \THH(\BP{n-1}/\MU)$, the class $\theta_n$ is sent to the class $\sigma^2(v_n)$ from \cref{hahn-wilson-review}.
    There are also $p$-complete isomorphisms
    \begin{align*}
    \pi_\ast \THH(\BP{n-1}/X(p^n))^{t\Cpn} & \cong \BP{n}^{t\Cpn}[\B\Delta_n]_\ast, \\
    \pi_\ast \TC^-(\BP{n-1}/X(p^n)) & \cong \BP{n}[\B\Delta_n]_\ast\pw{\hbar}[\tfrac{v_n}{\hbar}] \\
    & \cong \BP{n}[\B\Delta_n]_\ast\pw{\hbar}[\theta_n]/(\theta_n\hbar - v_n),\\
    \pi_\ast \TP(\BP{n-1}/X(p^n)) & \cong \BP{n}^{tS^1}[\B\Delta_n]_\ast \\
    & \cong \BP{n}[\B\Delta_n]_\ast\ls{\hbar}.
    \end{align*}
    Here, the equation $\theta_n\hbar = v_n$ is to be understood modulo decomposables.
    These isomorphisms satisfy the following property: under the maps 
    \begin{align*}
        \TC^-(\BP{n-1}/X(p^n)) & \to \TC^-(\BP{n-1}/\MU), \\
        \TP(\BP{n-1}/X(p^n)) & \to \TP(\BP{n-1}/\MU),
    \end{align*}
    the classes $\{v_i\}_{0\leq i\leq n}$ on the left-hand side are sent to the eponymous classes in the right-hand side (via \cref{hahn-wilson-review}).
    \item 
    There is an equivalence of $\BP{n}$-modules:
    $$\THH(\BP{n}/X(p^n))^\wedge_p \cong \BP{n}[\B\Delta_n]^\wedge_p \oplus \left(\bigoplus_{j\geq 1} \Sigma^{2jp^{n+1}-1} \BP{n}[\B\Delta_n]/p^{v_p(j)+1}\right)^\wedge_p.$$
    In particular, there is an additive equivalence
    $$\THH(\BP{n}/X(p^n))/p \cong \BP{n}[S^{2p^{n+1}-1} \times \Omega S^{2p^{n+1}+1} \times \B\Delta_n]/p.$$
    Moreover, $\pi_{2p^{n+1}-3} \TC^-(\BP{n}/X(p^n))^\wedge_p$ detects the class $\sigma_n\in \pi_{2p^{n+1}-3} X(p^n)$ from \cite[Lemma 3.1.12]{bpn-thom}.
\end{enumerate}
\end{theorem}
\begin{remark}\label{mod-p-vn}
Let $v_{[j,m)}$ denote the regular sequence $v_j,\cdots,v_{m-1}$ in $\pi_\ast \BPP$. Then the argument used to prove \cref{thh-calculations} in fact shows the following (somewhat more general) result: for $j\leq n$, there is an isomorphism of $\BP{n-1}_\ast$-modules
\begin{equation}\label{modulo-moore}
    \pi_\ast \THH(\BP{n-1}/X(p^j))/v_{[0,n-j)} \cong \BP{n-1}[\B\Delta_j]_\ast[\theta_n]/v_{[0,n-j)} \otimes_{\FF_p} \Lambda_{\FF_p}(\lambda_{j+1}, \cdots, \lambda_n),
\end{equation}
where $|\lambda_i| = 2p^i-1$. When $j=0$, \cref{modulo-moore} recovers \cite[Proposition 2.9]{thh-truncated-bp}. For brevity, the discussion below only includes the cases $j=n$ and $j=n-1$. Similarly, using that $T(1)$ (resp. $T(2)_{(2)}$) is a Thom spectrum over $\Omega S^{2p+1}$ (resp. $\Omega \Sp(2)$), there are equivalences
\begin{align*}
    \THH(\Z_p)/p & \simeq \FF_p[S^{2p-1} \times \Omega S^{2p+1}], \\
    \THH(\ku)/(2,\beta) & \simeq \FF_2[\Sp(2) \times \Omega S^9].
\end{align*}
\end{remark}
\begin{remark}
If we write $\pi_\ast \MU = \Z[x_1, x_2, \cdots]$ where $|x_i| = 2i$, and define $\MU\pdb{n-1} = \MU/(x_n, x_{n+1}, \cdots)$, then one can similarly prove an analogue of \cref{thh-calculations} with $\BP{n-1}$ replaced by $\MU\pdb{n-1}$. Namely, if $n$ is a power of $p$, there is an equivalence 
$$\THH(\MU\pdb{n-1}/X(n))^\wedge_p \simeq \MU\pdb{n-1}[\Omega S^{2n+1}]^\wedge_p$$
of $\MU\pdb{n-1}$-modules. There is also a $p$-complete isomorphism 
$$\pi_\ast \TP(\MU\pdb{n-1}/X(n))^\wedge_p \cong \pi_\ast (\MU\pdb{n}^{tS^1})^\wedge_p.$$
We expect (see \cref{conjecture-thh} below) that this refines to a $p$-complete equivalence $\TP(\MU\pdb{n-1}/X(n))^\wedge_p \simeq (\MU\pdb{n}^{tS^1})^\wedge_p$.
\end{remark}
\begin{example}
One can make \cref{thh-calculations}(a) very explicit for $\Z_p$ (note that \cref{thh-calculations}(b) for $\Z_p$ is B\"okstedt's result). For instance,
$$\pi_\ast \TC^-(\Z_p/T(1)) \cong \Z_p[v_1]\pw{\hbar}[\theta]/(\hbar\theta = v_1).$$
Let us view $\BP{1}$ as $(\ku^\wedge_p)^{h\FF_p^\times}$, and let $\beta\in \pi_2 \ku$ be the Bott class. Then, $\pi_\ast \ku^{tS^1} \cong \Z[\beta]\ls{\hbar}$ is isomorphic to $\Z\pw{q-1}\ls{\hbar}$, where $q = 1 + \beta \hbar$ lives in degree $0$. If $\Z_p\pw{\ptl}$ is as in \cite[Corollary 3.8.8]{apc}, then $\pi_\ast \BP{1}^{tS^1} \cong \Z_p\pw{\ptl}\ls{\hbar}$. If we assume (for simplicity) that $T(1)$ is an $\Efr{2}$-algebra, then replacing $X(p)$ by $T(1)$, we obtain:
\begin{align*}
    \pi_\ast \TP(\Z_p/T(1)) & \cong \Z_p\pw{\ptl}\ls{\hbar}.
\end{align*}
Here, $\FF_p^\times$ acts on $\Z_p\pw{q-1}$ as specified before \cite[Proposition 3.8.6]{apc}; indeed, the $\Z_p^\times$-action on $\Z_p\pw{q-1} = \pi_0 (\ku^\wedge_p)^{tS^1}$ agrees with the action of the Adams operations on $\pi_\ast (\ku^\wedge_p)^{tS^1}$, as one can check by calculating the Adams operations on the $p$-completed complex K-theory of $\CP^\infty$. Indeed, if $g\in \Z_p^\times$, then
$$\psi^g(\hbar) = \frac{1}{g} \sum_{j\geq 1} \binom{g}{j} \beta^{j-1} \hbar^j = \frac{1}{g} \frac{(1+\beta \hbar)^g - 1}{\beta},$$
so that
$$\psi^g(q) = \psi^g(1+\beta\hbar) = 1 + g\beta \psi^g(\hbar) = (1+\beta\hbar)^g = q^g.$$
\end{example}
\begin{remark}
Recall from \cite[Section 3.4]{rotinv} that there is an $\E{2}$-monoidal functor $\sh: \Sp^\gr \to \Sp^\gr$ given by shearing: this functor sends $M_\bull \mapsto M_\bull[2\bull]$. Assume for simplicity that $T(n)$ admits the structure of an $\Efr{2}$-algebra. From this perspective, part of \cref{thh-calculations}(a) simply states that there is an equivalence of ungraded $\BP{n-1}$-modules
$$\THH(\BP{n-1}/T(n)) \simeq \sh(\gr_{v_n} \BP{n}),$$
where $\sh(\gr_{v_n} \BP{n})$ denotes the shearing of the associated graded of the $v_n$-adic filtration $\F^\star_{v_n} \BP{n}$ on $\BP{n}$.
\end{remark}

An immediate implication of \cref{thh-calculations} is the following.
\begin{corollary}[{\cite[Corollary 5.0.2]{hahn-wilson-bpn}}]\label{k-thy-redshift}
Fix an $\E{3}$-form of the truncated Brown-Peterson spectrum $\BP{n-1}$. We have $\Lk K(\BP{n-1}) \neq 0$.
\end{corollary}
\begin{proof}
There is a trace map $K(\BP{n-1}) \to \TP(\BP{n-1})$, which is a map of $\E{2}$-rings. It therefore suffices to exhibit a nonzero module over $\Lk \TP(\BP{n-1})$ --- but we may take the module $\Lk \TP(\BP{n-1}/X(p^n))$, which is nonzero by \cref{thh-calculations}(a). (In fact, \cref{thh-calculations}(a) implies $\pi_\ast \Lk \TP(\BP{n-1}/X(p^n))$ is isomorphic to $\Z_p[v_1, \cdots, v_{n-1}, v_n^{\pm 1}]^\wedge_{(p, \cdots, v_{n-1})}\ls{\hbar}$ tensored with the $\Z_p$-homology of $\B\Delta_n$.)
\end{proof}
\begin{remark}
It is easy to see that $T(n) \to \BP{n}$ is a nilpotent extension. This implies in particular that the following square is Cartesian by the Dundas-Goodwillie-McCarthy theorem \cite[Theorem 7.2.2.1]{dgm}:
$$\xymatrix{
K(T(n)) \ar[r] \ar[d] & K(\BP{n}) \ar[d]\\
\TC(T(n)) \ar[r] & \TC(\BP{n}).
}$$
Note that there is also a commutative square
$$\xymatrix{
\TC(T(n)) \ar[r] \ar[d] & \TC(\BP{n}) \ar[d]\\
\TC^-(T(n)) \ar[r] & \TC^-(\BP{n}),
}$$
and \cref{thh-calculations} and \cref{topological-sen} give an inductive approach to calculating the bottom row. One might therefore view the results of this article as a first step to fully computing $K(\BP{n})$. It would be very interesting to describe $\TC(T(n))$. For example, we expect that for a general odd prime, the spectrum $\TP(T(1))$ is closely related to the $\E{1}$-quotient $S\mmod\alpha_{p/p}$. (Here, $\alpha_{p/p}\in \pi_{2p(p-1)-1}(S)$ is an element in the $\alpha$-family.)

However, more is true about the map $T(n) \to \BP{n}$: in fact, every element in $\ker(\pi_\ast T(n) \to \pi_\ast \BP{n})$ is nilpotent.
To see this, first observe that this map is a rational equivalence (indeed, it is an equivalence on $Q_0$-Margolis homology), so $\fib(T(n)\to \BP{n})$ is torsion. Moreover, the map $T(n) \to \BP{n}$ is surjective on homotopy (since it is a ring map, and the generators $p, v_1, \cdots, v_n\in \pi_\ast \BP{n}$ lift to $T(n)$), so that the map $\fib(T(n)\to \BP{n}) \to T(n)$ induces an injection on homotopy. If $x\in \pi_\ast T(n)$ is in the image of the map $\fib(T(n)\to \BP{n}) \to T(n)$, then the image of $x$ under the Hurewicz map $\pi_\ast T(n) \to \MU_\ast T(n)$ is also torsion; but $\MU_\ast T(n) \cong \MU_\ast[t_1, \cdots, t_n]$ is torsion-free, so $x$ must be nilpotent by the main theorem of \cite{dhs-i}\footnote{In some sense, this is a rather perverse argument, because the heart of the proof of the nilpotence theorem relies crucially on showing that every element in $\ker(\pi_\ast T(n) \to \pi_\ast \BP{n})$ is nilpotent.}. This is the desired claim.

More generally, recall \cite[Table 1]{bpn-thom}, reproduced here as \cref{the-table} (for the definitions of these spectra, see \cite{mahowald-thom} for $A$, where it is denoted $X_5$; \cite[Construction 3.1]{tmf-witten} and \cite{hopkins-mahowald-orientations} for $B$; \cite{mrs} for $y(n)$; and \cite{yz} for $y_\Z(n)$).
\begin{table}[h!]
    \centering
    \begin{tabular}{c | c c c c c c c c c c}
	Height & 0 & 1 & 2 & $n$ & $n$ & $n$\\
	\hline
	Base $\E{1}$-ring $R$ & $(S^0)^\wedge_p$ & $A$ & $B$ & $T(n)$ & $y(n)$ & $y_\Z(n)$ \\
	\hline
	Designer chromatic spectrum $\Theta(R)$ & $\Z_p$ & $\bo$ & $\tmf$ & $\BP{n}$ & $k(n)$ & $k_\Z(n)$\\
    \end{tabular}
    \vspace{0.5cm}
    \caption{The relation between $R$ and $\Theta(R)$ is analogous to the relationship between $T(n)$ and $\BP{n}$.}
    \label{the-table}
\end{table}

In a manner similar to above, if $R$ is an $\E{1}$-ring as in the second line of \cref{the-table}, and $\Theta(R)$ is the associated designer spectrum, one can show that every element in $\ker(\pi_\ast R \to \pi_\ast \Theta(R))$ is nilpotent. It follows, for example, that there is a Cartesian square
$$\xymatrix{
K(R) \ar[r] \ar[d] & K(\Theta(R)) \ar[d]\\
\TC(R) \ar[r] & \TC(\Theta(R)).
}$$
Moreover, the proof of \cref{thh-calculations} shows that were $R$ to admit the structure of an $\E{2}$-ring (which is generally \textit{not true}\footnote{For instance, $y(n)$ cannot admit the structure of an $\E{2}$-ring, thanks to the Steinberger identity on the action of the Dyer-Lashof operation $Q_1$ on the dual Steenrod algebra (see \cite[Theorems III.2.2 and III.2.3]{bmms}).}, $\THH(\Theta(R)/R)$ would be $p$-completely equivalent to $R \oplus \bigoplus_{j\geq 1} \Sigma^{2jp^{n+1}-1} R/pj$ (where $n$ is the ``height'' of $R$). If $R = y(n)$ or $y_\Z(n)$, this result is literally true by \cref{thh-calculations}, as long as one assumes \cref{tn-e2} and interprets $\THH(\Theta(R)/R)$ to mean $\THH(\BP{n}/T(n)) \otimes_{T(n)} R$. This does not cover the cases $R=A,B$, though; see \cref{analogues-ko-tmf} for further discussion of these cases.
\end{remark}
\begin{remark}\label{relative-to-large-xp}
It is natural to ask whether \cref{thh-calculations} can be generalized to describe $\THH(\BP{n-1}/X(p^m))$ if $m\neq n$. For $m<n$, we do not know a full description (after killing $p, \cdots, v_{n-m-1}$, see \cref{mod-p-vn}); but the techniques of \cref{topological-sen} below provide a conceptual approach to addressing this question. For $m>n$, the proof of \cref{thh-calculations} easily implies that there is an additive isomorphism
\begin{align*}
    \pi_\ast \THH(\BP{n-1}/X(p^m)) & \cong \pi_\ast \THH(\BP{n-1}/X(p^n)) \otimes_{\BP{n-1}_\ast} \BP{n-1}_\ast\pdb{y_i | p^n < i \leq p^m}\\
    & \cong \BP{n-1}[\Omega S^{2p^n+1}]_\ast \pdb{y_i | 1\leq i \leq p^m \text{ such that } i\neq p^k \text{ for } 0\leq k\leq n}.
\end{align*}
Here, $y_i$ lives in degree $2i$. For example, if $n=0$, the divided power factor is just $\BP{n-1}_\ast[\BSU(p^m)]$. For instance, in the limit as $m\to \infty$, we recover the statement that $\pi_\ast \THH(\FF_p/\MU) \simeq \FF_p[\BSU \times \Omega S^3]_\ast$.
\end{remark}
\begin{remark}\label{descent-sseq}
\cref{thh-calculations}(b) implies that 
$$\pi_\ast \THH(\BP{n-1}/X(p^n-1)) \cong \BP{n-1}[\B\Delta_n]_\ast \oplus \bigoplus_{j\geq 1} \BP{n-1}[\B\Delta_n]_{\ast - 2jp^{n}+1}/p^{v_p(j)+1}.$$
This can be compared to \cref{thh-calculations}(a) (we will study in this in further detail in \cref{applications-thh}): the complexity of $\pi_\ast \THH(\BP{n-1}/X(p^n-1))$ compared to $\pi_\ast \THH(\BP{n-1}/X(p^n))$ can be understood as arising via the descent spectral sequence for the map $\THH(\BP{n-1}/X(p^n-1)) \to \THH(\BP{n-1}/X(p^n))$. Note that $X(p^n) \otimes_{X(p^n-1)} X(p^n) \simeq X(p^n)[\Omega S^{2p^n-1}]$; using this, one can calculate using methods similar to the proof of \cref{thh-calculations} that the $E_2$-page of the descent spectral sequence is
$$E_2^{\ast,\ast} 
\cong \pi_\ast \THH(\BP{n-1}/X(p^n))[\epsilon]/\epsilon^2,$$
where $|\epsilon| = 2p^n-1$. Calculating the differentials gives an ``alternative'' proof of \cref{thh-calculations}(b) given \cref{thh-calculations}(a); we will expand on this below in \cref{bockstein-serre}. In fact, inductively studying $\THH$ of $\BP{n-1}$ relative to $X(p^j)$ for $j\leq n$ gives a conceptual explanation for the families of differentials visible in the calculations of $\pi_\ast \THH(\BP{n-1})$ in \cite[Section 8]{angeltveit-rognes}, \cite{mcclure-staffeldt}, and \cite{thh-bp1}; see \cref{topological-sen} and \cref{algebraic-sen}.
\end{remark}

The proof of \cref{thh-calculations} will be broken into several components. Let us begin by illustrating \cref{thh-calculations}(a) in the case $n=0,1$.
\begin{proof}[Proof of \cref{thh-calculations}(a) for $n=0,1$]
We need to show that there are equivalences of spectra $\THH(\FF_p) \simeq \FF_p[\Omega S^3]$ and $\THH(\Z_p/X(p)) \simeq \Z_p[\BSU(p-1)\times \Omega S^{2p+1}]$.
The first equivalence is classical (see \cite{bokstedt-thh}), so we argue the second equivalence. There is a $p$-local map $f: \SU(p) \to \Omega S^3\pdb{3}$ of spaces given by the composite
$$\SU(p) \to \SU(p)/\SU(p-1) \simeq S^{2p-1} \xar{\alpha_1} \Omega S^3\pdb{3}.$$
In \cite[Remark 4.1.4]{bpn-thom}, we described a fiber sequence (which was also known to Toda in \cite{toda})
\begin{equation}\label{toda-fibration}
    S^{2p-1} \xar{\alpha_1} \Omega S^3\pdb{3} \to \Omega S^{2p+1}.
\end{equation}
This induces a fiber sequence
of $\E{1}$-spaces
$$\Omega\SU(p) \xar{f} \Omega^2 S^3\pdb{3} \to \SU(p-1) \times \Omega^2 S^{2p+1}.$$
We now compute:
\begin{align*}
    \THH(\Z_p/X(p)) & \simeq \THH(\Z_p) \otimes_{\THH(X(p))} X(p)\\
    & \simeq \THH(\Z_p) \otimes_{\Z_p \otimes \THH(X(p))} \Z_p.
\end{align*}
The map $X(p) \to \Z_p$ is precisely the map induced by $f: \Omega\SU(p) \to \Omega^2 S^3\pdb{3}$, so the above tensor product is given by $\Z_p[\Omega S^{2p+1} \times \BSU(p-1)]$, as desired.
\end{proof}
\begin{remark}\label{relation-to-hopkins-mahowald}
Recall that the calculation $\THH(\FF_p) \simeq \FF_p[\Omega S^3]$ follows from \cite{thh-thom} and the Hopkins-Mahowald theorem that $\FF_p$ is the Thom spectrum of the $\E{2}$-map $\Omega^2 S^3 \to \BGL_1(S^\wedge_p)$ which detects $1-p\in \pi_1 \BGL_1(S^\wedge_p) \cong \Z_p^\times$ on the bottom cell of $\Omega^2 S^3$. In \cite[Corollary B]{bpn-thom}, we prove (unconditionally!) that $\Z_p$ is the Thom spectrum of a map $\mu: \Omega^2 S^{2p+1} \to \BGL_1(T(1))$ which detects $v_1\in \pi_{2p-1} \BGL_1(T(1)) \cong \pi_{2p-2} T(1)$ on the bottom cell of $\Omega^2 S^{2p+1}$. (Unlike in the classical Hopkins-Mahowald theorem, the map $\mu$ is not an $\E{2}$-map.) This result implies that $\Z_p$ is the Thom spectrum of a map $\SU(p-1) \times \Omega^2 S^{2p+1} \to \BGL_1(X(p))$, which can also be used to prove \cref{thh-calculations}(a) for $n=1$.
\end{remark}
We now turn to \cref{thh-calculations}(a) in the general case; the strategy is to compute the homology of each of the spectra under consideration, and run the Adams spectral sequence. In the case of $\THH^{t\Cpn}$, $\TC^-$, and $\TP$, we will need the ``continuous homology'' of \cite[Equation 2.3]{homological-hfpss}.
\begin{prop}\label{homology-thh-bpn}
\begin{enumerate}
    \item There are isomorphisms
    \begin{align*}
    \H_\ast(\THH(\BP{n-1}/X(p^n)); \FF_p) & \cong 
    \begin{cases}
    \H_\ast(\BP{n-1}[\B\Delta_n]; \FF_2) [\sigma(\zeta_{n+1})] & p=2,\\
    \H_\ast(\BP{n-1}[\B\Delta_n]; \FF_p) [\sigma(\tau_n)] & p>2.
    \end{cases}
    \end{align*}
    \item There are isomorphisms
    \begin{align*}
    \H_\ast(\THH(\BP{n}/X(p^n)); \FF_p) & \cong \begin{cases}
    \H_\ast(\BP{n}[\B\Delta_n]; \FF_2) [\sigma(\zeta_{n+2})] \otimes_{\FF_2} \Lambda_{\FF_2}(\sigma(\zeta_{n+1}^2)) & p=2,\\
    \H_\ast(\BP{n}[\B\Delta_n]; \FF_p) [\sigma(\tau_{n+1})] \otimes_{\FF_p} \Lambda_{\FF_p}(\sigma(\zeta_{n+1})) & p>2.
    \end{cases}
    \end{align*}
    Moreover, there is a Bockstein $\beta: \sigma(\zeta_{n+2})\mapsto \sigma(\zeta_{n+1}^2)$ for $p=2$, and a Bockstein $\beta: \sigma(\tau_{n+1})\mapsto \sigma(\zeta_{n+1})$ for $p>2$.
\end{enumerate}
\end{prop}
\begin{proof}
We begin by proving (a). We will use the B\"okstedt spectral sequence, which runs
$$E^2_{\ast,\ast} = \HH_\ast(\H_\ast(\BP{n-1}; \FF_p)/\H_\ast(X(p^n); \FF_p)) \Rightarrow \H_\ast(\THH(\BP{n-1}/X(p^n)); \FF_p).$$
Since $\H_\ast(X(p^n); \FF_p) \cong \H_\ast(T(n); \FF_p) \otimes_{\FF_p} \H_\ast(\Omega \Delta_n; \FF_p)$ and the action of $\H_\ast(X(p^n); \FF_p)$ on $\H_\ast(\BP{n-1}; \FF_p)$ factors through the map $\H_\ast(X(p^n); \FF_p) \to \H_\ast(T(n); \FF_p)$ induced by the map crushing $\Omega \Delta_n$ to a point, we will ignore the contribution from $\Delta_n$ in this discussion. The final contribution from these terms will only be $\H_\ast(\B\Delta_n; \FF_p)$. (The following may therefore be interpreted as a computation of $\H_\ast(\THH(\BP{n-1}/T(n)); \FF_p)$; however, since \cref{tn-e2} is not known to be true, the spectrum $\THH(\BP{n-1}/T(n))$ cannot yet be defined.) We will continue to write $E^2_{\ast,\ast}$ to denote the Hochschild homology groups of $\H_\ast(\BP{n-1}; \FF_p)$ over $\H_\ast(T(n); \FF_p)$.

Recall that if $R$ is any discrete commutative ring, there are isomorphisms $\pi_\ast \HH(R[x]/R) \simeq R[x] \otimes \Lambda_R(\sigma x)$ and $\pi_\ast \HH(\Lambda_R(x)) \simeq \Lambda_R(x) \otimes R\pdb{\sigma x}$. It therefore follows from \cref{homology-bpn} that we have
\begin{align*}
    E^2_{\ast,\ast} & =
    \begin{cases}
    \H_\ast(\BP{n-1}; \FF_2) \otimes_{\FF_2} \Lambda_{\FF_2}(\sigma \zeta_j | j \geq n+1) & p=2, \\
    \H_\ast(\BP{n-1}; \FF_p) \otimes_{\FF_p} \Lambda_{\FF_p}(\sigma \zeta_j | j\geq n+1) \otimes_{\FF_p} \FF_p\pdb{\sigma \tau_j| j\geq n} & p>2.
    \end{cases}
\end{align*}
The map $\THH(\BP{n-1}) \to \THH(\BP{n-1}/X(p^n))$ induces a map from the B\"okstedt spectral sequence computing $\H_\ast(\THH(\BP{n-1}))$ to our spectral sequence. The differentials in  the B\"okstedt spectral sequence computing $\H_\ast(\THH(\BP{n-1}))$ are calculated in \cite[Proposition 5.6]{angeltveit-rognes}, where it is shown that for $p$ odd, $j\geq p$, and $m \geq n$, there are differentials
\begin{equation}\label{bokstedt-diffl}
    d^{p-1}(\gamma_j(\sigma \tau_m)) = \sigma(\zeta_{m+1}) \gamma_{j-p}(\sigma \tau_m).
\end{equation}
The argument of \cite[Proposition 5.7]{angeltveit-rognes} implies that
\begin{align*}
    E^\infty_{\ast,\ast} & = 
    \begin{cases}
    \H_\ast(\BP{n-1}; \FF_2) \otimes_{\FF_2} \Lambda_{\FF_2}(\sigma \zeta_j | j\geq n+1) & p=2, \\
    \H_\ast(\BP{n-1}; \FF_p) \otimes_{\FF_p} \FF_p[\sigma \tau_j | j \geq n]/(\sigma \tau_j)^p & p>2.
    \end{cases}
\end{align*}
The extensions on the $E^\infty$-page of the B\"okstedt spectral sequence computing $\H_\ast(\THH(\BP{n-1}))$ are determined by \cite[Theorem 5.12]{angeltveit-rognes}: there, it is shown that for $j\geq n+1$, we have $(\sigma \zeta_j)^2 = \sigma \zeta_{j+1}$ when $p=2$, and $(\sigma \tau_j)^p = \sigma \tau_{j+1}$. These imply extensions on the $E^\infty$-page of the B\"okstedt spectral sequence for $\THH(\BP{n-1}/X(p^n))$, and the resulting answer is that of the proposition.

We now turn to (b). The calculation is similar to (a), the only difference being that the $E^2$-page of the B\"okstedt spectral sequence is now
\begin{align*}
    E^2_{\ast,\ast} & = 
    \begin{cases}
    \H_\ast(\BP{n}; \FF_2) \otimes_{\FF_2} \Lambda_{\FF_2}(\sigma(\zeta_{n+1}^2), \sigma \zeta_j | j \geq n+2) & p=2, \\
    \H_\ast(\BP{n}; \FF_p) \otimes_{\FF_p} \Lambda_{\FF_p}(\sigma \zeta_j | j\geq n+1) \otimes_{\FF_p} \FF_p\pdb{\sigma \tau_j| j\geq n+1} & p>2.
    \end{cases}
\end{align*}
Again, the differentials in the B\"okstedt spectral sequence computing $\H_\ast(\THH(\BP{n-1}))$ give rise to differentials in the above B\"okstedt spectral sequence, and we have
\begin{align*}
    E^\infty_{\ast,\ast} & = 
    \begin{cases}
    \H_\ast(\BP{n}; \FF_2) \otimes_{\FF_2} \Lambda_{\FF_2}(\sigma(\zeta_{n+1}^2), \sigma \zeta_j | j\geq n+2) & p=2, \\
    \H_\ast(\BP{n-1}; \FF_p) \otimes_{\FF_p} \FF_p[\sigma \tau_j | j \geq n+1]/(\sigma \tau_j)^p \otimes_{\FF_p} \Lambda_{\FF_p}(\sigma \zeta_{n+1}) & p>2.
    \end{cases}
\end{align*}
Again, the extensions on the $E^\infty$-page of the B\"okstedt spectral sequence computing $\H_\ast(\THH(\BP{n-1}))$ imply extensions on the above $E^\infty$-page, and the resulting answer is that of the proposition. The Bockstein follows from the fact that $\beta(\tau_i) = \zeta_i$ for $p$ odd and $\beta(\zeta_i) = \zeta_{i-1}^2$ for $p=2$.
\end{proof}
\begin{prop}\label{homology-tc-bpn}
There are isomorphisms
\begin{align*}
    \H_\ast^c(\TC^-(\BP{n-1}/X(p^n)); \FF_p) & \cong \H_\ast(\BP{n}[\B\Delta_n]; \FF_p) \pw{\hbar} \oplus \hbar\text{-}\mathrm{torsion}, \\
    \H_\ast^c(\TP(\BP{n-1}/X(p^n)); \FF_p) & \cong \H_\ast(\BP{n}[\B\Delta_n]; \FF_p) \ls{\hbar}, \\
    \H_\ast^c(\THH(\BP{n-1}/X(p^n))^{t\Cpn}; \FF_p) & \cong \H_\ast(\BP{n}^{t\Cpn}[\B\Delta_n]; \FF_p) \ls{\hbar}.
\end{align*}
Here, $|\hbar| = -2$, and the $\hbar$-torsion terms will be specified in the proof.
\end{prop}
\begin{proof}
As in \cref{homology-thh-bpn}, the contribution from $\Delta_n$ is just the $\FF_p$-homology of $\B\Delta_n$, and we will ignore this term in the calculations. Moreover, the calculation for $\H_\ast^c(\THH(\BP{n-1}/X(p^n))^{t\Z/p^k}; \FF_p)$ is similar to the calculation of $\H_\ast^c(\TC^-(\BP{n-1}/X(p^n)); \FF_p)$ (and $\H_\ast^c(\TP(\BP{n-1}/X(p^n)); \FF_p)$), so we will only do the latter. (The only difference is that $\FF_p\ls{\hbar}$ below is replaced by $\FF_p\ls{\hbar}[\epsilon_k]/\epsilon_k^2$.) The $E^2$-page of the homological homotopy fixed points spectral sequence computing $\H_\ast^c(\TC^-(\BP{n-1}/X(p^n)); \FF_p)$ is given by
\begin{align*}
    E^2_{\ast,\ast} & \cong \H_\ast(\THH(\BP{n-1}/X(p^n)); \FF_p) \otimes_{\FF_p} \FF_p\pw{\hbar} \\
    & \cong \begin{cases}
    \FF_2[\sigma(\zeta_{n+1}), \hbar, \zeta_1^2, \cdots, \zeta_n^2, \zeta_j | j\geq n+1] & p=2 \\
    \FF_p[\sigma(\tau_n), \hbar, \zeta_i | i\geq 1] \otimes_{\FF_p} \Lambda_{\FF_p}[\tau_j | j\geq n] & p>2.
    \end{cases}
\end{align*}
There is a map to the above spectral sequence from the homological homotopy fixed points spectral sequence computing $\H_\ast^c(\TC^-(\BP{n-1}); \FF_p)$, and \cite[Proposition 6.1]{homological-hfpss} calculates that there are differentials $d^2(x) = \hbar \sigma(x)$ for every $x\in \H_\ast(\THH(\BP{n-1}/X(p^n)); \FF_p)$. For $j\geq n$, the following classes survive to the $E^3$-page:
\begin{align*}
    \zeta'_{j+1} & = \zeta_{j+1} + \zeta_j \sigma(\zeta_j) = \zeta_{j+1} + \zeta_j \sigma(\zeta_{n+1})^{2^{n+1-j}}, \ p=2\\
    \tau'_{j+1} & = \tau_{j+1} + \tau_j \sigma(\tau_j)^{p-1}, \ p>2.
\end{align*}
Moreover, (powers of) the classes $\sigma(\zeta_{n+1})$ at $p=2$ and $\sigma(\tau_n)$ at $p>2$ are simple $\hbar$-torsion: for example, $\hbar \sigma(\zeta_{n+1})^{2^{n+1-j}}$ is killed by a $d^2$-differential on $\zeta_j$, and the case for a general power of $\sigma(\zeta_{n+1})$ follows from taking a binary expansion of the exponent. This leaves
\begin{align*}
    E^3_{\ast,\ast} & \cong \FF_2\pw{\hbar}[\zeta_1^2, \cdots, \zeta_n^2, \zeta_{n+1}^2, \zeta'_{j+1} | j\geq n+1], \ p=2, \\
    E^3_{\ast,\ast} & \cong \FF_p\pw{\hbar}[\zeta_i | i\geq 1] \otimes_{\FF_p} \Lambda_{\FF_p}[\tau'_{j+1} | j \geq n], \ p>2,
\end{align*}
and the image of $\sigma$ in filtration zero (these classes being simple $\hbar$-torsion). We claim that the spectral sequence degenerates at the $E^3$-page, which then implies the desired result. (In the case of $\THH(\BP{n-1}/X(p^n))^{t\Cp}$, for instance, the class $\epsilon_1 \hbar^{1-p^n}$ plays the role of $\tau_n$ in $\H_\ast^c(\THH(\BP{n-1}/X(p^n))^{t\Cp}; \FF_p)$ for $p$ odd.) 
As with the proof of \cref{homology-thh-bpn}, this follows from \cite[Proposition 6.1]{homological-hfpss}: were there any differentials in the homological homotopy fixed points spectral sequence for $\H_\ast(\TC^-(\BP{n-1}/X(p^n)); \FF_p)$, there would also exist corresponding differentials in the homological homotopy fixed points spectral sequence for $\H_\ast(\TC^-(\BP{n-1}); \FF_p)$.

However, the statement of \cite[Proposition 6.1]{homological-hfpss} assumes that $\BP{n-1}$ admits the structure of an $\Eoo$-algebra; this is not necessary, since their appeal to \cite[Proposition 5.1]{homological-hfpss} only uses the existence of the Dyer-Lashof operations $Q_0$ and $Q_1$ on $\H_\ast(\THH(\BP{n-1}); \FF_p)$, which already exist in the homology of any $\E{2}$-algebra. It therefore suffices to know that $\THH(\BP{n-1})$ admits the structure of an $\E{2}$-algebra, which is a consequence of our assumption that $\BP{n-1}$ is an $\E{3}$-form of the truncated Brown-Peterson spectrum.
\end{proof}
\begin{proof}[Proof of \cref{thh-calculations}(a)]
We will ignore the contribution from $\B\Delta_n$ below: the contribution from this term is simply its homology.
We will first calculate $\pi_\ast \THH(\BP{n-1}/X(p^n))$ via the Adams spectral sequence
$$E_2^{\ast,\ast} = \Ext_{\cA_\ast}^{\ast,\ast}(\FF_p, \H_\ast(\THH(\BP{n-1}/X(p^n)); \FF_p)) \Rightarrow \pi_\ast \THH(\BP{n-1}/X(p^n))^\wedge_p.$$
Using \cref{homology-thh-bpn}(a), there is a change-of-rings isomorphism
$$E_2^{\ast,\ast} \cong \Ext_{\ce(n-1)_\ast}^{\ast,\ast}(\FF_p, \FF_p[\sigma(\zeta_{n+1})]) \cong \FF_p[\sigma(\zeta_{n+1}), v_j | 0\leq j\leq n-1],$$
where $v_j$ lives in bidegree $(s,t-s) = (1, 2p^j-2)$. The Adams spectral sequence is concentrated in even total degree (and therefore degenerates at the $E_2$-page). The class $\sigma(\zeta_{n+1})$ in degree $|\zeta_{n+1}| + 1 = 2p^n$ is denoted $\theta_n$, so that the above calculation says that there is an isomorphism
$$\pi_\ast \THH(\BP{n-1}/X(p^n)) \simeq \BP{n-1}[\B\Delta_n]_\ast[\theta_n].$$
Since $\THH(\BP{n-1}/X(p^n))\simeq \THH(\BP{n-1}) \otimes_{\THH(X(p^n))} X(p^n)$, we see that $\THH(\BP{n-1}/X(p^n))$ admits the structure of a $\THH(\BP{n-1})$-module. There is an $\E{2}$-map $\BP{n-1} \to \THH(\BP{n-1})$, so that $\THH(\BP{n-1}/X(p^n))$ acquires the structure of a $\BP{n-1}$-module by restriction of scalars. Therefore, each of the $\BP{n-1}_\ast$-module generators of $\pi_\ast \THH(\BP{n-1}/X(p^n))$ lift to maps of spectra from shifts of $\BP{n-1}$ to $\THH(\BP{n-1}/X(p^n))$. Moreover, the resulting map $\BP{n-1}[\B\Delta_n\times \Omega S^{2p^n+1}] \to \THH(\BP{n-1}/X(p^n))$ induces an isomorphism on homotopy by construction, so we obtain the first part of \cref{thh-calculations}(a).

The calculation for $\pi_\ast \THH(\BP{n-1}/X(p^n))^{t\Cpn}$ is similar to the calculation of $\pi_\ast \TC^-(\BP{n-1}/X(p^n))$ (and $\pi_\ast \TP(\BP{n-1}/X(p^n))$); moreover, it will be illustrative to calculate $\pi_\ast \TP(\BP{n-1}/X(p^n))$, since the case of $\pi_\ast \TC^-(\BP{n-1}/X(p^n))$ will just involve bookkeeping of the $\hbar$-torsion terms in \cref{homology-tc-bpn}. There is an Adams spectral sequence 
$$E_2^{\ast,\ast} = \Ext_{\cA_\ast}^{\ast,\ast}(\FF_p, \H_\ast^c(\TP(\BP{n-1}/X(p^n)); \FF_p)) \Rightarrow \pi_\ast \TP(\BP{n-1}/X(p^n))^\wedge_p,$$
which is in general only conditionally convergent, but is strongly convergent in this case. (This is because $\H_\ast(\THH(\BP{n-1}/X(p^n)); \FF_p)$ is bounded-below and of finite type.) By \cref{homology-tc-bpn}, there is a change-of-rings isomorphism 
$$E_2^{\ast,\ast} \cong \Ext_{\ce(n)_\ast}^{\ast,\ast}(\FF_p, \FF_p\ls{\hbar}) \cong \FF_p[v_j | 0\leq j\leq n]\ls{\hbar},$$
so that the Adams spectral sequence is concentrated in even total degree (and therefore degenerates at the $E_2$-page); this gives the desired calculation.
\end{proof}
\begin{remark}\label{theta-hbar-vn}
The homotopy fixed points spectral sequence for $\pi_\ast \TC^-(\BP{n-1}/X(p^n))$ has $E_2$-page given by
$$E_2^{\ast,\ast} = \BP{n-1}[\B\Delta_n]_\ast[\theta_n]\pw{\hbar}.$$
By evenness, this spectral sequence degenerates at the $E_2$-page. The calculation of \cref{thh-calculations}(a) tells us that the class $\hbar\theta_n$ on the $E_\infty$-page represents the class $v_n\in \pi_\ast \BP{n}$ (modulo decomposables). 

Note that \cref{thh-calculations}(a) says in particular that $\pi_\ast \THH(\BP{n-1}/X(p^n))^{t\Cp} \cong \pi_\ast \BP{n}^{t\Cp}[\B\Delta_n]$. There is an isomorphism $\pi_\ast \BP{n}^{t\Cp} \cong \pi_\ast \BP{n-1}^{tS^1}$ (which was proved in \cite[Proposition 2.3]{ando-morava-sadofsky}, and conjectured to lift to an equivalence of spectra in \cite[Conjecture 1.2]{five-author-tate}), so that $\pi_\ast \THH(\BP{n-1}/X(p^n))^{t\Cp} \cong \BP{n-1}[\B\Delta_n]_\ast\ls{\hbar}$. Note that unless $n=0$, this is \textit{not} isomorphic to $\pi_\ast \THH(\BP{n-1}/X(p^n))[\theta_n^{-1}]$, since $\pi_\ast \THH(\BP{n-1}/X(p^n))[\theta_n^{-1}]$ is $2p^n$-periodic, while $\pi_\ast \THH(\BP{n-1}/X(p^n))^{t\Cp}$ is $2$-periodic.
\end{remark}
\begin{proof}[Proof of \cref{thh-calculations}(b)]
We now calculate $\pi_\ast \THH(\BP{n}/X(p^n))$, this time with the use of Bockstein spectral sequences. (Similar arguments can be found in \cite{thh-truncated-bp}.) Again, we will ignore the contribution from $\B\Delta_n$ below: the contribution from this term is simply its homology. For simplicity, let us write
\begin{align*}
    x = \begin{cases}
    \sigma(\zeta_{n+1}^2) & p=2,\\
    \sigma(\zeta_{n+1}) & p>2
    \end{cases}, & \ y = \begin{cases}
    \sigma(\zeta_{n+2}) & p=2,\\
    \sigma(\tau_{n+1}) & p>2,
    \end{cases}
\end{align*}
so that $|x| = 2p^{n+1}-1$ and $|y| = 2p^{n+1}$. If $M$ is a (left) $\BP{n}$-module, let $\THH(\BP{n}/X(p^n); M)$ denote $\THH(\BP{n}/X(p^n)) \otimes_{\BP{n}} M$, so that we may informally view $\THH(\BP{n}/X(p^n); \FF_p)$ as $\THH(\BP{n}/X(p^n))/(p, \cdots, v_n)$. Using \cref{homology-thh-bpn}(b), one can show that
$$\pi_\ast \THH(\BP{n}/X(p^n); \FF_p) \cong \FF_p[x,y]/x^2;$$
we will compute $\THH(\BP{n}/X(p^n); \BP{n})$ using this calculation and $n+1$ Bockstein spectral sequences.
The $v_0$-Bockstein spectral sequence is given by
\begin{equation}\label{v0-bockstein}
    E_1^{\ast,\ast} = \pi_\ast \THH(\BP{n}/X(p^n); \FF_p)[v_0] \cong \FF_p[v_0,x,y]/x^2 \Rightarrow \pi_\ast \THH(\BP{n}/X(p^n); \Z_p).
\end{equation}
It follows from the Bockstein calculation in \cref{homology-thh-bpn}(b) that there is a $d_1$-differential
\begin{equation}\label{d2-y}
    d_1(y) = v_0 x,
\end{equation}
which implies $d_1(yv_0^n) = v_0^{n+1} x$ (by $\FF_p[v_0]$-linearity).
However, \cref{d2-y} does not immediately imply differentials on powers of $y$, since $\THH(\BP{n}/X(p^n))$ does not admit the structure of a ring (so the spectral sequence is not multiplicative). However, this is easily resolved: there is a map to the above Bockstein spectral sequence from the Bockstein spectral sequence computing $\pi_\ast \THH(\BP{n}; \Z_p)$, whose $E_1$-page is
$$'E_1^{\ast,\ast} \cong \pi_\ast \THH(\BP{n}; \FF_p)[v_0].$$
The calculation of $\H_\ast(\THH(\BP{n}); \FF_p)$ is described in \cite[Theorem 5.12]{angeltveit-rognes}; from this, one can compute $\pi_\ast \THH(\BP{n}; \FF_p)$. Here, we will only need to observe that the classes $x,y\in E_1^{\ast,\ast}$ lift along the map $'E_1^{\ast,\ast} \to E_1^{\ast,\ast}$. We will continue to denote these lifts by $x$ and $y$; there is still a $d_1$-differential $d_1(y) = v_0 x$ in $'E_1^{\ast,\ast}$. Since $\THH(\BP{n}; \Z_p)$ admits the structure of an $\E{2}$-ring, the above spectral sequence is multiplicative. Therefore, we may appeal to \cite[Proposition 6.8]{may-steenrod}, which gives higher differentials on powers of $y$. In particular, we claim:
\begin{equation}\label{diff-on-yj}
    d_{v_p(j)+1}(y^j) = v_0^{v_p(j)+1} xy^{j-1},
\end{equation}
up to a unit in $\FF_p^\times$. By taking base-$p$ expansions, it suffices to prove this differential when $j$ is a power of $p$, say $j=p^k$: then, \cref{diff-on-yj} says that $d_{k+1}(y^{p^k}) = v_0^{k+1} xy^{p^k - 1}$. Using \cite[Proposition 6.8]{may-steenrod} for $k>1$, we have
$$d_{k+1}((y^{p^{k-1}})^p) = v_0 (y^{p^{k-1}})^{p-1} d_{k}(y^{p^{k-1}}) = v_0 y^{p^k - p^{k-1}} d_k(y^{p^{k-1}});$$
this inductively implies \cref{diff-on-yj} once we establish the case $k=1$. 

For $p=2$, \cite[Proposition 6.8]{may-steenrod} says that
$$d_2(y^2) = v_0 y d_1(y) + Q_1(d_1(y)) = v_0^2 xy^2 + Q_1(v_0 x).$$
But 
$$Q_1(x) = Q_1(\sigma(\zeta_{n+1}^2)) = \sigma(Q_2(\zeta_{n+1}^2)) = \sigma(\zeta_{n+2}^2),$$
which is zero. Therefore, we see that $d_2(y^2) = v_0 xy^2$, as desired. For $p>2$, \cite[Proposition 6.8]{may-steenrod} says that
$$d_2(y^p) = v_0 y^{p-1} d_1(y) + \sum_{1\leq j\leq r} j[d_1(y) y^{j-1}, d_1(y) y^{p-j-1}],$$
for some integer $r$.
The ``correction'' term is a $v_0$-multiple of sum of terms of the form $[x y^{j-1}, x y^{p-j-1}]$. Note that this class lives in $\pi_\ast \THH(\BP{n}; \FF_p)$, but for the calculation of \cref{v0-bockstein}, we are only concerned with the image of this class in $\pi_\ast \THH(\BP{n}/X(p^n); \FF_p)$. We claim that the image of $[x y^{j-1}, x y^{p-j-1}]$ in $\pi_\ast \THH(\BP{n}/X(p^n); \FF_p)$ vanishes, so the correction terms above vanish. To prove this, observe that the Leibniz rule implies that, in $\pi_\ast \THH(\BP{n}; \FF_p)$, we have
\begin{align*}
    [x y^{j-1}, x y^{p-j-1}] & = x[y^{j-1}, xy^{p-j-1}] + y^{j-1}[x, xy^{p-j-1}] \\
    & = x^2[y^{j-1}, y^{p-j-1}] + xy^{p-j-1}[y^{j-1}, x] + y^{p-2}[x,x] + xy^{j-1}[x, y^{p-j-1}].
\end{align*}
Here, all terms are written up to sign; this will not matter, since we will show that each of the terms in the sum above vanish. The first term vanishes since $x^2 = 0$, and the third term vanishes since $[x,x] = 0$. For the second and fourth term, we will argue more generally that the image of $[x,y^k]$ in $\pi_\ast \THH(\BP{n}/X(p^n); \FF_p)$ vanishes for any $k\geq 0$. The Leibniz rule implies that $[x,y^k] = ky^{k-1}[x,y]$, so it suffices to show that the image of $[x,y]$ in $\pi_\ast \THH(\BP{n}/X(p^n); \FF_p)$ vanishes. 

Since $[x,y]$ lives in degree $|x| + |y| + 1 = (2p^{n+1}-1) + 2p^{n+1} + 1 = 4p^{n+1}$ and $\pi_{4p^{n+1}} \THH(\BP{n}/X(p^n); \FF_p) \cong \FF_p\{y^2\}$, we must have $[x,y] \dot{=} y^2$ in $\pi_\ast \THH(\BP{n}/X(p^n); \FF_p)$ if $[x,y]$ is nonzero. To show that $[x,y] \dot{\neq} y^2$, we observe that the $\E{2}$-map $\iota: \THH(\BP{n}; \FF_p) \to \THH(\BP{n}/\MU; \FF_p)$ factors through $\THH(\BP{n}/X(p^n); \FF_p)$. The classes $x$ and $y$ are in the image of the map $\THH(\BP{n}; \FF_p) \to \THH(\BP{n}/X(p^n); \FF_p)$, and $x$ is killed by the map $\iota$. Since $\iota$ is an $\E{2}$-map, we must have $\iota([x,y]) = [\iota(x), \iota(y)] = 0$; however, $\iota(y^2) = \iota(y)^2$ is nonzero. Therefore, $[x,y] \dot{\neq} y^2$; but since $\pi_{4p^{n+1}} \THH(\BP{n}/X(p^n); \FF_p)$ is a $1$-dimensional $\FF_p$-vector space spanned by $y^2$, we must have $[x,y] = 0$.

The upshot of this discussion is that the $E_r$-page of \cref{v0-bockstein} is given by 
$$E_r^{\ast,\ast} = \FF_p[v_0, y^{p^{r-1}}]\{1, x, xy, xy^2, \cdots\}/(v_0^i xy^{p^{i-1} j - 1}, 1\leq i \leq r-1, 1\leq j \leq p-1).$$
In particular, no power of $y$ survives to the $E_\infty$-page, and since $v_0$ represents $p$, we can resolve the $v_0$-extensions to conclude that
\begin{equation}\label{zp-homotopy-thh}
    \pi_\ast \THH(\BP{n}/X(p^n); \Z_p) \cong \Z_p \oplus \bigoplus_{j\geq 1} \Z_p/p^{v_p(j)+1}\{xy^{j-1}\}.
\end{equation}
Note that $|xy^{j-1}| = 2jp^{n+1}-1$.

The higher Bockstein spectral sequences (for $v_1,\cdots, v_n$) all collapse at the $E_1$-page for degree reasons, as we now explain. 
For the $v_m$-Bockstein spectral sequence with $1\leq m\leq n$, one can argue by induction on $m$ (the base case is the same argument as the inductive step). First, observe that $v_1, \cdots, v_n$ survive the Bockstein spectral sequence, since $\BP{n}$ splits off $\THH(\BP{n}/X(p^n))$. In particular, there cannot be any differential with target given by a product of monomials in the $v_i$s. By $\Z_p[v_1, \cdots, v_m]$-linearity, any differential must therefore be of the form 
$$d_r(xy^{j-1}) = v_{i_1}^{r_1} \cdots v_{i_a}^{r_a} v_m^r xy^{k-1}$$
for some $j,k$, exponents $r_1, \cdots, r_a$, and $1\leq i_1, \cdots, i_a < m$. (More precisely, it will be a sum of monomials of the above form, but this point will not matter.) But $d_r(xy^{j-1})$ has bidegree $(t-s,s) = (2jp^{n+1}-2, r)$, while $v_{i_1}^{r_1} \cdots v_{i_a}^{r_a} v_m^r xy^{k-1}$ has bidegree $(t-s,s) = (2r_1(p^{i_1}-1) + \cdots + 2r_a(p^{i_a}-1) + 2r(p^m-1) + 2kp^{n+1}-1, r)$. Such a differential is therefore not possible, since $2jp^{n+1}-2$ is even, while $2r_1(p^{i_1}-1) + \cdots + 2r_a(p^{i_a}-1) + 2r(p^m-1) + 2kp^{n+1}-1$ is odd. The calculation of $\pi_\ast \THH(\BP{n}/X(p^n))$ now follows from \cref{zp-homotopy-thh}.

Since $\THH(\BP{n}/X(p^n))\simeq \THH(\BP{n}) \otimes_{\THH(X(p^n))} X(p^n)$, we see that $\THH(\BP{n}/X(p^n))$ admits the structure of a $\THH(\BP{n})$-module. There is an $\E{2}$-map $\BP{n} \to \THH(\BP{n})$, so that $\THH(\BP{n}/X(p^n))$ acquires the structure of a $\BP{n}$-module by restriction of scalars. Therefore, each of the $\BP{n}_\ast$-module generators of $\pi_\ast \THH(\BP{n}/X(p^n))$ lift to maps of spectra from shifts of $\BP{n}$ to $\THH(\BP{n}/X(p^n))$. Moreover, the resulting map $\BP{n}[\B\Delta_n] \oplus \bigoplus_{j\geq 1} \Sigma^{2jp^{n+1}-1} \BP{n}[\B\Delta_n]/p^{v_p(j)+1} \to \THH(\BP{n}/X(p^n))$ induces an isomorphism on homotopy by construction, so we obtain \cref{thh-calculations}(b).
\end{proof}

\begin{remark}\label{bockstein-serre}
When $n=0$, one may view the Bockstein calculation of \cref{thh-calculations}(b) as a translation of the Serre spectral sequence for the fibration \cref{toda-fibration}. Assume that $p>2$. Indeed, the Serre spectral sequence is given by
$$E^2_{\ast,\ast} = \H_\ast(S^{2p-1}; \Z_p) \otimes \H_\ast(\Omega S^{2p+1}; \Z_p) \cong \Z_p[x,y]/x^2 \Rightarrow \H_\ast(\Omega S^3\pdb{3}; \Z_p).$$
There is a single family of differentials, determined multiplicatively from
$$d^{2p}(y) = px;$$
this implies that $d^{2p}(y^m) = mpy^{m-1} x$. The Serre spectral sequence collapses at the $E^{2p+1}$-page, and the resulting answer is precisely \cref{zp-homotopy-thh}. In fact, if $\phi_n: \Omega^2 S^{2p^n+1}\to S^{2p^n-1}$ is a charming map in the sense of \cite[Definition 4.1.1]{bpn-thom} (such as the Cohen-Moore-Neisendorfer map of \cite{cmn-1, cmn-2, cmn-3}), the proof of \cref{thh-calculations}(b) can be understood as a calculation of $\pi_\ast \BP{n-1}[\B\fib(\phi_n)]$ using the Serre spectral sequence for the Cohen-Moore-Neisendorfer type fibration
\begin{equation}\label{b-fib-phin}
    S^{2p^n-1} \to \B\fib(\phi_n) \to \Omega S^{2p^n+1}.
\end{equation}
The Serre spectral sequence for \cref{b-fib-phin} is exactly the same as that of \cref{toda-fibration}: the $E^2$-page is given by
$$E^2_{\ast,\ast} = \H_\ast(S^{2p^n-1}; \Z_p) \otimes \H_\ast(\Omega S^{2p^n+1}; \Z_p) \cong \Z_p[x,y]/x^2 \Rightarrow \H_\ast(\B\fib(\phi_n); \Z_p).$$
There is a single family of differentials, determined multiplicatively from
$$d^{2p^n}(y) = px;$$
this implies that $d^{2p^n}(y^m) = mpy^{m-1} x$, and the Serre spectral sequence collapses at the $E^{2p^n+1}$-page. The upshot is that
$$\H_i(\B\fib(\phi_n); \Z_p) \cong \begin{cases}
\Z_p & i=0,\\
\Z_p/pk & 2kp^n-1, \\
0 & \text{else}.
\end{cases}$$
In fact, \cref{thh-calculations}(b) implies that there is an equivalence of $\BP{n-1}$-modules 
$$\THH(\BP{n-1}/X(p^{n-1})) \simeq \BP{n-1}[\B\Delta_{n-1} \times \B\fib(\phi_{n})].$$
The calculations of \cref{thh-calculations} can be predicted from the results of \cite{bpn-thom}. Let us suppose that $p$ is odd for simplicity. Assuming \cite[Conjectures D and E]{bpn-thom}, \cite[Corollary B]{bpn-thom} implies that there is a map $\Omega^2 S^{2p^n+1} \to \BGL_1(X(p^n))$ whose Thom spectrum is $\BP{n-1}[\Omega\Delta_n]$. This implies that there is an equivalence of spectra $\THH(\BP{n-1}/X(p^n)) \simeq \BP{n-1}[\B\Delta_n \times \Omega S^{2p^n+1}]$; this is precisely the first part of \cref{thh-calculations}(a). Moreover, \cite[Theorem A]{bpn-thom} says (still assuming the aforementioned conjectures) that the Thom spectrum of the composite $\fib(\phi_n) \to \Omega^2 S^{2p^n+1} \to \BGL_1(X(p^n))$ is $\BP{n}[\Omega \Delta_n]$. This can be shown to imply that $\pi_\ast \TP(\BP{n-1}/X(p^n)) \simeq \pi_\ast \BP{n}^{tS^1}[\B\Delta_n]$, which is indeed confirmed by \cref{thh-calculations}(a). This result also implies that there is an equivalence of spectra $\THH(\BP{n}/X(p^n)) \simeq \BP{n}[\B\Delta_n \times \B\fib(\phi_{n+1})]$, which is indeed true by \cref{thh-calculations}(b).
We will state the results predicted by this discussion as a conjecture.
\end{remark}
\begin{conjecture}\label{conjecture-thh}
Fix an $\E{3}$-form of the truncated Brown-Peterson spectrum $\BP{n-1}$. Then $\THH(\BP{n-1}/X(p^n))$ admits the structure of an $S^1$-equivariant $\BP{n}$-module (where $S^1$ acts trivially on $\BP{n}$), and the equivalences of \cref{thh-calculations}(a) refine to $p$-complete equivalences of spectra
\begin{align*}
    \THH(\BP{n-1}/X(p^n))^{t\Cpn} & \simeq \BP{n}^{t\Cpn}[\B\Delta_n], \\
    \TP(\BP{n-1}/X(p^n)) & \simeq \BP{n}^{tS^1}[\B\Delta_n].
\end{align*}
The first equivalence is $S^1$-equivariant for the residual $S^1/\mu_m$-action on $\THH(\BP{n-1}/X(p^n))^{t\Cpn}$ and $\BP{n}^{t\Cpn}$.
\end{conjecture}
\begin{remark}
The primary difficulty with proving \cref{conjecture-thh} is that it is not clear how to endow $\TP(\BP{n-1}/X(p^n))$ or $\THH(\BP{n-1}/X(p^n))^{t\Cpn}$ with the structure of $\BP{n}$-modules. Nevertheless, a small part of the final equivalence in \cref{conjecture-thh} can be proved unconditionally when $n=1$. Namely, there is a map $\TP(\Z_p/X(p)) \to \bigoplus_{j>-(p-1)} \Sigma^{2j} \BP{1}$ which induces the inclusion of summands on mod $p$ cohomology. (This is the ``easy'' range, since the first predicted summand of $\TP(\Z_p/X(p))$ which is not covered by this claim is $\Sigma^{-2(p-1)} \BP{1}$; but $\pi_0$ of this spectrum this is exactly where the class $v_1$ lives.)
We computed the mod $p$ homology of $\TP(\Z_p/X(p))$ in \cref{homology-tc-bpn}. This implies that $\H^{\ast,c}(\TP(\Z_p/X(p)); \FF_p) \cong \H^\ast(\BP{1}; \FF_p)\ls{\hbar} \otimes_{\FF_p} \H^\ast(\BSU(p-1); \FF_p)$. There is an Adams spectral sequence
$$\Ext^{s,t+2j}_{\cA_\ast}(\cA\mmod \ce(1), \cA \mmod \ce(1))\ls{\hbar} \otimes_{\FF_p} \H^\ast(\BSU(p-1); \FF_p) \Rightarrow \pi_0\Map(\TP(\Z_p/X(p)), \Sigma^{2j} \BP{1})^\wedge_p.$$
We wish to show that for $j>-(p-1)$, any class in bidegree $(s,t-s) = (0, 2j)$ survives to the $E_\infty$-page. For this, it suffices to show that there can be no nonzero $d_r$-differential off this class for $r\geq 2$. This differential would necessarily land in $(r, 2j-1)$. By \cite[Proposition 4.1]{adams-priddy}, $\Ext^{s,t}_{\cA_\ast}(\cA\mmod \ce(1), \cA \mmod \ce(1))$ vanishes for $s\geq 1$, $t-s$ odd, and $t-s\geq -2(p-1)$. In particular, we see that taking $(s,t-s) = (r, 2j-1)$, we have $2j-1\geq -2(p-1)$ precisely when $j>-(p-1)$. Therefore, we get a map $\TP(\Z_p/X(p)) \to \Sigma^{2j} \BP{1}$ for every $j>-(p-1)$, which gives the desired claim.
\end{remark}

%% file: bokstedt/over-laurent.tex
In \cref{thh-calculations}, we saw a ``polynomial'' generator in degree $2p^n$, where $n$ is the height. When $n=0$, this reduces the B\"okstedt generator in degree $2$; we will now discuss a variant of \cref{thh-calculations} when $n=1$, where one obtains a generator in degree $2$.
\begin{construction}\label{xp-polynomial}
Let $\U(1) \to \SU(p)$ denote the inclusion given by the homomorphism
$$\lambda\mapsto 
\mathrm{diag}(\lambda, \cdots, \lambda, \lambda^{1-p}).$$
There is an induced map $\BU(1) \to \BSU(p)$, which defines an $\E{2}$-map $\Omega\U(1) \simeq \Z \to \Omega \SU(p)$. Let $J(p)$ denote the Thom spectrum of the composite $\E{2}$-map $\mu: \Omega\U(1) \to \Omega \SU(p) \to \Omega\SU \simeq \BU$. Then $J(p)$ admits an $\Efr{2}$-structure by \cref{xn-framed-e2} such that there is an $\Efr{2}$-algebra map $J(p) \to X(p)$.
Note that the underlying $\E{1}$-map of $\mu$ is null, since $\B\mu: S^1 \to \mathrm{B}^2\U \simeq \SU$ is a class in $\pi_1(\SU) = 0$. Therefore, the underlying $\E{1}$-ring of $J(p)$ is $S[\Z] = S[t^{\pm 1}]$. Moreover, the underlying $\E{1}$-map of $J(p) \to X(p) \to \Z_p$ is the map $S[t^{\pm 1}] \to \Z_p$ sending $t\mapsto 1$.
\end{construction}
\begin{prop}\label{t1-jp}
There is an equivalence $\THH(T(1)/J(p)) \simeq T(1)[J_{p-1}(S^2)]$. Similarly, $\THH(X(p)/J(p)) \simeq X(p)[J_{p-1}(S^2) \times \SU(p-1)]$.
\end{prop}
\begin{proof}
Indeed, $\THH(T(1)/J(p)) \simeq \THH(T(1)) \otimes_{\THH(J(p))} J(p)$ is equivalent to $T(1)[S^{2p-1}] \otimes_{T(1)[S^1]} T(1)$; but there is a fiber sequence
$$S^1 \to S^{2p-1} \to S^{2p-1}/S^1 = \CP^{p-1} \simeq J_{p-1}(S^2),$$
from which the desired claim follows.
\end{proof}
\begin{prop}\label{laurent-xp}
The following statements are true:
\begin{enumerate}
    \item There is an equivalence $\THH(\Z_p/J(p)) \simeq \Z_p[\Omega S^3]$. In particular, $\pi_\ast \THH(\Z_p/J(p)) \cong \Z_p[x]$ with $|x| = 2$.
    On homotopy, the map $\THH(\Z_p/J(p)) \to \THH(\Z_p/X(p))$ is given by
    $$x^j\mapsto \begin{cases}
    \theta^{j/p} & j\in p\Z,\\
    0 & \text{else}.
    \end{cases}$$
    \item The canonical map $\THH(\Z_p/J(p)) \to \THH(\FF_p/J(p))$ factors through the unit $\THH(\FF_p) \to \THH(\FF_p/J(p))$, and defines an equivalence $\FF_p\otimes_{\Z_p} \THH(\Z_p/J(p)) \xar{\sim} \THH(\FF_p)$ of $\THH(\Z_p)$-modules.
\end{enumerate}
\end{prop}
\begin{proof}
For part (a), we begin by observing that there is an equivalence
$$\THH(\Z_p/J(p)) \simeq \THH(\Z_p) \otimes_{\THH(J(p))} J(p) \simeq \Z_p[\Omega S^3\pdb{3}] \otimes_{\Z_p[\U(1)]} \Z_p.$$
The map $\Z_p\otimes_{J(p)} \THH(J(p)) \to \THH(\Z_p)$ factors through $\Z_p \otimes_{X(p)} \THH(X(p)) \to \THH(\Z_p)$, and can be identified with $\Z_p$-chains of the composite
$$\U(1) \to \SU(p) \to S^{2p-1} \xar{\alpha_1} \Omega S^3\pdb{3}.$$
Note that the map $\U(1) \to S^{2p-1}$ is the fiber of the map $S^{2p-1} \to \CP^{p-1}$.
This composite can be identified with action of $S^1$ on $\Omega S^3\pdb{3}$. Since there is a fiber sequence
$$S^1 \to \Omega S^3\pdb{3} \to \Omega S^3,$$
we see that $\THH(\Z_p/J(p))\simeq \Z_p[\Omega S^3]$. To identify the map $\THH(\Z_p/J(p)) \to \THH(\Z_p/X(p))$, observe that $\CP^{p-1} \simeq J_{p-1}(S^2)$ and that there is a square where each row and column is a fiber sequence:
$$\xymatrix{
\Omega (\SU(p-1) \times \CP^{p-1}) \simeq \Omega \SU(p)/S^1 \ar[r] \ar[d] & \ast \ar[r] \ar[d] & \CP^{p-1} \times \SU(p-1) \ar[d] \\
S^1 \ar[r] \ar[d] & \Omega S^3\pdb{3} \ar[r] \ar[d] & \Omega S^3 \ar[d]^-{H_p}\\
\SU(p) \ar[r] & \Omega S^3\pdb{3} \ar[r] & \Omega S^{2p+1} \times \BSU(p-1).
}$$
The effect of the map $\THH(\Z_p/J(p)) \to \THH(\Z_p/X(p))$ is dictated by the bottom-right vertical map, which is induced by the James-Hopf map $H_p: \Omega S^3 \to \Omega S^{2p+1}$. On $\Z_p$-homology, the effect of the James-Hopf map is as stated in \cref{laurent-xp}(a).

For part (b), there is an equivalence
$$\THH(\FF_p/J(p)) \simeq \THH(\FF_p) \otimes_{\THH(J(p))} J(p) \simeq \FF_p[\Omega S^3] \otimes_{\FF_p[\U(1)]} \FF_p.$$
However, the map $\FF_p \otimes_{J(p)} \THH(J(p)) \to \THH(\FF_p)$ factors through $\FF_p \otimes_{\Z_p} \THH(\Z_p) \to \THH(\FF_p)$, and can be identified with $\FF_p$-chains of the composite of $\U(1) \to \Omega S^3\pdb{3}$ with the canonical map $\Omega S^3\pdb{3} \to \Omega S^3$. This composite is null as an $\E{1}$-map (in fact, as an $\E{2}$-map), since there is a fiber sequence of $\E{1}$-spaces
$$\BU(1) \simeq \CP^\infty \to S^3\pdb{3} \to S^3.$$
Therefore, we see that 
$$\THH(\FF_p/J(p)) \simeq \FF_p[\Omega S^3] \otimes_{\FF_p} (\FF_p \otimes_{\FF_p[\U(1)]} \FF_p) \simeq \FF_p[\Omega S^3 \times \CP^\infty].$$
This implies that the map $\THH(\Z_p/J(p)) \to \THH(\FF_p/J(p))$ factors through $\THH(\FF_p) \to \THH(\FF_p/J(p))$. In turn, we obtain a map $\FF_p\otimes_{\Z_p} \THH(\Z_p/J(p)) \to \THH(\FF_p)$ which sends the generators in $\pi_\ast(\FF_p\otimes_{\Z_p} \THH(\Z_p/J(p))) \cong \FF_p[x]$ to the generators in $\pi_\ast \THH(\FF_p) \cong \FF_p[\sigma]$. Therefore, the map $\FF_p\otimes_{\Z_p} \THH(\Z_p/J(p)) \to \THH(\FF_p)$ is an equivalence, as desired.
\end{proof}
\begin{remark}\label{pth-power-sigma}
The map $J(p) \to X(p)$ induces a map $u: \THH(\Z_p/J(p)) \to \THH(\Z_p/X(p))$. Under \cref{thh-calculations} and \cref{laurent-xp}, the map $u$ can be identified with the $\Z_p$-chains of the composite
$$\Omega S^3 \to \Omega S^{2p+1} \to \Omega S^{2p+1} \times \BSU(p-1);$$
here, the map $\Omega S^3 \to \Omega S^{2p+1}$ is the Hopf map. This claim follows from the proof of \cref{laurent-xp}, \cref{t1-jp}, and the EHP fibration
$$J_{p-1}(S^2) \to \Omega S^3 \to \Omega S^{2p+1}.$$
In particular, the map $u$ induces the map $\Z_p[x] \to \Z_p[\theta] \otimes_{\Z_p} \Z_p[\BSU(p-1)]$ which sends $x^m\mapsto \theta^{m/p}$ if $p\mid m$ and $x^m\mapsto 0$ otherwise.

Note that if $T(1)$ were an $\Efr{2}$-algebra, the map $u$ would factor through $\THH(\Z_p/J(p)) \to \THH(\Z_p/T(1))$; and under the equivalences of \cref{thh-calculations} and \cref{laurent-xp}, this would identify with the $\Z_p$-chains of the Hopf map.
\end{remark}
\begin{remark}
\cref{laurent-xp} demonstrates the dependence of $\THH(R'/R)$ on the $\E{1}$-$R$-algebra structure on $R'$. Indeed, recall that the underlying $\E{1}$-map of the $\E{2}$-map $J(p) \to X(p) \to \Z_p$ is the map $S[t^{\pm 1}] \to \Z_p$ sending $t\mapsto 1$. \cref{laurent-xp} states that $\THH(\Z_p/J(p)) \simeq \Z_p[\Omega S^3]$. However, suppose that $S[t^{\pm 1}] = S[\Z]$ is equipped with its standard $\E{2}$-structure, and $\Z_p$ is viewed as an $\E{1}$-$S[\Z]$-algebra via the composite $S[\Z] \to S \to \Z_p$. Then $\THH(\Z_p/S[\Z]) \simeq \THH(\Z_p) \otimes S[\CP^\infty] \simeq \Z_p[\Omega S^3\pdb{3} \times \CP^\infty]$. Since $\Z_p[\Omega S^3\pdb{3} \times \CP^\infty] \not \simeq \Z_p[\Omega S^3]$, we conclude that $\THH(\Z_p/S[\Z]) \not \simeq \THH(\Z_p/J(p))$.
\end{remark}
\begin{corollary}\label{tp-jp}
There is an isomorphism $\pi_\ast \TP(\Z_p/J(p)) \simeq \Z_p[t^{\pm 1}]^\wedge_{(t-1)}\ls{\hbar}$ with $|\hbar| = -2$. 
\end{corollary}
\begin{corollary}\label{mod-p-thh}
If $\cC$ is a $\Z_p$-linear $\infty$-category, there is a (non-$S^1$-equivariant) equivalence $\THH(\cC/J(p)) \otimes_{\Z_p} \FF_p \simeq \THH(\cC\otimes_{\Z_p} \FF_p)$.
\end{corollary}
\begin{proof}
By \cref{laurent-xp}(b), there is an equivalence $\THH(\Z_p/J(p)) \otimes_{\Z_p} \FF_p \simeq \THH(\FF_p)$ of $\THH(\Z_p)$-modules. It follows that
\begin{align*}
    \THH(\cC/J(p)) \otimes_{\Z_p} \FF_p & \simeq \THH(\cC) \otimes_{\THH(\Z_p)} \THH(\Z_p/J(p)) \otimes_{\Z_p} \FF_p \\
    & \xar{\sim} \THH(\cC) \otimes_{\THH(\Z_p)} \THH(\FF_p) \simeq \THH(\cC\otimes_{\Z_p} \FF_p),
\end{align*}
as desired.
\end{proof}
\begin{remark}
Recall from \cite[Theorem 3.5]{amn-kunneth} that if $S[z] = S[\Z_{\geq 0}]$ denotes the flat polynomial ring on a class in degree $0$, then there is an isomorphism $\pi_\ast \THH(\Z_p/S[z]) \cong \Z_p[\sigma^2(z-p)]$, where the $\Eoo$-map $S[z] \to \Z_p$ sends $z\mapsto p$. This implies that $\pi_\ast \TP(\Z_p/S[z]) \cong \Z_p[z]^\wedge_{(z-p)}\ls{\hbar}$. Similarly, there is an isomorphism $\pi_\ast \TP(\Z_p/S\pw{\ptl}) \cong \Z_p\pw{\ptl}^\wedge_{(\ptl-p)}\ls{\hbar}$, where $\ptl\mapsto p$ and $S\pw{\ptl} = \left(S[q^{\pm 1}]^\wedge_{(p,q-1)}\right)^{h\FF_p^\times}$.

In the same way, there is an isomorphism $\pi_\ast \THH(\Z_p/S[t^{\pm 1}]) \cong \Z_p[\sigma^2(t+p-1)]$, where the $\Eoo$-map $S[t^{\pm 1}] \to \Z_p$ sends $t\mapsto 1-p$. This implies that $\pi_\ast \TP(\Z_p/S[t^{\pm 1}]) \cong \Z_p[t^{\pm 1}]^\wedge_{(t+p-1)}\ls{\hbar}$. In light of the obvious analogy to \cref{laurent-xp} and \cref{tp-jp}, it is natural to ask: what is the role of $J(p)$?

To answer this, let us assume for simplicity that $T(1)$ admits the structure of an $\E{2}$-ring. The main utility of $J(p)$ is that it admits, by construction, a direct comparison to $T(1)$; one can view $J(p)$ as containing roughly the same ``height $1$'' information as $T(1)$. On the other hand, we do not know how to directly compare $S[t^{\pm 1}]$ (with the standard $\E{2}$-structure) to $T(1)$. (Both admit $\E{1}$-algebra maps to $T(1)[t^{\pm 1}]$, but this is somewhat unsatisfactory.) One can therefore view \cref{xp-polynomial} as an explicit modification of the $\E{2}$-structure on $S[t^{\pm 1}]$ such that the resulting $\E{2}$-algebra admits an interesting map to $T(1)$.
\end{remark}

It is natural to ask if \cref{laurent-xp} admits a generalization to $\BP{n-1}$.
At height $1$ and $p=2$, we can explicitly construct some $\Efr{2}$-rings which give higher analogues of $J(p)$, but a general construction at higher heights and other primes eludes us.
\begin{construction}\label{higher-jp}
Recall from \cref{t2-thom} that there is an $\E{2}$-map $\Omega \Sp(2) \to \BU$ whose Thom spectrum is equivalent to $T(2)$ at $p=2$.
Let $T_2(2)$ denote the $\Efr{2}$-ring defined as the Thom spectrum of the composite $\E{2}$-map
$$\Omega\Spin(4) \to \Omega \Sp(2) \to \BU,$$
where the first map is induced by the inclusion $\Spin(4) \subseteq \Spin(5) \cong \Sp(2)$. Similarly, let $T_4(2)$ denote the $\Efr{2}$-ring defined as the Thom spectrum of the composite $\E{2}$-map
$$\Omega\U(2) \to \Omega \Sp(2) \to \BU,$$
where the first map is induced by the inclusion $\U(2) \subseteq \Sp(2)$. Note that this inclusion factors as $\U(2) \to \Spin(4) \to \Sp(2)$, so that there is a composite map of $\Efr{2}$-rings
$$T_4(2) \to T_2(2) \to T(2).$$
\end{construction}
\begin{remark}
There is a fiber sequence
$$\Omega S^3 \to \Omega \Spin(4) \to \Omega S^3,$$
which implies that $\MU_\ast(T_2(2)) \simeq \MU_\ast[t_1, x_2]$ where $|x_2| = 2$. Similarly, there is a fiber sequence
$$\Omega S^3 \to \Omega \U(2) \to \Omega S^1 \simeq \Z,$$
which implies that $\MU_\ast(T_4(2)) \simeq \MU_\ast[t_1, x_0^{\pm 1}]$ where $|x_0| = 0$.
\end{remark}
\begin{lemma}\label{sp2-quotients}
There is a diffeomorphism $\Sp(2)/\Spin(4) \cong S^4$, as well as a homotopy equivalence $\Sp(2)/\U(2) \simeq J_3(S^2)$.
\end{lemma}
\begin{proof}
The first diffeomorphism follows immediately from the isomorphism $\Sp(2) \cong \Spin(5)$ and the resulting chain
$$\Sp(2)/\Spin(4) \cong \Spin(5)/\Spin(4) \cong \SO(5)/\SO(4) \cong S^4.$$
To prove the second equivalence, the key input is \cite[Proposition 4.3]{amelotte}, which says that there is a fiber sequence
$$V_2(\RR^5) \to J_3(S^2) \to \CP^\infty;$$
in other words, there is an $S^1$-action on the Stiefel manifold $V_2(\RR^5)$ such that $V_2(\RR^5)/S^1 \cong J_3(S^2)$.
Recall that $V_2(\RR^5)$ is diffeomorphic to $\SO(5)/\SO(3) \cong \Spin(5)/\SU(2)$. It is not difficult to see that the claimed $S^1$-action on $V_2(\RR^5)$ via the above fiber sequence is precisely the residual action of $\U(2)/\SU(2) \cong S^1$ on $\Spin(5)/\SU(2)$; in particular, we may identify $J_3(S^2) \simeq \Spin(5)/\U(2)$, as desired.
\end{proof}
\begin{remark}
The quotient $\Sp(2)/\U(2)$ is also known as the complex Lagrangian Grassmannian $\Gr_2^\Lag(T^\ast \cc^2)$ of Lagrangian subspaces of $T^\ast \cc^2$.
\end{remark}
\begin{warning}
One should not confuse $\Sp(2)/\U(2)$ with the quotient $\Sp(2)/(\Sp(1)\times \U(1))$: indeed, \cref{sp2-quotients} says that the former is homotopy equivalent to $J_3(S^2)$, while the latter is diffeomorphic to $S^7/\U(1) = \CP^3$. These spaces are not homotopy equivalent (although they do become equivalent after inverting $6$).
\end{warning}
\cref{sp2-quotients} has the following amusing (inconsequential?) consequence:
\begin{corollary}
Let $Q \subseteq \CP^4$ be a complex quadric, and let $\Gr_2^+(\RR^5)$ denote the Grassmannian of oriented $2$-planes in $\RR^5$. Then, there are diffeomorphisms $Q \cong \Gr_2^\Lag(T^\ast \cc^2) \cong \Gr_2^+(\RR^5)$, and these are homotopy equivalent to $J_3(S^2)$.
\end{corollary}
\begin{proof}
Since $\Sp(2)/\U(2) \cong \SO(5)/(\SO(3)\cdot \SO(2))$, we can identify $\Sp(2)/\U(2) = \Gr_2^\Lag(T^\ast \cc^2)$ with $\Gr_2^+(\RR^5)$. Therefore, \cref{sp2-quotients} gives a homotopy equivalence $\Gr_2^+(\RR^5) \simeq J_3(S^2)$. The desired claim now follows from the observation that $\Gr_2^+(\RR^5)$ is diffeomorphic to a quadric $Q \subseteq \CP^4$ via the map $\Gr_2^+(\RR^5) \to \Gr_1(\cc^5) \cong \CP^4$ induced by the isomorphism $\RR^{10} \xar{\sim} \cc^5$; see \cite[Example 10.6, Page 280]{kobayashi-nomizu}.
\end{proof}
\begin{remark}
There is a fibration\footnote{The fibration \cref{s2 J3S2 fibration} is analogous to the ``twistor'' fibration (see \cref{s2-bundle-hpn}) $S^2 \to \CP^3 \to S^4$.} (see \cref{partial-ehp} for a more general statement)
\begin{equation}\label{s2 J3S2 fibration}
    S^2 \to J_3(S^2) \to S^4,
\end{equation}
which, under the diffeomorphism 
$$\Spin(4)/\U(2) \cong (\SU(2)\times \SU(2))/\U(2) \cong \SU(2)/\U(1) \cong S^2,$$
can be identified via \cref{sp2-quotients} with the fibration
$$\Spin(4)/\U(2) \to \Sp(2)/\U(2) \to \Sp(2)/\Spin(4).$$
There is also a commutative diagram where each row and column is a fibration:
$$\xymatrix{
\U(1) \ar[r] \ar[d] & \U(2) \ar[r] \ar[d] & S^3 \ar[d] \\
\Sp(1) \ar[r] \ar[d] & \Sp(2) \ar[r] \ar[d] & S^7 \ar[d] \\
S^2 \ar[r] & J_3(S^2) \ar[r] & S^4;
}$$
the rightmost vertical fiber sequence is the Hopf fibration. This diagram captures the relationships between $J(2)$, $T_4(2)$, $T(1)$, and $T(2)$.
\end{remark}
\begin{remark}\label{T2 and y2 relationship}
The equivalence $\Sp(2)/\U(2) = \Gr_2^\Lag(T^\ast \cc^2) \simeq J_3(S^2)$ of \cref{sp2-quotients} can be used to understand the relationship between $T(2)$ and the Mahowald-Ravenel-Shick spectrum $y(2)$ from \cite{mrs} (at the prime $2$).\footnote{A simpler version of this discussion simply states that if $\Omega S^2 \to \BO$ is the map extending the M\"obius bundle $S^1 \to \BO$, then \cite[Proposition 2.1.6]{bpn-thom} along with loops on the fibration
$$S^3 \xar{\eta} S^2 \to \CP^\infty$$
implies that there is a map $S^1 \to \BGL_1(T(1))$ whose Thom spectrum is the $\E{1}$-quotient $S\mmod 2 = y(1)$. The map $S^1 \to \BGL_1(T(1))$ detects $1-2\in \pi_0(T(1))^\times$ on the bottom cell of the source, so we recover the fact that $T(1)/2 \simeq y(1)$. In particular, $\HH(y(1)/T(1)) \simeq y(1)[\CP^\infty]$. Since $y(1) \otimes_{T(1)} \Z_2 \simeq \FF_2$, this recovers the well-known observation that $\HH(\FF_2/\Z_2) \simeq \FF_2[\CP^\infty]$, at least as \textit{modules} over $\FF_2$. This argument does not give the \textit{$\FF_2$-algebra} structure, since $\HH(y(1)/T(1))$ is not a ring.}
Recall from \cref{t2-thom} that there is an $\E{2}$-map $\Omega \Sp(2) \to \BU$ whose Thom spectrum is equivalent to $T(2)$ at $p=2$.
Similarly, recall that $y(2)$ is the Thom spectrum of the bundle determined by the map $\mu: \Omega J_3(S^2) \to \Omega^2 S^3 \to \BO$, where the second map is the extension of the M\"obius bundle $S^1 \to \BO$. Under the equivalence $\Sp(2)/\U(2) \simeq \Sp/\U$, the map $\mu: \Omega J_3(S^2) \to \BO$ can be identified with the composite
$$\Omega(\Sp(2)/\U(2)) \to \Omega(\Sp/\U) \to \mathrm{B}^2\O \xar{\eta} \BO;$$
the middle map is obtained via Bott periodicity. Applying \cite[Proposition 2.1.6]{bpn-thom} to loops on the fibration
$$\Sp(2) \to J_3(S^2) \to \BU(2),$$
we conclude that $y(2) = \Omega J_3(S^2)^\mu$ is equivalent as an $\E{1}$-ring to the Thom spectrum of an $\E{1}$-map $\U(2) \to \BGL_1(T(2))$. This implies, for instance, that $\THH(y(2)/T(2)) \simeq y(2)[\BU(2)]$. Since $k(2) \simeq y(2) \otimes_{T(2)} \BP{2}$, this implies that $\THH(k(2)/\BP{2}) \simeq k(2)[\BU(2)]$. Similarly, since $y(2) \otimes_{T(2)} \ku \simeq \FF_2$, we also recover the observation that $\FF_2$ is equivalent as an $\E{1}$-ring to the Thom spectrum of an $\E{1}$-map $\U(2) \to \BGL_1(\ku)$, and hence that $\HH(\FF_2/\ku) \simeq \FF_2[\BU(2)]$ as $\FF_2$-modules.
\end{remark}
\begin{prop}\label{t2-thh}
There is an equivalence $\THH(T(2)/T_2(2)) \simeq T(2)[S^4]$, as well as an equivalence $\THH(T(2)/T_4(2)) \simeq T(2)[J_3(S^2)]$.
\end{prop}
\begin{proof}
Note that $\eta$ is nullhomotopic in $T_4(2)$ (and hence in $T_2(2)$), since the inclusion $\SU(2) \to \U(2)$ defines a map $S^2 \to \Omega \U(2)$, which in turn Thomifies to a map $C\eta \to T_4(2)$ which factors the unit. By \cref{sp2-quotients}, there are fiber sequences of $\E{1}$-spaces
\begin{align*}
    \Omega \Spin(4) & \to \Omega \Sp(2) \to \Omega S^4, \\
    \Omega \U(2) & \to \Omega \Sp(2) \to \Omega J_3(S^2),
\end{align*}
which by \cite[Proposition 2.1.6]{bpn-thom} (see also \cite{beardsley-thom}) imply that $T(2)$ is a Thom spectrum of an $\E{1}$-map $\Omega S^4 \to \BGL_1(T_2(2))$ (resp. $\Omega J_3(S^2) \to \BGL_1(T_4(2))$). Together with \cite{thh-thom}, this implies the desired claim.
\end{proof}
\begin{remark}
Recall that $\SU(4)/\Sp(2) \cong S^5$. It follows that $\THH(X(4)/T(2)) \simeq X(4)[S^5]$. Similarly, recall that $\SU(4) \cong \Spin(6)$; therefore, there is an diffeomorphism
$$\SU(4)/\Spin(4) \cong \Spin(6)/\Spin(4) \cong \SO(6)/\SO(4) \cong V_2(\RR^6).$$
It follows that $\THH(X(4)/T_2(2)) \simeq X(4)[V_2(\RR^6)]$. (Note also that $\SU(4)/\Spin(4) \cong \SU(4)/(\SU(2)\times \SU(2))$ can be viewed as an ``oriented complex Grassmannian'' $\tilde{\Gr}_2(\cc^4)$.) Finally, $\THH(X(4)/T_2(2)) \simeq X(4)[\SU(4)/\U(2)]$.
\end{remark}
\begin{corollary}\label{thh-rel-t42}
There are $2$-complete equivalences of $\ku$-modules
\begin{align*}
    \THH(\ku/T_2(2)) & \simeq \ku[\Omega S^5], \\
    \THH(\ku/T_4(2)) & \simeq \ku[\Omega S^3].
\end{align*}
Under these equivalences, the maps
$$\THH(\ku/T_4(2)) \to \THH(\ku/T_2(2)) \to \THH(\ku/T(2))$$
are induced by taking $\ku$-chains of the Hopf maps
$$\Omega S^3 \xar{H} \Omega S^5 \xar{H} \Omega S^9.$$
\end{corollary}
\begin{proof}
Using \cref{t2-thh}, this follows from \cref{thh-calculations}(a) (more precisely, the version with $p=2$ and $n=2$ for $\THH(\BP{1}/T(2)) \simeq \ku[\Omega S^9]$), and the fiber sequences of $\E{1}$-spaces
\begin{align*}
    \Omega S^4 \simeq \Omega (\Sp(2)/\Spin(4)) & \to \Omega^2 S^5 \to \Omega^2 S^9, \\
    \Omega J_3(S^2) \simeq \Omega (\Sp(2)/\U(2)) & \to \Omega^2 S^3 \to \Omega^2 S^9
\end{align*}
obtained by looping the $2$-local EHP fiber sequences for $S^4$ and $S^2$. The identification of the maps $\THH(\ku/T_4(2)) \to \THH(\ku/T_2(2))$ and $\THH(\ku/T_2(2)) \to \THH(\ku/T(2))$ is an immediate consequence.
\end{proof}
\begin{remark}
Recall from \cref{thh-calculations}(a) that the generator $\theta_2\in \pi_8 \THH(\ku/T(2))$ can be understood as $\sigma^2(v_2)$ (up to decomposables). Taking $\THH$ relative to the Thom spectrum $T_2(2)$ over $\Omega\Spin(4)$ can be regarded as extracting a square root of $\theta_2\in \pi_8 \THH(\ku/T(2))$. Similarly, taking $\THH$ relative to the Thom spectrum $T_4(2)$ over $\Omega\U(2)$ can be regarded as extracting a fourth root of $\theta_2\in \pi_8 \THH(\ku/T(2))$; hence the subscript $4$.
(Roughly, the generator of $\pi_4 \THH(\ku/T_2(2))$ can be thought of as $\sigma^2(v_1)$; and the generator of $\pi_2 \THH(\ku/T_4(2))$ can be thought of as $\sigma^2(2)$.)
In particular, one should regard $T_4(2) = (\Omega \U(2))^\mu$ as the appropriate analogue of $J(p)$ at height $1$ and $p=2$.
\end{remark}
\begin{remark}
\cref{thh-rel-t42} suggests that $\ku^\wedge_2$ is equivalent to the Thom spectrum of an $\E{1}$-map $\Omega^2 S^3 \to \BGL_1(T_4(2))$. This could also be rephrased in a manner similar to the results of \cite{bpn-thom}: assuming \cite[Conjectures D and E]{bpn-thom}, \cite[Corollary B]{bpn-thom} says that $\ku^\wedge_2$ is the Thom spectrum of a map $\Omega^2 S^9 \to \BGL_1(T(2))$. It follows from \cref{t2-thh} that $T(2) \simeq \colim_{\Omega J_3(S^2)} T_4(2)$, so that \cite[Corollary B]{bpn-thom} implies
$$\ku^\wedge_2 \simeq \colim_{\Omega^2 S^9} T(2) \simeq \colim_{\Omega^2 S^9} \colim_{\Omega J_3(S^2)} T_4(2) \simeq \colim_{\Omega^2 S^3} T_4(2),$$
where the final equivalence comes from the $\E{1}$-equivalence $\colim_{\Omega^2 S^9} \Omega J_3(S^2) \simeq \Omega^2 S^3$ arising from the EHP sequence.
\end{remark}
This leads to the following, which we only state for $T(n)$; there is an analogue for $X(p^n)$, too.
\begin{conjecture}\label{jpn-conjecture}
Fix a prime $p$ and $n\geq 0$. For each $0\leq j\leq n$, there are $\Efr{2}$-rings $T_{p^j}(n)$ equipped with $\Efr{2}$-maps 
$$T_{p^n}(n) \to \cdots \to T_{p^j}(n) \to T_{p^{j-1}}(n) \to \cdots \to T_0(n) = T(n)$$
such that there are $p$-complete equivalences
\begin{align*}
    \THH(T(n)/T_{p^j}(n)) & \simeq \BP{n-1}[J_{p^j-1}(S^{2p^{n-j}})], \\
    \THH(\BP{n-1}/T_{p^j}(n)) & \simeq \BP{n-1}[\Omega S^{2p^{n-j}+1}].
\end{align*}
The map $\THH(\BP{n-1}/T_{p^j}(n)) \to \THH(\BP{n-1}/T_{p^{j-1}}(n))$ induced by the $\Efr{2}$-map $T_{p^j}(n) \to T_{p^{j-1}}(n)$ is given by $\BP{n-1}$-chains on the Hopf map $\Omega S^{2p^{n-j}+1} \to \Omega S^{2p^{n-j+1}+1}$. In other words, if $\theta_n^{1/p^j}\in \pi_{2p^{n-j}} \THH(\BP{n-1}/T_{p^j}(n))$ denotes the generator (roughly, thought of as $\sigma^2(v_{n-j})$), then
$$\pi_{2p^{n-j}} \THH(\BP{n-1}/T_{p^j}(n)) \ni \theta_n^{1/p^j} \mapsto (\theta_n^{1/p^{j-1}})^p\in \pi_{2p^{n-j}} \THH(\BP{n-1}/T_{p^{j-1}}(n)).$$
\end{conjecture}
In particular, \cref{jpn-conjecture} says that for the putative $\Efr{2}$-ring $T_{p^n}(n)$, there is an equivalence $\THH(\BP{n-1}/T_{p^n}(n)) \simeq \BP{n-1}[\sigma]$ with $|\sigma|=2$.
\begin{example}
There is an inclusion $\Spin^c(5) \cong \Sp(2) \cdot \U(1) \subseteq \Sp(3)$ (whose quotient is $\CP^5$), so that composition with the inclusion $\Sp(3)\subseteq \SU(6)$ defines an inclusion $\Sp(2)\cdot \U(1) \subseteq \SU(6)$. In particular, we obtain an $\E{2}$-map $\Omega(\Sp(2)\cdot \U(1)) \to \Omega\SU(6)$. The Thom spectrum of the resulting composite $\E{2}$-map
$$\Omega(\Sp(2)\cdot \U(1)) \to \Omega\SU(6) \to \Omega\SU \simeq \BU$$
defines an $\Efr{2}$-ring, which we expect can be identified with $T_8(3)$ for $p=2$.
\end{example}

%% file: sen/topological-sen.tex
\subsection{Constructing the topological Sen operator}

There is a much simpler description of the descent spectral sequence of \cref{descent-sseq}, following the perspective of \cref{bockstein-serre} that \cref{thh-calculations}(b) is essentially a calculation of a Serre spectral sequence. We will continue to fix $\E{3}$-forms of the truncated Brown-Peterson spectra $\BP{n-1}$ and $\BP{n}$.
\begin{notation}
Let $R$ be an $\Eoo$-$\Z_p$-algebra. We will write $\epsilon^R$ to denote $R[\BSU(p-1)]$ and $\epsilon^R_\ast$ to denote $\pi_\ast \epsilon^R$. (The notation is meant to indicate that $\epsilon$ only plays a ``small'' role in the below discussion.)
\end{notation}
\begin{definition}[Spectral Gysin sequence]
Suppose $S^{n-1} \to E \to B$ is a fibration. Since $E \simeq \hocolim_B S^{n-1}$ in pointed spaces, we have $E_+ \simeq \hocolim_B S^{n-1}_+$. There is a cofiber sequence $S^{n-1}_+ \to S^0 \to S^n$, so we obtain a cofiber sequence
$$E_+ \to \hocolim_B (S^0) \simeq B_+ \to \hocolim_B (S^n) \simeq \Sigma^n(B_+).$$
If $R$ is an $\E{1}$-ring, we get a cofiber sequence of left $R$-modules:
$$R[E] \to R[B] \to \Sigma^n R[B].$$
\end{definition}
\begin{construction}[Topological Sen operator]\label{construction-sen}
Let $\cC$ be an $X(n)$-linear $\infty$-category. There is an $S^1$-equivariant equivalence
\begin{align*}
    \THH(\cC/X(n-1)) & \simeq \THH(\cC) \otimes_{\THH(X(n-1))} X(n-1) \\
    & \simeq \THH(\cC) \otimes_{X(n) \otimes_{X(n-1)} \THH(X(n-1))} X(n),
\end{align*}
a tautological $S^1$-equivariant equivalence
$$\THH(\cC/X(n)) \simeq \THH(\cC) \otimes_{\THH(X(n))} X(n).$$
Since $\THH(X(n-1))\simeq X(n-1)[\SU(n-1)]$, there is an equivalence $X(n) \otimes_{X(n-1)} \THH(X(n-1)) \simeq X(n)[\SU(n-1)]$. Note that $X(n) \otimes_{X(n-1)} \THH(X(n-1))$ admits the structure of an $\E{1}$-ring, and that the $\E{1}$-algebra map $\THH(X(n-1)) \to \THH(X(n))$ induces an $\E{1}$-algebra map $X(n) \otimes_{X(n-1)} \THH(X(n-1)) \to \THH(X(n)) \simeq X(n)[\SU(n)]$. The fiber sequence
$$S^{2n-1} \to \BSU(n-1) \to \BSU(n)$$
implies:
\end{construction}
\begin{theorem}\label{topological-sen}
Let $\cC$ be a left $X(n)$-linear $\infty$-category. Then there is a cofiber sequence
\begin{equation}\label{sen-thh-C}
    \THH(\cC/X(n-1)) \xar{\iota} \THH(\cC/X(n)) \xar{\Theta_\cC} \Sigma^{2n} \THH(\cC/X(n)),
\end{equation}
where the map $\iota$ is $S^1$-equivariant, and the cofiber of $\iota$ is (at least nonequivariantly) identified with $\Sigma^{2n} \THH(\cC/X(n))$. We will call the map $\Theta_\cC: \Sigma^{-2n} \THH(\cC/X(n)) \to \THH(\cC/X(n))$ the \textit{topological Sen operator}.
\end{theorem}
\begin{remark}\label{gauss-manin-sen}
A simpler analogue of \cref{topological-sen} can be described as follows. Let $A$ be an $\Efr{2}$-ring, and let $A[t]$ be the flat polynomial ring over $A$ on a generator in degree $0$. Suppose $\cC$ is an $A[t]$-linear $\infty$-category. The nonequivariant equivalence $\HH(A[t]/A) \simeq A[t][S^1]$ defines a cofiber sequence
\begin{equation}\label{hh-gm}
    \HH(\cC/A) \to \HH(\cC/A[t]) \xar{\nabla} \Sigma^2 \HH(\cC/A[t])
\end{equation}
analogous to \cref{topological-sen}, which exhibits $\nabla: \HH(\cC/A[t]) \to \Sigma^2 \HH(\cC/A[t])$ as a ``Gauss-Manin connection''. This cofiber sequence is often quite useful; for example, if we regard $\Z_p$ as a $S\pw{t}$-algebra by the $\Eoo$-map $S\pw{t} \to \Z_p$ sending $t\mapsto p$, we have $\pi_\ast \THH(\Z_p/S\pw{t}) \simeq \Z_p[y]$ with $|y|=2$ (more precisely, $y = \sigma^2(t-p)$); see \cite{krause-nikolaus-dvr}. It is not difficult to show that the map $\nabla: \THH(\Z_p/S\pw{t}) \to \Sigma^2 \THH(\Z_p/S\pw{t})$ sends $y^n\mapsto ny^{n-1}$, which implies B\"okstedt's calculation of $\pi_\ast \THH(\Z_p)$.

Just as in \cref{topological-sen}, the map $\HH(\cC/A) \to \HH(\cC/A[t])$ in \cref{hh-gm} is $S^1$-equivariant, but we can only nonequivariantly identify its cofiber with $\Sigma^2 \HH(\cC/A[t])$. To identify the cofiber equivariantly, observe that if $\lambda$ denotes the rotation representation of $S^1$, then $\HH(A/A[t]) \simeq A[\B^\lambda \Z_{\geq 0}]$. Here, $\B^\lambda \Z_{\geq 0}$ is the $\lambda$-delooping of $\Z_{\geq 0}$. This implies that there is an \textit{equivariant} cofiber sequence
\begin{equation}\label{equiv-hh-gm}
    \HH(\cC/A) \to \HH(\cC/A[t]) \xar{\nabla} \Sigma^\lambda \HH(\cC/A[t]).
\end{equation}
See \cref{alt-construction-tcp-jp} for some further discussion.
\end{remark}
\begin{remark}\label{sen-bpn}
At the level of homotopy, the map $\Theta$ in \cref{sen-thh-C} for $\cC = \LMod_{\BP{n-1}}$ can be identified using \cref{thh-calculations}. Namely, recall that $\pi_\ast \THH(\BP{n-1}/X(p^n)) \cong \BP{n-1}[\B\Delta_n]_\ast[\theta_n]$ by \cref{thh-calculations}(a); it then follows from \cref{thh-calculations}(b) that $\Theta$ must send
$$\Theta: \theta_n^j \mapsto jp \theta_n^{j-1}.$$
Therefore, we may informally write $\Theta = p\partial_{\theta_n}$.\footnote{This action of $\Theta$ on $\theta_n = \sigma^2(v_n)$ is related to the observation from \cite[Lemma 3.2.8(d)]{lee-thh} that there is a choice of $v_n$ such that the right unit $\eta_R: \BPP_\ast \to \BPP_\ast \BPP \cong \BPP_\ast[t_1, t_2, \cdots]$ satisfies $d(v_n) = \eta_R(v_n) - v_n \equiv pt_n \pmod{t_1, \cdots,t_{n-1}}$.}
From the point of view of \cref{bockstein-serre}, the map $\Theta$ can be interpreted as the $d^{2p^n}$-differential in the Serre spectral sequence computing the $\BP{n-1}$-homology of the total space of the fibration \cref{b-fib-phin}. Determining the action of $\Theta$ on $\THH(\BP{n-1}/X(p^j))$ for $j\leq n-1$ can therefore be viewed as an analogue of determining the differentials in the Serre spectral sequence/Gysin sequence of a putative analogue of the Cohen-Moore-Neisendorfer fibration \cref{b-fib-phin} (where $p$ is replaced by $v_{n-j}$).

One can make some qualitative observations about the action of $\Theta$ on $\THH(\BP{n-1}/X(p^j))$ for $j\leq n-1$. Indeed, recall from \cref{modulo-moore} that there is an isomorphism
$$\pi_\ast \THH(\BP{n-1}/X(p^j))/v_{[0,n-j)} \cong \BP{n-1}[\B\Delta_j]_\ast[\theta_n]/v_{[0,n-j)} \otimes_{\FF_p} \Lambda_{\FF_p}(\lambda_{j+1}, \cdots, \lambda_n).$$
An easy calculation shows that there is an isomorphism
$$\pi_\ast \THH(X(p^n)/X(p^j)) \cong X(p^n)\left[\prod_{i=j+1}^n \ol{\Delta}_i\right]_\ast \otimes_{\Z_{(p)}} \Z_{(p)}(\lambda_{j+1}, \cdots, \lambda_n).$$
Therefore, the calculation of $\pi_\ast \THH(\BP{n-1}/X(p^j))/v_{[0,n-j)}$ implies that the image of a class $y\in \pi_\ast \THH(\BP{n-1}/X(p^j))$ under $\Theta: \THH(\BP{n-1}/X(p^j)) \to \Sigma^{2p^j} \THH(\BP{n-1}/X(p^j))$ lives in the ideal generated by $v_{[0,n-j+1)} = (p,\cdots,v_{n-j})$.
\end{remark}

\begin{remark}
The fact that the cofiber of the $S^1$-equivariant map $\iota: \THH(\BP{n-1}/X(p^n-1)) \to \THH(\BP{n-1}/X(p^n))$ is (at least nonequivariantly) identified with $\Sigma^{2p^n} \THH(\BP{n-1}/X(p^n))$ makes it more difficult to determine $\TP(\BP{n-1}/X(p^n))$ (even modulo $v_{n-1}$) from our calculation of $\pi_\ast \TP(\BP{n-1}/X(p^n))$ in \cref{thh-calculations} and the preceding description of $\Theta$ as an endomorphism of $\THH(\BP{n-1}/X(p^n))$. One fundamental question is therefore to describe the $S^1$-action on $\cofib(\iota)$.
This is already complicated modulo $p$ when $n=1$, and a description of $\TP(\Z_p/X(p-1)) \simeq \TP(\Z_p)[\BSU(p-1)]$ from $\TP(\Z_p/X(p))$ was essentially done in \cite[Conjecture 4.3]{bokstedt-madsen} and \cite[Theorem 7.4]{tsalidis}. Recall from \cref{thh-calculations}(a) that there is an isomorphism 
$$\pi_\ast \TP(\Z_p/X(p))/p \cong \FF_p[v_1, \hbar^{\pm 1}] \otimes_{\FF_p} \epsilon^{\FF_p}_\ast \cong \pi_\ast k(1)^{tS^1}[\BSU(p-1)].$$
Then, the map $\pi_\ast \TP(\Z_p/X(p))/p \to \pi_{\ast-2p} \TP(\Z_p/X(p))/p$ is given by
$$\hbar^{p^k} \mapsto \hbar^{p^k(p+1)} v_1^{\frac{p^{k+1}-p}{p-1}}, \ v_1^k\mapsto 0.$$
This is a direct consequence of \cite[Theorem 7.4]{tsalidis}, once one notes that the the formula $t^{p^k + \phi(k+1)} f^{\phi(k)}$ from \textit{loc. cit.} becomes precisely $\hbar^{p^k(p+1)} v_1^{\frac{p^{k+1}-p}{p-1}}$, via the translation in notation given by
$$t\leadsto \hbar, \ f\leadsto \sigma^2(v_1), \ tf\leadsto v_1, \ \phi(k) = \frac{p^{k+1}-p}{p-1} = v_p((p^k)!^p).$$
One could also prove this using an argument similar to \cite[Theorem 6.5.1]{even-filtr}.

Moreover, the image of $\hbar$ under the boundary map $\pi_{-2} \TP(\Z_p/X(p))/p \to \pi_{2p-3} \TP(\Z_p/X(p-1))/p$ is the class $\alpha_1\in \pi_{2p-3} \TP(\Z_p)/p$; note that since $\hbar$ lives in $\pi_{-2} \TP(\Z_p/X(p))$, the class $\alpha_1$ in fact extends to an element of $\pi_{2p-3} \TP(\Z_p/X(p-1))$. The problem of calculating $\pi_\ast \TP(\Z_p)$ from $\TP(\Z_p/X(p))$ is very similar to the problem of $\pi_\ast \TP(\Z_p)$ from $\TP(\Z_p/S\pw{t})$, discussed in \cite{liu-wang} (see \cref{gauss-manin-sen}).
\end{remark}
If we assume \cref{tn-e2}, then \cref{topological-sen} can be refined: namely, if $\cC$ is a left $T(n)$-linear $\infty$-category, then there is a cofiber sequence
\begin{equation}\label{tn-sen}
    \THH(\cC/T(n-1)) \xar{\iota} \THH(\cC/T(n)) \xar{\Theta_\cC} \Sigma^{2p^n} \THH(\cC/T(n)).
\end{equation}
\begin{remark}
Suppose $n=1$ and $\cC = \Mod_{\Z_p}$ for $p$ odd. Then there is a map $\TP(\Z_p) \to \TP(\Z_p/T(1))$, and a trace map $K(\Z_p) \to \TP(\Z_p)$. Let $j = \tau_{\geq 0} \Lone S$; upon $p$-adic completion, there is an equivalence (see \cite[Theorem 9.17]{bokstedt-madsen})
$$K(\Z_p)^\wedge_p \simeq j \vee \Sigma j \vee \Sigma^3 \ku.$$
The summand $j$ is the unit component, i.e., there is an $\Eoo$-ring map $j \to K(\Z_p)^\wedge_p$. It follows that after $p$-completion, there is a ring map $j \to \TP(\Z_p)$. Assuming the equivalence $\TP(\Z_p/T(1)) \simeq \BP{1}^{tS^1}$ of \cref{conjecture-thh}, the following diagram commutes:
$$\xymatrix{
j \ar[r] \ar[d]_-{\mathrm{unit}} & \BP{1} \ar[d]^-{\mathrm{unit}} \\
\TP(\Z_p) \ar[r] & \TP(\Z_p/T(1)).
}$$
Let $\ell$ be a topological generator of $\Z_p^\times$, and let $\psi^\ell: \BP{1} \to \Sigma^{2p-2} \BP{1}$ be the associated Adams operation. Then, the fiber of $\psi^\ell-1$ is $j$. Based on the above commutative diagram, one expects that under the equivalence $\TP(\Z_p/T(1)) \simeq \BP{1}^{tS^1}$ of \cref{conjecture-thh}, the map $\psi^\ell - 1$ is closely related to $\Theta_{\Z_p}^{tS^1}$. Note, for example, that if we take $\ell = p+1$, the map $\psi^\ell-1$ sends $v_1^j \mapsto p^{v_p(j)+1} v_1^j$ up to $p$-adic units; this should be compared to the fact that $\Theta_{\Z_p}$ sends $\theta_1^j \mapsto jp\theta_1^{j-1}$ by \cref{sen-bpn}. This discussion, as well as the classical discussion in \cite{bokstedt-madsen}, suggests that $\TP(\Z_p)^\wedge_p \simeq (j^{tS^1})^\wedge_p$. 
In fact, something stronger is true: in forthcoming work \cite{calculating-THH-Z} with Arpon Raksit, we will show that $\THH(\Z_p) = \tau_{\geq 0}(j^{t\Cp})$ as cyclotomic $\Eoo$-rings.
\end{remark}
\begin{example}\label{sen-on-bp1}
Let $n=1$, and let $\cC = \Mod_{\BP{1}}$. Then \cref{topological-sen} gives a cofiber sequence
$$\THH(\BP{1}/X(p-1)) \to \THH(\BP{1}/X(p)) \xar{\Theta_{\BP{1}}} \Sigma^{2p} \THH(\BP{1}/X(p)).$$
Moreover, recall from \cref{thh-calculations}(b) that there is a $p$-complete equivalence
$$\THH(\BP{1}/X(p)) \simeq \BP{1}[\BSU(p-1)] \oplus \bigoplus_{j\geq 1} \Sigma^{2jp^2 - 1} \BP{1}[\BSU(p-1)]/pj.$$
Let $a_j$ denote the $\BP{1}$-module generator of the summand $\Sigma^{2jp^2 - 1} \BP{1}/pj$. Since $\THH(\BP{1}/X(p-1))\simeq \THH(\BP{1})[\BSU(p-1)]$, the calculations of \cite[Section 6]{thh-bp1} can be rephrased as follows. For $0\leq k \leq v_p(j)$, $\Theta_{\BP{1}}$ is given on homotopy by
$$\Theta_{\BP{1}}: p^k a_j \mapsto \left(\frac{j}{p^k}-1\right) a_{j-p^k} v_1^{p\frac{p^{k+1}-1}{p-1}},$$
up to $p$-adic units. A different perspective on this computation is given in \cite{lee-thh}.
\end{example}
\begin{variant}\label{topological-sen-jpn}
One can prove a variant of \cref{topological-sen} by replacing $X(p)$ with $J(p)$.
If $\cC$ is a left $J(p)$-linear $\infty$-category, then \cref{t1-jp} produces a cofiber sequence:
\begin{equation}\label{jp-sen}
    \THH(\cC) \xar{\iota} \THH(\cC/J(p)) \xar{\Theta'} \Sigma^{2} \THH(\cC/J(p)).
\end{equation}
Here, the map $\iota$ is $S^1$-equivariant, and $\cofib(\iota)$ is (at least nonequivariantly) identified with $\Sigma^{2} \THH(\cC/J(p))$.
\cref{laurent-xp} shows that $\THH(\Z_p/J(p)) \simeq \Z_p[\Omega S^3]$. On homotopy, the map $\THH(\Z_p/J(p)) \to \Sigma^2 \THH(\Z_p/J(p))$ is given by the $d^2$-differential in the Serre spectral sequence for the fibration
$$S^1 \to \Omega S^3\pdb{3} \to \Omega S^3.$$
For example, under the isomorphism $\pi_\ast \THH(\Z_p/J(p)) \cong \Z_p[x]$ with $|x|=2$, the map $\Theta'$ in the cofiber sequence \cref{jp-sen} for $n=1$ sends $x^j\mapsto jx^{j-1}$.

Suppose $\cC$ is in fact a $\Z_p$-linear $\infty$-category. Base-changing \cref{jp-sen} along the map $\Z_p \to \FF_p$ and using \cref{mod-p-thh}, we obtain a cofiber sequence
\begin{equation}\label{mod-p-sen}
    \THH(\cC) \otimes_{\Z_p} \FF_p \xar{\iota} \THH(\cC\otimes_{\Z_p} \FF_p) \xar{\Theta'} \Sigma^{2} \THH(\cC\otimes_{\Z_p} \FF_p).
\end{equation}
Note that the map $\Theta': \THH(\FF_p) \to \Sigma^2 \THH(\FF_p)$ sends $\sigma^j\mapsto j\sigma^{j-1}$ on homotopy. It follows that upon composition with $\sigma: \Sigma^2 \THH(\cC\otimes_{\Z_p} \FF_p) \to \THH(\cC\otimes_{\Z_p} \FF_p)$, $\Theta'$ acts by multiplication by $j$ on the homotopy of the $j$th graded piece $\gr^j_\sigma \THH(\cC\otimes_{\Z_p} \FF_p)$ of the $\sigma$-adic filtration on $\THH(\cC\otimes_{\Z_p} \FF_p)$.
\end{variant}
\begin{remark}
Let $p=2$. Using the fiber sequence
$$S^3 \to \BU(1) \to \BU(2),$$
one can similarly show that if $T_4(2)$ denotes the $\Efr{2}$-ring from \cref{higher-jp} and $\cC$ is a left $T_4(2)$-linear $\infty$-category, there is a cofiber sequence
$$\THH(\cC/J(2)) \to \THH(\cC/T_4(2)) \to \Sigma^4 \THH(\cC/T_4(2)).$$
\end{remark}
%

\begin{remark}
Let $R$ be an animated $\Z_p$-algebra. Let $\hat{\prism}_{R}$ denote the Nygaard-completed prismatic cohomology of $R$, and $\cN^i \hat{\prism}_{R}$ denote the $i$th graded piece of the Nygaard filtration $\cN^{\geq \star}(\hat{\prism}_{R})$. Note that \cite[Remark 5.5.15]{apc} gives an isomorphism $\cN^i (\hat{\prism}_{R}\{i\}) \cong \cN^i \hat{\prism}_{R}$, where $\hat{\prism}_{R}\{i\}$ denotes the Breuil-Kisin twisted prismatic cohomology of $R$. Using the methods of \cite{bhatt-mathew-syntomic}, one can construct a cofiber sequence
\begin{equation}\label{bhatt-mathew-cofiber}
    (\cN^i \hat{\prism}_R)/p \to \F^\conj_i \dR_{(R/p)/\FF_p} \cong \cN^i \hat{\prism}_{R/p} \to \F^\conj_{i-1} \dR_{(R/p)/\FF_p}.
\end{equation}
As explained in \textit{loc. cit.}, the second map is closely related to the Sen operator.
Recall (see \cite[Example 6.4.17]{apc} and \cite{bms-ii}) that $\THH(R/p)$ admits a motivic filtration such that $\gr^i_\mot \THH(R/p) = \cN^i(\hat{\prism}_{R/p})[2i]$. Taking $\cC = \Mod_R$, \cref{mod-p-sen} says that there is a self-map $\Theta': \THH(R/p) \to \Sigma^2 \THH(R/p)$ whose fiber is $\THH(R)/p$. Presumably, the cofiber sequence \cref{mod-p-sen} can be shown to respect the motivic filtration, so taking graded pieces would recover the cofiber sequence \cref{bhatt-mathew-cofiber}. Given this discussion, it is natural to ask if $\THH(R/J(p))$ admits a motivic filtration such that \cref{jp-sen} is a cofiber sequence of motivically-filtered spectra. 
\end{remark}
\begin{recall}
Let $(\Z_p\pw{\ptl}, \ptl)$ denote the prism of \cite[Notation 3.8.9]{apc}, and if $R$ is a $p$-complete animated $\Z_p$-algebra, let $\ptl\Omega_R$ denote $\prism_{R/\Z_p\pw{\ptl}}$. In particular, $\ptl\Omega_R \simeq (q\Omega_R)^{h\FF_p^\times}$, via the $\FF_p^\times$-action on the prism $(\Z_p\pw{q-1}, [p]_q)$. Let $\diffr_R$ denote the diffracted Hodge complex of \cite[Construction 4.7.1]{apc}, so that $\diffr_R$ is isomorphic to $\ptl\Omega_R/\ptl$ by \cite[Remark 4.8.6]{apc}. Recall the cofiber sequence of \cite[Remark 5.5.8]{apc}:
$$\cN^i\hat{\prism}_R \to \F^\conj_i \diffr_R \xar{\Theta + i} \F^\conj_{i-1} \diffr_R.$$
The mod $p$ reduction of this fiber sequence produces \cref{bhatt-mathew-cofiber}.
\end{recall}
Since $\THH(R)$ admits a motivic filtration whose graded pieces are $\cN^i(\hat{\prism}_R)[2i]$, the cofiber sequence \cref{jp-sen} motivates the following conjecture (which essentially states that $\THH(R/J(p))$ is a sheared Rees construction on the conjugate filtration of $\diffr_R$):
\begin{conjecture}\label{rel-jp-diffracted}
Let $R$ be an animated $\Z_p$-algebra. Then there is a filtration $\F^\star_\mot \THH(R/J(p))$ on $\THH(R/J(p))$ such that:
\begin{itemize}
    \item $\gr^i_\mot \THH(R/J(p)) \simeq (\F^\conj_i \diffr_R)[2i]$; and
    \item the map $\Theta'_R: \THH(R/J(p)) \to \Sigma^2 \THH(R/J(p))$ respects the motivic filtration and induces the map $\Theta + i: \F^\conj_i \diffr_R \to \F^\conj_{i-1} \diffr_R$ on $\gr^i_\mot$; and
    \item $\gr^i_\mot (\THH(R/J(p))[x^{-1}]) \simeq \diffr_R[2i]$, such that the localization map $\THH(R/J(p)) \to \THH(R/J(p))[x^{-1}]$ induces the inclusion $(\F^\conj_i \diffr_R)[2i] \to \diffr_R[2i]$ on $\gr^i_\mot$.
\end{itemize}
\end{conjecture}

\begin{remark}\label{rel-t1-diffracted}
Recall from \cref{t1-jp} that there is an equivalence $\THH(X(p)/J(p)) \simeq X(p)[\SU(p-1) \times J_{p-1}(S^2)]$. 
Using this, it is not difficult to show that \cref{rel-jp-diffracted} implies that if $R$ is an animated $\Z_p$-algebra, then $\THH(R/X(p))$ admits a motivic filtration such that $\gr^i_\mot \THH(R/X(p)) \simeq (\F^\conj_{pi} \diffr_R)[2pi] \otimes_R \epsilon^R$. If \cref{tn-e2} were true\footnote{Or at least the weaker statement that $T(1)$ admits the structure of an $\E{2}$-ring.} for $n=1$, then $\THH(R/T(1))$ would admit a motivic filtration such that $\gr^i_\mot \THH(R/T(1)) \simeq (\F^\conj_{pi} \diffr_R)[2pi]$. Therefore, $\THH(R/X(p))$ precisely extracts the pieces of the conjugate filtration on $\diffr_R$ which are not automatically split by the Sen operator.
From the point of view of \cref{rel-jp-diffracted}, the utility of the discussion in \cref{higher-jp} is that although describing a higher chromatic analogue of $J(p)$ is tricky (see \cref{jpn-conjecture}), $\THH(\cC/X(p^n))$ furnishes a natural higher chromatic and noncommutative analogue of the diffracted Hodge complex when $\cC$ is a left $\BP{n}$-linear $\infty$-category.
\end{remark}
\begin{remark}\label{evidence-diffracted}
We collect some further evidence for \cref{rel-jp-diffracted}:
\begin{enumerate}
    \item Recall that if $\cd$ is an $\FF_p$-linear $\infty$-category, then the canonical map $\THH(\cd) \to \HH(\cd/\FF_p)$ is given by quotienting by $\sigma\in \pi_2 \THH(\FF_p)$. Moreover, if $R$ is an animated $\FF_p$-algebra, then $\gr^i_\mot \THH(R) \simeq (\F_i^\conj \dR_{R/\FF_p})[2i]$, and $\F^\sigma_\star \THH(R)$ is a noncommutative analogue of the conjugate filtration $\F^\conj_\star \dR_{R/\FF_p}$. In particular, the induced motivic filtration on $\THH(R)/\sigma$ has $\gr^i_\mot (\THH(R)/\sigma) \simeq L\Omega^i_{R/\FF_p}[-i]$.
    
    This picture admits an analogue over $J(p)$. Recall from \cref{laurent-xp}(a) that $\pi_\ast \THH(\Z_p/J(p)) \cong \Z_p[x]$ with $|x|=2$. Let $\cC$ be a $\Z_p$-linear $\infty$-category. One could attempt to define the quotient $\THH(\cC/J(p))/x$ as a relative tensor product of $\THH(\cC/J(p))$ with $\Z_p$ over $\THH(\Z_p/J(p))$. Unfortunately, this tensor product does not make sense, since $\THH(\Z_p/J(p))$ does not naturally acquire the structure of an $\E{1}$-algebra. However, were $J(p)$ to admit the structure of an $\E{3}$-algebra, the above relative tensor product would precisely be computing $\HH(\cC/\Z_p) = \THH(\cC) \otimes_{\THH(\Z_p)} \Z_p$. It is therefore reasonable to view the canonical map $\THH(\cC/J(p)) \to \HH(\cC/\Z_p)$ as a quotient by $x$. If $R$ is an animated $\Z_p$-algebra, then $\HH(R/\Z_p)$ is a noncommutative analogue of the Hodge complex $\bigoplus_{n\geq 0} L\widehat{\Omega}_{R/\Z_p}^n[-n]$. Under \cref{rel-jp-diffracted}, the perspective that the map $\THH(\cC/J(p)) \to \HH(\cC/\Z_p)$ is given by ``killing $x$'' can be regarded as an analogue of \cite[Remark 4.7.14]{apc}, which identifies $\gr^\conj_i \diffr_R \simeq L\widehat{\Omega}^i_{R/\Z_p}[-i]$.
    \item Let $R$ be a smooth $\Z_p$-algebra. Then the prismatic-crystalline comparison theorem (see \cite[Remark 4.7.18]{apc}) implies that the base-change $\FF_p \otimes_{\Z_p} \F^\conj_\star \diffr_R$ can be identified with $\Fr_\ast \F^\conj_\star \Omega^\bull_{R/p/\FF_p}$, where $\Fr: R \to R$ is the absolute Frobenius. Under \cref{rel-jp-diffracted}, \cref{mod-p-thh} can be viewed as a noncommutative analogue of this result.
    \item By \cref{laurent-xp}, the class $x$ is sent to $\sigma\in \pi_2 \THH(\FF_p)$ under the map $\iota: \THH(\Z_p/J(p)) \to \THH(\FF_p)$. Since the cyclotomic Frobenius induces an equivalence $\varphi: \THH(\FF_p)[1/\sigma] \xar{\sim} \THH(\FF_p)^{t\Cp}$,
    the cofiber sequence of \cref{mod-p-sen} predicts a cofiber sequence
    \begin{equation}\label{tcp-jp}
        \THH(\cC)^{t\Cp} \otimes_{\Z_p} \FF_p \xar{\iota} \THH(\cC\otimes_{\Z_p} \FF_p)^{t\Cp} \xar{\Theta'} \THH(\cC\otimes_{\Z_p} \FF_p)^{t\Cp}.
    \end{equation}
    Such a cofiber sequence does indeed exist, and we will construct it below in \cref{alt-construction-tcp-jp} (albeit using slightly different methods).
    
    Suppose that the cofiber sequence \cref{tcp-jp} respects the motivic filtration when $\cC = \Mod_R$. Since $\THH(R)^{t\Cp} \simeq \HP((R/p)/\FF_p)$ (see \cite[Proposition 2.12]{mathew-kaledin}) and $\HP((R/p)/\FF_p)$ has a motivic filtration such that $\gr^i_\mot \HP((R/p)/\FF_p) \simeq \dR_{(R/p)/\FF_p}[2i]$, the cofiber sequence \cref{tcp-jp} would presumably be related under \cref{rel-jp-diffracted} to the following cofiber sequence related to \cref{bhatt-mathew-cofiber} (whose existence was told to me by Akhil Mathew):
    \begin{equation}\label{dr-modp}
        \prismht_R/p \to \dR_{(R/p)/\FF_p} \to \dR_{(R/p)/\FF_p}.
    \end{equation}
\end{enumerate}
\end{remark}
For completeness, we give an argument for \cref{dr-modp}.
\begin{proof}[Proof of the cofiber sequence \cref{dr-modp}]
Recall from \cite[Corollary 3.16]{bhatt-mathew-syntomic} that if $A$ is an animated $\Z_p[x]$-algebra, there is a cofiber sequence
\begin{equation}\label{hodge-tate-affine}
    \prismht_A\{i\}/x \to \prismht_{A/x}\{i\} \to \prismht_{A/x} \{i-1\}.
\end{equation}
This implies (by setting $i=0$ and viewing $R/p$ as the base-change $R \otimes_{\Z_p[x]} \Z_p$, where the map $\Z_p[x] \to R$ sends $x\mapsto p$, and the map $\Z_p[x] \to \Z_p$ is the augmentation) that there is a cofiber sequence
$$\prismht_R/p \to \prismht_{R/p} \to \prismht_{R/p}.$$
The de Rham/crystalline comparison theorems tell us that $\prism_{R/p} \simeq \prism_{(R/p)/\Z_p} \simeq (\dR_R)^\wedge_p$, where $\prism_{(R/p)/\Z_p}$ denotes prismatic cohomology with respect to the crystalline prism $(\Z_p, (p))$ (i.e., the derived crystalline cohomology of $R/p$). But then $\prismht_{R/p} \simeq \dR_{(R/p)/\FF_p}$, as desired.

Let us remark that \cref{hodge-tate-affine} can be constructed using $\WCart^\HT_{\GG_a}$. Indeed, we can reduce to the case when $A$ is the $p$-completion of $\Z_p[x] = \co_{\GG_a}$. Then, \cite[Example 9.1]{prismatization} implies that $\spec(\Z_p) \times_{\GG_a} \WCart^\HT_{\GG_a} \cong B(\GG_a^\sharp \rtimes \GG_m^\sharp)$. Let $\alpha: \WCart^\HT_{\Z_p} \to \WCart^\HT_{\GG_a}$ be the tautological map, so that it factors through a map $f: \WCart^\HT_{\Z_p} \to \spec(\Z_p) \times_{\GG_a} \WCart^\HT_{\GG_a}$, which can in turn be identified with the map $B\GG_m^\sharp \to B(\GG_a^\sharp \rtimes \GG_m^\sharp)$. It follows that there is a Cartesian square
$$\xymatrix{
\GG_a^\sharp \ar[d] \ar[r] & \spec(\Z_p) \ar[d] \\
\WCart^\HT_{\Z_p} \ar[r]_-f & \spec(\Z_p) \times_{\GG_a} \WCart^\HT_{\GG_a}.
}$$
Let $\cf$ be a quasicoherent sheaf on $\WCart^\HT_{\GG_a}$, and let $\cf/x$ be the associated quasicoherent sheaf on $\spec(\Z_p) \times_{\GG_a} \WCart^\HT_{\GG_a}$. Our goal is to identify the cofiber of the map $\cf/x \to f_\ast \alpha^\ast \cf \simeq f_\ast f^\ast(\cf/x)$ in the case when $\cf$ is the Breuil-Kisin twisting line bundle $\co_{\WCart^\HT_{\GG_a}}\{i\}$ on $\WCart^\HT_{\GG_a}$. The preceding Cartesian square along with the cofiber sequence\footnote{Here, we declare $\gamma_{-1}(x) = 0$.}
$$\Z_p \to \Z_p\pdb{t} = \co_{\GG_a^\sharp} \xar{\partial_t} \Z_p\pdb{t}, \ \gamma_n(t) \mapsto \gamma_{n-1}(t)$$
implies that $\cofib(\cf/x \to f_\ast \alpha^\ast \cf)$ can be identified with $\co_{\WCart^\HT_{\Z_p}}\{-1\} \otimes f_\ast \alpha^\ast \cf$. Setting $\cf = \co_{\WCart^\HT_{\GG_a}}\{i\}$ and taking global sections produces \cref{hodge-tate-affine}.
\end{proof}
We now construct a more general version of the cofiber sequence \cref{tcp-jp}. We first need the following lemma:
\begin{lemma}\label{building-rot-rep}
Let $G\subseteq S^1$ be a nontrivial finite subgroup of $S^1$, and let $\lambda$ denote the rotation representation of $S^1$ on $\cc$. 
\begin{enumerate}
    \item Define $(S^\lambda)^{(1)}$ via the cofiber sequence
    $$G_+ \to S^0 \to (S^\lambda)^{(1)}.$$
    Then there is a cofiber sequence
    $$\Sigma (G_+) \to (S^\lambda)^{(1)} \to S^\lambda.$$
    \item Let $X$ be a spectrum with $G$-action. Then $X^{tG} \xar{\sim} (\Sigma^\lambda X)^{tG}$.
\end{enumerate}
\end{lemma}
\begin{proof}
Part (a) describes an equivariant CW-structure on $S^\lambda$; we leave this as an exercise to the reader. Part (b) follows by observing that the cofiber sequence
$$G_+ \otimes X \to X \to (S^\lambda)^{(1)} \otimes X$$
implies that $X^{tG} \xar{\sim} ((S^\lambda)^{(1)} \otimes X)^{tG}$; and the cofiber sequence
$$X \otimes \Sigma (G_+) \to X \otimes (S^\lambda)^{(1)} \to \Sigma^\lambda X$$
implies that $(X \otimes (S^\lambda)^{(1)})^{tG} \xar{\sim} (\Sigma^\lambda X)^{tG}$.
\end{proof}
\begin{prop}
Let $S[\pi] = S[\Z_{\geq 0}]$. For any $S[\pi]$-linear $\infty$-category $\cC$, there are cofiber sequences
\begin{align}
    \THH(\cC)^{t\Cp} \otimes_{S[\pi]} S & \to \THH(\cC  \otimes_{S[\pi]} S)^{t\Cp} \xar{\nabla^{t\Cp}} \THH(\cC \otimes_{S[\pi]} S)^{t\Cp}, \label{refined-tcp-cofib} \\
    \TP(\cC) & \to \TP(\cC/S[\pi]) \xar{\nabla^{tS^1}} \TP(\cC/S[\pi]). \label{tate-cofib}
\end{align}
\end{prop}
\begin{proof}
We will use \cref{equiv-hh-gm} with $A = S$ (here, the variable $t$ is relabeled as $\pi$). This gives us an $S^1$-equivariant cofiber sequence
\begin{equation}\label{spi-cofib}
    \THH(\cC) \to \THH(\cC/S[\pi]) \to \Sigma^\lambda \THH(\cC/S[\pi]).
\end{equation}
To prove the cofiber sequence \cref{refined-tcp-cofib}, we first apply $t\Cp$ to the preceding cofiber sequence:
$$\THH(\cC)^{t\Cp} \to \THH(\cC/S[\pi])^{t\Cp} \to (\Sigma^\lambda \THH(\cC/S[\pi]))^{t\Cp}.$$
Observe that the tensor product $\THH(\cC)^{t\Cp} \otimes_{S[\pi]} S$ along the augmentation $S[\pi] \to S$ sending $\pi\mapsto 0$ is precisely $\THH(\cC)^{t\Cp} \otimes_{S[\pi]} S$. Similarly, $\THH(\cC/S[\pi])^{t\Cp} \otimes_{S[\pi]} S \simeq \THH(\cC \otimes_{S[\pi]} S)^{t\Cp}$. It therefore suffices to show that $(\Sigma^\lambda \THH(\cC/S[\pi]))^{t\Cp} \simeq \THH(\cC/S[\pi])^{t\Cp}$; but this is exactly \cref{building-rot-rep}.

The cofiber sequence \cref{tate-cofib} is even easier to construct: applying $tS^1$ to \cref{spi-cofib}, we obtain a cofiber sequence
$$\TP(\cC) \to \TP(\cC/S[\pi]) \to (\Sigma^\lambda \THH(\cC/S[\pi]))^{tS^1}.$$
Since there is a cofiber sequence
$$S^1_+ \to S^0 \to S^\lambda,$$
we see that there is an equivalence $X^{tS^1} \xar{\sim} (\Sigma^\lambda X)^{tS^1}$ for any $S^1$-spectrum $X$. In particular, $(\Sigma^\lambda \THH(\cC/S[\pi]))^{tS^1} \simeq \TP(\cC/S[\pi])$, as desired.
\end{proof}
\begin{corollary}\label{alt-construction-tcp-jp}
Let $K$ be a number field, let $\fr{p}\subseteq \co_K$ be a prime ideal over $p$, and let $R$ denote the localization of $\co_K$ at $\fr{p}$. Denote by $\pi\in R$ a uniformizer, and let $k = R/\pi$ be the residue field, so that there is an $\Eoo$-map $S[\pi] \to R$ sending $\pi\mapsto \pi$. For any $R$-linear $\infty$-category $\cC$, there are cofiber sequences
\begin{align}
    \THH(\cC)^{t\Cp} \otimes_R k & \to \THH(\cC \otimes_R k)^{t\Cp} \xar{\nabla^{t\Cp}} \THH(\cC \otimes_R k)^{t\Cp}, \\
    \TP(\cC) & \to \TP(\cC/S[\pi]) \xar{\nabla^{tS^1}} \TP(\cC/S[\pi]). \label{pi-tate-cofib}
\end{align}
\end{corollary}
\begin{remark}
The cofiber sequence \cref{pi-tate-cofib} was used in \cite{liu-wang} to calculate $\TP(\co_K)$ by computing the resulting endomorphism of $\TP(\co_K/S[\pi])$.
\end{remark}

%% file: sen/some-calculations.tex
We now calculate the topological Sen operator for perfectoid rings; these calculations lend further evidence for \cref{rel-jp-diffracted}.
\begin{recall}
Let $R$ be a perfectoid ring. Recall that $\Ainf(R) = W(R^\flat)$, so that $L_{\Ainf(R)/\Z_p}$ is $p$-completely zero. Let $\Ainf^+(R)$ denote the spherical Witt vectors $W^+(R^\flat)$ of \cite[Example 5.2.7]{elliptic-ii}. 
\end{recall}
\begin{lemma}\label{thh-ainf}
Let $\xi$ be a generator of the kernel of Fontaine's map $\theta: \Ainf(R) \to R$. Let $\Omega^2 S^3 \to \BGL_1(\Ainf^+(R))$ denote the $\E{2}$-map which detects $1-\xi\in \Ainf(R)^\times$ on the bottom cell of the source. Then there is an equivalence of $\E{2}$-$\Ainf^+(R)$-algebras between the $\xi$-adic completion of $\Ainf(R)$ and the $\xi$-adic completion of the Thom spectrum of the following composite:
$$g_\xi: \Omega^2 S^3\pdb{3} \to \Omega^2 S^3 \to \BGL_1(\Ainf^+(R)).$$
In particular, there is an equivalence $\THH(\Ainf(R)^\wedge_\xi/\Ainf^+(R)^\wedge_\xi) \simeq \Ainf(R)^\wedge_\xi[\Omega S^3\pdb{3}]$ of $\E{2}$-$\Ainf(R)^\wedge_\xi$-algebras.
\end{lemma}
\begin{proof}
Recall from \cite[Theorem 1.13]{zhouhang-mao} that the Thom spectrum of the map $\Omega^2 S^3 \to \BGL_1(\Ainf^+(R))$ is equivalent to $R$ as an $\E{2}$-$\Ainf^+(R)$-algebra. The fiber sequence
$$\Omega^2 S^3\pdb{3} \to \Omega^2 S^3 \to S^1$$
implies that there is a class $\xi\in \pi_0 (\Omega^2 S^3\pdb{3})^{g_\xi}$ and a map $S^1 \to \BGL_1(\Omega^2 S^3\pdb{3})^{g_\xi}$ detecting $1-\xi$, such that its Thom spectrum is $R$.
This implies that there is a cofiber sequence
$$(\Omega^2 S^3\pdb{3})^{g_\xi} \xar{\xi} (\Omega^2 S^3\pdb{3})^{g_\xi} \to R.$$
It follows that the $\xi$-adic completion $(\Omega^2 S^3\pdb{3})^{g_\xi}$ is equivalent to $\Ainf(R)^\wedge_\xi$. The claim about $\THH$ follows in the standard manner using \cite{thh-thom}.
\end{proof}
\begin{remark}\label{perfd-equiv}
In fact, the calculation from \cite[Theorem 6.1]{bms-ii} that $\pi_\ast \THH(R) \cong R[\sigma]$ is equivalent to \cite[Theorem 1.13]{zhouhang-mao} (which constructs $R$ as the Thom spectrum of the map $\Omega^2 S^3 \to \BGL_1(\Ainf^+(R))$). The equivalence between these two statements can be proved similarly to \cite[Remark 1.5]{krause-nikolaus-dvr}.
\end{remark}
\begin{prop}\label{thh-perfd}
Let $R$ be a $p$-complete perfectoid ring. Then there is a $p$-complete equivalence
$$\THH(R/X(p)) \simeq R[\CP^\infty \times \Omega S^{2p+1}] \otimes_R \epsilon^R.$$
In particular, if $\theta$ denotes the ``polynomial\footnote{Recall that $\THH(R/X(p))$ is not a ring; the word polynomial simply means the subspace generated by $R[\Omega S^{2p+1}]_\ast$.}'' generator in degree $2p$ arising via the James filtration on $\Omega S^{2p+1}$ and $R\pdb{u} = \pi_\ast R[\CP^\infty]$ is (the underlying $R$-module of) a divided power algebra on a class $u$ in degree $2$, then  there is a $p$-complete isomorphism
$$\pi_\ast \THH(R/X(p)) \simeq R[\theta]\pdb{u} \otimes_R \epsilon^R_\ast.$$
\end{prop}
\begin{proof}
Let $X(p)_\xi$ denote the $\xi$-adic completion of the Thom spectrum of the composite
$$\Omega \SU(p) \to \Omega S^{2p-1} \xar{\alpha_1} \Omega^2 S^3\pdb{3} \to \BGL_1(\Ainf^+(R)).$$
Then, the map $\THH(X(p)_\xi) \to \THH(X(p)) \otimes \Ainf^+(R)^\wedge_\xi$ is a $(p,\xi)$-complete equivalence: indeed, the above composite is determined as an $\E{1}$-map by the composite
$$\SU(p) \to S^{2p-1} \xar{(1-\xi)\alpha_1} \B^2 \GL_1(\Ainf^+(R)).$$
Since $1-\xi$ is a unit in $\pi_0 \Ainf^+(R) \cong \Ainf(R)$, it suffices to prove that the map $\THH(\Ainf^+(R)^\wedge_\xi) \to \Ainf^+(R)^\wedge_\xi$ is a $(p,\xi)$-complete equivalence. But this is clear: after killing $\xi$ and tensoring with $\FF_p$, we obtain the map $\HH(R^\flat/\FF_p) \to R^\flat$, which is an equivalence since $R^\flat$ is perfect.

It then follows from \cref{thh-ainf} and the same argument used to prove \cref{thh-calculations}(a) that there are $(p,\xi)$-complete equivalences
$$\THH(\Ainf(R)^\wedge_\xi/X(p)) \simeq \THH(\Ainf(R)^\wedge_\xi/X(p)_\xi) \simeq \Ainf(R)[\Omega S^{2p+1} \times \BSU(p-1)].$$
Therefore, there are $p$-complete equivalences
\begin{align*}
    \THH(R/X(p)) & \simeq \THH(R/X(p)_\xi)\\
    & \simeq \THH(R/\Ainf^+(R)^\wedge_\xi) \otimes_{\THH(\Ainf(R)^\wedge_\xi/\Ainf^+(R)^\wedge_\xi)} \THH(\Ainf(R)^\wedge_\xi/X(p)_\xi)\\
    & \simeq \THH(R/\Ainf^+(R)^\wedge_\xi) \otimes_{\Ainf(R)^\wedge_\xi[\Omega S^3\pdb{3}]} \Ainf(R)^\wedge_\xi[\Omega S^{2p+1} \times \BSU(p-1)].
\end{align*}
Since $R$ is perfectoid, \cite[Theorem 6.1]{bms-ii} implies that $\THH(R/\Ainf^+(R)) \simeq R[\Omega S^3]$. The map $\THH(W(R^\flat)) \to \THH(R)$ induced by the unit can be identified with the composite $W(R^\flat)[\Omega S^3\pdb{3}] \to R[\Omega S^3]$, induced by Fontaine's map $\theta: \Ainf(R) \to R$. There is a $p$-local Cartesian square
\begin{equation}\label{pullback-perfd}
    \xymatrix{
    \Omega S^3\pdb{3} \ar[d] \ar[r] & \Omega S^3 \ar[d]^-{H_p\times \iota} \\
    \Omega S^{2p+1} \times \BSU(p-1) \ar[r] & \Omega S^{2p+1} \times \CP^\infty \times \BSU(p-1),
    }
\end{equation}
which implies that
$$\THH(R/X(p)) \simeq R[\Omega S^{2p+1} \times \CP^\infty \times \BSU(p-1)],$$
as desired. Alternatively, there are equivalences
\begin{align*}
    \THH(R/X(p)_\epsilon) & \simeq \THH(R/\Ainf^+(R)^\wedge_\xi) \otimes_{\THH(X(p)_\epsilon/\Ainf^+(R)^\wedge_\xi)} X(p)_\epsilon\\
    & \simeq R[\Omega S^3] \otimes_{R[\SU(p)]} R.
\end{align*}
The desired calculation follows from the observation that there is a $p$-local fibration
$$\SU(p) \simeq \SU(p-1) \times S^{2p-1} \xar{\ast\times \alpha_1} \Omega S^3 \xar{H_p\times \iota} \Omega S^{2p+1} \times \CP^\infty \times \BSU(p-1)$$
which is induced by the Cartesian square \cref{pullback-perfd}.
\end{proof}
\begin{remark}\label{jp-perfd}
\cref{thh-perfd} has the following slight variant: if $R$ is a $p$-complete perfectoid ring, then there is a $p$-complete equivalence $\THH(R/J(p)) \simeq R[\Omega S^3\times \CP^\infty]$. The only modification is that one instead has to use the $p$-local Cartesian square
$$\xymatrix{
\Omega S^3\pdb{3} \ar[d] \ar[r] & \Omega S^3 \ar[d] \\
\Omega S^3 \ar[r] & \Omega S^3 \times \CP^\infty,
}$$
which supplies a fibration
$$S^1 \to \Omega S^3 \to \Omega S^3 \times \CP^\infty.$$
In particular, the above discussion shows that $\pi_\ast \THH(R/J(p)) \cong R[x]\pdb{u}$.
This is compatible with \cref{rel-jp-diffracted}: 
\begin{enumerate}
    \item First, $\pi_\ast \THH(R/J(p))[x^{-1}] \cong R[x^{\pm 1}]\pdb{\tfrac{u}{x}}$. Since $\tfrac{u}{x}$ lives in degree $0$, \cref{rel-jp-diffracted} predicts that $\diffr_R \cong R\pdb{\tfrac{u}{x}}$. This is indeed true: \cite[Example 4.7.6]{apc} implies that the diffracted Hodge complex of a $p$-complete perfectoid ring $R$ is a divided power $R$-algebra on a single class in degree zero.
    \item Second, $\tau_{(2n-2, 2n]} \THH(R/J(p))$ is equivalent to $\bigoplus_{0\leq j\leq n} R\cdot \gamma_j(u) x^{n-j}$, so that \cref{rel-jp-diffracted} predicts that $\F^\conj_i \diffr_R$ is isomorphic to the $R$-submodule of $\diffr_R$ generated by $\{\gamma_j(\tfrac{u}{x})\}_{0\leq j\leq n}$. This is indeed true: see $(\ast_n)$ in the proof of \cite[Lemma 5.6.14]{apc}. In the same way, $\tau_{(2(n-1)p, 2np]} \THH(R/T(1))$ is a free $R$-module spanned by $\theta^i \gamma_j(u)$ for $(n-1-i)p < j \leq (n-i)p$. This includes $\gamma_j(u)$ for $(n-1)p<j \leq np$, but also terms such as $\theta^n$ and $\theta^{n-1} \gamma_p(u)$.
\end{enumerate}
\end{remark}
\begin{remark}\label{perfectoid-sen-operator}
We can understand the calculation of \cref{thh-perfd} more algebraically as follows. There is a $p$-local fiber sequence
\begin{equation}\label{cannot-deloop}
    S^{2p-1} \to \Omega S^3 \to \CP^\infty \times \Omega S^{2p+1},
\end{equation}
where the second map is given by the product of the canonical map $\Omega S^3 \to \CP^\infty$ with the James-Hopf map $\Omega S^3 \to \Omega S^{2p+1}$.
The Serre spectral sequence in $\Z_p$-homology for \cref{cannot-deloop} is given by
$$E^2_{\ast,\ast} = \Z_p\pdb{u} \otimes_{\Z_p} \Z_p[\theta, \epsilon]/\epsilon^2 \Rightarrow \pi_\ast \Z_p[\Omega S^3] \cong \Z_p[\sigma],$$
where $\epsilon$ lives in degree $2p-1$. It is not difficult to show that there is a single family of differentials given by
$$d^{2p}(\gamma_{p^n}(u)) = \epsilon \prod_{j=1}^{n-1} \gamma_{p^j}(u)^{p-1}, \ d^{2p}(\theta^j) = jp \theta^{j-1} \epsilon.$$
where the equality is to be understood up to $p$-adic units.
The above description implies that the map $d^{2p}:E^2_{2np,0} \to E^2_{2np-2p,2p-1}$ is surjective, and its kernel is a free $\Z_p$-module of rank $1$ (for example, one can calculate an explicit $(n+1)\times n$-matrix with coefficients in $\Z_p$ which describes $d^{2p}$).
If $R$ is a perfectoid ring, this discussion determines the Serre spectral sequence in $R$-homology for \cref{cannot-deloop}. Since the $d^{2p}$-differential in this spectral sequence is just the effect of the topological Sen operator $\Theta_R: \THH(R/X(p)) \to \Sigma^{2p} \THH(R/X(p))$ on homotopy, we see that $\Theta_R$ is given (up to $p$-adic units) by the map
$$\gamma_{p^n}(u) \mapsto \prod_{j=1}^{n-1} \gamma_{p^j}(u)^{p-1}.$$
Treat $u$ as a variable, and write $\frac{u^j}{j!}$ to denote $\gamma_j(u)$; then\footnote{
For the last equality, note that if $n\geq 1$, then $(p^n-1)!$ is a $p$-adic unit multiple of $\prod_{j=1}^{n-1} (p^j!)^{p-1}$.
Indeed, observe that $p^n - 1 = \sum_{j=0}^{n-1} p^j(p-1)$. By Legendre's formula for the $p$-adic valuation of factorials, we have $v_p(p^j!) = \frac{p^j-1}{p-1}$, so that
$$v_p((p^n-1)!) = \frac{p^n-1 - n(p-1)}{p-1} = -n + \sum_{j=0}^{n-1} p^j = \sum_{j=1}^{n-1} (p^j-1) = v_p\left(\prod_{j=1}^{n-1} (p^j!)^{p-1}\right),$$
as desired.
}
\begin{align*}
    (u^{1-p}\partial_u)(\gamma_{p^n}(u)) & = \frac{u^{p^n-p}}{(p^n-1)!} = \frac{(u^{p^{n-1}})^{p-1}}{(p^n-1)!} = \frac{u^{\sum_{j=1}^{n-1} p^j(p-1)}}{(p^n-1)!} \dot{=} \prod_{j=1}^{n-1} \left(\frac{u^{p^j}}{p^j!}\right)^{p-1}.
\end{align*}
Therefore, we may informally write $\Theta_R = u^{1-p} \partial_u$. The division by $u^p$ can be viewed as accounting for the shift by $2p$ in $\Theta_R$. Note that if $R$ is $p$-torsionfree, this operator can in turn be interpreted as $p\partial_{u^p}$. Similarly, under the isomorphism $\pi_\ast \THH(R/J(p)) \cong R[x]\pdb{u}$, the operator $\Theta'_R: \THH(R/J(p)) \to \Sigma^2 \THH(R/J(p))$ can be interpreted as $\partial_u$.
\end{remark}
A slight variant of the above discussion proves an analogous statement for $\Z/p^n$.
\begin{definition}
Let $Y_n$ denote the fiber of the composite
$$\HHP^\infty \to K(\Z, 4) \to K(\Z/p^{n-1}, 4).$$
\end{definition}
\begin{prop}\label{zpn-xp}
Fix an odd prime $p$. There are equivalences
\begin{align*}
    \THH(\Z/p^n/X(p)) & \simeq \Z/p^n[\Omega S^{2p+1} \times \B^2(p^{n-1}\Z)] \otimes_{\Z/p^n} \epsilon^{\Z/p^n}, \\
    \THH(\Z/p^n/J(p)) & \simeq \Z/p^n[\Omega S^3 \times \B^2(p^{n-1}\Z)],
\end{align*}
where the map $\FF_p \otimes_{\Z/p^n} \THH(\Z/p^n/J(p)) \to \THH(\FF_p/J(p))$ is given by the map $\FF_p[\Omega S^3 \times \B^2(p^{n-1}\Z)] \to \FF_p[\Omega S^3 \times \B^2(\Z)]$ induced by $p^{n-1} \Z\subseteq \Z$.
\end{prop}
\begin{proof}
In \cite{zpk-thom}, it was shown that $\Z/p^n$ is the Thom spectrum of the $\E{2}$-map
$$\Omega^3 Y_n \to \Omega^3 \HHP^\infty \simeq \Omega^2 S^3 \to \BGL_1(S^0),$$
which implies that $\THH(\Z/p^n) \simeq \Z/p^n[\Omega^2 Y_n]$. Note that there is a canonical map $\Omega^2 Y_n \to \Omega S^3$, and hence a map $\Omega^2 Y_n \to \CP^\infty$. Just as with \cref{thh-perfd}, we have
\begin{align*}
    \THH(\Z/p^n/X(p)) & \simeq \THH(\Z/p^n) \otimes_{\THH(X(p))} X(p)\\
    & \simeq \THH(\Z/p^n) \otimes_{\THH(\Z_p)} \THH(\Z_p/X(p))\\
    & \simeq \Z/p^n[\Omega^2 Y_n] \otimes_{\Z/p^n[\Omega S^3\pdb{3}]} \Z/p^n[\Omega S^{2p+1} \times \BSU(p-1)].
\end{align*}
There is still a $p$-local Cartesian square
$$\xymatrix{
\Omega S^3\pdb{3} \ar[d] \ar[r] & \Omega^2 Y_n \ar[d] \\
\Omega S^{2p+1} \times \BSU(p-1) \ar[r] & \Omega S^{2p+1} \times \B^2(p^{n-1} \Z) \times \BSU(p-1),
}$$
which implies the calculation of $\THH(\Z/p^n/X(p))$. The calculation of $\THH(\Z/p^n/J(p))$ is similar.
\end{proof}
\begin{remark}
One could also deduce \cref{zpn-xp} for $n\geq 2$ from \cref{thh-perfd} for $\FF_p$, using descent and the fact that $\HH(\FF_p/\Z/p^n) = \FF_p[K(\Z/p^{n-1}, 2)]$. Indeed, the composite $S^1 \xar{p^{n-1}} S^1 \xar{1-p} \BGL_1(S)$ detects the class $(1-p)^{p^{n-1}} = 1-p^n u \in \Z_p^\times$ for some $p$-adic unit $u$. Therefore, its Thom spectrum is equivalent to $\Z/p^n$. In turn, \cite[Proposition 2.1.6]{bpn-thom} (or \cite{beardsley-thom}) and the fiber sequence
$$S^1 \xar{p^{n-1}} S^1 \to B\Z/p^{n-1}$$
imply that $\FF_p$ is the Thom spectrum of a map $B\Z/p^{n-1} \to \BGL_1(\Z/p^n)$ which detects $1-p\in (\Z/p^n)^\times$ on the bottom cell of the source. Applying \cite{thh-thom} implies the desired calculation of $\HH(\FF_p/\Z/p^n)$.
\end{remark}
\begin{remark}
There is a higher chromatic analogue of \cref{zpn-xp}. To explain this, recall from \cite[Construction 3.5.1]{rotinv} that there is an $\E{2}$-algebra $S\ls{\hbar}$ over the sphere spectrum with $|\hbar|=-2$. It follows from \cite[Corollary 3.12]{framed-e2} that $S\ls{\hbar}$ can be upgraded to an $\Efr{2}$-algebra. Tensoring with $X(p^n)$ therefore defines an $\Efr{2}$-ring $X(p^n)\ls{\hbar}$; in particular, one can define THH relative to $X(p^n)\ls{\hbar}$. The $\E{2}$-map $X(p^n) \to \BP{n-1} \to \BP{n-1}^{tS^1}$ factors through an $\E{2}$-map $X(p^n)\ls{\hbar} \to \BP{n-1}^{tS^1}$, where $\hbar$ is sent to a complex orientation of $\BP{n-1}$ (viewed as a class in $\pi_{-2} \BP{n-1}^{tS^1}$). The calculation of \cref{thh-calculations} implies that
$$\THH(\BP{n-1}^{tS^1}/X(p^n)\ls{\hbar}) \simeq \BP{n-1}^{tS^1}[\Omega S^{2p^n+1} \times \B\Delta_n].$$
The spectrum $\BP{n-1}^{t\Z/m}$ is the quotient $\BP{n-1}^{tS^1}/\tfrac{[m](\hbar)}{\hbar}$, where $[m](\hbar)$ denotes the $m$-series of the formal group law over $\BP{n-1}_\ast$. This can be viewed as the Thom spectrum of a map $S^1 \to \BGL_1(\BP{n-1}^{tS^1})$ detecting $1 + \tfrac{[m](\hbar)}{\hbar} \in \pi_0(\BP{n-1}^{tS^1})^\times$. It follows that 
\begin{equation}\label{tZm-thh}
    \THH(\BP{n-1}^{t\Z/m}/X(p^n)\ls{\hbar}) \simeq \BP{n-1}^{t\Z/m}[\B S^1 \times \Omega S^{2p^n+1} \times \B\Delta_n].
\end{equation}
When $n=1$, there is an equivalence $\BP{0}^{t\Z/m} \simeq (\Z/m)^{tS^1}$, and \cref{tZm-thh} can be viewed as the equivalence of \cref{zpn-xp}, base-changed along $\Z/m \to (\Z/m)^{tS^1}$.
\end{remark}
Since $\B^2(p^{n-1} \Z) \cong \CP^\infty$ (more canonically, it is the total space of the line bundle $\co(p^{n-1})$ over the standard $\CP^\infty$), \cref{zpn-xp} implies that $\pi_\ast \THH(\Z/p^n/J(p)) \cong \Z/p^n[x]\pdb{u_n}$ with $|u_n|=|x|=2$. Were \cref{rel-jp-diffracted} to hold, \cref{zpn-xp} would imply that $\diffr_{\Z/p^n}$ is a (discrete) divided power algebra over $\Z/p^n$. In \cite[Example 5.15]{prismatization}, it is shown that if $\GG_a^\sharp$ denotes the PD-completion of $\GG_a$ at the origin, then $\spec(\Z/p^n)^{\slashed{D}} \otimes \FF_p \cong \GG_a^\sharp \otimes \FF_p$ in the notation of \cite{prismatization}. This implies that $\diffr_{\Z/p^n} \otimes_{\Z/p^n} \FF_p$ is isomorphic to the divided power algebra $\FF_p\pdb{t_n}$ for $|t_n| = 0$. However, as predicted by \cref{rel-jp-diffracted}, there is in fact no need to reduce modulo $p$: \cref{diffr-zpn} below says that $\diffr_{\Z/p^n}$ is indeed isomorphic to the divided power algebra $\Z/p^n\pdb{t_n}$ for $|t_n| = 0$.

I am grateful to Bhargav Bhatt for the statement of the following lemma, which is analogous to the calculation that if $R$ is a commutative ring and $x\in R$ is a regular element, then there is a $p$-complete equivalence $\dR_{R/x/R} \simeq R\pdb{x}/x$ (see \cite[Theorem 8.4]{bhatt-ddr}). The argument for \cref{ht-div-pow} below is my interpretation of Bhatt's explanation. The topological discussion above can be regarded as an analogue of the calculation that $\HH(R/x/R) \simeq R[\CP^\infty]/x$. We will freely use notation from \cite{apc, prismatization} below.
\begin{lemma}\label{ht-div-pow}
Let $(A, I)$ be a transversal prism (i.e., $A/I$ is $p$-torsionfree). Let $x\in A$ be an element such that $x\pmod{I}$ is regular in $\ol{A} := A/I$, and such that $(x)\subseteq A$ is $\phi$-stable. Then $\WCart^\HT_{A/(I,x)/A}$ is $p$-completely isomorphic to $\GG_a^\sharp \times \spf(A/(I,x))$, so that $\prismht_{A/(I,x)/A} \cong A/(I,x)\pdb{t}$ with $|t| = 0$.
\end{lemma}
\begin{proof}
By \cite[Proposition 5.12]{prismatization}, the map $\WCart^\HT_{A/(I,x)/A} \to \spf(A/(I,x))$ is a split gerbe, banded by $T_{A/(I,x)/\ol{A}}\{1\}^\sharp$. In this case, since $x\pmod{I}$ is a regular element of $\ol{A}$, we see that $L_{A/(I,x)/\ol{A}} = (x)/(x^2)[1]$, so that $T_{A/(I,x)/\ol{A}} = \spf \Sym_{A/(I,x)}(L_{A/(I,x)/\ol{A}})^\wedge_p$ is isomorphic to $\Omega \GG_a$ over $A/(I, x)$. It follows that $\WCart^\HT_{A/(I,x)/A}$ is isomorphic to a trivial $\GG_a^\sharp$-torsor over $\spf(A/(I,x))$. Since $\prismht_{A/(I,x)/A}$ is the global sections of the structure sheaf of $\WCart^\HT_{A/(I,x)/A}$, the lemma follows. 
\end{proof}
\begin{remark}
In fact, the conjugate filtration $\F^\conj_i \prismht_{A/(I,x)/A}$ is isomorphic to the divided power filtration on $A/(I,x)\pdb{t}$ under \cref{ht-div-pow}.
\end{remark}
\begin{remark}\label{sseq-divpow}
Sticking with the assumptions of \cref{ht-div-pow}, let us mention without proof that \cref{ht-div-pow} is also a consequence of \cite[Example 7.9]{bhatt-scholze}, which states that $\prism_{A/(I,x)/A} \cong A\{\frac{x}{I}\}^\wedge_{(p,I)}$. 
If $I = (d)$ is principal, the $p$-complete isomorphism
$$\beta: A/(I,x)\pdb{t}^\wedge_p \xar{\sim} \prismht_{A/(I,x)/A} \cong A\left\{\frac{x}{I}\right\}^\wedge_p/I$$
leads to an $I$-adic Bockstein spectral sequence
$$E_1^{\ast,\ast} = A/(I,x)\pdb{t}^\wedge_p[\ol{d}] \cong A\left\langle \frac{x}{d}\right\rangle^\wedge_p[\ol{d}]/d \Rightarrow A\left\{\frac{x}{d}\right\}^\wedge_{(p,d)},$$
where $\ol{d}$ represents $d$ on the $E_1$-page.

The map $\beta$ sends $\gamma_{p^n}(t)\mapsto \delta^n(\frac{x}{d})$ (up to $p$-adic units). This can be proved by showing that in the setting of \cref{ht-div-pow}, $\phi(\delta^n(\frac{x}{d})) \in (d) \subseteq A\{\frac{x}{d}\}$ if $n\geq 0$ (see \cref{phi-deltan-div-by-d} below). The fact that
$$\phi\left(\delta^n\left(\frac{x}{d}\right)\right) = \delta^n\left(\frac{x}{d}\right)^p + p\delta^{n+1}\left(\frac{x}{d}\right)$$
then implies that $\delta^n(\frac{x}{d})^p \equiv -p\delta^{n+1}(\frac{x}{d})\pmod{d}$. Therefore, the elements $\delta^n(\frac{x}{d})$ can be used to define divided powers of $\frac{x}{d}\pmod{d}$. In particular, we obtain the desired map $\beta: A/(I,x)\pdb{t} \to A\{\frac{x}{d}\}^\wedge_{(p,d)}/d$, but further work is required to show that it is a $p$-complete isomorphism.
\end{remark}
\begin{lemma}\label{phi-deltan-div-by-d}
Fix notation as in \cref{ht-div-pow}. Then $\phi(\delta^n(\frac{x}{d})) \in (d) \subseteq A\{\frac{x}{d}\}$.
\end{lemma}
\begin{proof}
Let $t = \frac{x}{d}$. The desired claim can be proved by induction on $n$. For the base case, we need to show that $\phi(t)\in I$. By reduction to the universal case, we may assume that $(p,d)$ is regular in $A$. Then \cite[Lemma 3.6]{anschutz-le-bras-TC} implies that the sequence $(d, \phi(d))$ is regular in $A$. Since $(x)$ is $\phi$-stable, we see that $d$ divides $\phi(x)$; it then follows from the formula $\phi(d) \phi(t) = \phi(x)$ that $d$ divides $\phi(t)$, as desired.
For the inductive step, observe that
$$p\phi(\delta^{n+1}(t)) = p\delta(\phi(\delta^n(t))) = \phi^2(\delta^n(t)) - \phi(\delta^n(t))^p.$$
The inductive hypothesis says that $\phi(\delta^n(t)) \in (d)$ for every $k\geq 1$, so that $d$ divides $p\phi(\delta^{n+1}(t))$. Since $(p,d)$ is a regular sequence, this implies that $d$ divides $\phi(\delta^{n+1}(t))$, as desired.
\end{proof}
This implies the following result, which is also proved in \cite[Lemma 6.13]{petrov-hdr}.
\begin{corollary}\label{diffr-zpn}
There is an isomorphism $\spec(\Z/p^n)^{\slashed{D}} \cong \GG_a^\sharp \times \spec(\Z/p^n)$ of $\Z/p^n$-schemes. In particular, the scaling action of $\GG_m^\sharp$ on $\GG_a^\sharp$ over $\Z/p^n$ gives an isomorphism $\WCart^\HT_{\Z/p^n} \cong \GG_a^\sharp/\GG_m^\sharp$ of $\Z/p^n$-stacks.
\end{corollary}
\begin{proof}
Recall that $\prismht_{\Z/p^n/\Z_p\pw{\ptl}} = \diffr_{\Z/p^n}$. \cref{ht-div-pow} implies that $\prismht_{\Z/p^n/\Z_p\pw{\ptl}} \cong \Z/p^n\pdb{t}$ with $|t| = 0$; this gives the desired claim. (It is useful to view $\gamma_{p^m}(t)$ as a $p$-adic unit multiple of $\delta^m(\frac{p^n}{\ptl})$, as described in \cref{sseq-divpow}.)

Alternatively, consider the transversal prism $(A,I) = (\Z_p\pw{q-1}, [p]_q)$, and let $x = (q-1)^{n(p-1)}$. Note that $\phi(x) \in (x)$, so $(x)$ is $\phi$-stable. Then $A/I \cong \Z_p[\zeta_p]$, and $A/(I,x)$ is isomorphic to $\Z_p[\zeta_p]/(\zeta_p - 1)^{n(p-1)} \cong \Z/p^n[\zeta_p]$ since the $p$-adic valuation of $(\zeta_p - 1)^{n(p-1)}$ is $n$. It follows from \cref{ht-div-pow} that $\prismht_{\Z/p^n[\zeta_p]/\Z_p\pw{q-1}} \cong \Z/p^n[\zeta_p]\pdb{t'}$ with $|t'| = 0$. There is an action of $\Z_p^\times$ (and hence $\FF_p^\times \subseteq \Z_p^\times$) on $(A, I)$; taking $\FF_p^\times$-fixed points produces an isomorphism 
$$\prismht_{\Z/p^n/\Z_p\pw{\ptl}} \cong (\prismht_{\Z/p^n[\zeta_p]/\Z_p\pw{q-1}})^{h\FF_p^\times} \cong \Z/p^n\pdb{t}$$
with $|t| = 0$, as desired. Note that as described in \cref{sseq-divpow}, the divided power $\gamma_{p^m}(t')$ can be viewed as a $p$-adic multiple of $\delta^m(\frac{(q-1)^{n(p-1)}}{[p]_q}) = \delta^m(\frac{(q-1)^{np-n+1}}{q^p-1})$.
\end{proof}
An alternative (and more hands-on) proof of \cref{diffr-zpn} is given in \cref{alt-pf-zpn}; this alternative argument is also presented as \cite[Lemma 6.13]{petrov-hdr}.
\begin{example}
Let us describe the topological Sen operator on $\THH(\Z/p^n/X(p))$ for $n\geq 2$ (recall that $p>2$). This is equivalent to describing the Serre spectral sequence in $\Z/p^n$-homology for the fibration
$$S^{2p-1} \to \Omega^2 Y_n \to \Omega S^{2p+1} \times \B(p^{n-1} \Z).$$
Note that this fibration is an analogue of the fibration \cref{toda-fibration}.

It will be simpler to analyze the Serre spectral sequence in $\Z_p$-homology, since all the differentials in the Serre spectral sequence in $\Z/p^n$-homology arise from the Serre spectral sequence in $\Z_p$-homology. The analysis is similar to \cref{perfectoid-sen-operator}; the Serre spectral sequence runs
\begin{equation}\label{serre-yn}
    E^2_{\ast,\ast} = \Z_p\pdb{u_n} \otimes_{\Z_p} \Z_p[\theta, \epsilon]/\epsilon^2 \Rightarrow \pi_\ast \Z_p[\Omega^2 Y_n],
\end{equation}
where $\epsilon$ lives in degree $2p-1$ and $u_n$ lives in degree $2$. There are several ways to determine the differentials in this spectral sequence. Our approach will be to describe the pattern of differentials by first calculating $\pi_\ast \Z_p[\Omega^2 Y_n]$; in turn, we will do this by computing $\pi_\ast C^\ast(\Omega^2 Y_n; \Z_p)$. For this, we use the Serre spectral sequence for the fibration
$$\B\Z/p^{n-1} \to \Omega^2 Y_n \to \Omega S^3.$$
Since $\H^\ast(\B\Z/p^{n-1}; \Z) \cong \Z[c]/p^{n-1} c$ with $|c|=2$, the Serre spectral sequence collapses on the $E_2$-page, and we find that $\pi_\ast C^\ast(\Omega^2 Y_n; \Z_p) \cong \Z_p\pdb{x}[c]/(x-p^{n-1}c)$ with $|x|=2$. (If $n=1$, then $\Omega^2 Y_n \simeq \Omega S^3$, and the cohomology ring is $\Z_p\pdb{x}$.) For $n\geq 2$, this is isomorphic to $\Z_p\pdb{y}[c]/y$, where $y = x - p^{n-1}c$. Indeed, observe that if $n\geq 2$, then
$$\gamma_j(y) := \sum_{i=0}^j \frac{p^{i(n-1)}}{i!} c^i \gamma_{j-i}(x)$$
is a well-defined class in $\Z_p\pdb{x}[c]/(x-p^{n-1}c)$ since $p$ has divided powers in $\Z_p$, and that these classes form a basis for $\Z_p\pdb{x}[c]/(x-p^{n-1}c)$ as a $\Z_p[c]$-module. 
Recall that in homological grading, there is an equivalence:
$$\Z_p\pdb{y}/y \simeq \Z_p \oplus \bigoplus_{n\geq 1} \Z_p/n\{\gamma_n(y)\} [-2n],$$
which implies that if $n\geq 2$, then
$$\H^i(\Omega^2 Y_n; \Z_p) \cong \begin{cases}
\Z_p \oplus \bigoplus_{j=1}^k \Z_p/j \{\gamma_j(y) c^{k-j}\} & i = 2k\geq 0 \text{ even}\\
0 & \text{else}.
\end{cases}$$
Using the universal coefficients theorem, we find that if $n\geq 2$, then
$$\pi_i \Z_p[\Omega^2 Y_n] \cong \begin{cases}
\Z_p & i\in 2\Z_{\geq 0}\\
\bigoplus_{j=1}^k \Z_p/j & i = 2k-1.
\end{cases}$$
The generator of $\pi_{2j} \Z_p[\Omega^2 Y_n]$ is the linear dual to $c^j\in \Z_p\pdb{y}[c]/y$, while the generator of $\pi_{2k-1} \Z_p[\Omega^2 Y_n]$ which is killed by $j$ is dual to $\gamma_j(y) c^{k-j}$.
Note that the homotopy groups $\pi_\ast \Z_p[\Omega^2 Y_n]$ are \textit{independent} of $n$ if $n\geq 2$ (but the generators of these groups do depend on $n$). 

Let us now return to the Serre spectral sequence \cref{serre-yn}. Comparison with the Serre spectral sequence for the fibration \cref{toda-fibration} (i.e., with the topological Sen operator on $\THH(\Z_p/X(p))$; see \cref{sen-bpn}) forces the differentials in \cref{serre-yn} to be given by (up to $p$-adic units):
$$d^{2p}(\gamma_{p^k}(u_n)) = p^{n-1} \epsilon \prod_{j=1}^{k-1} \gamma_{p^j}(u_n)^{p-1} = p^{n} \epsilon \partial_{u_n^p}(\gamma_{p^k}(u_n)), \ d^{2p}(\theta^j) = jp\theta^{j-1} \epsilon.$$
Reducing modulo $p^n$, we get the topological Sen operator on $\THH(\Z/p^n/X(p))$ for $n\geq 2$:
$$\Theta: \gamma_{p^k}(u_n) \mapsto p^{n-1} \prod_{j=1}^{k-1} \gamma_{p^j}(u_n)^{p-1}, \ \Theta: \theta^j\mapsto jp \theta^{j-1}.$$
Observe that this acts as ``$p^n \partial_{u_n^p}$''.
Of course, one can similarly deduce the action of the topological Sen operator on $\THH(\Z/p^n/J(p))$.
This recovers the calculation
$$\pi_j \THH(\Z/p^n) = \begin{cases}
\bigoplus_{i=0}^j \Z/\gcd(j,p^n) & \text{even }j\geq 0, \\
\bigoplus_{i=0}^j \Z/\gcd(j,p^n) & \text{odd }j\geq 0, \\
0 & j<0.
\end{cases}$$
\end{example}
Another example of the topological Sen operator comes from studying complete DVRs, where the relationship between $\THH$ relative to $J(p)$ and the diffracted Hodge complex predicted by \cref{rel-jp-diffracted} can be seen directly.
\begin{example}\label{cdvr-sen}
Let $R$ be a $p$-torsionfree complete DVR of mixed characteristic $(0,p>0)$ whose residue field $k$ is perfect. Then we have
$$\pi_\ast \THH(R/X(p)) \cong \HH_\ast(R/\Z_p)[\theta] \otimes_{\Z_p} \epsilon^{\Z_p}_\ast,$$
and the map $\Theta: \pi_\ast \THH(R/X(p)) \to \pi_{\ast-2p} \THH(R/X(p))$ sends $\theta^j \mapsto jp \theta^{j-1}$. To compute the action of the topological Sen operator on the remainder of $\THH(R/X(p))$, it will be simpler to assume that $T(1)$ is an $\E{2}$-ring and work instead with $\THH(R/T(1))$; this is merely cosmetic, and it is not difficult to modify the below argument to use $\THH(R/X(p))$ instead. Then, we have $\pi_\ast \THH(R/T(1)) \cong \HH_\ast(R/\Z_p)[\theta]$. We will compute $\THH(R)$ using the topological Sen operator on $\THH(R/T(1))$ and \cref{tn-sen}. Let $\pi\in R$ be a uniformizer, let $E(u)\in W(k)\pw{u}$ be its minimal polynomial, and let $E'(u)\in W(k)\pw{u}$ denote its derivative with respect to $u$.
Recall that $R = W(k)\pw{u}/E(u)$, that $W(k)$ is \'etale over $\Z_p$, $\pi_\ast \HH(W(k)\pw{u}/W(k)) \cong \Lambda_{W(k)\pw{u}}(du)$ with $|du|=1$, and $\pi_\ast \HH(R/W(k)\pw{u}) \cong R\pdb{\sigma_E}$, where $\sigma_E := \sigma^2(E(u))$. The transitivity sequence for the composite $W(k) \to W(k)\pw{u} \to R$ implies that $\HH(R/W(k)) \simeq \HH(R/\Z_p)$ is the fiber of a map $R\pdb{\sigma_E} \to \Sigma^2 R\pdb{\sigma_E}$ sending $\gamma_n(\sigma_E) \mapsto E'(\pi) \gamma_{n-1}(\sigma_E)$. In particular,
$$\pi_n \HH(R/\Z_p) \cong \begin{cases}
R & n=0, \\
R/E'(\pi) & n=2j+1, \ j\geq 0,\\
0 & \text{else}.
\end{cases}$$
Let us denote the generator of $\pi_{2j-1} \HH(R/\Z_p)$ by $z_j$, so that $\gamma_{j-1}(\sigma_E) \in \pi_{2j} \Sigma^2 R\pdb{\sigma_E}$ is sent to $z_j$ under the boundary map $\Sigma^2 R\pdb{\sigma_E} \to \Sigma \HH(R/\Z_p)$. We then have
\begin{equation}\label{cdvr-t1}
    \pi_n \THH(R/T(1))
\cong \begin{cases}
R\cdot \theta^j & n=2pj, j \geq 0\\
\bigoplus_{0\leq i < j/p} R/E'(\pi) \cdot z_{j-pi} \theta^i & n = 2j-1, j\geq 0,\\
0 & \text{else}.
\end{cases}
\end{equation}
From this, we can describe the topological Sen operator on $\THH(R/T(1))$. For this, it will be useful to rephrase the above calculations somewhat, and use $J(p)$ instead of $T(1)$. It is easy to compute that $\pi_\ast \THH(R/J(p)) \cong \HH_\ast(R/\Z_p)[x]$, where $x$ is the class in degree $2$ from \cref{laurent-xp}. In other words,
$$
\pi_n \THH(R/J(p)) \cong \begin{cases}
R \cdot x^j & n = 2j \text{ for }j \geq 0, \\
\bigoplus_{0\leq i < j} R/E'(\pi) \cdot z_{j-i} x^i & n = 2j-1, j\geq 1.
\end{cases}
$$
Since $\HH(R/\Z_p)$ is the fiber of a map $R\pdb{\sigma_E} \to \Sigma^2 R\pdb{\sigma_E}$, it follows that there is a cofiber sequence
\begin{equation}\label{jp-dvr}
    \THH(R/J(p)) \to R\pdb{\sigma_E}[x] \xar{\nabla} \Sigma^2 R\pdb{\sigma_E}[x],
\end{equation}
where we have denoted the second map by $\nabla$. The map $\nabla$ is given on homotopy by a derivation, sending $\sigma_E \mapsto E'(\pi)$. Informally, $\THH(R/J(p))$ can be written as $R\pdb{\sigma_E}[x]^{\nabla=0}$.


The topological Sen operator $\Theta: \THH(R/J(p)) \to \Sigma^2 \THH(R/J(p))$ is described on homotopy by the operator on $R\pdb{\sigma_E}[x]$ sending $x\mapsto nx^{n-1}$. Note that this operator commutes with $\nabla$ (so that it does indeed define an operator on $\pi_\ast \THH(R/J(p))$).
Observe that since $\THH(R)$ is the fiber of $\Theta: \THH(R/J(p)) \to \Sigma^2 \THH(R/J(p))$, and $\THH(R/J(p))$ is the fiber of $\nabla: R\pdb{\sigma_E}[x] \to \Sigma^2 R\pdb{\sigma_E}[x]$, we can write $\THH(R)$ as the total fiber of the square
\begin{equation}\label{total-fiber}
    \xymatrix{
    R\pdb{\sigma_E}[x] \ar[r]^-{\nabla} \ar[d]^-\Theta & \Sigma^2 R\pdb{\sigma_E}[x] \ar[d]^-\Theta \\
    \Sigma^2 R\pdb{\sigma_E}[x] \ar[r]^-{\nabla} & \Sigma^4 R\pdb{\sigma_E}[x],
    }
\end{equation}
where the map denoted $\Theta$ sends $x^n\mapsto nx^{n-1}$. In turn, it follows that $\THH(R)$ is also the total fiber of the square
\begin{equation}\label{theta-nabla}
    \xymatrix{
    R\pdb{\sigma_E}[x] \ar[r]^-{\nabla} \ar[d]^-{\Theta + \nabla} & \Sigma^2 R\pdb{\sigma_E}[x] \ar[d]^-{\Theta + \nabla} \\
    \Sigma^2 R\pdb{\sigma_E}[x] \ar[r]^-\nabla & \Sigma^4 R\pdb{\sigma_E}[x].
    }
\end{equation}
The operator $\nabla + \Theta$ acts on $R\pdb{\sigma_E}[x]$ by
\begin{equation}\label{psi-deg2nm}
    \nabla + \Theta: x^n \gamma_m(\sigma_E) \mapsto nx^{n-1} \gamma_m(\sigma_E) + E'(\pi) x^n \gamma_{m-1}(\sigma_E).
\end{equation}
Let us now invert $x$, and write $y = \sigma_E x^{-1}$ in $R\pdb{\sigma_E}[x^{\pm 1}]$. Then $y$ has divided powers and lives in degree $0$, and there is an isomorphism $R\pdb{\sigma_E}[x^{\pm 1}] \cong R\pdb{y}^\wedge_{(y)}[x^{\pm 1}]$ on homotopy. We can formally define $\Theta$ on $x^n$ for $n\leq 0$ by the same formula: $\Theta(x^n)=nx^{n-1}$. It follows from \cref{psi-deg2nm} that $\Psi := x(\nabla + \Theta)$ sends
$$\Psi: \gamma_n(y) \mapsto -nx^{-n} \gamma_n(\sigma_E) + E'(\pi) x^{-n+1} \gamma_{n-1}(\sigma_E) = E'(\pi) \gamma_{n-1}(y) - n\gamma_n(y).$$
We claim that the action of $-t\partial_t$ on $R\pdb{(1-t)E'(\pi)}$ agrees the action of $\Psi$ on $R\pdb{y}^\wedge_{(y)}$, if we identify $y = (1-t) E'(\pi)$. Indeed:
\begin{align*}
    \gamma_n(y) & = \gamma_n((1-t) E'(\pi)) = -E'(\pi)^n \frac{(1-t)^n}{n!} \\
    & \xmapsto{-t\partial_t} E'(\pi)^n \frac{t(1-t)^{n-1}}{(n-1)!} = E'(\pi)^n \left(\frac{(1-t)^{n-1}}{(n-1)!} - n\frac{(1-t)^n}{n!}\right) \\
    & = E'(\pi) \gamma_{n-1}((1-t) E'(\pi)) - n\gamma_n((1-t) E'(\pi))\\
    & = E'(\pi) \gamma_{n-1}(y) - n\gamma_n(y).
\end{align*}
In particular, we can rewrite the square \cref{theta-nabla} after inverting $x$ as
$$\xymatrix{
R\pdb{(1-t)E'(\pi)}[x^{\pm 1}] \ar[r] \ar[d]^-{t\partial_t} & R\pdb{(1-t)E'(\pi)}[x^{\pm 1}] \ar[d]^-{t\partial_t}\\
R\pdb{(1-t)E'(\pi)}[x^{\pm 1}] \ar[r] & R\pdb{(1-t)E'(\pi)}[x^{\pm 1}],
}$$
where the horizontal maps act by sending $\gamma_n((1-t)E'(\pi)) \mapsto E'(\pi) \gamma_{n-1}((1-t)E'(\pi))$. The fiber of either of the horizontal maps in the above square can be identified with $\THH(R/J(p))[x^{-1}]$.

Using the above description of $\Theta$, one can calculate (with some tedium) that
$$\pi_n \THH(R) = \begin{cases}
R & n=0,\\
R/jE'(\pi) & n = 2j-1 \geq 0,\\
0 & \text{else}.
\end{cases}$$
This is exactly the calculation of $\pi_\ast \THH(R)$ from \cite[Theorem 5.1]{lindenstrauss-madsen} (reproved in \cite[Theorem 4.4]{krause-nikolaus-dvr}).

The above discussion can be compared to \cite[Remark 9.7]{prismatization}, which says that if $X = \spec R$, then $\WCart^\HT_X \cong X \times BG$, where $G = \{(a,t)\in \GG_a^\sharp \rtimes \GG_m^\sharp | t-1 = E'(\pi) a\}$. The canonical map $\WCart^\HT_X \to X \times \WCart^\HT \cong (B\GG_m^\sharp)_X$ can be identified with the map induced on classifying stacks by the quotient map $G \to \GG_m^\sharp$ of group schemes over $X$.
Recall that the diffracted Hodge stack $X^{\slashed{D}}$ can be identified with $\WCart^\HT_X \times_{\WCart^\HT} \spec(\Z_p) \cong \WCart^\HT_X \times_{\WCart^\HT \times X} X$. In particular, $X^{\slashed{D}} \cong (\GG_m^\sharp)_X/G$, i.e., the classifying stack of the group scheme $\GG_a^\sharp[E'(\pi)] = \{a\in \GG_a^\sharp | E'(\pi) a = 0\}$. One can show from this description that the $2$-periodification of the cohomology of the diffracted Hodge complex $\diffr_R \cong \Gamma(B\GG_a^\sharp[E'(\pi)]; \co)$ can be identified with $\pi_\ast \THH(R/J(p))[x^{-1}]$ (which is additively the $2$-periodification of $\pi_\ast \HH(R/\Z_p)$), as predicted by \cref{rel-jp-diffracted}.

Note that the extensions in the following long exact sequence in homotopy for \cref{tn-sen} are \textit{always} nontrivial:
\begin{align}
    \cdots \to \pi_{2j} \THH(R/T(1)) \to \pi_{2(j-p)} \THH(R/T(1)) \to \pi_{2j-1} \THH(R) \to \label{lexseq} \\
    \pi_{2j-1} \THH(R/T(1)) \to \pi_{2(j-p)-1} \THH(R/T(1)) \to \pi_{2j-2} \THH(R) \to \cdots \nonumber
\end{align}
For example, when $j=p$, there is a long exact sequence
$$\pi_{2p} \THH(R/T(1)) \cong R \cdot \theta \to \pi_{0} \THH(R/T(1)) \cong R \to \pi_{2p-1} \THH(R) \to \pi_{2p-1} \THH(R/T(1)) \to 0,$$
which in particular gives a short exact sequence
$$0 \to R/p \to \pi_{2p-1} \THH(R) \to R/E'(\pi) \cong \pi_{2p-1} \THH(R/T(1)) \to 0.$$
Since $\pi_{2p-1} \THH(R) \cong R/pE'(\pi)$, this extension must be nontrivial.
\end{example}

%% file: sen/relation-to-p-dr.tex
We now describe some additional calculations which give further evidence for \cref{rel-jp-diffracted}.
\begin{remark}\label{arpon}
Assume that $R$ is the $p$-completion of $\Z_p[t]$. Forthcoming work of Arpon Raksit (\cite{arpon-qdr}) shows that (a completion of) $q\Omega_R$ arises as the associated graded of a motivic filtration on $\HP(\ku[t]/\ku)$. In fact, Raksit studies $\HP(A[t]/A)$ for a general $\Eoo$-ring $A$ with even homotopy groups.
\end{remark}
Using \cref{arpon}, one can show that (a completion of) $\ptl\Omega_R$ arises as the associated graded of a motivic filtration on $\HP(\BP{1}[t]/\BP{1})$. Moreover, the class $\ptl$ is identified as the image of $v_1 \hbar^{p-1}$ in the associated graded.
For the sake of completeness, let us explicitly compute $\pi_\ast \TP(\Z_p[t]/X(p))$. As in \cref{cdvr-sen}, it will be convenient to assume that $T(1)$ is an $\Efr{2}$-ring and work instead with $\THH(R/T(1))$; again, this is merely cosmetic. We first need the following result, which is a special case of \cite[Proposition 3.1.1]{riggenbach-thesis} and $S^1$-equivariant Poincar\'e duality for $S^1/\mu_n$.
\begin{lemma}\label{tate-product}
Let $X$ be a bounded-below spectrum equipped with an action of $S^1$. Then there is an equivalence
$$\left(\bigoplus_{n \geq 1} X \otimes (S^1/\mu_n)_+\right)^{tS^1} \simeq \lim_{k\to \infty} \bigoplus_{n\geq 1} \left(\Sigma \tau_{\leq k} X\right)^{t\Cn}.$$
\end{lemma}
\begin{example}\label{polynomial-t1}
Let $S$ be the sphere spectrum. Recall that $\Z_p[t] \simeq \Z_p \otimes S[t]$, so that $\THH(\Z_p[t]/T(1)) \simeq \THH(\Z_p/T(1)) \otimes \THH(S[t])$. Let $\THH(S[t],(t))$ denote the fiber of the map $\THH(S[t]) \to \THH(S) \simeq S$ induced by the augmentation $S[t] \to S$ sending $t\mapsto 0$; note that the map $\THH(S[t]) \to S$ admits an $S^1$-equivariant splitting. Similarly, we write $\THH(\Z_p[t]/T(1), (t))$ to denote the fiber of the map $\THH(\Z_p[t]/T(1)) \to \THH(\Z_p/T(1))$ induced by the augmentation $\Z_p[t] \to \Z_p$. Then $\THH(S[t],(t)) \simeq \bigoplus_{n\geq 1} (S^1/\mu_n)_+$, so that
\begin{equation}\label{formula-poly-t1}
    \THH(\Z_p[t]/T(1), (t)) \simeq \THH(\Z_p/T(1)) \otimes \THH(S[t],(t)) \simeq \bigoplus_{n\geq 1} (S^1/\mu_n)_+ \otimes \THH(\Z_p/T(1)).
\end{equation}
It follows from \cref{tate-product} that there is an equivalence
$$\THH(\Z_p[t]/T(1), (t))^{tS^1} \simeq \lim_{k\to \infty} \bigoplus_{n\geq 1} \Sigma (\tau_{\leq k} \THH(\Z_p/T(1)))^{t\Cn}.$$
Using \cref{thh-calculations}(a), we have $\tau_{\leq 2kp} \THH(\Z_p/T(1)) \simeq \Z_p[J_k(S^{2p})]$. A simple calculation using \cref{thh-calculations}(a) shows that there is an isomorphism 
$$\pi_\ast (\tau_{\leq 2kp} \THH(\Z_p/T(1)))^{t\Cn} \cong \pi_\ast (\BP{1}/v_1^{k+1})^{t\Cn} \cong \pi_\ast(\tau_{\leq 2k(p-1)} \BP{1})^{t\Cn}.$$
Let $\pdb{n}(\hbar) := \tfrac{[n](\hbar)}{\hbar}$, so that $\pi_\ast \BP{1}^{t\Cn} \cong \Z_p[v_1]\ls{\hbar}/\pdb{n}(\hbar)$. In analogy to $q = \beta\hbar + 1$, if we define $\ptl = v_1 \hbar^{p-1}$, then $\pdb{n}(\hbar)$ defines an element of $\Z_p\pw{\ptl}$ which we will denote $\pdb{n}_\ptl$.
We conclude that
$$\pi_\ast (\tau_{\leq 2kp} \THH(\Z_p/T(1)))^{t\Cn} \cong \Z_p\pw{\ptl}\ls{\hbar}/(\ptl^{k+1}, \pdb{n}(\hbar)).$$
It follows that
$$\pi_\ast \THH(\Z_p[t]/T(1), (t))^{tS^1} \cong \lim_{k\to \infty} \bigoplus_{n\geq 1} \Sigma \Z_p\pw{\ptl}\ls{\hbar}/(\ptl^{k+1}, \pdb{n}_\ptl),$$
i.e., that
$$\pi_\ast \TP(\Z_p[t]/T(1)) \cong \Z_p\pw{\ptl}\ls{\hbar} \times \lim_{k\to \infty} \bigoplus_{n\geq 1} \Sigma \Z_p\pw{\ptl}\ls{\hbar}/\pdb{n}_\ptl.$$
\end{example}
In a manner similar to \cref{polynomial-t1}, one calculates that if we write $\Z_p[\beta]\ls{\hbar} = \Z_p\pw{q-1}\ls{\hbar}$ by setting $q = 1+\beta \hbar$, and $\pdb{n}_{\GG_m}(\hbar) = \tfrac{[n]_{\GG_m}(\hbar)}{\hbar}$ is the divided $n$-series of the rescaled multiplicative formal group law $x + y + (q-1) xy$, then
\begin{equation}\label{HP-ku}
    \pi_\ast \HP({\ku^\wedge_p}[t]/{\ku^\wedge_p}) \cong \Z_p\pw{q-1}\ls{\hbar} \times \lim_{k\to \infty} \bigoplus_{n\geq 1} \Sigma \Z_p\pw{q-1}\ls{\hbar}/((q-1)^{k+1},\pdb{n}_{\GG_m}(\hbar)).
\end{equation}
Moreover, $\pi_\ast \HP(\BP{1}[t]/\BP{1}) \cong \pi_\ast \HP({\ku^\wedge_p}[t]/{\ku^\wedge_p})^{\FF_p^\times}$, where $\FF_p^\times$ acts on $\HP({\ku^\wedge_p}[t]/{\ku^\wedge_p})$ via its action by Adams operations on ${\ku^\wedge_p}$. Note that since $\FF_p^\times$ has order coprime to $p$, taking $\FF_p^\times$-invariants preserves small limits and colimits after $p$-localization. In particular, $\pi_\ast \HP(\BP{1}[t]/\BP{1})$ is isomorphic to $\pi_\ast \TP(\Z_p[t]/T(1))$.

The following is also a consequence of the forthcoming work of Arpon Raksit (\cite{arpon-qdr}) mentioned above.
\begin{lemma}\label{n-series-BP1}
There is an $\Z_p^\times$-equivariant isomorphism
$$\H^\ast(q\Omega_{\Z_p[t]})\ls{\hbar} \cong \Z_p\pw{q-1}\ls{\hbar} \times \bigoplus_{n\geq 1} \Sigma \Z_p\pw{q-1}\ls{\hbar}/\pdb{n}_{\GG_m}(\hbar).$$
\end{lemma}
\begin{proof}
For the formal group law over ${\ku^\wedge_p}$, we have 
$$\pdb{n}_{\GG_m}(\hbar) = \sum_{i=1}^{n} \binom{n}{i} \hbar^{i-1} \beta^{i-1} = [n]_q \in \Z_p[\beta]\ls{\hbar},$$
where $q := 1+\beta\hbar$. The claim now follows from the fact that the differential $\nabla_q$ in $q\Omega_{\Z_p[t]}$ sends $t^n\mapsto [n]_q t^{n-1} dt$.
\end{proof}
In particular, $\pi_\ast \HP(\BP{1}[t]/\BP{1}) \cong \pi_\ast \TP(\Z_p[t]/T(1))$ is a $2$-periodification of a completion of $\H^\ast(\ptl\Omega_{\Z_p[t]})$. This calculation leads to the following expectation related to \cref{rel-jp-diffracted}:
\begin{conjecture}\label{p-tilde-t1}
Let $R$ be an animated $\Z_p$-algebra. Then $\TP(R/X(p))$ admits a motivic filtration $\F^\star_\mot \TP(R/X(p))$ such that $\gr^i_\mot \TP(R/X(p)) \simeq \hat{\prism}_{R/\Z_p\pw{\ptl}}[2i] \otimes_R \epsilon^R$, where $\hat{\prism}_{R/\Z_p\pw{\ptl}}$ is the Nygaard completion of $\ptl \Omega_R$.
\end{conjecture}

We now turn to a higher chromatic analogue of (part of) this picture.
\begin{definition}\label{hh-polynomial}
Let $R$ be an $\E{2}$-ring, and equip $\HH(R[t]/R) := \THH(S[t]) \otimes R$ with the $S^1$-action inherited from $\THH(S[t])$ and the trivial action on $R$.
\end{definition}
\begin{warning}\label{abusive HH notation}
If $R$ is only an $\E{2}$-ring, one \textit{cannot} define Hochschild homology relative to $R$; in particular, the notation $\HH(R[t]/R)$ is rather abusive. As explained in \cite[Corollary 2.9]{framed-e2}, if $R'$ is an $\E{1}$-$R$-algebra, then $\HH(R'/R)$ only exists (and has a natural $S^1$-action) when $R$ is a \textit{framed} $\E{2}$-ring\footnote{Suppose that $R$ is an $\E{3}$-algebra, and $\cC$ is an $R$-linear $\infty$-category. The choice of a framed knot in $\RR^3$ also defines an $\E{n-3}$-map $\int_{S^1} R \to R$, and hence allows one to define relative Hochschild homology $\HH(\cC/R)$. However, this does not define an $S^1$-action on $\HH(\cC/R)$! Thanks to Robert Burklund for this point.}. In other words, if $R$ is merely an $\E{2}$-ring, it would not be clear how to define $\HH(R[t]/R)$, had we not known that $R[t]$ admits a lift to the sphere spectrum. This leads to the following unfortunate warning: if $R$ is an $\Efr{2}$-ring with a nontrivial $S^1$-action, then the (more natural) circle action on $\HH(R[t]/R)$ arising via the $S^1$-action on $R$ \textit{cannot} be necessarily identified with the circle action from \cref{hh-polynomial}. However, for this article, we will only use the circle action from \cref{hh-polynomial}.
\end{warning}

View $\BP{n-1}[t_1, \cdots, t_j]$ as a $\Z_{\geq 0}^j$-graded ring, where $t_i$ has weight $(0,\cdots,1,\cdots,0)$. Then, define $\HP^\gr(\BP{n}[t_1, \cdots, t_j]/\BP{n})$ to be the $S^1$-Tate construction of $\HH(\BP{n}[t_1, \cdots, t_j]/\BP{n})$ taken internally to $\Z_{\geq 0}^j$-graded $\BP{n}$-modules. Similarly, define $\TP^\gr(\BP{n-1}[t_1, \cdots, t_j]/X(p^n))$ to be the $S^1$-Tate construction of $\THH(\BP{n-1}[t_1, \cdots, t_j]/X(p^n))$ taken internally to $\Z_{\geq 0}^j$-graded $\BP{n-1}$-modules. Then, related to \cref{p-tilde-t1}, we have the following result (which, when $n=0$, is a very special case of the main result of \cite{petrov-vologodsky}):
\begin{prop}\label{graded-drw-lift}
There is a $p$-complete isomorphism of $\Z_{\geq 0}^j$-graded modules equipped with a map from $\pi_\ast \BP{n}^{tS^1}[\B\Delta_n] \cong \pi_\ast \TP(\BP{n-1}/X(p^n))$:
$$\pi_\ast \HP^\gr(\BP{n}[t_1, \cdots, t_j]/\BP{n})[\B\Delta_n] \cong \pi_\ast \TP^\gr(\BP{n-1}[t_1, \cdots, t_j]/X(p^n)).$$
The map $\TP^\gr(\BP{n-1}[t]/X(p^n)) \to \TP(\BP{n-1}/X(p^n))$ is an equivalence after $K(n)$-localization.
\end{prop}
\begin{proof}
For simplicity, we assume that $j=1$ and write $t$ instead of $t_1$. In the graded setting, we may commute the $S^1$-Tate construction with the infinite direct sum (i.e., \cref{tate-product} is not necessary). It follows that there are \textit{graded} equivalences
\begin{align*}
    \TP^\gr(\BP{n-1}[t]/X(p^n), (t)) & \simeq \bigoplus_{m\geq 1} \Sigma \THH(\BP{n-1}/X(p^n))^{t\Z/m} (m), \\
    \HP^\gr(\BP{n}[t]/\BP{n}, (t)) & \simeq \bigoplus_{m\geq 1} \Sigma \BP{n}^{t\Z/m} (m).
\end{align*}
The desired result now follows from \cref{thh-calculations}(a). The second statement follows from the above equivalences and the fact that $\Lk (\BP{n}^{t\Z/p^m}) = 0$ by \cref{lk-bpn}.
\end{proof}
The proof above used the following (well-known) fact.
\begin{lemma}\label{lk-bpn}
Let $\BP{n}$ denote any form of the truncated Brown-Peterson spectrum. Then we have $\Lk (\BP{n}^{t\Z/p^m}) = 0$.
\end{lemma}
\begin{proof}
We first observe that $\BP{n}^{t\Z/p^m}$ depends only on the $p$-completion of $\BP{n}$; indeed, the obvious variant of \cite[Lemma I.2.9]{nikolaus-scholze} shows that if $X$ is a bounded-below spectrum with $\Z/p^m$-action, then $X^{t\Z/p^m}$ is $p$-complete, and the map $X^{t\Z/p^m} \to (X^\wedge_p)^{t\Z/p^m}$ is an equivalence. Since all forms of $\BP{n}$ are equivalent after $p$-completion by \cite{angeltveit-lind}, we may therefore reduce to proving the claim for a single form of $\BP{n}$.

To show that $\Lk (\BP{n}^{t\Z/p^m}) = 0$, it suffices to show (since $\BP{n}^{t\Z/p^m}$ is an $\MU$-module) that $(\pi_\ast \BP{n}^{t\Z/p^m}[\frac{1}{v_n}])/(p,\cdots,v_{n-1}) \cong \pi_\ast k(n)^{t\Z/p^m}[\frac{1}{v_n}] = 0$. Recall that $\pi_\ast \BP{n}^{t\Z/p^m} \cong \BP{n}_\ast\ls{\hbar}/[p^m](\hbar)$. We will work with the form of $\BP{n}$ such that the associated formal group law over $\pi_\ast \BP{n}$ induces the Honda formal group law over $\pi_\ast k(n)$. Then, the $p^m$-series of the formal group law over $\pi_\ast k(n)$ satisfies $[p^m](\hbar) = v_n^{\frac{p^{nm}-1}{p^n-1}} \hbar^{p^{mn}}$; so $\pi_\ast k(n)^{t\Z/p^m} \cong \FF_p[v_n]\ls{\hbar}/v_n^{\frac{p^{nm}-1}{p^n-1}}$. In particular, $v_n$ is nilpotent in $k(n)^{t\Z/p^m}$, so that $\pi_\ast k(n)^{t\Z/p^m}[\frac{1}{v_n}] = 0$.
\end{proof}
\begin{remark}\label{generalized-de-rham}
In general, $\pi_\ast \HP(\BP{n}[t]/\BP{n})$ looks like a completion of the $2$-periodification of the cohomology of the following two-term complex:
\begin{equation}\label{two-term-bpn}
    \BP{n}_\ast\pw{\hbar}[t] \xar{\nabla} \BP{n}_\ast\pw{\hbar}[t] dt, \ \nabla: t^m \mapsto \tfrac{[m]_{\BP{n}}(\hbar)}{\hbar} t^{m-1} dt.
\end{equation}
This is a variant of the $q$-de Rham complex, and was first considered by Arpon Raksit (in forthcoming work). Note that an analogue of \cref{two-term-bpn} can be defined for a formal group law $F(x,y)$ over any commutative ring $A$:
\begin{equation}\label{any-ring-two-term}
    F\Omega_{A[t]/A} := \left(A\pw{\hbar}[t] \xar{\nabla} A\pw{\hbar}[t] dt\right), \ \nabla: t^m\mapsto \tfrac{[m](\hbar)}{\hbar} t^{m-1} dt;
\end{equation}
we will study basic combinatorial properties of such complexes in \cite{generalized-n-series}.
After base-changing to $\QQ$, the operator $\nabla$ can be characterized by the formula $\hbar t\nabla = \exp_F(t\partial_t \log_F(\hbar))$.
\end{remark}
We also have:
\begin{prop}\label{a1-inv}
If $\cC$ is a left $\BP{n-1}$-linear $\infty$-category, and $\cC[t]$ denotes $\cC \otimes_{\BP{n-1}} \BP{n-1}[t]$, then \cref{conjecture-thh} implies that the map $\Lk \TP^\gr(\cC[t]/X(p^n)) \to \Lk \TP(\cC/X(p^n))$ is an equivalence.
\end{prop}
\begin{proof}
Observe that 
$$\THH(\cC[t]/X(p^n)) \simeq \THH(\cC/X(p^n)) (0) \oplus \bigoplus_{m\geq 1} (S^1/\mu_m)_+ \otimes \THH(\cC/X(p^n)) (m),$$
so that
$$\TP^\gr(\cC[t]/X(p^n)) \simeq \TP(\cC/X(p^n)) (0) \oplus \bigoplus_{m\geq 1} \Sigma \THH(\cC/X(p^n))^{t\Z/p^m} (m).$$
Now, \cref{conjecture-thh} implies that $\THH(\cC/X(p^n))^{t\Z/p^m}$ is a $\BP{n}^{t\Z/p^m}$-module. But $\Lk (\BP{n}^{t\Z/p^m}) = 0$ by \cref{lk-bpn}, so that $\Lk \TP^\gr(\cC[t]/X(p^n)) \simeq \Lk \TP(\cC/X(p^n))$, as desired.
\end{proof}
\begin{example}
Let $n=0$, and suppose $\cC$ is the $\infty$-category of quasicoherent sheaves on an $\FF_p$-scheme $X$. Then \cref{a1-inv} says that the map $\TP^\gr(\AA^1\times X) \to \TP(X)$ is a rational equivalence. This is generally not true in the non-graded setting.
\end{example}
\begin{remark}
Note that the functor $L_{K(0)} \TP$ is \textit{not} nil-invariant; the same is true of the functor $\Lk \TP(-/T(n))$ on $\BP{n-1}$-algebras. Indeed, \cite[Theorem 1.1]{nilinv-tp} says that the map $L_{K(0)} \TP(\FF_p[t]/t^k) \to L_{K(0)} \TP(\FF_p) \simeq \QQ_p^{tS^1}$ is an isomorphism if and only if $k$ is a power of $p$. We can also see this at the level of algebra by calculating the crystalline cohomology of $\FF_p[t]/t^k$. If $R$ denotes the $p$-completion of the PD-envelope of the quotient map $\Z_p[t] \to \FF_p[t]/t^k$ (so that $R = \Z_p\left[t, \frac{t^{kj}}{j!} | j\geq 1\right]^\wedge_p$), then \cite[Theorem 7.23]{berthelot-ogus} implies that $\Gamma_\crys((\FF_p[t]/t^k) / \Z_p)$ is quasi-isomorphic to the de Rham complex $\Omega^\bull_{R/\Z_p}$. Note that $R$ is additively isomorphic to the $p$-completion of $\bigoplus_{0\leq i \leq k-1} \bigoplus_{j\geq 0} \Z_p \{\frac{t^{kj+i}}{j!}\}$.

Since the derivative of $\frac{t^{kj+i}}{j!}$ is $(kj+i) \frac{t^{kj+i-1}}{j!}$, which simplifies to $k \frac{t^{k(j-1)+k-1}}{(j-1)!}$ when $i=0$, we find that $\pi_0 \Gamma_\crys((\FF_p[t]/t^k) / \Z_p) \cong \Z_p$, and
\begin{align*}
    \pi_{-1} \Gamma_\crys((\FF_p[t]/t^k) / \Z_p) & \cong \left(\bigoplus_{j\geq 0} \Z_p/k \cdot \left\{ \frac{t^{k(j+1)-1}}{j!}\right\} \right)^\wedge_p\\
    & \ \ \ \oplus \left(\bigoplus_{0\leq i < k-1} \bigoplus_{j\geq 0} \Z_p/(kj+i+1) \cdot \left\{\frac{t^{kj+i}}{j!}\right\}\right)^\wedge_p.
\end{align*}
For instance, suppose $k=p$. Then $p j+i+1 \equiv i+1\pmod{p}$, which is never zero since $0\leq i< p-1$. Therefore, the second summand is zero since $pj + i + 1$ is a $p$-adic unit, and we find that $\pi_{-1} \Gamma_\crys((\FF_p[t]/t^p) / \Z_p) \cong \left(\bigoplus_{j\geq 0} \Z/p \cdot \left\{ \frac{t^{p(j+1)-1}}{j!}\right\} \right)^\wedge_p$. However, if $k$ is not a power of $p$, the second summand contains a non-torsion piece; for example, if $k=2$ and $p$ is odd, the second summand contains the $p$-completion of $\bigoplus_{m\geq 0} \Z/p^m$, which is non-torsion.
\end{remark}
\begin{example}
Let $n=1$; then, \cref{a1-inv} says that \cref{conjecture-thh} implies that up to a Nygaard-type completion, $\Lone \TP^\gr(R[t]/X(p)) \xar{\sim} \Lone \TP(R/X(p))$ for $R$ being an $\E{1}$-$\Z_p$-algebra. In the non-graded setting, this is generally not true; this is in contrast to \cite[Corollary 4.24]{lmmt} (for instance), which says that $K(1)$-local algebraic K-theory is $\AA^1$-invariant on connective $K(1)$-acyclic ring spectra (in particular, on connective $\E{1}$-$\Z_p$-algebras).
\end{example}

Let us now pivot somewhat to a slightly different topic, working at the famed prime $p=2$. Then $\ptl\Omega_R = q\Omega_R$, and there is an interesting action of $\Z/2\subseteq \Z_2^\times$ on $q\Omega_R$ sending $q\mapsto q^{-1}$. If we view $q$ as the Chern class (in K-theory) of the tautological line bundle on $\CP^\infty$, this corresponds to the action of $\Z/2$ on $\CP^\infty$ given by complex conjugation. This motivates the following discussion:
\begin{remark}
We expect that most of the results and conjectures in this article continue to hold with $\Z/2$-equivariance, where ``real'' topological Hochschild homology $\THH_\RR$ is interpreted to mean the construction described in \cite{real-thh, equiv-fact-homology}. Recall that $\Z/2$ acts on $\SU(n)$ by complex conjugation; we will denote this $\Z/2$-space by $\SU(n)_\RR$. Let $\sigma$ (resp. $\rho = \sigma+1$) denote the sign representation (resp. regular representation), and if $X$ is a $\Z/2$-space, let $\Omega^\sigma X$ denote the space of maps $\Map(S^\sigma, X)$. There is a $\Z/2$-equivariant $\E{\rho}$-map 
$$\Omega^\sigma \SU(n)_\RR \simeq \Omega^\rho \BSU(n)_\RR \to \Omega^\rho \BSU_\RR \simeq \BU_\RR,$$
which equips its Thom spectrum $X(n)_\RR$ with the structure of an $\E{\rho}$-ring. One can show that the equivariant Quillen idempotent on $\MU_\RR$ restricts to an idempotent on $X(n)_\RR$, and we will write $T(n)_\RR$ to denote the resulting summand of $X(2^n)_\RR$. Moreover, $\Phi^{\Z/2} T(n)_\RR \simeq y(n)$ as $2$-local $\E{1}$-algebras.
\end{remark}
We then expect:
\begin{conjecture}\label{c2-equiv-analogues}
The following are true:
\begin{enumerate}
    \item $T(n)_\RR$ admits the structure of an $\E{\rho} \rtimes \U(1)_\RR$-algebra, and $X(2^n)_\RR$ splits as a direct sum of shifts of $T(n)_\RR$ such that the inclusion $T(n)_\RR \to X(2^n)_\RR$ of the unit summand is a map of $\E{\rho}$-algebras. In particular, $\THH_\RR(\BP{n-1}_\RR/T(n)_\RR)$ exists and admits an $\U(1)_\RR$-action\footnote{Note that $\U(1)_\RR = S^\sigma$.}.
    \item Let $\BP{n}_\RR$ denote the Real truncated Brown-Peterson spectrum. Then there are equivalences
    \begin{align}
        \THH_\RR(\BP{n-1}_\RR/T(n)_\RR) & \simeq \BP{n-1}_\RR[\Omega S^{2^n\rho+1}], \label{equiv-thh-tn} \\
        \THH_\RR(\BP{n-1}_\RR/T(n-1)_\RR) & \simeq \BP{n-1}_\RR \oplus \bigoplus_{j\geq 1} \Sigma^{2^{n-1}j\rho-1} \BP{n}_\RR/j \label{equiv-thh-tn1}
    \end{align}
    of $\BP{n-1}_\RR$-modules. The second equivalence requires $n\geq 1$.
    Furthermore, the class in $\pi_{2^n\rho} \THH_\RR(\BP{n-1}_\RR/T(n)_\RR)$ induced by the map $E: S^{2^n\rho} \to \Omega S^{2^n\rho+1}$ detects $\sigma^\rho(\ul{v}_n)$.
    \item There is a $\Z/2$-equivariant space $\tilde{K}_n$ and an equivariant fibration
    $$S^{2^n\rho-1}\to \tilde{K}_n \to \Omega S^{2^n\rho+1}$$
    such that $\tilde{K}_0 = \Omega S^\rho$ and $\tilde{K}_1 = \Omega S^{\rho+1}\pdb{\rho+1}$. For $n=0$, this is simply the EHP sequence for $S^\sigma$. 
    The boundary map $\Omega^2 S^{2^{n+1}+1} \to S^{2^{n+1}-1}$ of the underlying fibration is degree $2$ on the bottom cell of the source, and $(\tilde{K}_n)^{\Z/2} = K_{n-1}$ as $(S^{2^n\rho-1})^{\Z/2} = S^{2n-1}$-fibrations over $(\Omega S^{2^n\rho+1})^{\Z/2} = \Omega S^{2^n + 1}$.
    \item For any $\Z/2$-equivariant $\E{\sigma}$-$T(n)_\RR$-algebra $R$, there is an equivariant cofiber sequence
    $$\THH_\RR(R/T(n-1)_\RR) \to \THH_\RR(R/T(n)_\RR) \to \Sigma^{2^n\rho} \THH_\RR(R/T(n)_\RR),$$
    where the second map is a $\Z/2$-equivariant analogue of the topological Sen operator.
    \item Let 
    $$\TP_\RR(\BP{n-1}_\RR/T(n)_\RR) := \THH_\RR(\BP{n-1}_\RR/T(n)_\RR)^{t_{C_2} \U(1)_\RR},$$
    where the notation ``$t_{C_2} \U(1)_\RR$'' means the parametrized Tate construction from \cite[Remark 1.17]{quigley-shah-real-cyclotomic}.
    Then there is a $\Z/2$-equivariant equivalence
    $$\TP_\RR(\BP{n-1}_\RR/T(n)_\RR) \simeq \BP{n}_\RR^{t_{C_2} \U(1)_\RR},$$
    where $\U(1)_\RR$ acts trivially on $\BP{n}_\RR$.
    \item Let $R$ be a $2$-complete animated commutative ring, equipped with the trivial $\Z/2$-action. Then there is a $\Z/2$-equivariant filtration on $\TP_\RR(R/T(1)_\RR)$ such that 
    $$\gr^i_\mot \TP_\RR(R/T(1)_\RR) \simeq (\hat{\prism}_{R/\Z\pw{q-1}})^\wedge_2[2i] = (\hat{q\Omega}_R)^\wedge_2[2i],$$
    where the $\Z/2$-action on the right-hand side is obtained by viewing $\Z/2\subseteq \Z_2^\times \cong \Z/2 \times (1+4\Z_2)$ and using the $\Z_2^\times$-action on the $2$-completed $q$-de Rham complex.
\end{enumerate}
\end{conjecture}
\begin{example}
Note that \cite[Theorem 5.18]{real-thh} and \cite[Theorem A.1]{equiv-fact-homology} prove \cref{equiv-thh-tn} for $n=0$ and \cref{equiv-thh-tn1} for $n=1$, respectively.
\end{example}
\begin{remark}
We expect that $\BP{n}_\RR^{\Z/2} \otimes T(n)$ is concentrated in even degrees.
\end{remark}
\begin{prop}
The equivalence \cref{equiv-thh-tn} is true for $n=1$.
\end{prop}
\begin{proof}[Proof sketch]
This can be proved analogously to \cref{thh-calculations}(a) for $n=1$ using the equivariant Toda fiber sequence
$$S^{\rho+\sigma}\to \Omega S^{\rho+1}\pdb{\rho+1} \to \Omega S^{2\rho+1}$$
of \cite[Equation 7.1]{bpn-thom}.\footnote{See also \cite[Theorem 7.2.1]{bpn-thom}, which says that $\ul{\Z}$ is the Thom spectrum of a map $\Omega^\rho S^{2\rho+1} \to \BGL_1(X(2)_\RR)$ whose bottom cell detects $\ul{v}_1\in \pi_\rho X(2)_\RR$.} Indeed, recall from \cite[Example 7.1.3]{bpn-thom} that $X(2)_\RR$ is the Thom spectrum of the map $\Omega^\sigma S^{\rho+\sigma} \to \BGL_1(S)$ detecting $\tilde{\eta}\in \pi_\sigma S$. By an argument similar to \cite{thh-thom}, this implies that $\THH_\RR(X(2)_\RR) \simeq X(2)_\RR[S^{\rho+\sigma}]$. Therefore:
\begin{align*}
    \THH_\RR(\ul{\Z}/X(2)_\RR) & \simeq \ul{\Z}[\Omega S^{\rho+1}\pdb{\rho+1}] \otimes_{X(2)_\RR[S^{\rho+\sigma}]} X(2)_\RR \\
    & \simeq \ul{\Z}[\Omega S^{\rho+1}\pdb{\rho+1}] \otimes_{\ul{\Z}[S^{\rho+\sigma}]} \ul{\Z} \simeq \Z[\Omega S^{2\rho+1}],
\end{align*}
where the last equivalence uses the equivariant Toda fiber sequence.
\end{proof}
\begin{example}\label{yan-summand}
Let us note some additional evidence for \cref{c2-equiv-analogues}(a): if $X$ is a $\Z/2$-space, then the cofiber sequence $(\Z/2)_+ \to S^0 \to S^\sigma$ of spaces implies that $(\Omega^\sigma X)^{\Z/2}$ is equivalent to the fiber of the canonical map $X^{\Z/2} \to X$. In particular, since $(\SU(n)_\RR)^{\Z/2} = \SO(n)$, we see that $(\Omega^\sigma \SU(n)_\RR)^{\Z/2} \simeq \Omega(\SU(n)/\SO(n))$. Since geometric fixed points preserves colimits, this implies that $\Phi^{\Z/2} X(n)_\RR$ is the Thom spectrum of the map 
$$(\Omega^\sigma \SU(n)_\RR)^{\Z/2} \simeq \Omega(\SU(n)/\SO(n)) \to \Omega(\SU/\SO) \simeq \BO \simeq \BU_\RR^{\Z/2}.$$
Since $\Phi^{\Z/2} T(n)_\RR \simeq y(n)$ as $2$-local $\E{1}$-algebras, \cref{c2-equiv-analogues}(a) would imply that $\Phi^{\Z/2} X(2^n)_\RR$ (i.e., the Thom spectrum of the map $\Omega(\SU(n)/\SO(n)) \to \BO$) is a direct sum of shifts of $y(n)$ such that the inclusion $y(n) \to \Phi^{\Z/2} X(2^n)_\RR$ of the unit summand is an $\E{1}$-map. This is indeed true, and was proved in \cite{yan-thom}.
\end{example}
\begin{example}
The strongest evidence for \cref{c2-equiv-analogues}(d) is the following. It follows from \cite[Construction 7.1.1]{bpn-thom} that there is a map $\Omega S^{n\rho-1} \to \BGL_1(X(n-1)_\RR)$ whose Thom spectrum is $X(n)_\RR$. The same construction used to prove \cref{topological-sen} then shows that for any $\Z/2$-equivariant $\E{\sigma}$-$X(n)_\RR$-algebra $R$, there is an equivariant cofiber sequence
\begin{equation}\label{equiv-sen}
    \THH_\RR(R/X(n-1)_\RR) \to \THH_\RR(R/X(n)_\RR) \to \Sigma^{n\rho} \THH_\RR(R/X(n)_\RR),
\end{equation}
where the second map is a $\Z/2$-equivariant analogue of the topological Sen operator. It is not difficult to see that given the first half of \cref{c2-equiv-analogues}(a), \cref{c2-equiv-analogues}(d) can be easily proved using the construction of \cref{topological-sen}.

For example, we have $X(2)_\RR = T(1)_\RR$, and the cofiber sequence of \cref{equiv-sen} is precisely \cref{c2-equiv-analogues}(d). For $R = \ul{\Z}$, \cref{equiv-sen} becomes a cofiber sequence
$$\ul{\Z}[\Omega S^{\rho+1}\pdb{\rho+1}] \simeq \THH_\RR(\ul{\Z}) \to \THH_\RR(\ul{\Z}/X(2)_\RR) \simeq \ul{\Z}[\Omega S^{2\rho+1}] \simeq \bigoplus_{n\geq 0} \Sigma^{2n\rho} \ul{\Z} \to \bigoplus_{m\geq 1} \Sigma^{2m\rho} \ul{\Z}.$$
A version of this fiber sequence was in fact already studied in \cite[Lemma A.3]{equiv-fact-homology}.
\end{example}
\begin{remark}
\cref{c2-equiv-analogues}(a) and \cref{c2-equiv-analogues}(d) together imply that 
$$\pi_\ast \Phi^{\Z/2} \THH_\RR(\BP{n-1}_\RR) \simeq \FF_2[t, \sigma^2(v_{n-1}), \sigma(v_{j-1}) | 1\leq j \leq n]/(\sigma(v_j)^2),$$
where $|t| = 2^n = |\sigma^2(v_{n-1})|$ and $|\sigma(v_{j-1})| = 2^j-1$.
This can also be proved unconditionally using methods similar to that of \cite[Theorem 5.23]{real-thh}, by writing
$$\Phi^{\Z/2} \THH_\RR(\BP{n-1}_\RR) \simeq \Phi^{\Z/2} \BP{n-1}_\RR \otimes_{\BP{n-1}_\RR} \Phi^{\Z/2} \BP{n-1}_\RR,$$
and using that $\pi_\ast \Phi^{\Z/2} \BP{n-1}_\RR \cong \FF_2[t]$.

Note that if we assume \cref{c2-equiv-analogues}(c), then $\THH_\RR(\BP{n-1}_\RR/T(n-1)_\RR) \simeq \BP{n-1}_\RR[\tilde{K}_n]$; the conjectural equivalence $\tilde{K}_n^{\Z/2} = K_{n-1}$ then gives an equivalence
\begin{equation}\label{real-thh-fixed}
    \Phi^{\Z/2} \THH_\RR(\BP{n-1}_\RR/T(n-1)_\RR) \simeq \Phi^{\Z/2} \BP{n-1}_\RR[K_{n-1}].
\end{equation}
Observe that
$$\pi_\ast \Phi^{\Z/2} \BP{n-1}_\RR[K_{n-1}] \cong \FF_2[t, x, e]/e^2,$$
where $|x| = 2^n$ and $|y| = 2^n - 1$. For instance, when $n=1$, there is an equivalence $(\Omega S^{\rho+1}\pdb{\rho+1})^{\Z/2} = \Omega S^2$, and \cref{real-thh-fixed} reduces to the equivalence
$$\Phi^{\Z/2} \THH_\RR(\ul{\Z}) \simeq \Phi^{\Z/2} \ul{\Z}[\Omega S^2] \simeq (\tau_{\geq 0} \FF_2^{tS^1})[\Omega S^2].$$
\end{remark}

%% file: sen/segal.tex
In this section, we make some brief remarks regarding the Segal conjecture; the reader is referred to \cite[Section 4]{hahn-wilson-bpn} and \cite[Section 5]{mathew-k1-local-tr} for a discussion of its algebraic interpretation and a review of the literature on this topic.
\begin{definition}\label{def-segal}
An $\Eoo$-ring $R$ is said to satisfy the Segal conjecture if the cyclotomic Frobenius $\THH(R) \to \THH(R)^{t\Cp}$ is an equivalence in large degrees.
\end{definition}
\begin{example}
Let $R$ be a commutative $\FF_p$-algebra. If $R$ is Cartier smooth in the sense of \cite[Section 2]{kelly-morrow-valuation} and $\Omega^n_{R/\FF_p} = 0$ for $n\gg 0$, then $R$ satisfies the Segal conjecture in the sense of \cref{def-segal} (see \cite[Corollary 9.5]{mathew-recent-advances}).

For instance, suppose $R = k$ is a field of characteristic $p>0$. Then $\THH(k) \simeq \HH(k/\FF_p)[\sigma]$ as a module over $\THH(\FF_p) \simeq \FF_p[\sigma]$, and $\pi_i \HH(k/\FF_p) = 0$ for $i>\log_p [k:k^p] = \dim_k \Omega^1_{k/\FF_p}$. This implies that the localization map $\THH(k) \to \THH(k)[\frac{1}{\sigma}] \simeq_\varphi \THH(k)^{t\Cp}$ is an equivalence in degrees $>\log_p [k:k^p] - 2$.
\end{example}
\begin{example}
The proof of \cite[Theorem 4.3.1 and Corollary 4.2.3]{hahn-wilson-bpn} can be used to show that the map $\THH(\BP{n-1}) \otimes_{\BP{n-1}} \FF_p \to \THH(\BP{n-1})^{t\Cp} \otimes_{\BP{n-1}} \FF_p$ is an equivalence in degrees $>n+\sum_{i=0}^{n-1} |v_i| = \sum_{i=0}^{n-1} (2p^i-1) = 2\frac{p^n-1}{p-1} - n$. Note that $2\frac{p^n-1}{p-1} - n$ is also precisely the shift appearing in Mahowald-Rezk duality for $\BP{n}$ (see \cite[Corollary 9.3]{mahowald-rezk}).
\end{example}
\begin{remark}
Assume \cref{tn-e2}, so that we can define the $\THH$ of a left $T(n)$-linear $\infty$-category relative to $T(n)$.
Since we do not know if $\THH$ relative to $T(n)$ admits the structure of a cyclotomic spectrum (presumably it does not), it does not seem possible to state a direct analogue of \cref{def-segal} in this context. However, recall that if $k$ is a perfect field of characteristic $p>0$ and $R$ is an animated $k$-algebra, the cyclotomic Frobenius $\varphi: \THH(R) \to \THH(R)^{t\Cp}$ is the Frobenius-linear map given by inverting $\sigma\in \pi_2 \THH(k)$: this is a consequence of the observation that the map $\varphi: \THH(k) \to \THH(k)^{t\Cp}$ is given by composing the localization $\THH(k) \to \THH(k)[\sigma^{-1}]$ with a Frobenius-linear equivalence $\THH(k)[\sigma^{-1}] \simeq_\Fr \THH(k)^{t\Cp}$.

This observation motivates the following terminology: we say that an $\E{1}$-$\BP{n-1}$-algebra $R$ satisfies the ``$T(n)$-Segal conjecture'' if the base-change of the localization map $\THH(R/T(n)) \to \THH(R/T(n))[\theta_n^{-1}]$ along $\BP{n-1} \to \FF_p = \BP{n-1}/(p,\cdots,v_{n-1})$ is an equivalence in large degrees. Note that if $n=1$, this is equivalent to saying that the $p$-completion of the map $\THH(R/T(1)) \to \THH(R/T(1))[\theta^{-1}]$ is an equivalence in large degrees. One can similarly say that an $\E{1}$-$\Z_p$-algebra $R$ satisfies the ``$J(p)$-Segal conjecture'' if the map $\THH(R/J(p)) \to \THH(R/J(p))[x^{-1}]$ is an equivalence in large degrees.
\end{remark}
\begin{prop}\label{poly-segal}
If we assume \cref{tn-e2}, the localization map 
$$\gamma: \THH(\BP{n-1}[x_1,\cdots,x_d]/T(n)) \to \THH(\BP{n-1}[x_1,\cdots,x_d]/T(n))[\theta_n^{-1}]$$
is an equivalence in degrees $>d-2p^n$ after base-changing along $\BP{n-1} \to \FF_p$. In particular, the flat polynomial algebra $\BP{n-1}[x_1,\cdots,x_d]$ satisfies the $T(n)$-Segal conjecture.
\end{prop}
\begin{proof}
Write $T := \THH(\BP{n-1}/T(n))$ for notational simplicity.
Using \cref{formula-poly-t1}, we have
$$\THH(\BP{n-1}[t]/T(n))[\theta^{-1}] \simeq T[\theta^{-1}] \oplus \bigoplus_{n\geq 1} T[\theta^{-1}] \otimes (S^1/\mu_n)_+.$$
Since the map $T \to T[\theta^{-1}]$ is an equivalence in degrees $>-2p^n$ after base-changing along $\BP{n-1} \to \FF_p$, the map $T \otimes (S^1/\mu_n)_+ \to T[\theta^{-1}] \otimes (S^1/\mu_n)_+$ is an equivalence in degrees $>-2p^n+1$ after base-changing along $\BP{n-1} \to \FF_p$. Because the map $\gamma: \THH(\BP{n-1}[t]/T(n)) \to \THH(\BP{n-1}[t]/T(n))[\theta^{-1}]$ preserves the summands, we see that $\gamma$ is an equivalence in degrees $>-2p^n+1$ after base-changing along $\BP{n-1} \to \FF_p$. Inducting on the number of variables, we find that the map $\gamma$ is an equivalence in degrees $>d-2p^n$ after base-changing along $\BP{n-1} \to \FF_p$, as desired.
\end{proof}
\begin{remark}
When $d=0$, \cref{poly-segal} should be compared to \cite[Theorem 4.3.1]{hahn-wilson-bpn}. In fact, we expect it is possible to recover their result using \cref{poly-segal}. We also note the following variant. Let $R := \BP{n-1}[t_1,\cdots,t_d]$ denote the flat polynomial $\E{2}$-$\BP{n-1}$-algebra on classes $t_i$ in even degrees (i.e., the base-change of the $\Eoo$-$\MU$-algebra $\MU[t_1,\cdots,t_d]$ along the $\E{3}$-map $\MU \to \BP{n-1}$). The argument of \cref{poly-segal} then shows that after base-change along the composite $R \to \BP{n-1} \to \FF_p$, the localization map $\gamma: \THH(R/T(n)) \to \THH(R/T(n))[\theta_n^{-1}]$ is an equivalence in degrees $>-2p^n + \sum_{j=1}^d (|t_j|+1)$. When $n=0$, this is \cite[Corollary 4.2.3]{hahn-wilson-bpn}.
\end{remark}
\begin{prop}\label{reg-noeth-jp-segal}
Let $R$ be a $p$-torsionfree discrete commutative ring such that $R/p$ is regular Noetherian. Suppose $L\Omega^n_R = 0$ for $n>d$. Then \cref{rel-jp-diffracted} implies that $R$ satisfies the $J(p)$-Segal conjecture: in fact, the map $\THH(R/J(p)) \to \THH(R/J(p))[x^{-1}]$ is an equivalence in degrees $>d-2$.
\end{prop}
\begin{proof}
Recall that \cref{rel-jp-diffracted} asserts that $\THH(R/J(p))$ has a filtration such that $\gr^i_\mot \THH(R/J(p)) \simeq (\F^\conj_i \diffr_R)[2i]$, and such that the map $\gamma: \THH(R/J(p)) \to \THH(R/J(p))[x^{-1}]$ induces the map $\F^\conj_i \diffr_R \to \diffr_R$ on $\gr^i_\mot[-2i]$. By \cite[Remark 4.7.4]{apc}, $\diffr_R/\F^\conj_i \diffr_R$ is concentrated in cohomological degrees $\geq i+1$, so that the cofiber of $\gr^i(\gamma)$ is concentrated in degrees $\leq 2i-(i+1) = i-1$. Moreover, the hypothesis that $L\Omega^n_R = 0$ for $n>d$ implies that $\gamma$ induces an equivalence on $\gr^i_\mot$ for $i\geq d$. Combining these observations gives the desired statement (see also the proof of \cite[Corollary 9.5]{mathew-recent-advances}).
\end{proof}
%

%% file: sen/cartier.tex
In this section, we study a topological analogue of the Cartier isomorphism for the two-term complexes from \cref{generalized-de-rham}; we will study basic algebraic properties of these complexes in future work. To avoid dealing with completion issues, we use the following (see \cref{abusive HH notation} for a remark about the notation $\HH(R[t]/R)$):
\begin{definition}
Let $R$ be an $\E{2}$-ring. The polynomial $\E{1}$-$R$-algebra $R[t] = R[\NN]$ acquires a natural $\Z$-grading, and we will write $\HH(R[t]/R)_{\leq m}$ to denote the graded left $R$-module given by truncating $\HH(R[t]/R) := R \otimes \THH(S[t])$ in weights $\geq m+1$. Explicitly, $\HH(R[t]/R)_{\leq m}$ is equivalent to $R \oplus \left(\bigoplus_{1\leq n\leq m} R \otimes (S^1/\mu_n)_+\right)$.
\end{definition}
\begin{lemma}\label{smash-with-sphere-tate}
If $X\in \Sp^{\B S^1}$, the following composite is an equivalence:
$$X^{t\Cp} \otimes (S^1/\mu_n)_+ \xar{\sim} (X \otimes (S^1/\mu_n)_+)^{t\Cp} \xar{\psi} (X \otimes (S^1/\mu_{np})_+)^{t\Cp}.$$
Moreover, if $p\nmid m$, then $(X \otimes (S^1/\mu_m)_+)^{t\Cp} = 0$.
\end{lemma}
\begin{proof}
If $\varphi: (S^1/\mu_n)_+ \to ((S^1/\mu_{np})_+)^{h\Cp}$ denotes the unstable Frobenius (sending $x\mapsto x^{1/p}$), the cofiber of the composite
$$\psi: (S^1/\mu_n)_+ \to ((S^1/\mu_{np})_+)^{h\Cp} \to (S^1/\mu_{np})_+$$
has induced $\Cp$-action, where $(S^1/\mu_n)_+$ and $((S^1/\mu_{np})_+)^{h\Cp}$ are equipped with the trivial $\Cp$-action. 
Therefore, the canonical map $X^{t\Cp} \otimes (S^1/\mu_n)_+ \to (X \otimes (S^1/\mu_n)_+)^{t\Cp}$ is an equivalence (since $(S^1/\mu_n)_+$ is a finite spectrum with trivial $\Cp$-action). This gives the first claim.
Finally, if $p\nmid m$, then the $\Cp$-action on $S^1/\mu_m$ is free, so that $(X \otimes (S^1/\mu_m)_+)^{t\Cp} = 0$, as desired.
\end{proof}
\begin{prop}[Cartier isomorphism]\label{general-cartier}
Let $R$ be an $\E{2}$-ring. Then:
\begin{enumerate}
    \item There is an $S^1$-equivariant map $\fr{C}: \HH(R^{t\Cp}[t]/R^{t\Cp}) \to \HH(R[t]/R)^{t\Cp}$, where $\HH(R[t]/R)^{t\Cp}$ is endowed with the residual $S^1/\mu_p$-action and $\HH(R^{t\Cp}[t]/R^{t\Cp})$ is endowed with the diagonal $S^1$-action arising from the $S^1$-action on $\HH$ and the residual $S^1/\mu_p$-action on $R^{t\Cp}$. Moreover, the map $\fr{C}$ sends $t\mapsto t^p$.
    \item For each $m\geq 1$, the map $\fr{C}$ induces an equivalence $\fr{C}_{\leq m}: \HH(R^{t\Cp}[t]/R^{t\Cp})_{\leq m} \xar{\sim} \left(\HH(R[t]/R)_{\leq mp}\right)^{t\Cp}$.
\end{enumerate}
\end{prop}
\begin{proof}
Recall that there is an equivalence $\HH(R[t]/R) \simeq R \otimes \THH(S[t])$. Since the $\Cp$-Tate construction is lax symmetric monoidal, we obtain the map $\fr{C}$ via the composite
\begin{align*}
    \HH(R^{t\Cp}[t]/R^{t\Cp}) & \simeq R^{t\Cp} \otimes \THH(S[t]) \\
    & \xar{\id\otimes \varphi} R^{t\Cp} \otimes \THH(S[t])^{t\Cp} \\
    & \to (R \otimes \THH(S[t]))^{t\Cp} \simeq \HH(R[t]/R)^{t\Cp}.
\end{align*}
For each $m\geq 1$, there is an equivalence
\begin{align*}
    \left(\HH(R[t]/R)_{\leq m}\right)^{t\Cp} & \simeq R^{t\Cp} \oplus \left(\bigoplus_{1\leq n\leq m} R \otimes (S^1/\mu_n)_+\right)^{t\Cp}.
\end{align*}
Since the maps $\varphi: (S^1/\mu_n)_+ \to ((S^1/\mu_{np})_+)^{h\Cp}$ define the Frobenius on $\THH(S[t]) \simeq S \oplus \bigoplus_{n\geq 1} (S^1/\mu_n)_+$, we see from \cref{smash-with-sphere-tate} that for each $m\geq 1$, the map $\fr{C}_{\leq m}$ defines an equivalence
$$\bigoplus_{1\leq j\leq m} R^{t\Cp} \otimes (S^1/\mu_j)_+ \xar{\sim} \left(\bigoplus_{1\leq n\leq mp} R \otimes (S^1/\mu_n)_+\right)^{t\Cp}.$$
The left-hand side is $\HH(R^{t\Cp}[t]/R^{t\Cp})_{\leq m}$, and the right-hand side is $\left(\HH(R[t]/R)_{\leq mp}\right)^{t\Cp}$.
\end{proof}
\begin{remark}
When $R$ is an $\Eoo$-ring, the map $\fr{C}: \HH(R^{t\Cp}[t]/R^{t\Cp}) \to \HH(R[t]/R)^{t\Cp}$ of \cref{general-cartier} can also be constructed using (a simple case of) \cite[Theorem 1.3]{lawson-tate-diag}. The cited result says the following. Suppose $k$ is an $\Eoo$-ring, so that the Tate-valued Frobenius $k \to k^{t\Cp}$ admits an extension $\THH(k) \to k^{t\Cp}$ to an $S^1$-equivariant map of $\Eoo$-rings.
If $A$ is an $\E{1}$-$k$-algebra, and $M$ is an $A$-bimodule in $\Mod_k$, then there is a relative Tate diagonal 
$$k^{t\Cp} \otimes_{\THH(k)} \THH(A, M) \to \THH^k(A, M^{\otimes_A p})^{t\Cp},$$
where $\THH^k$ denotes $\THH$ relative to $k$. To construct $\fr{C}$, take $k = R$ and $A = M = k[t]$. Then 
$$k^{t\Cp} \otimes_{\THH(k)} \THH(A, M) \simeq k^{t\Cp} \otimes \THH(S^0[t]) \simeq \HH(k^{t\Cp}[t]/k^{t\Cp}),$$
since $\THH(A,M) \simeq \THH(A) \simeq \THH(S^0[t]) \otimes \THH(k)$. Similarly, $\THH^k(A, M^{\otimes_A p}) \simeq \HH(A/k)$, and it is straightforward to check that Lawson's relative Tate diagonal agrees with the map $\fr{C}$. 

One advantage of the construction of $\fr{C}$ in \cref{general-cartier} is that it is manifestly $S^1$-equivariant, and does not rely on $R$ being an $\Eoo$-ring. More generally, one finds that if $\cC$ is a stable $\infty$-category and $R$ is any $\E{2}$-ring, the cyclotomic Frobenius on $\THH(\cC)$ defines an $S^1$-equivariant map $\fr{C}: \HH(\cC \otimes R^{t\Cp}/R^{t\Cp}) \to \HH(\cC \otimes R/R)^{t\Cp}$ which generalizes the map of \cref{general-cartier}. This map is furthermore an equivalence if $\cC$ is smooth and proper.
\end{remark}
\begin{remark}\label{degreewise-fg}
In \cref{general-cartier}, the map $\fr{C}: \HH(R^{t\Cp}[t]/R^{t\Cp}) \to \HH(R[t]/R)^{t\Cp}$ is itself almost an equivalence: the main issue is that the canonical map 
$$\colim_m \left(\HH(R[t]/R)_{\leq mp}\right)^{t\Cp} \to \left(\colim_m \HH(R[t]/R)_{\leq mp}\right)^{t\Cp} \simeq \HH(R[t]/R)^{t\Cp}$$
may not be an equivalence. However, \cref{general-cartier} implies that the \textit{graded} map $\fr{C}^\gr: \HH(R[t]/R)^{t\Cp} \to \HH(R[t]/R)^{t\Cp}$ is itself an equivalence.
\end{remark}
\begin{remark}
If $R$ is a complex-oriented $\E{2}$-ring, let $[p](\hbar)\in \pi_{-2} R^{hS^1}$ denote the $p$-series of the formal group law over $R$. If $M\in \LMod_R^{\B S^1}$, then it is not difficult to show that there is an equivalence $M^{tS^1}/\frac{[p](\hbar)}{\hbar} \xar{\sim} M^{t\Cp}$. (Although certainly well-known, the only reference in the literature for a statement in this generality seems to be \cite[Lemma 6.2.2]{even-filtr}.) In particular, $\HH(R[t]/R)^{t\Cp} \simeq \HP(R[t]/R)/\frac{[p](\hbar)}{\hbar}$, so that \cref{general-cartier} and \cref{degreewise-fg} imply that there is an $S^1$-equivariant \textit{graded} equivalence
$$\fr{C}: \HH(R^{t\Cp}[t]/R^{t\Cp}) \to \HP(R[t]/R)/\tfrac{[p](\hbar)}{\hbar} \simeq \HH(R[t]/R)^{t\Cp}.$$
In future work, we will show that if $R$ is further assumed to be an $\Eoo$-ring and $\cC$ is a $R$-linear $\infty$-category, then the $(R^{t\Cp})^{hS^1} \simeq (R^{tS^1})^\wedge_p$-module $(\THH(\cC) \otimes_{\THH(R)} R^{t\Cp})^{hS^1}$ behaves as a noncommutative analogue of $L\eta_{[p](\hbar)/\hbar}$ applied to $\HP(\cC/R)$. Here, $\hbar \in \pi_{-2} R^{hS^1}$ is the complex-orientation of $R$, and $[p](\hbar)/\hbar\in \pi_0 R^{tS^1}$ is the quotient of the $p$-series of the associated formal group law.
\end{remark}
\begin{remark}
There is no reason to restrict to polynomial rings in a single variable in the equivalence of \cref{general-cartier}(b); we leave the details of the resulting statement to the reader.
\end{remark}
\begin{example}\label{cartier-zp}
Let $R = \Z$. Then $\Z^{t\Cp}$ is an $\Eoo$-$\Z$-algebra, and has homotopy groups given by $\FF_p\ls{\hbar}$ with $|\hbar|=2$. Therefore, $\Z^{t\Cp} \simeq \FF_p^{tS^1}$ as $\E{2}$-rings\footnote{In fact, they are equivalent as $\Eoo$-rings. Although seemingly innocuous, even the weaker claim that $\Z^{t\Cp}$ admits the structure of an $\Eoo$-$\FF_p$-algebra is surprisingly difficult to prove from first principles (see \cite[Remark IV.4.17]{nikolaus-scholze}). One might try to argue as follows: since $p=0\in \pi_0 \Z^{t\Cp}$, there is a map from the free $\Eoo$-algebra with $p=0$ to $\Z^{t\Cp}$. However, oddly enough, the free $\Eoo$-algebra with $p=0$ is \textit{not} an $\Eoo$-$\FF_p$-algebra; this dashes any hopes of proving that $\Z^{t\Cp}$ is an $\Eoo$-$\FF_p$-algebra through this argument.

More generally, the free $\E{n}$-algebra $R$ with $p=0$ is not an $\E{n}$-$\FF_p$-algebra unless $n=2$: indeed, applying the $\E{n}$-cotangent complex to the composite $\FF_p \to R \to \FF_p$ of $\E{n}$-algebra maps shows that $L^{\E{n}}_{\FF_p}$ is a retract of $\FF_p \otimes_R L^{\E{n}}_R \simeq \FF_p$. This forces $L^{\E{n}}_{\FF_p} = \FF_p$, i.e., $\FF_p$ would be built from the sphere by attaching a single $\E{n}$-cell in degree $1$ -- but this is impossible, since $\FF_p \otimes \FF_p \not \simeq \FF_p[\Omega^n S^{n+1}]$ unless $n=2$.}, and \cref{general-cartier} (combined with \cref{degreewise-fg}) specializes to the statement that there is a Frobenius-linear equivalence
$$\fr{C}: \HH(\FF_p[t]/\FF_p)\ls{\hbar}_{\leq m} \simeq \HH(\FF_p^{tS^1}[t]/\FF_p^{tS^1})_{\leq m} \xar{\sim} (\HH(\Z[t]/\Z)_{\leq mp})^{t\Cp}.$$
Note that $\HH(\Z[t]/\Z)^{t\Cp} \simeq \HP(\Z[t]/\Z)/p \simeq \TP(\FF_p[t])/p$.
Since the HKR theorem implies that $\HH(\FF_p[t]/\FF_p)^{tS^1}$ is a $2$-periodification of the Hodge cohomology of $\AA^1_{\FF_p}$, and $\HP(\Z[t]/\Z)/p$ is a $2$-periodification of the de Rham cohomology of $\AA^1_{\Z_p}$ modulo $p$ (which is the de Rham cohomology of $\AA^1_{\FF_p}$), one can view $\fr{C}$ as a topological analogue of the Cartier isomorphism for the affine line. It reduces to the usual Cartier isomorphism on graded pieces. In this case, the statement of \cref{general-cartier} should also be compared to \cite{kaledin-1, kaledin-2, mathew-kaledin}. Taking homotopy fixed points for the $S^1$-equivariance of $\fr{C}$ from \cref{general-cartier}(a), we obtain a Frobenius-linear equivalence
\begin{equation}\label{ogus}
    \fr{C}^{hS^1}: (\HH(\Z^{t\Cp}[t]/\Z^{t\Cp})^{hS^1})_{\leq m} \xar{\sim} ((\HH(\Z[t]/\Z)_{\leq mp})^{tS^1})^\wedge_p.
\end{equation}
More succinctly, there is a \textit{graded} equivalence
$$(\fr{C}^\gr)^{hS^1}: \HH(\Z^{t\Cp}[t]/\Z^{t\Cp})^{hS^1} \xar{\sim} \HP^\gr(\Z[t]/\Z)^\wedge_p.$$
Using the HKR filtration on $\HH(\Z[t]/\Z)$, one can prove that $\HH(\Z^{t\Cp}[t]/\Z^{t\Cp})^{hS^1}$ admits a filtration whose graded pieces are given by even shifts of $L\eta_p \Gamma_\dR(\Z_p[t]/\Z_p) \simeq L\eta_p \Gamma_\crys(\FF_p[t]/\Z_p)$. We will explain this in greater detail in a future article.
Since $\HP(\Z_p[t]/\Z_p) \simeq \TP(\FF_p)$, \cref{ogus} can be regarded as a $2$-periodification of the ``Cartier isomorphism'' $L\eta_p \Gamma_\crys(\FF_p[t]/\Z_p) \simeq \Gamma_\crys(\FF_p[t]/\Z_p)$ for the crystalline cohomology of $\FF_p[t]$ (see \cite[Theorem 8.20]{berthelot-ogus} for the general case).
\end{example}
\begin{example}\label{cartier-ku}
Let $R = \ku$. Then $\pi_\ast \ku^{t\Cp} \cong \Z[\zeta_p]\ls{\hbar}$, and it is expected that this lifts to an equivalence $\ku^{t\Cp} \simeq \Z[\zeta_p]^{tS^1}$ of $\Eoo$-rings (see also \cref{tate-bpn} below).
Nevertheless, there \textit{is} an equivalence $\ku^{t\Cp} \simeq \Z[\zeta_p]^{tS^1}$ of $\E{2}$-rings (one can show this by using \cite[Theorem 1.19]{zhouhang-mao}; thanks to Arpon Raksit for pointing this out).
Therefore, \cref{general-cartier} and \cref{degreewise-fg} give an equivalence
$$\fr{C}: \HH(\Z[\zeta_p][t]/\Z[\zeta_p])\ls{\hbar}_{\leq m} \simeq \HH(\Z[\zeta_p]^{tS^1}[t]/\Z[\zeta_p]^{tS^1})_{\leq m} \xar{\sim} (\HH(\ku[t]/\ku)_{\leq mp})^{t\Cp}.$$
Note that $\HH(\ku[t]/\ku)^{t\Cp} \simeq \HP(\ku[t]/\ku)/[p]_q$.
Here, we identify $\frac{[p]_{\GG_m}(\hbar)}{\hbar}\in \pi_0 \ku^{tS^1} \cong \Z\pw{q-1}$ with the $q$-integer $[p]_q$.

By the HKR theorem, one can view $\HH(\Z[\zeta_p][t]/\Z[\zeta_p])\ls{\hbar}$ as a $2$-periodification of the Hodge cohomology of $\AA^1_{\Z}$ base-changed along the map $\Z \to \Z[\zeta_p]$. Similarly, the aforementioned work of Raksit (see \cite{arpon-qdr} and \cref{arpon}, as well as \cref{n-series-BP1}) implies that $\HP(\ku[t]/\ku)$ can be viewed as a $2$-periodification of the $q$-de Rham complex of $\Z[t]$. Since killing $[p]_q\in \Z\pw{q-1}$ amounts to specializing $q$ to a primitive $p$th root of unity, one can view \cref{general-cartier} as a topological analogue of the Cartier isomorphism for the $q$-de Rham complex of the affine line (see, e.g., \cite[Proposition 3.4]{scholze-q-def}).

Taking homotopy fixed points for the $S^1$-equivariance of $\fr{C}$ from \cref{general-cartier}(a), we obtain an equivalence
\begin{equation}\label{q-ogus}
    \fr{C}^{hS^1}: (\HH(\ku^{t\Cp}[t]/\ku^{t\Cp})^{hS^1})_{\leq m} \xar{\sim} ((\HH(\ku[t]/\ku)_{\leq mp})^{tS^1})^\wedge_p.
\end{equation}
More succinctly, there is a \textit{graded} equivalence
$$(\fr{C}^\gr)^{hS^1}: \HH(\ku^{t\Cp}[t]/\ku^{t\Cp})^{hS^1} \xar{\sim} \HP^\gr(\ku[t]/\ku)^\wedge_p.$$
In future work, we show that Raksit's filtration on $\HP(\ku[t]/\ku)$ can be refined to construct a filtration on $\HH(\ku^{t\Cp}[t]/\ku^{t\Cp})^{hS^1}$ whose graded pieces are given by even shifts of $L\eta_{[p]_q} q\Omega_{\Z_p[t]}$.
Then, \cref{q-ogus} can be regarded as a $2$-periodification of the ``Cartier isomorphism'' $\phi_{\Z_p\pw{q-1}}^\ast q\Omega_{\Z_p[t]} \simeq L\eta_{[p]_q} q\Omega_{\Z_p[t]}$ for the $q$-de Rham cohomology of $\Z_p[t]$. (See \cite[Theorem 1.16(4)]{bhatt-scholze} applied to the $q$-crystalline prism $(\Z_p\pw{q-1}, [p]_q)$.)
\end{example}
\begin{remark}
\cref{cartier-ku} admits a mild generalization. Namely, if $\ku^{\Z/p^{n-1}}$ denotes the strict fixed points (so $\ku^{\Z/p^{n-1}} = \tau_{\geq 0}(\ku^{h\Z/p^{n-1}})$), then one can calculate $\pi_\ast (\ku^{\Z/p^{n-1}})^{t\Cp} \cong \pi_\ast \Z_p[\zeta_{p^n}]^{tS^1}$.
One can show that this can be extended to an equivalence $(\ku^{\Z/p^{n-1}})^{t\Cp} \simeq \Z_p[\zeta_{p^n}]^{tS^1}$ of $\E{2}$-rings.
\cref{general-cartier} and \cref{degreewise-fg} give a graded equivalence
$$\fr{C}: \HH(\Z[\zeta_{p^n}][t]/\Z[\zeta_{p^n}])\ls{\hbar} \xar{\sim} \HH(\ku^{\Z/p^{n-1}}[t]/\ku^{\Z/p^{n-1}})^{t\Cp}.$$
Here, the action of $\Cp$ on $\HH(\ku^{\Z/p^{n-1}}[t]/\ku^{\Z/p^{n-1}}) = \THH(S[t]) \otimes \ku^{\Z/p^{n-1}}$ is via the \textit{diagonal} action on $\THH$ and $\ku^{\Z/p^{n-1}}$. In this case, one can therefore view \cref{general-cartier} as a topological analogue of the Cartier isomorphism for Hodge-Tate cohomology relative to the prism $(\Z_p\pw{q^{1/p^{n-1}}-1}, [p]_q)$ of the affine line.
\end{remark}
\begin{example}\label{tate-bpn}
More generally, let $R = \BP{n}$. As recalled in \cref{theta-hbar-vn}, \cite[Proposition 2.3]{ando-morava-sadofsky} proved that there is an isomorphism $\pi_\ast \BP{n}^{t\Cp} \cong \pi_\ast \BP{n-1}^{tS^1}$, and this was conjectured to lift to an equivalence of spectra in \cite[Conjecture 1.2]{five-author-tate}. If we assume that there is in fact an equivalence $\BP{n}^{t\Cp} \simeq \BP{n-1}^{tS^1}$ of $\E{2}$-rings, \cref{general-cartier} and \cref{degreewise-fg} give an equivalence
$$\fr{C}: \HH(\BP{n-1}[t]/\BP{n-1})\ls{\hbar}_{\leq m} \xar{\sim} (\HH(\BP{n}[t]/\BP{n})_{\leq mp})^{t\Cp}.$$
Therefore, \cref{general-cartier} in this case can be viewed as an analogue of the Cartier isomorphism for the affine line in the setting of ``$v_n$-adic Hodge theory''. Taking homotopy fixed points for the $S^1$-equivariance of $\fr{C}$ from \cref{general-cartier}(a), we obtain an equivalence
\begin{equation}\label{c-ts1}
    \fr{C}^{hS^1}: (\HH(\BP{n}^{t\Cp}[t]/\BP{n}^{t\Cp})^{hS^1})_{\leq m} \xar{\sim}
    ((\HH(\BP{n}[t]/\BP{n})_{\leq mp})^{tS^1})^\wedge_p.
\end{equation}
More succinctly, there is a \textit{graded} equivalence
$$(\fr{C}^\gr)^{hS^1}: \HH(\BP{n}^{t\Cp}[t]/\BP{n}^{t\Cp})^{hS^1} \xar{\sim} \HP^\gr(\BP{n}[t]/\BP{n})^\wedge_p.$$
Note that \cref{conjecture-thh} in particular implies that if $T(n)$ admits the structure of an $\Efr{2}$-ring, then $\HP(\BP{n}[t]/\BP{n})$ is closely related to $\TP(\BP{n-1}[t]/T(n))$ by \cref{graded-drw-lift}. In this form, \cref{c-ts1} holds when $\BP{n}$ is replaced by any complex-oriented $\E{2}$-ring $R$. As in the preceding examples, we believe that when $R$ is connective, this can be regarded as a $2$-periodification of a ``Cartier isomorphism'' for the two-term complex \cref{any-ring-two-term}. See \cite{generalized-n-series} for further discussion.
\end{example}

%% file: mfg/algebraic-sen.tex
In \cref{applications-thh}, we showed that the descent spectral sequence of \cref{descent-sseq} admits a generalization given by the topological Sen operator (\cref{topological-sen}). This has an incarnation over $\Mfg$, as we now explain. The analogues of \cref{thh-calculations}, \cref{topological-sen}, etc., that we discuss in this section are useful for making topological predictions since the calculations involved are easier.
\begin{recall}[Even filtration]\label{even-filtr}
Let $\F^\star_\ev: \CAlg \to \CAlg(\Sp^\fil)$ be the \textit{even filtration} of \cite{even-filtr}: if $\CAlg^\ev$ denotes the full subcategory of $\CAlg$ spanned by the $\Eoo$-rings with even homotopy, then $\F^\star_\ev$ is the right Kan extension of the functor $\tau_{\geq 2\star}: \CAlg^\ev \to \CAlg(\Sp^\fil)$ along the inclusion $\CAlg^\ev \hookrightarrow \CAlg$. Note that since $\tau_{\geq 2\star}$ is lax symmetric monoidal and $\F^\star_\ev$ is defined by a right Kan extension, it is also a lax symmetric monoidal functor. We will need the following result from \cite{even-filtr}: if $R$ is an $\Eoo$-ring such that $\MU \otimes R \in \CAlg^\ev$, then $\F^\star_\ev R$ is $p$-completely equivalent to the underlying filtered $\Eoo$-ring of its Adams-Novikov tower $\nu(R)\in \Syn_\MU^\ev(\Sp) = \Mod_{\Tot(\tau_{\geq 2\star} \MU^{\otimes \bull+1})}(\Sp^\fil)$. (Also see \cite{piotr-synthetic, c-mot-mod}.) In this case, the associated graded Hopf algebroid $(\MU_\ast(R), \MU_\ast(\MU\otimes R))$ defines a stack over $B\GG_m$. If $R$ is complex-oriented, then this stack is isomorphic to $\spec(\pi_\ast R)/\GG_m$, where the $\GG_m$-action encodes the grading on $\pi_\ast R$.
\end{recall}
\begin{observe}
In order to define the stack $\tilde{\cM}_R$ associated to the graded Hopf algebroid $(\MU_\ast(R), \MU_\ast(\MU\otimes R))$, one does not need $R$ to be an $\Eoo$-ring: it only needs to admit the structure of a homotopy commutative ring such that $\MU_\ast(R)$ is concentrated in even degrees. This perspective is explained in Hopkins' lecture in \cite[Chapter 9]{tmf}. In particular, one can define the stack associated to $X(n)$: this is the moduli stack of graded formal groups equipped with a coordinate of order $\leq n$, and strict isomorphisms between them. (See, e.g., \cite[Section 2]{miller-comodules}.)
\end{observe}
\begin{variant}
We will find it convenient to work with the $p$-typical variant of the graded Hopf algebroid $(\MU_\ast(R), \MU_\ast(\MU\otimes R))$. Namely, if $R$ is a $p$-local homotopy commutative ring such that $\BPP_\ast(R)$ is concentrated in even degrees, then we will write $\cM_R$ to denote the graded stack associated to the graded Hopf algebroid $(\BPP_\ast(R), \BPP_\ast(\BPP\otimes R))$. For example, $\cM_{T(n)}$ is the moduli stack of $p$-typical graded formal groups equipped with a coordinate up to order $\leq p^{n+1}-1$; by $p$-typicality, this is further isomorphic to the moduli stack of $p$-typical graded formal groups equipped with a coordinate up to order $\leq p^n$. In particular, $\cM_{S^0}$ is isomorphic to the moduli stack $\Mfg$ of $p$-typical graded formal groups. Similarly, if $R$ is a $p$-local complex-oriented homotopy commutative ring, then $\cM_R$ is isomorphic to $\spec(\pi_\ast R)/\GG_m$.
\end{variant}
\begin{example}\label{witt-Mfg}
The unit map $S^0 \to \MU$ induces the map $\tilde{\cM}_\MU \cong \spec(\MU_\ast)/\GG_m \to \tilde{\cM}_{S^0}$ which describes the flat cover of the moduli stack of graded formal groups given by the graded Lazard ring. This map exhibits $\tilde{\cM}_{S^0}$ as the quotient of $\spec(\MU_\ast)/\GG_m$ by the group scheme $\spec(\pi_\ast(\MU \otimes \MU))/\GG_m$. Note that $\MU \otimes \MU \simeq \MU[\BU]$; since $\pi_\ast \Z[\BU]$ is the coordinate ring of the big Witt ring scheme, we see that $\spec(\pi_\ast(\MU \otimes \MU))/\GG_m$ is a lift of the big Witt ring scheme to $\spec(\MU_\ast)/\GG_m$. Similarly, $\cM_{S^0} = \Mfg$ is the quotient of $\cM_{\BPP} = \spec(\BPP_\ast)/\GG_m$ by a lift of the $p$-typical Witt ring scheme $\W$ to $\spec(\BPP_\ast)/\GG_m$.
\end{example}
\begin{remark}\label{anss-bokstedt}
If $A \to B$ is a map of $p$-local $\Eoo$-rings such that $\MU_\ast(A)$ and $\MU_\ast(B)$ are even, then there is an induced map $\cM_B \to \cM_A$ of graded stacks. Recall that $\THH(B/A)$ is the geometric realization of the simplicial $A$-algebra $B^{\otimes_A \bull+1}$. Applying $\F^\star_\ev$ levelwise to $B^{\otimes_A \bull+1}\in \Fun(\Deltab^\op, \CAlg_A)$ produces an Adams-Novikov analogue of the B\"okstedt spectral sequence:
$$\pi_\ast \HH(\cM_B/\cM_A) \Rightarrow \pi_\ast \gr^\bull_\ev \THH(B/A).$$
In particular, note that $\HH(\cM_B/\Mfg)$ is an approximation to $\gr^\bull_\ev \THH(B)$. For this spectral sequence to exist, it is not necessary that $A$ and $B$ be $\Eoo$-rings: for example, it suffices that $A \to B$ be a map of $p$-local $\E{2}$-rings such that $\MU_\ast(A)$ and $\MU_\ast(B)$ are even, and such that $\THH(B/A)$ is bounded below and has even $\MU$-homology. Then, $\gr^\bull_\ev \THH(B/A)$ must be interpreted as the associated graded of the Adams-Novikov filtration on $\THH(B/A)$; see \cite[Corollary 1.1.6]{even-filtr}.
\end{remark}
\begin{example}
Let $\Mfg^s$ denote the total space of the canonical line bundle over $\Mfg$ (determined by the map $\Mfg \to B\GG_m$).
If $R$ is a $p$-quasisyntomic ring, then \cite[Theorem 1.12]{bms-ii} and \cref{anss-bokstedt} give a spectral sequence
$$\pi_\ast \HH(\spec(R)/\Mfg^s) \Rightarrow \pi_\ast \cN^\ast \hat{\prism}_R[2\ast].$$
Indeed, $\cM_R = \spec(R)/\GG_m \cong \spec(R) \times B\GG_m$, so that the underlying $R$-algebra of $\HH(\cM_R/\Mfg)$ is $\HH(\spec(R)/\Mfg^s)$.
\end{example}
\begin{example}
The complex orientation $\BPP \to \BP{n}$ induces a map $\cM_{\BP{n}} \to \cM_{\BPP}$ which factors the structure map $\cM_{\BP{n}} \to \Mfg$. Explicitly, we have the following composite map of stacks over $B\GG_m$:
$$\spec(\BP{n}_\ast)/\GG_m \to \spec(\BPP_\ast)/\GG_m \to \Mfg.$$
Taking cotangent complexes gives the following transitivity cofiber sequence in $\Mod_{\BP{n}_\ast}^\gr$:
$$\BP{n}_\ast \otimes_{\BPP_\ast} L_{\spec(\BPP_\ast)/\Mfg} \to L_{\BP{n}_\ast/\Mfg} \to L_{\BP{n}_\ast/\BPP_\ast}.$$
Since $\BPP_\ast/(v_{n+1}, v_{n+2}, \cdots) \cong \BP{n}_\ast$, observe that $L_{\BP{n}_\ast/\BPP_\ast}$ is a free $\BP{n}_\ast$-module generated by classes $\sigma(v_{n+1}), \sigma(v_{n+2}), \cdots$.
Similarly, the discussion in \cref{witt-Mfg} implies that $L_{\spec(\BPP_\ast)/\Mfg}$ is a free $\BPP_\ast$-module generated by classes $d(t_i)$.
From this, one can deduce that $L_{\spec(\BP{n}_\ast)/\GG_m/\Mfg}$ is a free $\BP{n}_\ast$-module generated by classes $\sigma(v_j)$ with $j\geq n+1$ and $d(t_i)$ with $i\geq 1$. By the HKR theorem, $\pi_\ast \HH(\spec(\BP{n}_\ast)/\GG_m/\Mfg)$ is isomorphic to $\Sym_{\BP{n}_\ast}(L_{\BP{n}_\ast/\Mfg}[1])$,
which can be identified as
$$\pi_\ast \HH(\spec(\BP{n}_\ast)/\GG_m/\Mfg) \cong \BP{n}_\ast\pdb{\sigma^2 v_j | j\geq n+1} \otimes_{\BP{n}_\ast} \Lambda_{\BP{n}_\ast}(dt_i | i\geq 1).$$
Since $v_j$ lives in degree $2p^j-2$ and weight $p^j-1$, the class $\sigma^2 v_j$ lives in degree $2p^j = |v_j| + 2$ and weight $p^j$; similarly, since $t_i$ lives in degree $2p^i-2$ and weight $p^j-1$, the class $dt_i$ lives in degree $2p^i-1$ and weight $p^j$.
\end{example}
\begin{example}\label{hh-bpn-tn-mfg}
The same discussion for the following composite map of stacks over $B\GG_m$
$$\spec(\BP{n-1}_\ast)/\GG_m \to \spec(\BPP_\ast)/\GG_m \to \cM_{T(n)}$$
shows that $L_{\spec(\BP{n-1}_\ast)/\GG_m/\cM_{T(n)}}$ is a free $\BP{n-1}_\ast$-module generated by classes $\sigma(v_j)$ with $j\geq n$ and $d(t_i)$ with $i\geq n+1$. Therefore, the HKR theorem implies that $\pi_\ast \HH(\spec(\BP{n-1}_\ast)/\GG_m/\M_{T(n)})$ is isomorphic to a symmetric algebra over $\BP{n-1}_\ast$ on classes $\sigma^2(v_i)$ for $i \geq n$,
and $d(t_i)$ for $i\geq n+1$.
Explicitly,
$$\pi_\ast \HH(\spec(\BP{n-1}_\ast)/\GG_m/\M_{T(n)}) \cong \BP{n-1}_\ast\pdb{\sigma^2 v_j | j\geq n} \otimes_{\BP{n-1}_\ast} \Lambda_{\BP{n-1}_\ast}(dt_i | i\geq n+1).$$
The class $\sigma^2 v_j$ lives in degree $2p^j = |v_j| + 2$ and weight $p^j$, and the class $dt_i$ lives in degree $2p^i-1$ and weight $p^j$.
This mirrors the calculation of the $E^2$-term of the B\"okstedt spectral sequence in \cref{homology-thh-bpn}.

In fact, one can recover \cref{thh-calculations} in this way by running the Adams-Novikov-B\"okstedt spectral sequence (\cref{anss-bokstedt}) and using the $\E{2}$-Dyer-Lashof argument of \cref{homology-thh-bpn} to resolve the extension problems on the $E^\infty$-page. We use the term ``recover'' in a very weak sense here: the differentials in the Adams-Novikov-B\"okstedt spectral sequence are forced by the differentials in the usual B\"okstedt spectral sequence (\cref{homology-thh-bpn}). Explicitly, we have
$$d^{p-1}(\gamma_j(\sigma^2 v_m)) = \gamma_{j-p}(\sigma^2 v_m) dt_m$$
modulo decomposables, and the spectral sequence collapses on the $E_p$-page. There are topologically determined extensions $(\sigma^2 v_m)^p = \sigma^2 v_{m+1}$ modulo decomposables, which give an isomorphism (as implied by \cref{thh-calculations})
$$\pi_\ast \gr^\bull_\ev \THH(\BP{n-1}/T(n)) \cong \BP{n-1}_\ast[\sigma^2(v_n)].$$
\end{example}
\begin{recall}\label{gauss-manin}
Let $Y$ be a scheme, and let $q: X \to \AA^1\times Y$ be a morphism, so that $X$ is a scheme over $Y$ via the projection $\pr: \AA^1\times Y \to Y$. Then the transitivity cofiber sequence in $\QCoh(X)$ runs
$$q^\ast L_{\AA^1\times Y/Y} \to L_{X/Y} \to L_{X/\AA^1\times Y}.$$
Since $q^\ast L_{\AA^1\times Y/Y}$ is a free $\co_X$-module of rank $1$ generated by $dt$ (where $t$ is a coordinate on $\AA^1$), we obtain a cofiber sequence
$$\dR^\ast_{X/Y} \to \dR^\ast_{X/\AA^1\times Y} \xar{\nabla} \dR^\ast_{X/\AA^1\times Y} dt,$$
where $\dR^\ast_{X/Y} = \bigoplus_{i\geq 0} (\wedge^i L_{X/Y})[-i]$ denotes the underlying derived commutative algebra of the downwards-shearing of $\Sym_{\co_X}(L_{X/Y}[1](1))$. The map $\nabla$ is the Gauss-Manin connection for the morphism $q$. Note that $\nabla$ satisfies Griffiths transversality: it sends the $n$th piece of the Hodge filtration to the $(n-1)$st piece.
\end{recall}
\begin{remark}
Observe that if $q$ is taken to be the morphism $Y \to \AA^1\times Y$ given by the inclusion of the origin into $\AA^1$, then $\dR_{Y/\AA^1\times Y}$ is $p$-completely isomorphic to the divided power algebra $\co_Y\pdb{t}$. Using the fact that $\dR_{Y/Y} \cong \co_Y$, it is not difficult to see that the Gauss-Manin connection $\nabla$ must send $\gamma_j(t) \mapsto \gamma_{j-1}(t) dt$. Here, we set $\gamma_{-1}(t) = 0$. In particular, $\nabla$ is a PD-derivation.
\end{remark}
\begin{example}[The topological Sen operator and $\Mfg$]\label{algebraic-sen}
The map $T(n-1) \to T(n)$ of homotopy commutative rings induces a map $\cM_{T(n)} \to \cM_{T(n-1)}$ of graded stacks, which sends a $p$-typical graded formal group equipped with a coordinate up to order $\leq p^n$ to the underlying $p$-typical graded formal group equipped with a coordinate up to order $\leq p^n-1$. The map $\cM_{T(n)} \to \cM_{T(n-1)}$ is an affine bundle: in other words, it exhibits $\cM_{T(n-1)}$ as the quotient of $\cM_{T(n)}$ by the group scheme $\GG_a^{(p^n-1)}/\GG_m$ over $B\GG_m$, where $\GG_a^{(p^n-1)}$ denotes the affine line with $\GG_m$-action of weight $p^n-1$. This follows, for instance, from \cite[Reduction of Lemma 3.2.3 to Lemma 3.2.7]{ericbook}. If $f: X \to \cM_{T(n)}$ is a stack over $\cM_{T(n)}$, the transitivity cofiber sequence in $\QCoh(X)$ is given by
$$f^\ast L_{\cM_{T(n)}/\cM_{T(n-1)}} \to L_{X/\cM_{T(n-1)}} \to L_{X/\cM_{T(n)}}.$$
Since $\cM_{T(n)} \to \cM_{T(n-1)}$ is a $\GG_a$-bundle, we see that $L_{\cM_{T(n)}/\cM_{T(n-1)}}$ is a free $\co_{\cM_{T(n)}}$-module of rank $1$ generated by the class $dt_i$.
It follows that there is a cofiber sequence
\begin{equation}\label{hh-mtn}
    \HH(X/\cM_{T(n-1)}) \to \HH(X/\cM_{T(n)}) \xar{\Theta_\mot} \Sigma^{2p^n,p^n} \HH(X/\cM_{T(n)})
\end{equation}
of quasicoherent sheaves on $X$, where $\Sigma^{n,w}$ denotes a shift by degree $n$ and weight $w$. As indicated by the notation, the map $\Theta_\mot$ behaves as an analogue on $\Mfg$ of the topological Sen operator of \cref{topological-sen}; more precisely, it is the effect of the topological Sen operator at the level of the $E^2$-page of the Adams-Novikov-B\"okstedt spectral sequence of \cref{anss-bokstedt}. Moreover, the discussion in \cref{gauss-manin} says that $\Theta_\mot$ can be understood as an analogue of the Gauss-Manin connection.
\end{example}
\begin{example}\label{relation-to-stacky}
The topological Sen operator on $\THH(\Z_p/J(p)) \cong \Z_p[x]$ sends $x^j\mapsto jx^{j-1}$, so that the action of the Sen operator is precisely the action of $\GG_a^\sharp$ on $\GG_a = \spec \Z_p[x]$ given by $\partial_x: \Z_p[x] \to \Z_p[x]$. Therefore, there is a $p$-complete graded isomorphism $\gr^\bull_\ev \THH(\Z_p) \cong \Gamma(\GG_a/\GG_a^\sharp; \co)$. In the same way, one can argue that there is a $p$-complete isomorphism $\gr^\bull_\ev \THH(\Z_p)^{t\Cp} \cong \Gamma(\GG_m/\GG_m^\sharp; \co)$.

This perspective is related to the stacky approach to Hodge-Tate cohomology \`a la \cite{drinfeld-prism, apc} in the following way. By \cite[Proposition 3.5.1]{drinfeld-prism}, there is an isomorphism $\GG_a/\GG_a^\sharp \cong \GG_a^\dR$; similarly, $\GG_m/\GG_m^\sharp \cong \GG_m^\dR$. Therefore:
\begin{align}
    \gr^\bull_\ev \THH(\Z_p) & \cong \Gamma(\GG_a^\dR; \co), \label{grmot-thhz}\\
    \gr^\bull_\ev \THH(\Z_p)^{t\Cp} & \cong \Gamma(\GG_m^\dR; \co). \label{grmot-thhz-tcp}
\end{align}
Since $\gr^\bull_\ev \THH(\Z_p)^{t\Cp}$ is supposed to arise as the cohomology of the total space $\Tot(\co_{\WCart^\HT}\{1\})$ of the Breuil-Kisin twisting line bundle $\co_{\WCart^\HT}\{1\}$ over $\WCart^\HT$, the isomorphism \cref{grmot-thhz-tcp} suggests that $\Tot(\co_{\WCart^\HT}\{1\}) \cong \GG_m^\dR$. In turn, this suggests that $\WCart^\HT$ should be $\GG_m^\dR/\GG_m = B\GG_m^\sharp$. This is indeed true: it is precisely \cite[Theorem 3.4.13]{apc}.

Similarly, $\gr^\bull_\ev \THH(\Z_p)$ is supposed to arise as the cohomology of the total space of the Breuil-Kisin twisting line bundle over the ``extended Hodge-Tate locus'' $\Delta_0'$ in Drinfeld's $\Sigma'$. (The stack $\Delta_0'$ is defined in \cite[Section 5.10.1]{drinfeld-prism}.) In \cite{bhatt-f-gauge-lectures}, the stack $\Sigma'$ is denoted by $\spf(\Z_p)^\cN$, and one might therefore denote $\Delta_0'$ by $\spf(\Z_p)^{\cN, \HT}$. The isomorphism \cref{grmot-thhz} then suggests that the total space of the Breuil-Kisin line bundle over $\spf(\Z_p)^{\cN, \HT}$ is $\GG_a^\dR$, which in turn suggests that $\spf(\Z_p)^{\cN, \HT}$ should be $\GG_a^\dR/\GG_m \cong \GG_a/(\GG_a^\sharp \rtimes \GG_m)$. This is indeed true: it is precisely \cite[Lemma 5.12.4]{drinfeld-prism}.

Had we worked with the evenly faithfully flat cover $\gr^\bull_\ev \THH(\Z_p) \to \gr^\bull_\ev \THH(\Z_p/S[t])$ (where $t\mapsto p$) instead, the stack associated to the even filtration on $\THH(\Z_p)$ would in fact be presented by (and is therefore isomorphic to) $\GG_a^\dR/\GG_m$.
\end{example}

\begin{variant}\label{gm-sharp}
One can also study the stack $\cM_{J(p)}$ associated to the $\Efr{2}$-ring $J(p)$. It is not difficult to show that the morphism $\cM_{J(p)} \to \Mfg$ exhibits $\cM_{J(p)}$ as a $\GG_m$-bundle over $\Mfg$; for example, the fiber product $\spec(\MU_\ast)/\GG_m \times_{\Mfg} \cM_{J(p)}$ is isomorphic to $\spec(\pi_\ast(\MU \otimes J(p)))/\GG_m$, but there is an equivalence of $\E{2}$-$\MU$-algebras $\MU \otimes J(p) \simeq \MU[t^{\pm 1}]$ with $|t|=0$.

Since $\cM_{J(p)}$ is a $\GG_m$-bundle over $\Mfg$, descent in Hochschild homology is controlled by a Gauss-Manin connection. If $Y$ is a scheme and $q: X \to \GG_m\times Y$ is a morphism, then there is a cofiber sequence
$$\dR^\ast_{X/Y} \to \dR^\ast_{X/\GG_m\times Y} \xar{\nabla} \dR^\ast_{X/\GG_m\times Y} d\log(t).$$
If $X$ is a stack over $\cM_{J(p)}$, we then obtain a cofiber sequence
$$\HH(X/\Mfg) \to \HH(X/\cM_{J(p)}) \xar{\Theta_\mot} \Sigma^{2,1} \HH(X/\cM_{J(p)})$$
of quasicoherent sheaves on $X$. This is an analogue on $\Mfg$ of the topological Sen operator of \cref{jp-sen}.
\end{variant}
\begin{remark}
Suppose that $T(1)$ admits the structure of an $\Efr{2}$-ring (this is true at $p=2$). The unit map on $T(1)$ defines a map $\TP(\Z_p) \to \TP(\Z_p/T(1))$. Since $\TP(\Z_p/T(1))$ is concentrated in even degrees by \cref{thh-calculations}, one can define the motivic filtration on $\TP(\Z_p/T(1))$ using the double-speed Postnikov filtration. Under the isomorphism $\pi_\ast \TP(\Z_p/T(1)) \cong \pi_\ast \BP{1}^{tS^1} \cong \Z_p\pw{\ptl}^{tS^1}$, one can view $\gr^0$ of the motivic filtration $\TP(\Z_p/T(1))$ as $\Z_p\pw{\ptl}$. Recall that $\TP(\Z_p)$ is a homotopical analogue of the Cartier-Witt stack $\WCart_{\Z_p}$ from \cite{prismatization}. One can then view the map $\TP(\Z_p) \to \TP(\Z_p/T(1))$ as an analogue of the following map induced by the $q$-de Rham point:
$$\rho_{\ptl\dR}:\spf \Z_p\pw{\ptl} \cong (\spf \Z_p\pw{q-1})/\FF_p^\times \to (\spf \Z_p\pw{q-1})/\Z_p^\times \xar{\rho_{q\dR}} \WCart_{\Z_p}.$$
This map classifies the prism $(\Z_p\pw{\ptl}, (\ptl))$, and can reasonably be called the $\ptl$-de Rham point.

As explained in the end of the introduction to \cite{even-filtr}, one hopes that the unit map $S^0 \to \TP(\Z_p)$ induces the map $\WCart_{\Z_p} \to \Mfg$ classifying Drinfeld's formal group over $\WCart_{\Z_p} = \Sigma$ from \cite{drinfeld-formal-group} on the associated graded of the motivic filtration. If \cref{conjecture-thh} were true (i.e., there is an equivalence $\TP(\Z_p/T(1)) \simeq \BP{1}^{tS^1}$ of spectra), the resulting unit map $S^0 \to \TP(\Z_p/T(1)) \to \BP{1}^{tS^1}$ would just be the unit of the $\Eoo$-ring $\BP{1}^{tS^1}$. Since $\BP{1}$ is complex-oriented, the formal group over $\pi_0 \BP{1}^{tS^1} \cong \Z_p\pw{\ptl}$ must be isomorphic to the formal group of $\BP{1}$, i.e., the $p$-typification of the multiplicative formal group. In particular, the aforementioned expectation about the formal group over $\WCart_{\Z_p}$ and its relation to $\TP(\Z_p)$ would predict that the pullback of Drinfeld's formal group over $\WCart_{\Z_p}$ along the map $\rho_{\ptl\dR}$ is the $p$-typification of the multiplicative formal group over $\Z_p\pw{\ptl}$. This is indeed true, and was proved in \cite[Section 2.10.6]{drinfeld-formal-group}. This lends further evidence to the idea that the map $\TP(\Z_p) \to \TP(\Z_p/T(1))$ is a homotopical analogue of the $\ptl$-de Rham point of $\WCart_{\Z_p}$.
\end{remark}

%% file: mfg/pth-powers.tex
Recall from \cref{thh-calculations} that $\pi_\ast(\THH(\BP{n}/T(n+1)) \otimes_{\BP{n}} \BP{n-1})$ is (additively) equivalent to the ``subalgebra'' $\BP{n-1}_\ast[\theta_{n-1}^p]$ of $\pi_\ast \THH(\BP{n-1}/T(n)) \cong \BP{n-1}_\ast[\theta_{n-1}]$. This picture has an analogue over $\Mfg$, as we now explain. We first need a simple calculation.
\begin{remark}\label{quotient-hh}
Let $R$ be a commutative ring, and let $x\in R$ be a regular element. Then there is a $p$-completed equivalence $\dR^\ast_{R[t]/x / R} \simeq R[t]\pdb{x'}/x \otimes_{R[t]/x} \Lambda_{R[t]/x}(dt)$ with $|x'|=0$. Indeed, this follows from combining the observation that $R[t]/x \cong R[t] \otimes_R R/x$ with the following $p$-completed equivalences: $\dR^\ast_{R[t]/R} \simeq \Lambda_{R[t]}(dt)$, $\dR^\ast_{R/x / R} \simeq R\pdb{x'}/x$. Similarly, there is an equivalence $\HH(R[t]/x / R) \simeq R[t][S^1 \times \CP^\infty]/x$.
\end{remark}
\begin{example}\label{frobenius-thh}
Let $i_{n-1}: \cZ(v_{n-1}) \hookrightarrow \cM_{T(n)}$ denote the closed substack cut out by the global section $v_{n-1}\in \H^0(\cM_{T(n)}; \co_{\cM_{T(n)}})$. If $f: X \to \cM_{T(n)}$ is a stack over $\cM_{T(n)}$, let $X^{v_{n-1}=0}$ denote the pullback of $X$ along $i_{n-1}$, and let $f: X^{v_{n-1}=0} \to \cZ(v_{n-1})$ denote the structure morphism. Then $i_{n-1}^\ast \HH(X/\cM_{T(n)}) = \HH(X^{v_{n-1}=0}/\cZ(v_{n-1}))$. In the case $X = \spec(\BP{n}_\ast)/\GG_m$, there is an isomorphism $X^{v_{n-1}=0} = \spec(\BP{n-1}_\ast)/\GG_m$. We will now relate $\HH(X^{v_{n-1}=0}/\cZ(v_{n-1}))$ to $\HH(X^{v_{n-1}=0}/\cM_{T(n-1)})$ by calculating $\HH(\cZ(v_{n-1})/\cM_{T(n-1)})$.

Recall from \cref{algebraic-sen} that there is a $\GG_a$-bundle $\cM_{T(n)} \to \cM_{T(n-1)}$. Note that $L_{\cM_{T(n)}/\cM_{T(n-1)}}$ is a free $\co_{\cM_{T(n)}}$-module of rank $1$ generated by a class $d(t_n)$,
and that $L_{\cZ(v_{n-1})/\cM_{T(n)}}$ is a free $\co_{\cZ(v_{n-1})}$-module of rank $1$ generated by a class $\sigma^2(v_{n-1})$.
Applying \cref{quotient-hh}, we find that 
\begin{equation}\label{hh-z-vn1}
    \pi_\ast \HH(\cZ(v_{n-1})/\cM_{T(n-1)}) \cong \co_{\cZ(v_{n-1})}\pdb{\sigma^2(v_{n-1})} \otimes_{\co_{\cZ(v_{n-1})}} \Lambda_{\co_{\cZ(v_{n-1})}}(dt_n).
\end{equation}
We therefore see that $\HH(X^{v_{n-1}=0}/\cM_{T(n-1)})$ a subquotient of the tensor product of $\HH(X^{v_{n-1}=0}/\cZ(v_{n-1}))$ and $f^\ast \HH(\cZ(v_{n-1})/\cM_{T(n-1)}) \cong \co_{X^{v_{n-1}=0}}\pdb{\sigma(v_{n-1})}[dt_n]/(dt_n)^2$.
Let us now take $f$ to be the morphism $\spec(\BP{n-1}_\ast)/\GG_m \to \cM_{T(n)}$. The $E^2$-page of the Adams-Novikov-B\"okstedt spectral sequence for $\THH(\BP{n-2}/T(n-1))$ is given by
$$E^2_{\ast,\ast} = \BP{n-2}_\ast\pdb{\sigma^2 v_j | j\geq n-1} \otimes_{\BP{n-1}_\ast} \Lambda_{\BP{n-1}_\ast}(dt_j | j \geq n),$$
and the extensions on the $E^\infty$-page are given by $(\sigma^2 v_j)^{p^{n-j}} = \sigma^2 v_n$. The above discussion therefore shows that $f^\ast \HH(\cZ(v_{n-1})/\cM_{T(n-1)})$ precisely detects the ``bottom piece'' of this $E^2$-page, i.e., the subalgebra $\BP{n-2}_\ast\pdb{\sigma^2 v_{n-1}} \otimes_{\BP{n-2}_\ast} \Lambda_{\BP{n-2}_\ast}(dt_n)$. Therefore, the preceding calculation of $\HH(\cZ(v_{n-1})/\cM_{T(n-1)})$ gives one explanation for why $\pi_\ast(\THH(\BP{n-1}/T(n)) \otimes_{\BP{n-1}} \BP{n-2})$ is (additively) equivalent to the ``subalgebra'' $\BP{n-2}_\ast[\theta_{n-2}^p]$ of $\pi_\ast \THH(\BP{n-2}/T(n-1)) \cong \BP{n-2}_\ast[\theta_{n-2}]$.
\end{example}
\begin{remark}\label{mod-vj-vn1}
We can extend the analysis of \cref{frobenius-thh} further. 
Let $0\leq j\leq n-1$, and let $i_{j,\cdots,n-1}: \cZ(v_{[j,n)}) \hookrightarrow \cM_{T(n)}$ denote the closed substack cut out by the global sections $v_j,\cdots,v_{n-1} \in \H^0(\cM_{T(n)}; \co_{\cM_{T(n)}})$. If $f: X \to \cM_{T(n)}$ is a stack over $\cM_{T(n)}$, let $X^{v_j,\cdots,v_{n-1}=0}$ denote the pullback of $X$ along $i$, and let $f: X^{v_j,\cdots,v_{n-1}=0} \to \cZ(v_{[j,n)})$ denote the structure morphism. Then $i_{j,\cdots,n-1}^\ast \HH(X/\cM_{T(n)})$ is equivalent to $\HH(X^{v_j,\cdots,v_{n-1}=0}/\cZ(v_{[j,n)}))$. In the case $X = \spec(\BP{n-1}_\ast)/\GG_m$, there is an isomorphism $X^{v_j,\cdots,v_{n-1}=0} = \spec(\BP{j-1}_\ast)/\GG_m$. We can now relate $\HH(X^{v_j,\cdots,v_{n-1}=0}/\cZ(v_{[j,n)}))$ to $\HH(X^{v_j,\cdots,v_{n-1}=0}/\cM_{T(j)})$ by calculating $\HH(\cZ(v_{[j,n)})/\cM_{T(j)})$.

We claim that there is an isomorphism
$$\HH(\cZ(v_{[j,n)})/\cM_{T(j)}) \cong \co_{\cZ(v_{[j,n)})}\pdb{\sigma(v_i) | j\leq i\leq n-1} \otimes_{\co_{\cZ(v_{[j,n)})}} \Lambda_{\co_{\cZ(v_{[j,n)})}}(dt_i | j+1\leq i \leq n)$$
To prove this, we will use descending induction on $j$; the base case $j=n-1$ was studied in \cref{frobenius-thh}. For the inductive step, suppose we know the result for $j+1$. Let $i_j: \cZ(v_{[j,n)}) \to \cZ(v_{j+1}, \cdots, v_{n-1})$ denote the closed substack cut out by $v_j$. Then there are isomorphisms
\begin{align*}
\HH(\cZ(v_{[j,n)})/\cM_{T(j+1)}^{v_j=0}) & \cong i_j^\ast \HH(\cZ(v_{j+1}, \cdots, v_{n-1})/\cM_{T(j+1)})\\
& \cong \co_{\cZ(v_{[j,n)})}\pdb{\sigma^2(v_i) | j+1\leq i\leq n-1} \otimes_{\co_{\cZ(v_{[j,n)})}} \Lambda_{\co_{\cZ(v_{[j,n)})}}(dt_i | j+2\leq i \leq n)
\end{align*}
Recall that \cref{frobenius-thh} gives an isomorphism between $\HH(\cM_{T(j+1)}^{v_j=0}/\cM_{T(j)})$ and $\co_{\cM_{T(j+1)}^{v_j=0}}\pdb{\sigma^2(v_j)} \otimes_{\co_{\cM_{T(j+1)}^{v_j=0}}} \Lambda_{\co_{\cM_{T(j+1)}^{v_j=0}}}(dt_{j+1})$. The desired calculation of $\HH(\cZ(v_{[j,n)})/\cM_{T(j)})$ is now a simple computation with the transitivity sequence for the composite
$$\cZ(v_{[j,n)}) \to \cM_{T(j+1)}^{v_j=0} \to \cM_{T(j)}.$$
Let $X = \spec(\BP{j-1}_\ast)/\GG_m$, and let $f: X \to \cZ(v_{[j,n)})$ be the structure map. Then the above discussion implies that $\HH(\spec(\BP{j-1}_\ast)/\GG_m/\cM_{T(j)})$ is isomorphic to the tensor product of $f^\ast \HH(\cZ(v_{[j,n)})/\cM_{T(j)})$ and $\HH(\spec(\BP{j-1}_\ast)/\GG_m/\cZ(v_{[j,n)}))$. This gives the $E^2$-page of the Adams-Novikov-B\"okstedt spectral sequence computing $\pi_\ast \THH(\BP{j-1}/T(j))$ (see \cref{anss-bokstedt}), and one can run this spectral sequence as in \cref{homology-thh-bpn}. If $\BP{j-1}$ admits the structure of an $\E{3}$-algebra, there are extensions $\sigma^2(v_i)^p = \sigma^2(v_{i+1})$ modulo decomposables on the $E^\infty$-page of this spectral sequence.

Let $T(n)/v_{[j,n)}$ denote $T(n)/(v_j, \cdots, v_{n-1})$. Since $\theta_j\in \pi_\ast \THH(\BP{j-1}/T(j))$ is represented by $\sigma^2(v_j)$, we find that $\THH(T(n)/v_{[j,n)}/T(j))$ is (additively) equivalent to as $T(n)[\theta_j]/(v_{[j,n)}, \theta_j^{p^{n-j}})$. (See \cref{mod-jvn-thh} for a more topological perspective on this observation.) This discussion provides an algebraic perspective on why $\pi_\ast \THH(\BP{n-1}/T(n))/v_{[j,n)}$ is (additively) equivalent to as the ``subalgebra'' of $\pi_\ast \THH(\BP{j-1}/T(j))$ generated by $\theta_j^{p^{n-j}}$.
\end{remark}
\begin{remark}\label{frob-in-topology}
In topology, \cref{frobenius-thh} plays out as follows, if we assume\footnote{There is an unconditional variant of the following discussion, obtained by replacing $T(n)$ with $X(p^{n+1}-1)$. However, this comes at the cost of adding the spaces $\Delta_n$ into the mix.} \cref{tn-e2}. Let $n\geq 1$. We begin by observing that $T(n)/v_{n-1}$ is the Thom spectrum of an $\E{1}$ map $\mu: \Omega J_{p-1}(S^{2p^{n-1}}) \to \BGL_1(T(n-1))$; in particular, $T(n)/v_{n-1}$ admits the structure of an $\E{1}$-ring. To see this, we first define the map $\mu$ as follows. There is a map $S^{2p^{n-1}} \to \B^2\GL_1(X(p^n-1))$ which detects the class $v_{n-1}\in \pi_{2p^{n-1}-2} X(p^n-1)$, which naturally extends to a map $J_{p-1}(S^{2p^{n-1}}) \to \B^2\GL_1(X(p^n-1))$ since we are working $p$-locally. Therefore, we obtain an $\E{1}$-map $\Omega J_{p-1}(S^{2p^{n-1}}) \to \BGL_1(X(p^n-1))$. The projection $X(p^n-1) \to T(n-1)$ is a map of $\E{2}$-rings by \cref{tn-e2}, and therefore induces an $\E{1}$-map $\BGL_1(X(p^n-1))\to \BGL_1(T(n-1))$. Composition with
the $\E{1}$-map $\Omega J_{p-1}(S^{2p^{n-1}}) \to \BGL_1(X(p^n-1))$ produces the desired map $\mu$. 
The fact that the Thom spectrum of $\mu$ can be identified with $T(n)/v_{n-1}$ can be proved directly using \cref{tn-vj-vn1-homology} below.
It follows from this discussion that there is an equivalence 
$$\THH(T(n)/v_{n-1}/T(n-1)) \simeq T(n)[J_{p-1}(S^{2p^{n-1}})]/v_{n-1}.$$
Moreover, under the equivalence $\THH(\BP{n-2}/T(n-1)) \simeq \BP{n-2}[\Omega S^{2p^{n-1} + 1}]$ of \cref{thh-calculations}(a), the map $\THH(T(n)/v_{n-1}/T(n-1)) \to \THH(\BP{n-2}/T(n-1))$ induced by the map $T(n)/v_{n-1} \to \BP{n-2}$ is given by the skeletal inclusion of $J_{p-1}(S^{2p^{n-1}}) \to \Omega S^{2p^{n-1}+1}$. The projection $\THH(\BP{n-2}/T(n-1)) \to \THH(\BP{n-1}/T(n))/v_{n-1}$ can be identified with the effect on $\BP{n-2}$-chains of the James-Hopf map $\Omega S^{2p^{n-1}+1} \to \Omega S^{2p^n+1}$. Therefore, the EHP sequence
$$J_{p-1}(S^{2p^{n-1}}) \to \Omega S^{2p^{n-1}+1} \to \Omega S^{2p^n+1}$$
shows that $\THH(\BP{n-1}/T(n))/v_{n-1}$ is (additively) equivalent to precisely as the ``subalgebra'' of $\THH(\BP{n-2}/T(n-1))$ generated by $\theta_{n-1}^p$.
The above calculation of $\THH(T(n)/v_{n-1}/T(n-1))$ is a topological incarnation of the calculation of $\HH(\cZ(v_{n-1})/\cM_{T(n-1)})$ in \cref{frobenius-thh}. Indeed, the Adams-Novikov-B\"okstedt spectral sequence (see \cref{anss-bokstedt}) runs
\begin{equation}\label{anss-z-vn1}
    E^2_{\ast,\ast} = \pi_\ast \HH(\cZ(v_{n-1})/\cM_{T(n-1)}) \Rightarrow \pi_\ast \gr^\bull_\ev \THH(T(n)/v_{n-1}/T(n-1)),
\end{equation}
and the $E^2$-page is given by \cref{hh-z-vn1}. Again, one can establish analogues of the B\"okstedt differentials \cref{bokstedt-diffl} in the Adams-Novikov-B\"okstedt spectral sequence, and thereby obtain an alternative approach to the above calculation of $\THH(T(n)/v_{n-1}/T(n-1))$.
\end{remark}
\begin{remark}\label{mod-jvn-thh}
Let us continue to assume \cref{tn-e2}, and let $0\leq j \leq n-1$. Recall that $T(n)/v_{[j,n)}$ denote $T(n)/(v_j, \cdots, v_{n-1})$. Recall that
\begin{equation}\label{tn-vj-vn1-homology}
    \H_\ast(T(n)/v_{[j,n)}; \FF_p) =
    \begin{cases}
    \FF_2[\zeta_1^2, \cdots, \zeta_j^2, \zeta_{j+1}, \cdots, \zeta_n] & p=2, \\
    \FF_p[\zeta_1, \cdots, \zeta_n] \otimes \Lambda_{\FF_p}(\tau_j, \cdots, \tau_{n-1}) & p>2.
    \end{cases}
\end{equation}
It is natural to ask if the discussion in \cref{frob-in-topology} extends to a description of $\THH(T(n)/v_{[j,n)}/T(j))$, paralleling \cref{mod-vj-vn1}. This is an ill-posed question, since it is not clear that $T(n)/v_{[j,n)}$ admits the structure of an $\E{1}$-algebra. Nevertheless, if $T(n)/v_{[j,n)}$ did admit the structure of an $\E{1}$-$T(j)$-algebra, then an analysis similar to \cref{thh-calculations} shows that
$$\THH(T(n)/v_{[j,n)} / T(j)) \simeq T(n)[J_{p^{n-j}-1}(S^{2p^j})]/v_{[j,n)}.$$
This is the topological analogue of the calculation of \cref{mod-vj-vn1}. Under the equivalence $\THH(\BP{j-1}/T(j)) \simeq \BP{j-1}[\Omega S^{2p^j + 1}]$ of \cref{thh-calculations}(a), the map $\THH(T(n)/v_{[j,n)}/T(j)) \to \THH(\BP{j-1}/T(j))$ induced by the map $T(n)/v_{[j,n)} \to \BP{j-1}$ is given by the skeletal inclusion of $J_{p^{n-j}-1}(S^{2p^j}) \to \Omega S^{2p^j+1}$. The projection
$$\THH(\BP{j-1}/T(j)) \to \THH(\BP{n-1}/T(n))/v_{[j,n)}$$
can be identified with the effect on $\BP{j-1}$-chains of the James-Hopf map $\Omega S^{2p^j+1} \to \Omega S^{2p^n+1}$. Therefore, the EHP sequence
$$J_{p^{n-j}-1}(S^{2p^j}) \to \Omega S^{2p^j+1} \to \Omega S^{2p^n+1}$$
shows that $\pi_\ast \THH(\BP{n-1}/T(n))/v_{[j,n)}$ is (additively) equivalent to precisely the ``subalgebra'' of $\pi_\ast \THH(\BP{j-1}/T(j))$ generated by $\theta_j^{p^{n-j}}$.

Since $\THH(T(n)/v_{[j,n)} / T(j)) \simeq T(n)[J_{p^{n-j}-1}(S^{2p^j})]/v_{[j,n)}$, one expects $T(n)/v_{[j,n)}$ to have an $\E{1}$-cell structure over $T(j)$ described by the cell structure of $J_{p^{n-j}-1}(S^{2p^j})$. Although we do not know how to prove this unconditionally, it is not difficult to show if we further assume \cite[Conjectures D and E]{bpn-thom}. In this case, \cite[Corollary B]{bpn-thom} says that there is a map $f: \Omega^2 S^{2p^j+1} \to \BGL_1(T(j))$ which detects $v_j\in \pi_{2p^j-2} T(j)$ on the bottom cell of the source, such that the Thom spectrum of $f$ is a form of $\BP{j-1}$. Let\footnote{In \cref{frob-in-topology}, we described the map $f_{n,n-1}$ without assuming \cite[Conjectures D and E]{bpn-thom}. It is generally not possible to describe $f_{n,j}$ similarly if $j<n-1$: although there is a map $S^{2p^j} \to \B^2\GL_1(T(j))$ which detects $v_j\in \pi_{2p^j-2} T(j)$, there are $p$-local obstructions to extending along $S^{2p^j} \to J_{p^{n-j}-1}(S^{2p^j})$ if $n-j>1$. These obstructions can be viewed as the $\E{2}$-Browder brackets on $v_j$; \cite[Conjecture E]{bpn-thom} implies that these Browder brackets can be compatibly trivialized.} $f_{n,j}: \Omega J_{p^{n-j}-1}(S^{2p^j}) \to \BGL_1(T(j))$ denote the composite of $f$ with the $\E{1}$-map $\Omega J_{p^{n-j}-1}(S^{2p^j}) \to \Omega^2 S^{2p^j+1}$. Then the Thom spectrum of $f_{n,j}$ is equivalent to $T(n)/v_{[j,n)}$ as a $T(j)$-module. This is not quite an ``$\E{1}$-cell structure'' for $T(n)/v_{[j,n)}$, since $f_{n,j}$ is not an $\E{1}$-map; nevertheless, this construction of $T(n)/v_{[j,n)}$ suffices to calculate $\THH(T(n)/v_{[j,n)} / T(j))$.
\end{remark}
\begin{example}\label{similarity-ehp}
If $R$ is an $\E{2}$-algebra which is an $\E{1}$-$T(n)$-algebra, one can loosely interpret the above discussion as saying that the square
\begin{equation}\label{pushout-ehp}
    \xymatrix{
    R[J_{p^{n-j}-1}(S^{2p^j})] \simeq R \otimes_{T(n)} \THH(T(n)/v_{[j,n)}/T(j)) \ar[r] \ar[d] & \THH(R/v_{[j,n)}/T(j)) \ar[d] \\
    R \ar[r] & \THH(R/T(n))/v_{[j,n)}
    }
\end{equation}
exhibits the top-right corner as the ``tensor product of the top-left and bottom-right corners''. Note that the homotopy of the top-left corner is $R[\theta_j]/\theta_j^{p^{n-j}}$.
The bottom-right corner should be thought of as $\THH(R/v_{[j,n)}/T(n)/v_{[j,n)})$, although it is difficult to make this picture precise (since $T(n)/v_{[j,n)}$ does not admit the structure of an $\E{2}$-algebra).

For instance, if $R = \BP{n-1}$, then the square \cref{pushout-ehp} says that the square
\begin{equation}
    \xymatrix{
    \BP{n-1}[J_{p^{n-j}-1}(S^{2p^j})] \ar[r] \ar[d] & \THH(\BP{j-1}/T(j)) \simeq \BP{j-1}[\Omega S^{2p^j+1}] \ar[d] \\
    R \ar[r] & \THH(\BP{n-1}/T(n))/v_{[j,n)} \simeq \BP{j-1}[\Omega S^{2p^n+1}]
    }
\end{equation}
exhibits the top-right corner as the tensor product of the top-left and bottom-right corners. This is essentially the observation that the map $\THH(\BP{j-1}/T(j)) \to \THH(\BP{n-1}/T(n))/v_{[j,n)}$ sends $\theta_j^m\mapsto 0$ unless $p^{n-j}|m$, in which case $\theta_j^m \mapsto \theta_n^{m/p^{n-j}}$. This therefore explains the similarity between $\THH(\BP{n-1}/T(n))$ and $\THH(\BP{j-1}/T(j))$ given by \cref{thh-calculations}(a).
\end{example}
\begin{remark}
As mentioned before, the lack of structure on the objects involved above make it difficult to use the above picture to understand the multiplicative structure on $\THH(\BP{n-1})$; but it does point to a plan of attack. Namely, one can attempt to understand the even filtration on $\THH(\BP{n-1})/v_{[0,n)}$ by considering the natural map $\THH(\BP{n-1})/v_{[0,n)} \to \THH(\FF_p)$. It is not hard to see that this map is an eff cover, so that the stack associated to the even filtration on $\THH(\BP{n-1})/v_{[0,n)}$ is the quotient of the scheme associated to the even filtration on $\THH(\FF_p)$ by a certain group scheme. The scheme associated to the even filtration on $\THH(\FF_p)$ is precisely $\GG_a$, and the above discussion suggests that the stack associated to the even filtration on $\THH(\BP{n-1})/v_{[0,n)}$ is isomorphic to $\GG_a/W[F^n]$; this is also suggested by work of Lee in \cite{lee-thh}. We hope to study this in future work joint with Jeremy Hahn and Arpon Raksit. To this end, we set up some groundwork for future investigation of this stack in \cref{cartier-duals}, where we study some basic properties of $W[F^n]$.
\end{remark}
\begin{remark}
The calculation of $\THH(T(n)/v_{[j,n)} / T(j))$ in \cref{mod-jvn-thh} shows that more is true: if $n\geq k-1$, the structure of $\BP{n}$ as an $\E{1}$-$X(p^k)$-algebra (i.e., $\THH(\BP{n}/T(k))$) mirrors the structure of $\BP{n-k}$ as an $\E{1}$-algebra over the sphere (i.e., $\THH(\BP{n-k})$).
\end{remark}
\begin{remark}
Note that the Thom spectrum of the map $f_{n,0}$ has been studied in \cite{mrs}, where it was denoted $y(n)$. Just as the $y(n)$ describe a filtration of $y(\infty) = \FF_p$ by $\E{1}$-algebras, the spectra $T(n)/v_{[j,n)}$ describe a filtration of $\BP{j-1}$. For instance, it is not difficult to show that for $j\leq k \leq n-1$, the spectrum $T(n)/v_{[j,n)}$ is $\mathrm{Tel}(k)$-acyclic. Therefore, if $T(n)/v_{[j,n)}$ admits the structure of an $\E{1}$-ring, the same argument as \cite[Corollary 4.15]{lmmt} implies that the map $K(T(n)/v_{[j,n)}) \to K(\BP{j-1})$ is an $\mathrm{Tel}(k)$-equivalence for $j\leq k \leq n-1$. Since $K(\BP{j-1})$ is $\mathrm{Tel}(k)$-locally contractible for $k\geq j+1$, the only interesting case is $k = j$; in this case, we find that the maps 
$$K(T(j+1)/v_j) \to K(T(j+2)/(v_j, v_{j+1})) \to \cdots \to K(\BP{j-1})$$
are all $\mathrm{Tel}(j)$-equivalences.
\end{remark}
\begin{remark}\label{iterated-thom-tj}
Since $T(n)/v_{[j,n)}$ is closely related to $T(j)$ by \cref{mod-jvn-thh}, it is natural to wonder if there is a relationship between $T(n)/v_{[j,n)}$ and $T(n+k)/v_{[j,n+k)}$, in a manner compatible with their relationship to $T(j)$. By \cref{frob-in-topology}, $T(n+k)/v_{[n,n+k)}$ is the Thom spectrum of a map $\Omega J_{p^k-1}(S^{2p^n}) \to \BGL_1(T(n))$. It follows that if $T(n)/v_{[j,n)}$ admits the structure of an $\E{1}$-ring, then $T(n+k)/v_{[j,n+k)}$ is the Thom spectrum of a map $\Omega J_{p^k-1}(S^{2p^n}) \to \BGL_1(T(n)/v_{[j,n)})$. As mentioned in \cref{mod-jvn-thh}, if we further assume \cite[Conjectures D and E]{bpn-thom}, the spectrum $T(n)/v_{[j,n)}$ (resp. $T(n+k)/v_{[j,n+k)}$) is the Thom spectrum of a map $\Omega J_{p^{n-j}-1}(S^{2p^j}) \to \BGL_1(T(j))$ (resp. $\Omega J_{p^{n+k-j}-1}(S^{2p^j}) \to \BGL_1(T(j))$).

The relationship between the two presentations of $T(n+k)/v_{[j,n+k)}$ (as a Thom spectrum over $T(n)/v_{[j,n)}$ and over $T(j)$) is explained by the following observation in unstable homotopy theory: there is a fibration\footnote{To construct the fibration \cref{partial-ehp}, recall that there is an EHP sequence
$$J_{p^m-1}(S^{2d}) \to \Omega S^{2d+1} \xar{H} \Omega S^{2dp^m+1}.$$
By dimension considerations, the canonical map $J_{p^{m+k}-1}(S^{2d}) \to \Omega S^{2d+1}$ factors through $J_{p^{m+k}-1}(S^{2d}) \to J_{p^k-1}(S^{2dp^m}) \times_{\Omega S^{2dp^m+1}} \Omega S^{2d+1}$, and one can easily check that this map is an equivalence. This implies the desired fiber sequence \cref{partial-ehp}.}
\begin{equation}\label{partial-ehp}
    J_{p^m-1}(S^{2d}) \to J_{p^{m+k}-1}(S^{2d}) \xar{H} J_{p^k-1}(S^{2dp^m}).
\end{equation}
Indeed, applying \cref{partial-ehp} when $m = {n-j}$ and $d=p^j$, we obtain a fibration of $\E{1}$-spaces:
$$\Omega J_{p^{n-j}-1}(S^{2p^j}) \to \Omega J_{p^{n+k-j}-1}(S^{2p^j}) \xar{H} \Omega J_{p^k-1}(S^{2p^n}).$$
The composite of $f_{n+k,j}: \Omega J_{p^{n+k-j}-1}(S^{2p^j}) \to \BGL_1(T(j))$ with the map $\Omega J_{p^{n-j}-1}(S^{2p^j}) \to \Omega J_{p^{n+k-j}-1}(S^{2p^j})$ is $f_{n,j}: \Omega J_{p^{n-j}-1}(S^{2p^j}) \to \BGL_1(T(j))$. Therefore, \cite[Proposition 2.1.6]{bpn-thom} implies that there is a map $f_{n+k,j}: \Omega J_{p^k-1}(S^{2p^n}) \to \BGL_1(T(n)/v_{[j,n)})$ whose Thom spectrum is $T(n+k)/v_{[j,n+k)}$. This is the desired relationship between the various presentations of $T(n+k)/v_{[j,n+k)}$.
\end{remark}
\begin{remark}
Observe that the preceding discussion implies, in particular, that there is a map $q_k: \Omega J_{p^k-1}(S^{2p^n}) \to \BGL_1(y(n))$ such that the composite $S^{2p^n-1} \to \Omega J_{p^k-1}(S^{2p^n}) \to \BGL_1(y(n))$ detects $v_n\in \pi_{2p^n-2} y(n)$, and such that the Thom spectrum of the map $q_k$ is $y(n+k)$. Taking $k \to \infty$, this implies that there is a map $q_\infty: \Omega^2 S^{2p^n+1} \to \BGL_1(y(n))$ whose Thom spectrum is $y(\infty) = \FF_p$. The map $q_\infty$ is adjoint to the $\E{1}$-map $\Omega^3 S^{2p^n+1}_+ \to y(n)$ from \cite[Section 4.1]{mrs} which detects $v_n$ on the bottom cell of the source.
\end{remark}

%% file: appendix/ko-tmf-analogues.tex
Many of the results from the body of this article extend to the case of $\ko$ and $\tmf$. In this section, we will state these results; since the proofs are essentially the same, we will not give arguments unless the situation is substantially different. We will specialize to the case $p=2$ for simplicity. One of the main observations in \cref{thh-calculations} is that the structure of $\BP{n}$ as an $\E{1}$-algebra over $X(p^n)$ (or rather, $T(n)$) mirrors the structure of $\Z_p$ as an $\E{1}$-algebra over $S^0$. For $\ko$ and $\tmf$, there are analogues of $X(p^n)$, which we studied in \cite{bpn-thom}.
\begin{recall}
Let $A$ denote the free $\E{1}$-algebra $S\mmod \nu$ with a nullhomotopy of $\nu$, i.e., the Thom spectrum of the $\E{1}$-map $\Omega S^5 \to \BGL_1(S)$ which detects $\nu\in \pi_3(S)$ on the bottom cell\footnote{The spectrum $A$ has been studied before by Mahowald and his coauthors in \cite{mahowald-thom, davis-mahowald, mahowald-v2-periodic, mahowald-bo-res, mahowald-imj, mahowald-unell-bott}, where it is often denoted $X_5$.}. This spectrum has the property that $\H_\ast(A; \FF_2) \cong \FF_2[\zeta_1^4]$ (in fact, $\BPP_\ast(A) \cong \BPP_\ast[\frac{\eta_R(v_1^2) - v_1^2}{4}] \cong \BPP_\ast[t_1^2 + v_1 t_1]$). There is an $\E{1}$-map $i: A \to \ko$ such that under the isomorphism $\H_\ast(\ko; \FF_2) \cong \FF_2[\zeta_1^4, \zeta_2^2, \zeta_3, \cdots]$, the map $i$ corresponds to the inclusion of $\FF_2[\zeta_1^4]$. In particular, the map $A \to \ko$ is an equivalence in dimensions $\leq 4$. There is in fact an $\E{1}$-map $A \to \MSpin$, induced from an $\E{1}$-map $\Omega S^5 \to \BSpin$. There is also an $\E{1}$-map $A \to X(2) = T(1)$, such that $T(1) \simeq A \otimes C\eta$. We note that the ``$Q_0$-Margolis homology'' of $\H_\ast(\ko; \FF_2)$ (i.e., the homology of $\Sq^1$ viewed as a differential acting on $\H_\ast(\ko; \FF_2)$) is precisely $\H_\ast(A; \FF_2)$.

Similarly, let $B$ denote the $\E{1}$-algebra of \cite[Definition 3.2.17]{bpn-thom}\footnote{The $\E{1}$-ring $B$ has been briefly studied under the name $\ol{X}$ in \cite{hopkins-mahowald-orientations}.}, so that there is an $\E{1}$-space $N$ such that $B$ is the Thom spectrum of an $\E{1}$-map $N \to \BString$. We will not recall the construction of $N$ here; we only say that $B$ is obtained from the $\E{1}$-quotient $S\mmod\sigma$ by further taking an ``$\E{1}$-quotient'' by the class in $\pi_{11}(S\mmod\sigma)$ constructed from a nullhomotopy of $\nu\sigma\in \pi_{10}(S)$. This spectrum has the property that $\H_\ast(B; \FF_2) \cong \FF_2[\zeta_1^8, \zeta_2^4]$ (in fact, $\BPP_\ast(B) \cong \BPP_\ast[b_4, y_6]$, where $b_4 \equiv t_1^4$ and $y_6 \equiv t_2^2$ modulo $(2,v_1, \cdots)$)\footnote{For the sake of illustration, we remark that if $p=2$, then $b_4$ can be taken to be the following cobar representative for $\sigma = \alpha_{4/4}$, where the $v_i$s are Hazewinkel's generators:
\begin{align*}
    b_4 & = \frac{1}{2}\left(\frac{\eta_R(v_1^4) - v_1^4}{8} - (\eta_R(v_1 v_2) - v_1 v_2)\right)\\
    & = 5 t_1^4 + 9 t_1^3 v_1 + 7 t_1^2 v_1^2 - 2 t_1 t_2 + 2 t_1 v_1^3 - t_1 v_2 - t_2 v_1.
\end{align*}
Here, we used the formula
$$\eta_R(v_2) = v_2 - 5v_1 t_1^2 - 3v_1^2 t_1+ 2t_2 - 4t_1^3.$$
}. There is an $\E{1}$-map $i: B \to \tmf$ such that under the isomorphism $\H_\ast(\tmf; \FF_2) \cong \FF_2[\zeta_1^8, \zeta_2^4, \zeta_3^2, \zeta_4, \cdots]$, the map $i$ corresponds to the inclusion of $\FF_2[\zeta_1^8, \zeta_2^4]$. In particular, the map $B \to \tmf$ is an equivalence in dimensions $\leq 12$. There is in fact an $\E{1}$-map $B \to \MString$. There is also an $\E{1}$-map $B \to T(2)$ such that $T(2) \simeq B \otimes DA_1$, where $DA_1$ is an $8$-cell complex whose mod $2$ cohomology is isomorphic to the subalgebra of the Steenrod algebra generated by $\Sq^2$ and $\Sq^4$. We note that the ``$Q_0$-Margolis homology'' of $\H_\ast(\tmf; \FF_2)$ (i.e., the homology of $\Sq^1$ viewed as a differential acting on $\H_\ast(\tmf; \FF_2)$) is precisely $\H_\ast(B; \FF_2)$.
\end{recall}
The following was implicitly stated in \cite{bpn-thom}, but we make it explicit here:
\begin{conjecture}\label{a-b-e2}
The $\E{1}$-algebra structures on $A$ and $B$ admit extensions to $\Efr{2}$-algebra structures such that the maps $A \to X(2)$, $B \to X(4)_{(2)}$, $A\to \ko$, and $B \to \tmf$ admit the structure of $\Efr{2}$-maps.
\end{conjecture}
A calculation paralleling \cref{homology-thh-bpn} shows:
\begin{prop}
Assume \cref{a-b-e2}. There are isomorphisms
\begin{align*}
    \H_\ast(\THH(\ko/A); \FF_2) & \cong \H_\ast(\ko; \FF_2)[\sigma(\zeta_3)] \otimes_{\FF_2} \Lambda_{\FF_2}(\sigma(\zeta_2^2)), \\
    \H_\ast(\THH(\tmf/B); \FF_2) & \cong \H_\ast(\tmf; \FF_2)[\sigma(\zeta_4)] \otimes_{\FF_2} \Lambda_{\FF_2}(\sigma(\zeta_3^2)).
\end{align*}
Here, $|\sigma(\zeta_3)| = 8$, $|\sigma(\zeta_2^2)|=7$, $|\sigma(\zeta_4)| = 16$, and $|\sigma(\zeta_3^2)| = 15$.
\end{prop}
Using the Adams spectral sequence for $\pi_\ast \THH(\ko/A)$ and $\pi_\ast \THH(\tmf/B)$ as in \cref{thh-calculations}(b) (and using $\ko$- and $\tmf$-linearity), one finds:
\begin{theorem}\label{ko-a}
Assume \cref{a-b-e2}. Upon $2$-completion, there are equivalences
\begin{align*}
    \THH(\ko/A) & \simeq \ko \oplus \bigoplus_{j\geq 1} \Sigma^{8j - 1} \ko/2j, \\
    \THH(\tmf/B) & \simeq \tmf \oplus \bigoplus_{j\geq 1} \Sigma^{16j - 1} \tmf/2j.
\end{align*}
\end{theorem}
\begin{remark}
Since $\ko \otimes C\eta \simeq \ku$, \cref{ko-a} implies that $\THH(\ko/A) \otimes C\eta \simeq \THH(\ku/T(1))$. Relatedly, there is an equivalence $\ko \otimes T(1) \simeq \ku[\Omega S^5]$ of $\E{1}$-$T(1)$-algebras, which implies that 
$$\THH(\ko) \otimes T(1) \simeq \THH(\ko \otimes T(1)/T(1)) \simeq \ku[S^5] \oplus \bigoplus_{j\geq 1} \Sigma^{8j - 1} \ku[S^5]/2j.$$
Along similar lines, \cref{ko-a} implies that $\THH(\tmf/B) \otimes DA_1 \simeq \THH(\BP{2}/T(2))$. There is also a $2$-local equivalence $\tmf \otimes T(2) \simeq \BP{2}[N]$ of $\E{1}$-$T(2)$-algebras, so that 
$$\THH(\tmf) \otimes T(2) \simeq \BP{2}[\B N] \oplus \bigoplus_{j\geq 1} \Sigma^{16j - 1} \BP{2}[\B N]/2j.$$
Note that $\BP{2}[N] \simeq \BP{2}[\Omega S^9 \times \Omega S^{13}]$, so that $\pi_\ast(\tmf \otimes T(2)) \cong \mathbf{Z}_{(2)}[v_1, v_2, x_8, y_{12}]$, where $|v_1| = 2$, $|v_2| = 6$, $|x_8|=8$, and $|y_{12}|=12$.
This gives a potential approach to calculating $\THH(\ko)$ (resp. $\THH(\tmf)$) via the $T(1)$-based (resp. $T(2)$-based) Adams-Novikov spectral sequence. Describing this spectral sequence is essentially equivalent to calculating the analogue of the topological Sen operator for $\THH(\ko/A)$, whose construction is described below in \cref{A-thh}.
\end{remark}
\begin{remark}
Recall from \cref{bpn-vs-tj} that the structure of $\ko$ over $A$ mirrors the structure of $\tmf$ over $B$, which in turn mirrors the structure of $\BP{n}$ over $T(n)$; in other words, the calculation of \cref{ko-a} is along the diagonal line $(n,n)$ in \cref{bpn-vs-tj}. It is natural to wonder whether there is an $\E{1}$-ring $\tilde{A}$ equipped with an $\E{1}$-map $A \to \tilde{A}$ and an $\E{1}$-map $\tilde{A} \to \ko$ such that the structure of $\ko$ over $\tilde{A}$ mirrors the structure of $\BP{n-1}$ over $T(n)$. (This is the ``off-diagonal line'' $(n,n-1)$ in \cref{bpn-vs-tj}.) This question is only interesting when $p=2$, since $\ko_{(p)}$ splits as a direct sum of even shifts of $\BP{1}$ if $p>2$. Let us localize at $2$ for the remainder of this discussion. Examining the argument establishing \cref{ko-a} when $p=2$, one finds that the mod $2$ homology of $\tilde{A}$ must be $\H_\ast(\tilde{A}; \FF_2) \cong \FF_2[\zeta_1^4, \zeta_2^2]$.

If $A$ admits the structure of an $\Efr{2}$-ring, then $\tilde{A}$ can be constructed as follows. The class $\sigma_1\in \pi_5(A)$ determined by a nullhomotopy of $\eta\nu$ (see \cite[Remark 3.2.17]{bpn-thom}) defines a map $S^6 \to \BGL_1(A)$, which, thanks to our assumption on $A$, extends to an $\E{1}$-map $\Omega S^7 \to \BGL_1(A)$. The desired $\E{1}$-$A$-algebra $\tilde{A}$ can be defined as Thom spectrum of this map.
(According to \cite[Remark 5.1.5]{bpn-thom}, one should not expect $\tilde{A}$ to admit a natural construction as a Thom spectrum over the sphere.) Note that $\tilde{A} \otimes C\eta \simeq T(2)$.
\end{remark}
The same argument as \cref{thh-calculations}(a) shows:
\begin{prop}\label{tilde-A-thh}
If both $A$ and $\tilde{A}$ admit the structure of $\Efr{2}$-rings and $\ko$ admits the structure of an $\E{1}$-$\tilde{A}$-algebra, then there is a $2$-complete equivalence
$$\THH(\ko/\tilde{A}) \simeq \ko[\Omega S^9],$$
where the generator in $\pi_8 \THH(\ko/\tilde{A})$ is $\sigma^2(v_2)$. Moreover,
\begin{equation}\label{tilde-A-tp}
    \H^c_\ast(\TP(\ko/\tilde{A}); \FF_2) \cong \FF_2[\zeta_1^4, \zeta_2^2, \zeta_3^2, \zeta_4, \cdots]\ls{\hbar}.
\end{equation}
\end{prop}
\begin{remark}
Similarly, if $B$ admits the structure of an $\Efr{2}$-ring, the class $\sigma_2\in \pi_{13}(B)$ from \cite[Remark 3.2.24]{bpn-thom} defines a map $S^{14} \to \BGL_1(B)$. Thanks to our assumption on $B$, this extends to an $\E{1}$-map $\Omega S^{15} \to \BGL_1(B)$. Define $\tilde{B}$ to be Thom spectrum of this map, so that $\H_\ast(\tilde{B}; \FF_2) \cong \FF_2[\zeta_1^8, \zeta_2^4, \zeta_3^2]$. Note that $\tilde{B} \otimes DA_1 \simeq T(3)$.

If $\tilde{B}$ admits the structure of an $\Efr{2}$-ring and $\tmf$ admits the structure of an $\E{1}$-$\tilde{B}$-algebra, then the same argument as in \cref{thh-calculations}(a) shows that there is a $2$-complete equivalence
$$\THH(\tmf/\tilde{B}) \simeq \tmf[\Omega S^{17}],$$
where the generator in $\pi_{16} \THH(\tmf/\tilde{B})$ is $\sigma^2(v_3)$. Moreover,
\begin{equation}\label{tilde-B-tp}
    \H^c_\ast(\TP(\tmf/\tilde{B}); \FF_2) \cong \FF_2[\zeta_1^8, \zeta_2^4, \zeta_3^2, \zeta_4^2, \zeta_5, \cdots]\ls{\hbar}.
\end{equation}
\end{remark}

\begin{construction}[Topological Sen operator for $\THH$ relative to $A$]\label{A-thh}
By \cite[Theorem 1]{thh-thom}, $\THH(A)$ is equivalent to the Thom spectrum of the composite
$$\cL S^5 \xar{\cL \nu} \cL \B^2\GL_1(S) \simeq \B^2 \GL_1(S) \times \BGL_1(S) \xar{\id\times \eta} \BGL_1(S).$$
There is a \textit{nonsplit} fiber sequence
\begin{equation}\label{free-loops-S5}
    \Omega S^5 \to \cL S^5 \to S^5,
\end{equation}
and the restriction of the above composite along the map $\Omega S^5 \to \cL S^5$ is the map $\Omega S^5 \to \BGL_1(S)$ which defines $A$. It follows from the fiber sequence \cref{free-loops-S5} that $\THH(A)$ is the Thom spectrum of a map $S^5 \to \BGL_1(A)$ which detects a class $x\in \pi_4(A) \cong \pi_4(\ko)$. In particular, $\THH(A)$ is an $A$-module with two cells. Now assume \cref{a-b-e2}; then \cite[Corollary 2.8]{framed-e2} gives a splitting $\THH(A) \to A$, which implies that the class $x\in \pi_4(A)$ must be trivial. In other words, $\THH(A) \simeq A[S^5]$. If $\cC$ is an $A$-linear $\infty$-category, this implies the existence of a cofiber sequence
\begin{equation}\label{A-sen}
    \THH(\cC) \to \THH(\cC/A) \xar{\Theta'} \Sigma^6 \THH(\cC/A).
\end{equation}
One should be able to recover the calculation of $\THH(\ko)$ from \cite{thh-bp1} using \cref{A-sen} in the case $\cC = \Mod_\ko$ and the calculation of \cref{ko-a}. 
Similarly to \cref{mod-p-vn}, \cref{ko-a} and \cref{A-sen} imply that
\begin{align*}
    \THH(\ko/A)/2 & \simeq \ko[S^7 \times \Omega S^9]/2, \\
    \THH(\ko) \otimes_\ko \FF_2 & \simeq \FF_2[S^5 \times S^7 \times \Omega S^9].
\end{align*}
The latter of these has been proven by Angeltveit-Rognes in \cite[Theorem 6.2]{angeltveit-rognes}.
\end{construction}
\begin{remark}
Recall from \cite[Corollary 9.3]{mahowald-rezk} that Mahowald-Rezk duality gives an equivalence $W \ko \simeq \Sigma^6 \ko$ (resp. $W \BP{1} \simeq \Sigma^{2p} \BP{1}$); the shift of $6$ (resp. $2p$) in this equivalence arises for the same reason as in \cref{A-sen} (resp. \cref{tn-sen} with $n=1$): both correspond to the class $\sigma^2(t_1^2)$ (resp. $\sigma^2(t_1)$).
We hope to explore this further in future work.
\end{remark}
\begin{remark}
There is also an analogue of the topological Sen operator for $B$. To describe it, one observes using an argument similar to \cref{A-thh} that $\THH(B/S\mmod\sigma) \simeq B[S^{13}]$ and that $\THH(S\mmod\sigma) \simeq (S\mmod\sigma)[S^9]$. This implies that if $\cC$ is a $B$-linear $\infty$-category, there are cofiber sequences
\begin{align*}
    \THH(\cC/S\mmod\sigma) & \to \THH(\cC/B) \xar{\Theta_B} \Sigma^{14} \THH(\cC/B),\\
    \THH(\cC) & \to \THH(\cC/S\mmod\sigma) \xar{\Theta_B'} \Sigma^{10} \THH(\cC/S\mmod\sigma).
\end{align*}
However, it is significantly more complicated to describe these cofiber sequences in almost any nontrivial example, so we omit further discussion. Nevertheless, one can use \cref{ko-a} to show the following equivalences analogous to \cref{mod-p-vn}:
\begin{align*}
    \THH(\tmf/B)/2 & \simeq \tmf[S^{15} \times \Omega S^{17}]/2, \\
    \THH(\tmf) \otimes_\tmf \FF_2 & \simeq \FF_2[S^9 \times S^{13} \times S^{15} \times \Omega S^{17}];
\end{align*}
note that $\FF_2[S^9 \times S^{13}] \simeq \FF_2[BN]$. The latter of these has been proven by Angeltveit-Rognes in \cite[Theorem 6.2]{angeltveit-rognes}.
\end{remark}
Assume \cref{a-b-e2}, and let $p=2$. Then there is a map $\cM_{T(1)} \to \cM_A$ of stacks over $\Mfg$, which exhibits $\cM_{T(1)}$ as a $2$-fold fppf cover of $\cM_A$. Recall that $\cM_{T(1)}$ is isomorphic to the moduli stack of graded formal groups equipped with a coordinate up to order $\leq 2$ (equivalently, order $\leq 3$ for $2$-typical formal groups). Similarly, we have:
\begin{prop}\label{stack-A}
The stack $\cM_A$ is isomorphic to the moduli stack of graded formal groups equipped with an even coordinate up to order $\leq 5$.
\end{prop}
\begin{proof}
Recall that there is a fiber sequence
$$\SU(2)/\U(1) \cong S^2 \to \BU(1) \simeq \CP^\infty \to \BSU(2) \simeq \HHP^\infty.$$
Let $n\geq 1$. There is a homotopy equivalence $\HHP^n\times_{\HHP^\infty} \CP^\infty \simeq \CP^{2n+1}$ (since $S^{4n+3}/\SU(2) = \HHP^n$ and $S^{4n+3}/\U(1) = \CP^{2n+1}$), which produces the ``twistor fibration'', i.e., the fiber sequence
\begin{equation}\label{s2-bundle-hpn}
    S^2 \to \CP^{2n+1} \to \HHP^n.
\end{equation}
The map $\CP^{2n+1} \to \HHP^n$ is given in coordinates by the map $[z_1: \cdots: z_{2n+2}] \mapsto [z_1 + z_2 \mathbf{j} : \cdots: z_{2n+1} + z_{2n+2} \mathbf{j}]$.
Note that $\SU(2)/\U(1) = \CP^1$ is the unit sphere $S(\fr{su}(2))$ in the adjoint representation of $\SU(2)$, so \cref{s2-bundle-hpn} equivalently says that $\CP^{2n+1}$ is the sphere bundle of the adjoint bundle of rank $3$ over $\HHP^n$.

Let $R$ be a complex-oriented homotopy commutative ring with associated formal group $\GG$ over $\pi_\ast(R)$; we will assume for simplicity that $2$ is not a zero-divisor in $\pi_\ast R$. Then $R^\ast(\CP^5)$ is isomorphic to the ring of functions on $\GG$ which vanish to order $\geq 6$. The Serre spectral sequence associated to the fiber sequence \cref{s2-bundle-hpn} implies that $R^\ast(\HHP^2)$ is isomorphic to $R^\ast(\CP^5)^{\Z/2}$, where $\Z/2$ acts by inversion on the formal group. This implies the desired claim, since there is an equivalence $\HHP^2 \simeq \Sigma^4 C\nu$ of spectra, and $A$ is the free $\E{1}$-ring whose unit factors through the inclusion $S^0 \to C\nu$.
\end{proof}
\begin{remark}
The description of $\cM_A$ in \cref{stack-A} has concrete applications; for instance, in \cite{hodge}, we show that $\cM_{\tmf \otimes A} = \cM_A \times_{\Mfg} \Mell$ can be identified with the moduli stack of elliptic curves $\ce$ equipped with a splitting of the Hodge filtration on $\H^1_\dR(\ce)$, and use this to describe an topological analogue of the integral ring of quasimodular forms.
\end{remark}
\begin{remark}
As explained in \cite[Remark 7.1.7]{bpn-thom}, there is a $\Z/2$-equivariant $\E{1}$-algebra $A_{\Z/2}$ whose underlying $\E{1}$-algebra is $A$, and such that $\Phi^{\Z/2} A_{\Z/2} = X(2)_{(2)}$ as $\E{1}$-algebras. This is a topological interpretation of the following observation suggested by \cref{stack-A}: $\cM_A$ is ``half'' of $\cM_{X(2)}$; more precisely, there is a two-fold fppf cover $\cM_{X(2)} \to \cM_A$. This is an algebraic analogue of the equivalence $A \otimes C\eta \simeq X(2)$.

We also note that there is a $\Z/2$-equivariant analogue  of the fiber sequence \cref{s2-bundle-hpn}: namely, there is a $\Z/2$-equivariant twistor fibration
$$\xymatrix{
S^\rho \ar[d]^-\cong \ar[r] & \CP^{2n-1}_\RR \ar[d]^-\cong \ar[r] & \HHP^{n-1}_\RR \ar[d]^-\cong \\
S^{2\rho-1}/S^\sigma \ar[r] & S^{2n\rho-1}/S^\sigma \ar[r] & S^{2n\rho-1}/S^{2\rho-1}.
}$$
where $\Z/2$ acts on $\HHP^n$ via the action of $\Z/2\subseteq S^1\subseteq \SO(3)$ on $\bH$. The underlying fibration is \cref{s2-bundle-hpn}, while the $\Z/2$-fixed points gives the fibration
$$S^1 \to \RP^{2n-1} \to \CP^{n-1}$$
which exhibits $\RP^{2n-1}$ as the sphere bundle of the complex line bundle $\co(2)$ on $\CP^{n-1}$.
\end{remark}
\begin{construction}
One consequence of the identification of $\cM_A$ in \cref{stack-A} is that $\cM_A \to \Mfg$ is an affine bundle, so that the pullback of the cotangent complex $L_{\cM_A/\Mfg}$ to $\spec(\BPP_\ast(A))/\GG_m \cong \spec(\BPP_\ast[t_1^2 + v_1 t_1])/\GG_m$ can be identified with a free $\BPP_\ast[t_1^2 + v_1 t_1]$-module of rank $1$ generated by the class $d(t_1^2 + v_1 t_1)$ in weight $2$. Using \cref{gauss-manin}, we obtain the algebraic analogue of \cref{A-sen}: if $X$ is a stack over $\cM_A$, there is a cofiber sequence
\begin{equation*}
    \HH(X/\Mfg) \to \HH(X/\cM_A) \xar{\Theta_\mot} \Sigma^{6,3} \HH(X/\cM_A).
\end{equation*}
The stack $\cM_{\ko}$ can be identified with the moduli stack of curves of the form $y = x^2 + bx + c$ with change of coordinate $x\mapsto x+r$, and $\HH(\cM_{\ko}/\Mfg)$ describes the $E_1$-page of the Adams-Novikov-B\"okstedt spectral sequence calculating $\THH(\ko)$ (see \cref{anss-bokstedt}). It would be interesting to explicitly describe $\HH(\cM_{\ko}/\Mfg)$; note that
$$\pi_\ast \HH(\cM_\ko/\cM_A) \cong \co_{\cM_\ko}\pdb{\sigma^2 v_j | j\geq 2} \otimes_{\co_{\cM_\ko}} \Lambda_{\co_{\cM_\ko}}(dt_i | i\geq 2),$$
where $\sigma^2(v_j)$ lives in degree $2^{j+1}$ and weight $2^j$, and $dt_i$ lives in degree $2^{i+1}-1$ and weight $2^i$.
This can be proved exactly as in \cref{hh-bpn-tn-mfg}; weight considerations presumably allow one to fully describe $\Theta_\mot: \HH(X/\cM_A) \to \Sigma^{6,3} \HH(X/\cM_A)$, and hence $\HH(\cM_\ko/\Mfg)$.
\end{construction}
\begin{remark}\label{nonexistence-fibrations}
Assume \cref{a-b-e2}, and let $p=2$. It is trickier to describe the stack $\cM_B$ in a manner analogous to \cref{stack-A}. As a first approximation, if we assume that $S\mmod \sigma$ admits the structure of a homotopy commutative ring, one can attempt to describe the moduli stack $\cM_{S\mmod \sigma}$. However, it is provably impossible to construct a Hurewicz fibration
$$S^4 \to \HHP^5 \to \OP^2$$
in the point-set category. This is a consequence of \cite[Theorem 5.1]{schultz-fibration}, which states more generally that if $F \to E \to X$ is a Hurewicz fibration where $E$ is homotopy equivalent to $\HHP^{2n+1}$ and $F$ and $X$ are homotopy equivalent to finite CW-complexes, then either $F$ or $X$ must be contractible. Note that this result implies that there cannot even be a Hurewicz fibration 
$$S^4 \to \HHP^3 \to S^8.$$
Similarly, there cannot be Hurewicz fibrations
\begin{align*}
    \CP^3 & \to \CP^7 \to S^8, \\
    \CP^3 & \to \CP^{11} \to \OP^2;
\end{align*}
see \cite{no-quotient-hopf} for the impossibility of the first Hurewicz fibration (which implies the impossibility of the second Hurewicz fibration). These no-go results make it difficult to give a formal group-theoretic description of $R^\ast(\OP^2)$ (and hence of $\cM_{S\mmod \sigma}$, since $\OP^2 \simeq \Sigma^8 C\sigma$) where $R$ is a complex-oriented homotopy commutative ring.
\end{remark}

The story for $\ko$ admits a slightly different generalization to higher heights.
\begin{example}\label{recast-ko}
Observe that $S^5 = \SU(4)/\Sp(2)$, and that the map $\Omega S^5 \to \BU$ (whose Thom spectrum is $A$) can be viewed as the composite
$$\Omega(\SU(4)/\Sp(2)) \to \Omega(\SU/\Sp) \simeq \BSp \to \BU.$$
The equivalence $\Omega(\SU/\Sp) \simeq \BSp$ is given by Bott periodicity, and the map $\BSp \to \BU$ takes a symplectic bundle to its underlying unitary bundle.
\end{example}
Motivated by \cref{recast-ko}, we are led to the following definition:
\begin{definition}\label{XH-n}
Define an $\E{1}$-algebra $X_\bH(n)$ via the Thom spectrum of the composite
$$\Omega(\SU(2n)/\Sp(n)) \to \Omega(\SU/\Sp) \simeq \BSp \to \BU.$$
There is a canonical $\E{1}$-map $X_\bH(n) \to X_\bH(\infty) = \MSp$. 
\end{definition}
The spectrum $X_\bH(n)$ has been studied by Andy Baker.
\begin{remark}\label{c2-fixed}
See \cite{crabb-mitchell} for a detailed study of the space $\Omega(\SU(2n)/\Sp(n))$. Let us note that if $\SU(2n)_\bH$ denotes the $\Z/2$-equivariant loop space with the $\Z/2$-action given by the symplectic involution
$$A \mapsto J\ol{A} J^{-1}, \ J = \begin{pmatrix}
0 & -1\\
1 & 0
\end{pmatrix}^{\oplus n},$$
then $\Omega(\SU(2n)/\Sp(n)) \simeq (\Omega^\sigma \SU(2n)_\bH)^{\Z/2}$. Indeed, the fixed points of $\SU(2n)_\bH$ is $\Sp(n)$ (by definition), so we can apply the first sentence of \cref{yan-summand} to conclude.
\end{remark}
\begin{remark}
As the notation indicates, $X_\bH(n)$ should be viewed as a quaternionic analogue of the $X(n)$ spectra from \cite{ravenel-loc}; see \cref{Xn-analogies}.
\begin{table}[h!]
    \centering
    \begin{tabular}{c | c c c c c c c c c c}
	Normed division algebra & $\RR$ & $\cc$ & $\bH$ \\
	\hline
	$\E{1}$-ring & $y(n) \subseteq \Phi^{\Z/2} X(2^n)_\RR$ & $T(n) \subseteq X(2^n)$ & $X_\bH(2^n)$ \\
	\hline
	Limit as $n\to \infty$ & $\FF_2 \subseteq \Phi^{\Z/2} \MU_\RR = \MO$ & $\BPP \subseteq \MU$ & $\MSp$ \\
	\hline
	Mod $2$ homology & $\FF_2[\zeta_1, \cdots, \zeta_n]$ & $\FF_2[\zeta_1^2, \cdots, \zeta_n^2]$ & $\FF_2[\zeta_1^4, \cdots, \zeta_n^4]$
    \end{tabular}
    \vspace{0.5cm}
    \caption{The analogies between $y(n)$, $T(n)$, and $X_\bH(n)$, where the implicit prime is $p=2$. The inclusion $y(n) \subseteq \Phi^{\Z/2} X(2^n)_\RR$ is discussed in \cref{yan-summand}; see \cite{yan-thom}. The final row is to be interpreted as follows: $\H_\ast(\Phi^{\Z/2} X(2^n)_\RR; \FF_2)$ is a direct sum of even shifts of $\FF_2[\zeta_1, \cdots, \zeta_n]$; similarly for $\H_\ast(X(2^n); \FF_2)$ and $\H_\ast(X_\bH(2^n); \FF_2)$.}
    \label{Xn-analogies}
\end{table}

Note that there are isomorphisms of algebras
\begin{align*}
    \H_\ast(\Phi^{\Z/2} X(n)_\RR; \FF_2) & \cong \H_\ast(\Omega(\SU(n)/\SO(n)); \FF_2) \cong \FF_2[x_1, \cdots, x_{n-1}], \\
    \H_\ast(X_\bH(n); \Z) & \cong \H_\ast(\Omega(\SU(2n)/\Sp(n)); \Z) \cong \Z[y_1, \cdots, y_{n-1}],
\end{align*}
where $|x_j| = j$ and $|y_j| = 4j$.
\end{remark}
\begin{example}
By construction, $X_\bH(2) \simeq A = S\mmod \nu$.
\end{example}
\begin{construction}\label{partial-orientation}
Suppose that $X_\bH(n)$ admits the structure of a homotopy commutative ring. One can then also ask for an interpretation of the stack $\cM_{X_\bH(n)}$ analogous to \cref{stack-A}. It turns out that the difficulties of \cref{nonexistence-fibrations} are no longer an issue for $X_\bH(n)$. Indeed, the analogue of the map $S^4 \to \Omega S^5$ (whose Thomification is the map $C\nu \to A$ used in the proof of \cref{stack-A}) is given by a map $\iota: \HHP^{n-1} \to \Omega(\SU(2n)/\Sp(n))$ which exhibits $\HHP^{n-1}$ as the generating complex of $\Omega(\SU(2n)/\Sp(n))$. (See, e.g., \cite[Proposition 1.4]{crabb-mitchell}.)
Moreover, the composite
$$\HHP^{n-1} \xar{\iota} \Omega(\SU(2n)/\Sp(n)) \to \Omega(\SU/\Sp) \simeq \BSp$$
factors as $\HHP^{n-1} \hookrightarrow \HHP^\infty \simeq \BSp(1) \to \BSp$. Since the Thom spectrum of the tautological quaternionic line bundle over $\HHP^{n-1}$ is $\Sigma^{-4} \HHP^n$, the map $\iota$ Thomifies to a map $\Sigma^{-4} \HHP^n \to X_\bH(n)$.
\end{construction}
Using the twistor fibration \cref{s2-bundle-hpn} and the map $\Sigma^{-4} \HHP^n \to X_\bH(n)$ of \cref{partial-orientation}, one can argue as in \cref{stack-A} to show: 
\begin{prop}\label{stack-XHn}
The stack $\cM_{X_\bH(n)}$ is isomorphic to the moduli stack of graded formal groups equipped with an even coordinate up to order $\leq 2n+1$.
\end{prop}
\begin{remark}\label{quaternionic-sen}
Suppose that $X_\bH(n)$ admits the structure of an $\Efr{2}$-ring.
There is also a canonical map $\cM_{X_\bH(n-1)} \to \cM_{X_\bH(n)}$ which exhibits $\cM_{X_\bH(n)}$ as the quotient of $\cM_{X_\bH(n-1)}$ by the group scheme $\GG_a^{(2n-2)}/\GG_m$ over $B\GG_m$, where $\GG_a^{(2n-2)}$ denotes the affine line with $\GG_m$-action of weight $2n-2$. This is the algebraic analogue of the following:
\begin{lemma}
The spectrum $X_\bH(n)$ is equivalent to the Thom spectrum of a map $\Omega S^{4n-3} \to \BGL_1(X_\bH(n-1))$. 
\end{lemma}
\begin{proof}
By \cite[Proposition 2.1.6]{bpn-thom} (see also \cite{beardsley-thom}), it suffices to establish that there is a fiber sequence of $\E{1}$-spaces
\begin{equation*}
    \Omega(\SU(2n-2)/\Sp(n-1)) \to \Omega(\SU(2n)/\Sp(n)) \to \Omega S^{4n-3}.
\end{equation*}
To see this, observe that there is a diffeomorphism $\SU(2n)/\Sp(n) \cong \SU(2n-1)/\Sp(n-1)$, and hence a fibration
$$\xymatrix{
\SU(2n-2)/\Sp(n-1) \ar@{=}[d] \ar[r] & \SU(2n-1)/\Sp(n-1) \ar[d]^-\cong \ar[r] & \SU(2n-1)/\SU(2n-2) \ar[d]^-\cong \\
\SU(2n-2)/\Sp(n-1) \ar[r] & \SU(2n)/\Sp(n) \ar[r] & S^{4n-3}.
}$$
The desired fiber sequence is obtained by looping the bottom row.
\end{proof}
\begin{remark}
On the bottom cell of the source, the map $\Omega S^{4n-3} \to \BGL_1(X_\bH(n-1))$ defines a class $\chi_n^\bH\in \pi_{4n-5} X_\bH(n-1)$, and $\chi_{2^n}^\bH$ is detected in the $E_2$-page of the Adams-Novikov spectral sequence for $X_\bH(2^n-1)$ by $[t_n^2]$. Moreover, if $X_\bH(n-1)$ admits the structure of an $\Efr{2}$-ring and $X_\bH(n)$ admits the structure of an $\E{1}$-$X_\bH(n-1)$-algebra, then $\THH(X_\bH(n)/X_\bH(n-1)) \simeq X_\bH(n)[\Omega S^{4n-3}]$.
\end{remark}
We can then conclude (as in \cref{topological-sen} and \cref{algebraic-sen}) that if $\cC$ is an $X_\bH(n)$-linear $\infty$-category and $X$ is a stack over $\cM_{X_\bH(n)}$, then there are cofiber sequences
\begin{align*}
    \THH(\cC/X_\bH(n-1)) & \to \THH(\cC/X_\bH(n)) \xar{\Theta'} \Sigma^{4n-2} \THH(\cC/X_\bH(n)), \\
    \HH(X/\cM_{X_\bH(n-1)}) & \to \HH(X/\cM_{X_\bH(n)}) \xar{\Theta_\mot} \Sigma^{4n-2,2n-1} \HH(X/\cM_{X_\bH(n)}).
\end{align*}
Only the first cofiber sequence requires that $X_\bH(n-1)$ and $X_\bH(n)$ admit the structure of $\Efr{2}$-rings, and that $X_\bH(n)$ admits the structure of an $\E{1}$-$X_\bH(n-1)$-algebra; the second cofiber sequence only requires that $X_\bH(n)$ admit the structure of a homotopy commutative ring.
\end{remark}

%% file: appendix/diffracted-hodge-of-Zpn.tex
In this brief section, we give an alternative algebraic argument for \cref{diffr-zpn} following \cite[Example 5.15]{prismatization}. I am very grateful to Sasha Petrov for an illuminating discussion about this entire appendix; see also \cite[Lemma 6.13]{petrov-hdr}.
\begin{proof}[Alternative proof of \cref{diffr-zpn}]
Let $R$ be a (discrete) commutative $\Z/p^n$-algebra. Then \cite[Construction 3.8]{prismatization} implies that
\begin{align*}
    \spec(\Z/p^n)^{\slashed{D}}(R) & \simeq \Map_\CAlg(\Z/p^n, W(R)/V(1))\\
    & \simeq \{z\in W(R) | zV(1) = p^n\} = \{x\in W(R) | V(Fz) = p^n\}.
\end{align*}
Since $V$ is injective, this is a torsor for $\W[F](R) = \GG_a^\sharp(R)$. Moreover, this torsor is trivializable, i.e., $p^n$ is in the image of $VF$. In fact, we claim that
\begin{equation}\label{pn V}
    p^n = V(p^{n-1}) = VF(V(p^{n-2}))\in W(\Z/p^n).
\end{equation}
To see this, let us compute in ghost coordinates. Recall that if $w(x) = (w_0(x), w_1(x), \cdots)$ are the ghost coordinates of $x\in W(R)$, then $w_{n+1}(Vx) = pw_n(x)$. Since $w(p^n) = (p^n, p^n, \cdots)$ and $w(V(p^{n-1})) = (0, p^n, p^n, \cdots)$, we see that
$$w(p^n - V(p^{n-1})) = (p^n, 0, 0, \cdots).$$
Since the map $\GG_a^\sharp \cong W[F] \to W$ sends $x\in \GG_a^\sharp$ to the Witt vector whose ghost coordinates are $(x, 0, 0, \cdots)$, the claim \cref{pn V} follows from the observation that $p^n\in W[F](\Z_p)$ is sent to zero in $W[F](\Z/p^n)$. 
\end{proof}
\begin{remark}\label{Gm-sharp action}
As pointed out by Sasha Petrov, the preceding calculation also determines the $\GG_m^\sharp$-action on $\spec(\Z/p^n)^{\slashed{D}}$ as follows. The above discussion says that the isomorphism $\GG_a^\sharp \xar{\sim} \spec(\Z/p^n)^{\slashed{D}}$ sends $x\mapsto x + V(p^{n-2})$. Under this isomorphism, the action of $g\in \GG_m^\sharp$ on $x+V(p^{n-2}) \in \spec(\Z/p^n)^{\slashed{D}}$ is given by
$$g(x + V(p^{n-2})) = gx + gV(p^{n-2}) = gx + V(F(g) p^{n-2});$$
but $F(g) = 1$ since $\GG_m^\sharp = W^\times[F]$, so that this can be identified with $gx + V(p^{n-2})$. In other words, the isomorphism $\GG_a^\sharp \xar{\sim} \spec(\Z/p^n)^{\slashed{D}}$ is equivariant for the scaling action of $\GG_m^\sharp$ on $\GG_a^\sharp$.
\end{remark}
One can get a formula which is more ``accurate'' than \cref{pn V} via the following (see also \cite[Page 56]{illusie-slides}, where part of this statement is attributed to Gabber)\footnote{Our understanding is that this result is quite well-known; some form is heavily used in \cite{bms-i}.}.
\begin{lemma}\label{gabber-lemma}
Let $y$ denote the element of $W(\Z_p)$ associated to the ghost coordinates $(1-p^{p-1}, 1-p^{p^2-1}, \cdots)$. Then $[p] + V(y) = p$. Moreover, $y = Fx$ for some $x\in W(\Z_p)$ if and only if $p>2$; in this case, $x\in W(\Z_p)^\times$ (and hence $y\in W(\Z_p)^\times$). If $p=2$, then $y[2^m]$ is in the image of $F$ for any $m\geq 2$.
\end{lemma}
\begin{remark}
Let us assume $p$ is odd for simplicity. Then \cref{gabber-lemma} implies that $p-[p]\in W(\Z_p)$ is a unit multiple of $V(1)$, since $p-[p] = V(y) = xV(1)$ and $x\in W(\Z_p)^\times$.\footnote{Analogously, $[2](2 - [2]) = [2] V(y) = V(y[4])\in W(\Z_p)$ is divisible by $V(1)$ for $p=2$.} It follows from \cite[Construction 3.8]{prismatization} that if $X = \spf(R)$ is a bounded $p$-adic formal scheme, then the diffracted Hodge complex $X^{\slashed{D}}$ is given on $p$-nilpotent rings $S$ by $X^{\slashed{D}}(S) = X(W(R)/(p-[p]))$.
\end{remark}
\begin{remark}
Applying $F$ to the identity $[p]+V(y)=p$, we see that $[p^2] = p(1-y)$. In particular, the element $a\in W(\Z_p)$ of \cite[Lemma 4.7.3]{drinfeld-prism} can be identified with $1-y$.
\end{remark}
\begin{remark}
Using \cref{gabber-lemma}, we can give an ``alternative'' formula for a preimage of $p^n$ under $VF$. Indeed, we have $p = [p] + V(y)$ for some $y\in W(\Z_p)$, so that $p^n = [p^n] + \sum_{i=0}^{n-1} \binom{n}{i} [p^i] V(y)^{n-i}$ in $W(\Z_p)$. Because $V(a)b = V(aFb)$ and $FV = p$, we have $V(a)^n = V(p^{n-1} a^n)$ by an easy induction on $n$. Moreover, $[p^i] V(a) = V(aF[p^i]) = V([p^{pi}]a)$. Since $[p^n] = 0\in W(\Z/p^n)$ (and hence in $W(R)$), we have 
$$p^n = \sum_{i=0}^{n-1} \binom{n}{i} [p^i] V(y)^{n-i} = \sum_{i=0}^{n-1} \binom{n}{i} V(p^{n-i-1} y^{n-i} F[p^i]) \in W(\Z/p^n).$$
Assume $p>2$, so that \cref{gabber-lemma} implies that $y = Fx$ for some $x\in W(\Z_p)$. The multiplicativity of $F$ now lets us conclude that
$$p^n = VF\left(\sum_{i=0}^{n-1} \binom{n}{i} p^{n-i-1} x^{n-i} [p^i]\right) \in W(\Z/p^n),$$
so that $p^n\in W(R)$ is in the image of $VF$, as desired.

One can check that
$$\sum_{i=0}^{n-1} \binom{n}{i} p^{n-i-1} y^{n-i} [p^{pi}] = p^{n-1} \in W(\Z/p^n).$$
This is essentially an elaboration on the proof of \cref{gabber-lemma}. Indeed, applying $w_j$, we have
\begin{align*}
    w_j\left(\sum_{i=0}^{n-1} \binom{n}{i} p^{n-i-1} y^{n-i} [p^{pi}]\right) & = \frac{1}{p} \sum_{i=0}^{n-1} \binom{n}{i} (p-p^{p^{j+1}})^{n-i} p^{p^{j+1}i} \\
    & = p^{n-1} - p^{p^{j+1}n-1}.
\end{align*}
It therefore suffices to show that the Witt vector $a\in W(\Z_p)$ with coordinates $w_j(a) = p^{p^{j+1}n-1}$ vanishes in $W(\Z/p^n)$, which follows from a direct calculation.
\end{remark}

Let us end with a proof of \cref{gabber-lemma}; the explicit formulas below are unnecessary for any conceptual development, but we included it since the computation was rather fun.
\begin{proof}[Proof of \cref{gabber-lemma}]
First, it is easy to see that $y$ is well-defined.
Let us now check that $p = [p] + V(y)$. If $w(x) = (w_0(x), w_1(x), \cdots)$ are the ghost coordinates of $x\in W(R)$, then $w_{n+1}(Vx) = pw_n(x)$. It follows that $w_n(Vy) = p-p^{p^n}$. Since $w([p]) = (p, p^p, p^{p^2}, \cdots)$, we have 
$$w([p]+Vy) = w([p]) + w(Vy) = (p, p, \cdots) = w(p),$$
so that $p = [p] + V(y)$, as claimed. 

To prove the claim about $y$ being in the image of $F$, recall that if $x\in W(R)$, then the ghost coordinates of $Fx$ are given by $w_n(Fx) = w_{n+1}(x)$. In particular, $y = Fx$ for some $x\in W(\Z_p)$ if and only if we can solve
$$1-p^{p^n-1} = x_0^{p^n} + px_1^{p^{n-1}} + \cdots + p^n x_n$$
for some $x_0,\cdots, x_n\in \Z_p$ and all $n\geq 1$. This is impossible for $p=2$. Indeed, first note that we need $x_0^2 + 2x_1 = 1-p^{p-1} = -1$, so that $x_0^2 \equiv 1\pmod{2}$ (and hence $x_0 \equiv 1\pmod{2}$). Write $x_0 = 1 + 2s$, so that $x_0^2 + 2x_1 = 1 + 4s(1+s) + 2x_1$. In order for this to equal $1-p^{p-1}=-1$, we need $4s(1+s) + 2x_1 = -2$, i.e., $x_1\equiv 1\pmod{2}$. This implies that $x_0^4 \equiv 1\pmod{8}$ and $x_1^2 \equiv 1\pmod{4}$ (so $2x_1^2 \equiv 2\pmod{8}$). Since $1 - p^{p^2-1} = -7 = x_0^4 + 2x_1^2 + 4x_2$, we can reduce modulo $8$ to find that $1 \equiv 1 + 2 + 4x_2\pmod{8}$. But then $x_2$ would solve $4x_2 \equiv -2\pmod{8}$, which is impossible.

Now assume $p>2$. Since $x_0^p + px_1 = 1-p^{p-1}$, we have $x_0^p \equiv 1\pmod{p}$; this implies that $x_0^{p^n} \equiv 1\pmod{p^{n+1}}$. Writing $x_0^{p^n} = 1 - p^{n+1} s_n$ for some $s_n\in \Z_p$, we have $x_1 = ps_1-p^{p-2}$. Since $p>2$, we see that $x_1 = p(s_1 - p^{p-3})\in p\Z_p$. We claim that $x_n$ exists and is an element of $p\Z_p$ for $n\geq 1$. We established the base case $n=1$ above, so assume that $x_1, \cdots, x_{n-1} \in p\Z_p$, and let $x_i = pt_i$. We then have
\begin{align*}
    p^n x_n & = 1 - p^{p^n-1} - (x_0^{p^n} + px_1^{p^{n-1}} + \cdots + p^{n-1} x_{n-1}^p)\\
    & = p^{n+1} s_n - p^{p^n-1} - p^{p^{n-1}+1} t_1^{p^{n-1}} - \cdots - p^{p+n-1} t_{n-1}^p,
\end{align*}
so that
$$x_n = ps_n - p^{p^n-1-n} - p^{p^{n-1}+1-n} t_1^{p^{n-1}} - \cdots - p^{p-1} t_{n-1}^p.$$
This is clearly divisible by $p$ since $p>2$ (so that $p^n-1-n \geq 1$ for $n\geq 1$). Therefore, $x_n$ exists and lives in $p\Z_p$, as desired. (Note that if $p=2$ and $n=1$, then $p^n-1-n = 0$, so $x_1\not\in 2\Z_2$.) If one prefers an explicit formula, the above argument shows that once one writes $x_0 = 1-ps_0$, then $x_j = pt_j$ for $j\geq 1$ can be defined inductively by
$$t_n = 
\sum_{i=1}^{p^n} \frac{(-1)^{i+1}}{p^{n+1-i}} \binom{p^n}{i} s_0^i - p^{p^n-2-n} - \sum_{k=1}^{n-1} p^{p^k-k-1} t_{n-k}^{p^k}.
$$
The first term is $s_n$; note that $\frac{1}{p^{n+1-i}} \binom{p^n}{i}\in \Z$. Since $x_0 \equiv 1\pmod{p}$ and $x_i\equiv 0\pmod{p}$ for $i\geq 1$, it is easy to see that all the ghost components of $x$ lie in $1+p\Z_p\subseteq \Z_p^\times$; this implies that $x\in W(\Z_p)$ is invertible, as claimed.

Let us now assume that $p=2$, and show that $y[2^m]$ is in the image of $F$ for any $m\geq 2$. To see this, observe that the ghost components of $y[2^m]$ are given by 
$$w_n(y[2^m]) = w_n(y) w_n([2^m]) = 2^{m2^n}(1-2^{2^{n+1}-1}).$$
We therefore need to solve
$$2^{m2^{n-1}}(1-2^{2^n-1}) = x_0^{2^n} + 2x_1^{2^{n-1}} + \cdots + 2^n x_n$$
for some $x_0,\cdots, x_n\in \Z_2$ and all $n\geq 1$. When $n=1$, we have $x_0^2 + 2x_1 = -2^m$, so that $x_0^2 \equiv 0\pmod{2}$ since $m>0$. It follows that $x_0 = 2t_0$ for some $t_0\in \Z_2$.
We now claim that $x_n$ exists for $n\geq 0$ and lives in $2\Z_2$. We established the base case $n=0$ above, so assume $x_0,x_1, \cdots, x_{n-1}\in 2\Z_2$, and write $x_i = 2t_i$. Then
\begin{align*}
    2^n x_n & = 2^{m2^{n-1}}(1-2^{2^n-1}) - (x_0^{2^n} + 2x_1^{2^{n-1}} + \cdots + 2^{n-1} x_{n-1}^2)\\
    & = 2^{m2^{n-1}}(1-2^{2^n-1}) - (2^{2^n} t_0^{2^n} + 2^{2^{n-1}+1} t_1^{2^{n-1}} + \cdots + 2^{n+1} t_{n-1}^2),
\end{align*}
so that
$$x_n = 2^{m2^{n-1}-n}(1-2^{2^n-1}) - (2^{2^n-n} t_0^{2^n} + 2^{2^{n-1}+1-n} t_1^{2^{n-1}} + \cdots + 2 t_{n-1}^2).$$
Because $m\geq 2$ and $2^j-j\geq 1$ for every $j\geq 0$, we see that $x_n\in 2\Z_2$, as desired. (Of course, the key case is $m=2$; when $m=1$ and $n=1$, the term $2^{m2^{n-1}-n}(1-2^{2^n-1}) = -1\not\in 2\Z_2$.) If one prefers an explicit formula, note that the above argument shows that once one writes $x_0 = 2t_0$, then $x_j = 2t_j$ can be defined inductively by
$$t_n = 
2^{m2^{n-1}-n-1}(1-2^{2^n-1}) - \sum_{i=1}^n 2^{2^i - i - 1} t_{n-i}^{2^i}.
$$
Note that $x$ is \textit{not} invertible in $W(\Z_2)$; instead, since $x_j\in 2\Z_2$, the $n$th ghost component $w_n(x)\in 2^{n+1} \Z_2$.
\end{proof}

%% file: appendix/cartier-duals.tex
This section was inspired by the results proved above, but it does not play an essential role in the body of this article.
\cref{bwtimes-n} below can be viewed as an algebraic way to bookkeep the structure possessed by the topological Sen operators; and, as we hope to show in future work, it sits as an intermediary between the topological and algebraic Sen operators of \cref{topological-sen} and \cref{algebraic-sen} (see \cref{zpn-gradings}). We begin with the following (presumably well-known) result. I am (again) grateful to Sasha Petrov for a relevant discussion on it.
\begin{prop}\label{witt-cartier-dual}
There is an isomorphism of group schemes over $\Z_{(p)}$ between $W[F^n] := \ker(F^n: W \to W)$ and the Cartier dual of the completion of $W_n = W/V^n$ at the origin.
\end{prop}
\begin{proof}
Let us model $W$ by the $p$-typical big Witt vectors. Given $f(t)\in W$, let $a_0,a_1,a_{2},\cdots$ denote the ghost components of $f$, so that $td\log(f(t)) = \sum_{m\geq 0} a_{m} t^{p^m}$. Then $f(t)\in W[F^n]$ if and only if $a_{m} = 0$ for $m\geq n$. 

Let us first prove the claim of the proposition when $n=1$. Then, $d\log(f(t))$ is a constant, and $f(0) = 1$; we claim that this is equivalent to the condition that $f$ defines a homomorphism $\hat{\GG}_a \to \GG_m$, i.e., that $f(x+y) = f(x) f(y)$. To check this, first suppose that $f(x+y) = f(x) f(y)$. Then $\partial_x f(x+y) = f(y) f'(x)$, so that 
$$\frac{\partial_x f(x+y)}{f(x+y)} = \frac{f'(x)}{f(x)} = d\log(f(x))$$
is independent of $y$. Taking $x = 0$, we see that $d\log(f(x))$ is constant, as desired. The reverse direction (that $d\log(f(x))$ being constant and $f(0) = 1$ implies that $f(x+y) = f(x) f(y)$) is similar.

In the general case, note that since the Frobenius on $W$ shifts the ghost components by $F: (a_0, a_1, a_{2}, \cdots) \mapsto (a_1, a_{2}, a_{3}, \cdots)$, the Frobenius $F$ applied to $f$ satisfies:
$$d\log(F^j(f)(t)) = \sum_{m=0}^{n-j} a_{{m+j}} t^{p^m},$$
so that there is an equality of power series
$$F^j(f)(t) = \exp\left(\sum_{m=0}^{n-j} \frac{a_{{m+j}}}{p^m} t^{p^m}\right).$$
Note that this is a slight variant of the the Artin-Hasse exponential.
Define a map $g: \hat{W}_n \to \GG_m$ on Witt components $(x_0,\cdots,x_{n-1})$ (not ghost components!) as follows:
\begin{align*}
    g(x_0,\cdots,x_{n-1}) & = \prod_{j=0}^{n-1} F^j(f)(x_j) = \exp\left(\sum_{j=0}^{n-1} \sum_{m=0}^{n-j} \frac{a_{{m+j}}}{p^m} x_j^{p^m}\right) \\
    & = \exp\left(\sum_{m=0}^{n-1} \frac{a_{m}}{p^m} \left(\sum_{j=0}^{m} p^j x_j^{p^{m-j}}\right) \right).
\end{align*}
The coefficient of $\tfrac{a_{m}}{p^m}$ is precisely the $m$th Witt polynomial, so that the function $g$ is indeed additive on $\hat{W}_n$. Moreover, the assignment $f\mapsto g$ indeed gives an isomorphism $W[F^n] \xar{\sim} \Hom(\hat{W}_n, \GG_m)$, as one can check inductively using the case $n=1$ and the fact that it induces an isomorphism over $\QQ$.
\end{proof}
\begin{remark}[Integral case]
One does not need $p$-typicality for the above statement to hold. Namely, if $\WW$ denotes the big Witt ring scheme, it is a classical fact that the Cartier dual of $\WW$ over $\Z$ is canonically identified with $\hat{\WW}$. As in the $p$-typical case above, the pairing $\WW \times \hat{\WW} \to \GG_m$ sends
$$(a,b) \mapsto \exp\left(\sum_{n\geq 1} \frac{w_n(a) w_n(b)}{n}\right).$$
One only needs to check that this expression is in fact defined over $\Z$. To see this, first observe that if the Witt components of $b$ are $(b_1, b_2, \cdots)$, we have
$$\exp\left(\sum_{n\geq 1} \frac{w_n(a) w_n(b)}{n}\right) = \prod_{j\geq 1} \exp\left(\sum_{n\geq 1} \frac{w_{nj}(a) b_j^n}{n}\right).$$
Note that $w_{nj}(a) = w_n(F_j a)$, so that if $(F_j a)_d$ denote the Witt components of $F_j a$, we have
$$\exp\left(\sum_{n\geq 1} \frac{w_n(a) w_n(b)}{n}\right) = \prod_{j,d\geq 1} (1 - (F_j a)_d b_j^d),$$
giving the desired integral representation.
In fact, the last step can be generalized via the following rephrasing of the Dwork lemma:
\begin{lemma}
Let $R$ be a torsionfree ring equipped with ring maps $\phi_p: R \to R$ for each prime $p$ such that $\phi_p(r) \equiv r^p\pmod{p}$ for all $r\in R$. Let $(x_n)_{n\geq 1}$ be a sequence of elements such that $x_n \equiv \phi_p(x_{n/p})\pmod{p^{v_p(n)}}$ for each prime $p$ and every $n\in p\Z_{\geq 0}$. Then $f(t) := \exp\left(\sum_{n\geq 1} \frac{x_n t^n}{n}\right)$ lies in $1 + tR\pw{t}\subseteq 1 + t(R \otimes \QQ)\pw{t}$.
\end{lemma}
\begin{proof}
Let $g(t) = 1-t \in R\pw{t}$, so that there is an identity
$$g(t) = \exp\left(\log(1-t)\right) = \exp\left(-\sum_{n\geq 1} \frac{t^n}{n}\right).$$
Because $f(0) = 1$, we can write $f(t) = \prod_{j\geq 1} (1-r_j t^j) = \prod_{j\geq 1} g(r_j t^j)$ for unique $r_j\in R\otimes \QQ$. Since $g(t)$ is integral, it is sufficient to show that the elements $r_j$ are also integral. Applying $d\log$, we find that
$$\sum_{n\geq 1} \frac{x_n t^n}{n} = d\log(f)(t) = \sum_{j\geq 1} d\log(g)(r_j t^j) = -\sum_{j,m\geq 1} \frac{r_j^m t^{jm}}{m}.$$
It follows that $x_n = -\sum_{j|n} j r_j^{n/j}$. One can now argue in exactly the same way as the usual Dwork lemma (i.e., by induction on $r_j\in R$ for $j|n$ with $j\neq n$) to argue that each $r_j$ is integral.
\end{proof}
\end{remark}
\begin{corollary}\label{qcoh-BWn}
Write the underlying scheme of $W_n$ as $\prod_{i=0}^{n-1} \GG_a$ (where the $i$th copy of $\GG_a$ has coordinate $\Phi_i$).
There is a fully faithful functor $\QCoh(BW[F^n]) \hookrightarrow \QCoh(W_n)$ whose essential image consists of those $p$-complete $M\in \QCoh(W_n)$ such that $\Phi_i$ acts locally nilpotently on $\H^\ast(M/p)$ for each $0\leq i \leq n-1$. Furthermore, this functor is symmetric monoidal for the convolution tensor product on $\QCoh(W_n)$.

If $\cf\in \QCoh(BW[F^n])$ is sent to $M\in \QCoh(W_n)$ under this functor, one obtains a cube\footnote{Recall that $[n]$ denotes the set $\{0,\cdots,n\}$.} $\Phi_\bull: 2^{[n-1]} \to \Mod_{\Z_p}$ whose vertices are all $M$ and such that the edge from the subset $\{i_1, \cdots, i_{j-1}\}$ to $\{i_1, \cdots, i_j\}$ is given by the operator $\Phi_j: M \to M$. Then, the global sections $\Gamma(BW[F^n]; \cf)$ can be identified with the total fiber of the cube $\Phi_\bull$.
\end{corollary}
\begin{remark}
More generally, the argument of \cref{witt-cartier-dual} shows that there is an isomorphism of group schemes over $\Z_{(p)}$ between $W_m[F^n] := \ker(F^n: W_m \to W_m)$ and the Cartier dual of $W_n[F^m]$.
One can give a simpler proof of this fact over a perfect field $k$ of characteristic $p>0$ using the theory of Dieudonn\'e modules: the Dieudonn\'e module of $W_m[F^n]$ over $k$ is $W(k)[F,V]/(F^n, V^m)$, while the Dieudonn\'e module of $W_n[F^m]$ over $k$ is $W(k)[F,V]/(F^m, V^n)$.
\end{remark}
The argument of \cref{witt-cartier-dual} also shows the following (which is already discussed in \cite[Appendix D]{drinfeld-formal-group}):
\begin{prop}[{\cite[Appendix D]{drinfeld-formal-group}}]
Let $\hat{\GG}_\lambda$ be the degeneration of $\hat{\GG}_m$ to $\hat{\GG}_a$ given by $\spf \Z_{(p)}[t, \lambda, \frac{1}{1+t\lambda}]^\wedge_t$ with group law $x + y + \lambda xy$. Then the $\Z_{(p)}[\lambda]$-linear Cartier dual of $\hat{\GG}_\lambda$ is isomorphic to the group scheme $\bD(\hat{\GG}_\lambda) = \spec \Z_{(p)}[\lambda, z, \frac{\prod_{j=0}^{n-1} (z-j\lambda)}{n!}]$ over $\AA^1_\lambda = \spec \Z_{(p)}[\lambda]$ with coproduct $z\mapsto z\otimes 1 + 1\otimes z$.
\end{prop}
\begin{proof}
A homomorphism $f: \hat{\GG}_\lambda \to \GG_m \times \AA^1_\lambda$ is an element of $\Z_{(p)}[t, \lambda, \frac{1}{1+t\lambda}]^\wedge_t$ such that 
$$f(x+y+\lambda xy) = f(x) f(y).$$
This condition implies that 
$$(1+\lambda y) f'(x+y+\lambda xy) = f(y) f'(x),$$
so that dividing both sides by $f(x+y+\lambda xy)$, we have
$$(1+\lambda y) \cdot d\log(f)(x+y+\lambda xy) = d\log(f)(x).$$
Taking $x = 0$, we see that $d\log(f)(y)$ is a constant multiple of $\frac{1}{1+\lambda y}$ (where the constant is given by $\frac{f'(0)}{f(0)}$), and hence 
$$f(y) = (1+\lambda y)^{z/\lambda} = \sum_{n\geq 0} y^n \frac{\prod_{j=0}^{n-1} (z-j\lambda)}{n!}$$
for some fixed $z$. This gives the desired claim, similarly to \cref{witt-cartier-dual}. Note that $f(y) = \exp(z\log_F(y))$, where $\log_F$ is the logarithm of the formal group law $x + y + \lambda xy$ over $\AA^1_\lambda$.
\end{proof}
\begin{remark}
Observe that $\bD(\hat{\GG}_\lambda)$ is isomorphic to the subgroup $(W\times \AA^1_\lambda)[F+[-\lambda]^{p-1}]$ of $W \times \AA^1_\lambda$ cut out by $\{x | Fx = [-\lambda]^{p-1} x\}$; see \cite[Proposition 6.3.3]{toen-hkr} and \cite[Proposition D.4.10]{drinfeld-formal-group}. The key point is that if $f(x)\in W$ and $xd\log(f(x)) = \sum_{m\geq 1} a_{m} x^{m}$, then $f(t)\in (W\times \AA^1_\lambda)[F+[-\lambda]^{p-1}]$ if and only if 
$$a_{p^{n+1}} = ((-\lambda)^{p-1})^{p^n} a_{p^n} = (-\lambda)^{p^{n+1} - p^n} a_{p^n}.$$
To check this, note that 
$$xd\log(f)(x) = \frac{zx}{1+\lambda x} = \sum_{n\geq 0} (-\lambda)^n z x^{n+1},$$
so that $a_m = (-\lambda)^{m-1} z$, and $a_m = (-\lambda)^{m-n} a_n$ if $m\geq n$.
\end{remark}
\begin{remark}
A similar argument shows that if $\GG$ denotes the group scheme $\spec \Z/p^N[\lambda]\pdb{x}$ with group law $x + y + \lambda xy$ (so that when $\lambda = 0$, we get $\GG_a^\sharp$), then the $\Z/p^N[\lambda]$-linear Cartier dual of $\GG$ is isomorphic to the completion of $\spec \Z/p^N[\lambda, z]$ at the locus $\prod_{j=0}^{p-1} (z-j\lambda) = z(z^{p-1} - \lambda^{p-1})$ (see, e.g., \cite[Section B.4]{drinfeld-formal-group}). It follows that the $\infty$-category of $\GG$-representations is equivalent to the $\infty$-category of $\Z/p^N$-modules $M$ equipped with an operator $z: M \to M$ such that $z(z^{p-1} - \lambda^{p-1})$ acts locally nilpotently on $\H^\ast(M \otimes_{\Z/p^N} \FF_p)$.
\end{remark}
\begin{recall}
In \cite[Lemma 3.5.18]{apc}, Bhatt and Lurie show that the following is a Cartesian square of group schemes over $\Z/p^k$:
\begin{equation}\label{gm-sharp-log-square}
    \xymatrix{
    \GG_m^\sharp \ar[r]^-{\log} \ar[d] & \GG_a^\sharp \ar[d]^-{x\mapsto \exp(px)} \\
    \GG_m \ar[r]_-{x\mapsto x^p} & \GG_m^{(1)}.
    }
\end{equation}
We will generalize this below in \cref{wcross-cartier-dual}.
In \cite{generalized-n-series}, we prove another generalization of this square, albeit in a different direction: $\GG_a^\sharp$ is replaced by the Cartier dual of a formal group $\hat{\GG}$, and $\GG_m^\sharp$ is replaced by an appropriate $\hat{\GG}$-analogue of the divided power completion.
%
%
\end{recall}
\begin{corollary}\label{wcross-cartier-dual}
Let $k\geq 0$. There is an isomorphism of group schemes over $\Z/p^k$ between the Cartier dual of $W^\times[F^n] := \ker(F^n: W^\times \to W^\times)$ and the completion of $W_n$ at its $\FF_p$-rational points $W_n(\FF_p) \cong \Z/p^n$.
\end{corollary}
\begin{proof}
Following \cite[Remark 3.5.17]{apc}, it suffices to prove the following analogue of \cite[Lemma 3.5.18]{apc}: there is a Cartesian diagram of flat group schemes over $\Z/p^k$ given by
\begin{equation}\label{log-square}
    \xymatrix{
    W^\times[F^n] \ar[r]^-{\log} \ar[d] & W[F^n] \ar[d]^-{x\mapsto \exp(p^n x)} \\
    \GG_m \ar[r]_-{x\mapsto x^{p^n}} & \GG_m^{(n)}.
    }
\end{equation}
Here, the left vertical map $W^\times[F^n] \to \GG_m$ is the composite
$$W^\times[F^n] \to W^\times \to W^\times/V \cong \GG_m.$$
Indeed, taking the Cartier dual of \cref{log-square} and using \cref{witt-cartier-dual}, we obtain a pushout diagram of formal group schemes
$$\xymatrix{
p^n \Z \ar[r] \ar[d] & \Z \ar[d] \\
\hat{W}_n \ar[r] & \bD(W^\times[F^n]).
}$$
This implies that $\bD(W^\times[F^n])$ is the completion of $W_n$ at its $\FF_p$-rational points $W_n(\FF_p) \cong \Z/p^n$, as desired.

The proof that the square \cref{log-square} is Cartesian is in fact a consequence of \cite[Lemma 3.5.18]{apc}. As in \cite[Lemma 3.5.18]{apc}, since all group schemes involved are flat over $\Z/p^k$, it suffices to prove that the diagram is Cartesian after base-changing to $\FF_p$ (i.e., assume that $k=1$). We begin by noting that there is an isomorphism $V: W^\times \xar{\sim} \{x\in W | x_0 = 0\}$ of $p$-adic formal schemes sending $x\mapsto Vx$. This gives an isomorphism $W[F^n] \times \mu_{p^n} \cong W^\times[F^n]$ of group schemes over $\FF_p$, sending $(x, a) \mapsto [a] + Vx$. It therefore suffices to show that the composite
$$W[F^n] \xar{x\mapsto 1 + Vx} W^\times[F^n] \xar{\log} W[F^n]$$
is an isomorphism. But there is a commutative diagram
$$\xymatrix{
W[F^{n-1}] \ar[r]_-{x\mapsto 1 + Vx} \ar[d] & W^\times[F^{n-1}] \ar[r]^-{\log} \ar[d] & W[F^{n-1}] \ar[d] \\
W[F^{n}] \ar[r]^-{x\mapsto 1 + Vx} \ar[d] & W^\times[F^{n}] \ar[r]^-{\log} \ar[d] & W[F^{n}] \ar[d] \\
\GG_a^\sharp \ar[r]^-{x\mapsto 1 + Vx} & \GG_m^\sharp \ar[r]^-{\log} & \GG_a^\sharp,
}$$
where the columns exhibit the third term as the quotient of the first two. By induction on $n$ (with the base case being provided by \cite[Lemma 3.5.18]{apc}), we may conclude that the middle horizontal composite is an isomorphism.
\end{proof}
\begin{remark}
One could have alternatively/equivalently proved \cref{wcross-cartier-dual} by observing that the square \cref{log-square} for $n-1$ maps to \cref{log-square} for $n$; all components of this map of squares are the canonical ones, except on the bottom-right $\GG_m$ (where it is given by the $p$th power map $\GG_m \to \GG_m^{(1)}$). Diagramatically:
\begin{equation}
\begin{tikzcd}[row sep=2.5em]
W^\times[F^{n-1}] \arrow[rr,"\log"] \arrow[dr,swap] \arrow[dd,swap] &&
  W[F^{n-1}] \arrow[dd,swap, "x\mapsto \exp(p^{n-1} x)" near start] \arrow[dr] \\
& W^\times[F^n] \arrow[rr,crossing over,swap,"\log" near start] &&
  W[F^n] \arrow[dd,"x\mapsto \exp(p^n x)"] \\
\GG_m \arrow[rr,"x\mapsto x^{p^{n-1}}" near end] \arrow[dr,swap,"\id"] && \GG_m^{(n-1)} \arrow[dr,"x\mapsto x^p"] \\
& \GG_m \arrow[rr,"x\mapsto x^{p^n}"] \arrow[uu,<-,crossing over] && \GG_m^{(n)}.
\end{tikzcd}
\end{equation}
Taking Cartier duals, we obtain a map of pushout squares:
\begin{equation}
\begin{tikzcd}[row sep=2.5em]
p^n \Z \arrow[rr] \arrow[dr,swap] \arrow[dd,swap] &&
  \hat{W}_n \arrow[dd,swap] \arrow[dr] \\
& \Z \arrow[rr,crossing over] &&
  \bD(W^\times[F^n]) \arrow[dd] \\
p^{n-1} \Z \arrow[rr] \arrow[dr,swap] && \hat{W}_{n-1} \arrow[dr] \\
& \Z \arrow[rr] \arrow[uu,<-,crossing over] && \bD(W^\times[F^{n-1}]).
\end{tikzcd}
\end{equation}
Again, \cref{wcross-cartier-dual} follows by induction on $n$, using \cite[Lemma 3.5.18]{apc} for the base case.
\end{remark}
\begin{corollary}\label{bwtimes-n}
Write the underlying scheme of $W_n$ as $\prod_{i=0}^{n-1} \GG_a$ (where the $i$th copy of $\GG_a$ has coordinate $\Psi_i$).
There is a fully faithful functor $\QCoh(BW^\times[F^n]) \hookrightarrow \QCoh(W_n)$ whose essential image consists of those $p$-complete $M\in \QCoh(W_n)$ such that $\Psi_i^p-\Psi_i$ acts locally nilpotently on $\H^\ast(M/p)$ for each $0\leq i \leq n-1$. Furthermore, this functor is symmetric monoidal for the convolution tensor product on $\QCoh(W_n)$.

If $\cf\in \QCoh(BW^\times[F^n])$ is sent to $M\in \QCoh(W_n)$ under this functor, one obtains a cube $\Psi_\bull: 2^{[n-1]} \to \Mod_{\Z_p}$ whose vertices are all $M$ and such that the edge from the subset $\{i_1, \cdots, i_{j-1}\}$ to $\{i_1, \cdots, i_j\}$ is given by the operator $\Psi_j: M \to M$. Then, the global sections $\Gamma(BW^\times[F^n]; \cf)$ can be identified with the total fiber of the cube $\Psi_\bull$.
\end{corollary}
\begin{example}
If $\cf, \cg\in \QCoh(BW^\times[F^2])$ correspond to tuples $(M, \Psi_0^M, \Psi_1^M)$ and $(M', \Psi_0^{M'}, \Psi_1^{M'})$, then the global sections $\Gamma(BW^\times[F^2]; \cf)$ can be identified with the total fiber of the square
$$\xymatrix{
M \ar[r]^-{\Psi_0^M} \ar[d]_-{\Psi_1^M} & M \ar[d]^-{\Psi_1^M} \\
M \ar[r]_-{\Psi_0^M} & M.
}$$
Moreover, $\cf\otimes \cg$ corresponds to the module $M \otimes M'$, where
\begin{align*}
    \Psi_0^{M \otimes M'} & = \Psi_0^{M}\otimes 1 + 1 \otimes \Psi_0^{M'},\\
    \Psi_1^{M \otimes M'} & = \Psi_1^{M'} \otimes 1 + 1\otimes \Psi_1^{M'} - \frac{1}{p}\sum_{i=1}^{p-1} \binom{p}{i} (\Psi_0^M)^i \otimes (\Psi_0^{M'})^{p-i}.
\end{align*}
More generally, if $\cf, \cg\in \QCoh(BW^\times[F^n])$ correspond to tuples $(M, \Psi_0^M, \cdots, \Psi_{n-1}^M)$ and $(M', \Psi_0^{M'}, \cdots, \Psi_{n-1}^{M'})$, let us write $\Psi := (\Psi_0, \Psi_0, \cdots)$. Let $w_j(\Psi) = \sum_{i=0}^j p^i \Psi_{i}^{p^{j-i}}$ denote the corresponding Witt polynomial; then 
\begin{equation}\label{wj-tensor}
    w_j(\Psi^{M \otimes M'}) = w_j(\Psi^M) \otimes 1 + 1\otimes w_j(\Psi^{M'}).
\end{equation}
\end{example}
\begin{prop}
Define a homomorphism $W^\times[F^n] \to \GG_m$ via the composite
$$W^\times[F^n] \to W^\times \to (W/V)^\times \cong \GG_m.$$
Let $\co\{1\}$ denote the line bundle over $BW^\times[F^n]$ determined by the resulting map $BW^\times[F^n] \to B\GG_m$. Under the functor of \cref{bwtimes-n}, the total space of the line bundle $\co\{1\}$ corresponds to the $p$-completion of $\Z_p[x^{\pm 1}]$ with the action of $\Psi_{j}$ determined by the following requirement on Witt polynomials:
\begin{equation}\label{witt-poly-psi}
    w(\Psi) = (w_0(\Psi), w_1(\Psi), w_2(\Psi), \cdots) = (x\partial_x, x\partial_x, x\partial_x, \cdots).
\end{equation}
\end{prop}
\begin{proof}
The map $BW^\times[F^n] \to B\GG_m$ determines the stack $\GG_m/W^\times[F^n]$ over $BW^\times[F^n]$, so that the corresponding object in $\QCoh(W_n)$ under the functor of \cref{bwtimes-n} has underlying module given by $\co_{\GG_m} = \Z[x^{\pm 1}]$.
It is not too hard to show from the definition of the map $BW^\times[F^n] \to B\GG_m$ that under the functor of \cref{bwtimes-n}, the line bundle $\co\{1\}$ over $BW^\times[F^n]$ corresponds to the $p$-complete module $\Z_p$ (with generator $x$) where $\Psi_0$ acts on $x$ by $1$, and $\Psi_j$ acts on $x$ by zero for $j\geq 1$. The action of $w_j(\Psi)$ on $\co\{m\} = \Z_p \cdot x^m$ then follows from \cref{wj-tensor}.
\end{proof}
\begin{example}
For instance, it follows from \cref{witt-poly-psi} that
\begin{align*}
    \Psi_0 & = x\partial_x, \\
    \Psi_1 & = \frac{x\partial_x}{p}\left(1 - (x\partial_x)^{p-1}\right),\\
    \Psi_2
    & = \frac{x\partial_x}{p^2}\left(1 - (x\partial_x)^{p^2-1} - \frac{1}{p^{p-1}} \sum_{j=0}^p (-1)^{j} \binom{p}{j} (x\partial_x)^{(p-1)(j+1)} \right).
\end{align*}
\end{example}
\begin{remark}
Using \cref{qcoh-BWn}, a similar calculation can be used to describe the $\GG_a$-bundle $\GG_a/W[F^n]$ over $BW[F^n]$; and, in particular, the $\infty$-category $\QCoh((\GG_a/W[F^n])/\GG_m)$. Let us summarize this calculation as follows. Recall that $\GG_a/W[F] \cong \GG_a/\GG_a^\sharp$ is isomorphic to $\GG_a^\dR$, so that if $\cA_1 := \Z_p\{x,\partial_x\}/([\partial_x,x] = 1)$ is the Weyl algebra, then $\QCoh(\GG_a/W[F]) \simeq \LMod_{\cA_1}^{\partial_x\text{-nilp}}$. Similarly, if we write $\frac{1}{m!} \partial_x^m = \partial_x^{[m]}$ (so that $\partial_x^{[m]}(x^k) = \binom{k}{m} x^{k-m}$), define $\cA_1^{[n]}$ via
$$\cA_1^{[n]} = \Z_p\left\{x, \partial_x, \cdots, \partial_x^{[p^{n-1}]} \right\}/([\partial_x^{[p^j]}, x] = \partial_x^{[p^j-1]}).$$
Note that $\partial_x^{[p^j-1]}$ is a $p$-adic unit multiple of $\prod_{k=0}^{j-1} (\partial_x^{[p^k]})^{p-1}$. Then, the action of $W[F^n]$ on $\GG_a$ implies that there is an equivalence 
$$\QCoh(\GG_a/W[F^n]) \simeq \LMod_{\cA_1^{[n]}}^{\partial_x^{[p^j]}\text{-}\mathrm{nilp}};$$
this can be extended to an equivalence between $\QCoh((\GG_a/W[F^n])/\GG_m)$ and the $\infty$-category of graded $\cA_1^{[n]}$-modules such that $\partial_x^{[p^j]}$ acts nilpotently for $0\leq j \leq n-1$, where $x\in \cA_1^{[n]}$ has weight $1$ and $\partial_x^{[p^j]}\in \cA_1^{[n]}$ has weight $-p^j$. Algebras of divided power differential operators such as $\cA_1^{[n]}$ were initially studied by Berthelot in \cite{berthelot-dmod}.
\end{remark}
\begin{warning}
Note that $\GG_a/W[F^n]$ is not a ring stack. Indeed, the map $W[F^n] \to \GG_a$ is not a quasi-ideal: the $W$-module structure on $W[F^n]$ does not factor through $W \twoheadrightarrow W_1 = \GG_a$ (if $F^n(x) = 0$, then $xV(y) = V(F(x)y)$ need not vanish). However, it \textit{does} factor through $W \twoheadrightarrow W_n$ (if $F^n(x) = 0$, then $xV^n(y) = V^n(F^n(x)y) = 0$); indeed, $W_n/W[F^n] \cong W/p^n$ admits the structure of a ring stack.
\end{warning}
\begin{remark}\label{higher DI}
The proof of \cref{wcross-cartier-dual} showed that there is an isomorphism
$$W^\times[F^n] \cong W[F^n] \times \mu_{p^n}$$
over $\FF_p$. Let $\fr{X}$ be a smooth $p$-adic formal scheme over $\Z_p$, and let $X = \fr{X} \otimes_{\Z_p} \FF_p$. Suppose that the $\GG_m^\sharp$-action on $\diffr_{\fr{X}} \otimes_{\Z_p} \FF_p = F_{X,\ast} \Omega^\bull_{X/\FF_p}$ refines to a $W^\times[F^n]$-action. (For instance, let $(\diffr_{\fr{X}})_0$ denote the weight $0$ piece of the $\Z/p$-grading on $\diffr_{\fr{X}}$ inherited from the $\GG_m^\sharp$-action. The datum of a refinement to a $W^\times[F^2]$-action leads to an operator on $(\diffr_{\fr{X}})_0$ which acts on $\gr^{pi}_\conj \diffr_{\fr{X}}$ by multiplication by $-i$.)
In this case, the $\Z/p$-grading on $F_{X,\ast} \Omega^\bull_{X/\FF_p}$ from \cite[Remark 4.7.20]{apc} would refine to a $\Z/p^n$-grading; this would imply a refinement of the Deligne-Illusie theorem \cite{deligne-illusie}, stating that $\tau_{\geq -p^n+1} F_{X,\ast} \Omega^\bull_{X/\FF_p}$ would be decomposable. 
\end{remark}
\begin{remark}[``Witty'' interpretation of \cite{higher-dim-degen}]\label{zpn-gradings}
The work of \cite{lee-thh} suggests
that the base-change along $\BP{n-1}_\ast \to \FF_p$ (even along $\BP{n-1}_\ast \to \Z_p$) of the stack constructed from the associated graded of the motivic filtration \cite{even-filtr} on $\THH(\BP{n-1})^{t\Cp}$ (resp. $\THH(\BP{n-1})$) is isomorphic to the stack $(\GG_m/W^\times[F^n])/\GG_m \cong BW^\times[F^n]$ (resp. $(\GG_a/W[F^n])/\GG_m \cong (F_\ast W/pF^{n-1})/\GG_m$). We are currently investigating this and its consequences with Jeremy Hahn and Arpon Raksit. In particular, this suggests that if a $\Z_p$-scheme ``lifts to $\BP{n-1}$'', the $\GG_m^\sharp$-action on $\diffr_{\fr{X}}$ refines to a $W^\times[F^n]$-action. From this perspective, the operators $\Psi_j$ from above are closely related to the topological Sen operators $\Theta_j$ from the body of this article: roughly, $\Theta_j$ can be understood as $w_{j-1}(\Psi)$.

Given \cref{higher DI}, one is therefore naturally led to the following question: if $X$ is a smooth and proper $\FF_p$-scheme which ``lifts to $\BP{n-1}$'' and $\dim(X)<p^n$, does the Hodge-de Rham spectral sequence for $X$ degenerate at the $E_1$-page? This question need not make sense, since $\BP{n-1}$ is generally not an $\Eoo$-ring \cite{lawson-bp, senger-bp}.
However, since $\BP{n-1}$ admits the structure of an $\E{3}$-ring, one can nevertheless ask whether such a degeneration statement holds noncommutatively if $\QCoh(X)$ admits a lift to a left $\BP{n-1}$-linear $\infty$-category. 
This line of thinking was motivation for the following result (see \cite{higher-dim-degen}): if $\QCoh(X)$ lifts to a left $\BP{n-1}$-linear $\infty$-category, and $\dim(X)<p^n$, then the Tate spectral sequence for $\HP(X/\FF_p)$ degenerates at the $E_2$-page.
\end{remark}